\documentclass{amsbook}

\usepackage[english]{babel}
\usepackage{enumerate}
\usepackage{amssymb}
\usepackage{mathrsfs}
\usepackage{esint}
\usepackage{color}
\usepackage{amsthm}

\newtheorem{thm}{Theorem}[section]
\newtheorem{lem}[thm]{Lemma}
\newtheorem{hyp}[thm]{Hypothesis}
\newtheorem{cor}[thm]{Corollary}
\newtheorem{cond}[thm]{Condition}
\newtheorem{defi}[thm]{Definition}
\newtheorem{prop}[thm]{Proposition}
\newtheorem{ex}[thm]{Example}

\newtheorem{rem}[thm]{Remark}
\newtheorem*{prop*}{Proposition}
\newtheorem*{thm*}{Theorem}

\numberwithin{section}{chapter}
\numberwithin{equation}{chapter}

\DeclareMathOperator*{\esssup}{ess\,sup}

\begin{document}

\frontmatter

\title{The Connes character formula for locally compact spectral triples}

\author[Fedor Sukochev]{Fedor Sukochev}
\address{University of New South Wales, Kensington, NSW, 2052, Australia}
\email{f.sukochev@unsw.edu.au}
\author[Dmitriy Zanin]{Dmitriy Zanin}
\address{University of New South Wales, Kensington, NSW, 2052, Australia}
\email{d.zanin@unsw.edu.au}

\date{}



\dedicatory{Dedicated to Alain Connes for his 70-th anniversary. With gratitude and admiration.}

\begin{abstract} 
A fundamental tool in noncommutative geometry is Connes' character formula. This formula is used in an essential way in the applications of noncommutative geometry to index theory and to the spectral characterisation of manifolds. 
    
A non-compact space is modelled in noncommutative geometry by a non-unital spectral triple. Our aim is to establish the Connes character formula for non-unital spectral triples. This is significantly more difficult than in the unital case and we achieve it with the use of recently developed double operator integration techniques. Previously, only partial extensions of Connes' character formula to the non-unital case were known.
    
In the course of the proof, we establish two more results of importance in noncommutative geometry: an asymptotic for the heat semigroup of a non-unital spectral triple, and the analyticity of the associated $\zeta$-function.
    
We require certain assumptions on the underlying spectral triple, and we verify these assumptions in the case of spectral triples associated to arbitrary complete Riemannian manifolds and also in the case of Moyal planes.

\end{abstract}

\maketitle
\tableofcontents

\mainmatter

\chapter{Introduction}\label{intro chapter}

\section*{Acknowledgements}

    The authors thank Professor Alain Connes for encouragement and numerous useful comments, Dr Denis Potapov whose ideas eventually led us to Theorem \ref{csz key lemma}, 
    Dr Alexei Ber, Dr Galina Levitina and Mr Edward Mcdonald for their substantial effort in improving our initial arguments. Additionally we thank Dr Victor Gayral and Professor Yurii Kordyukov for their help with Section \ref{manifold section} and Professor Peter Dodds for his assistance with the arguments of Subsection \ref{peter subsection}. We also would like to thank Edward McDonald for editing the manuscript.
    
    We would like to extend our thanks and appreciation to Professor Nigel Higson, as it was due to his encouragement that we initiated this project.

\section{Introduction}

    One of the fundamental tools in noncommutative geometry
    is the Chern character. The Connes Character Formula (also known as the Hochschild character theorem) provides an expression for the 
    class of the Chern character in Hochschild cohomology, and it is an important
    tool in the computation of the Chern character. The formula has been applied to many areas
    of noncommutative geometry and its applications: such as the local index formula \cite{Connes-Moscovici}, the spectral characterisation of manifolds \cite{Connes-reconstruction} and recent work in mathematical physics \cite{Connes-Chamseddine-Mukhanov-quanta-of-geometry-2015}.

    In its original formulation, \cite{Connes-original-spectral-1995}, the Character Formula is stated as follows: Let $(\mathcal{A},H,D)$ be a $p$-summable compact spectral triple
    with (possibly trivial) grading $\Gamma$ (as defined in Section \ref{spectral triple subsection}).
    By the definition of a spectral triple, for all $a \in \mathcal{A}$ the commutator $[D,a]$ has an extension to a bounded operator $\partial(a)$ on $H$. Furthermore, if $F = \chi_{(0,\infty)}(D)-\chi_{(-\infty,0)}(D)$
    then for all $a \in \mathcal{A}$ the commutator $[F,a]$ is a compact operator in the weak Schatten ideal $\mathcal{L}_{p,\infty}$. 
    For simplicity assume that $\ker(D) = \{0\}$, and now consider the following two linear maps on the algebraic tensor power $\mathcal{A}^{\otimes(p+1)}$,
    defined on an elementary tensor $c = a_0\otimes a_1\otimes \cdots \otimes a_p \in \mathcal{A}^{\otimes(p+1)}$ by,
    \begin{equation*}
        \mathrm{Ch}(c) := \frac{1}{2}\mathrm{Tr}(\Gamma F[F,a_0][F,a_1]\cdots[F,a_p])
    \end{equation*} 
    and
    \begin{equation*}
        \Omega(c) := \Gamma a_0\partial a_1\partial a_2\cdots \partial a_p.
    \end{equation*}
    Then the Connes Character Formula states that if $c$ is a Hochschild cycle (as defined in Section \ref{hochschild subsection}) then
    \begin{equation*}
        \mathrm{Tr}_\omega(\Omega(c)(1+D^2)^{-p/2}) = \mathrm{Ch}(c)
    \end{equation*}
    for every Dixmier trace $\mathrm{Tr}_\omega$. In other words, the multilinear maps $\mathrm{Ch}$ and $c \mapsto \mathrm{Tr}_\omega(\Omega(c)(1+D^2)^{-p/2})$ define
    the same class in Hochschild cohomology.
    
%
    
    
    There has been great interest in generalising the tools and results of noncommutative geometry to the \lq\lq non-compact\rq\rq (i.e., non-unital) setting. {  The definition
    of a spectral triple associated to a non-unital algebra originates with Connes \cite{Connes-reality}, was furthered by the work of Rennie \cite{Rennie-2003, Rennie-2004}
    and Gayral, Gracia-Bond\'ia, Iochum, Sch\"ucker and Varilly \cite{gayral-moyal}. Earlier, similar ideas appeared in the work of Baaj and Julg \cite{Baaj-Julg}. Additional contributions to this area were made by Carey,
    Gayral, Rennie, and the first named author \cite{CGRS1,CGRS2}. The conventional definition 
    of a non-compact spectral triple is to replace the condition that $(1+D^2)^{-1/2}$ be compact with the assumption that for all $a \in \mathcal{A}$ the operator $a(1+D^2)^{-1/2}$ is compact.     
    }
    
    This raises an important question: is the Connes Character Formula true for locally compact spectral triples? 
    
    In this paper we are able to provide an affirmative answer to this question, provided that one assumes certain regularity properties on the spectral triple.
    There is a substantial difference between the theories of compact and non-compact spectral triples, in particular issues pertaining to summability are 
    more subtle. We achieve our proof of the non-unital Character Formula using recently developed techniques of operator integration.


\section{The main results}

    In this paper we prove three key theorems (Theorems \ref{heat thm}, \ref{zeta thm} and \ref{main thm}) and a new result concerning universal measurability (Theorem \ref{zeta measurability theorem}). 

    Essential to our approach is a certain set of assumptions on a spectral triple to be outlined below. The notion of a spectral triple, and all of the corresponding notations are explained fully in Section \ref{spectral triple subsection}. By definition, if $(\mathcal{A},H,D)$ is a spectral triple then for $a \in \mathcal{A}$, the notation $\partial(a)$ denotes the bounded
    extension of the commutator $[D,a]$, and for an operator $T$ on $H$ which preserves the domain of $D$, $\delta(T)$ denotes the bounded extension of $[|D|,T]$ when it exists.
    The notation $\mathcal{L}_{r,\infty}$, $r \geq 1$, denotes the ideal of compact operators $T$ whose singular value sequence $\{\mu(n,T)\}_{n=0}^\infty$ satisfies $\mu(n,T) = O(n^{-1/r})$.
    The norm $\|\cdot\|_1$ is the trace-class norm. 
    
    Our main assumption on $(\mathcal{A},H,D)$ is as follows:
    \begin{hyp}\label{main assumption} 
        The spectral triple $(\mathcal{A},H,D)$ satisfies the following conditions:
        \begin{enumerate}[{\rm (i)}]
            \item\label{ass0} $(\mathcal{A},H,D)$ is a smooth spectral triple.
            \item\label{ass1} There exists $p \in \mathbb{N}$ such that $(\mathcal{A},H,D)$ is $p-$dimensional, i.e., for every $a\in\mathcal{A},$
                \begin{align*}
                              a(D+i)^{-p} &\in\mathcal{L}_{1,\infty},\\
                    \partial(a)(D+i)^{-p} &\in\mathcal{L}_{1,\infty}.
                \end{align*}
            \item\label{ass2} for every $a\in\mathcal{A}$ and for all $k\geq0,$ we have
                \begin{align*}
                              \Big\|\delta^k(a)(D+i\lambda)^{-p-1}\Big\|_1 &= O(\lambda^{-1}),\quad\lambda\to\infty,\\
                    \Big\|\delta^k(\partial(a))(D+i\lambda)^{-p-1}\Big\|_1 &= O(\lambda^{-1}),\quad\lambda\to\infty.
                \end{align*}
        \end{enumerate}
    \end{hyp}
    Condition \ref{main assumption}.\eqref{ass0} is well known and widely used in the literature. The notion of ``smoothness" that we use here
    is identical to what is sometimes referred to as $QC^\infty$ (see Definition \ref{smoothness definition}).
    
    Condition \ref{main assumption}.\eqref{ass1} is also widely used, but we caution the reader that elsewhere in the literature an alternative definition
    of dimension is often used: where $(\mathcal{A},H,D)$ is said to be $p$-dimensional if for all $a \in \mathcal{A}$ we have $a(D+i)^{-1} \in \mathcal{L}_{p,\infty}$
    and $\partial(a)(D+i)^{-1} \in \mathcal{L}_{p,\infty}$. {  The definition of dimension in \ref{main assumption}.\eqref{ass1} is strictly stronger, and we discuss this issue in \ref{dimension discussion}.}
    
    Condition \ref{main assumption}.\eqref{ass2} is new and specific to the locally compact situation.
    Indeed, if $\mathcal{A}$ is unital then \ref{main assumption}.\eqref{ass2} is redundant, as it follows from \ref{main assumption}.\eqref{ass1}. 
    
    In order to show that Hypothesis \ref{main assumption}.\eqref{ass2} is reasonable, we prove that it is satisfied for spectral triples associated to the following two classes of examples:
    \begin{enumerate}[{\rm (i)}]
        \item{} Noncommutative Euclidean spaces, a.k.a. Moyal spaces. (Section \ref{nc section})
        \item{} Complete Riemannian manifolds. (Section \ref{manifold section}).
    \end{enumerate}
    
    In deciding on the conditions of Hypothesis \ref{main assumption}, we have avoided the assumption that the spectral dimension of $(\mathcal{A},H,D)$ is isolated: this is an assumption
    made in \cite{higson}, \cite{Connes-Moscovici} and in some parts of \cite{CGRS2}.
    
    Our first main result is established in Section \ref{heat section}. This result provides an asymptotic estimate of the trace of the heat operator $s \mapsto e^{-s^2D^2}$,
    and we remark that the following theorem is new even in the compact case.
    \begin{thm}\label{heat thm} 
        Let $p\in\mathbb{N}$ and let $(\mathcal{A},H,D)$ be a spectral triple satisfying Hypothesis \ref{main assumption}. If $c\in\mathcal{A}^{\otimes (p+1)}$ is a Hochschild cycle, then
        \begin{equation}\label{heat eq}
            {\rm Tr}(\Omega(c)(1+D^2)^{1-\frac{p}{2}}e^{-s^2D^2})=\frac{p}{2}{\rm Ch}(c)s^{-2}+O(s^{-1}),\quad s\downarrow0.
        \end{equation}
    \end{thm}
    
    Note that we do not require that the parity of the dimension of $p$ match the parity of the spectral triple (i.e., $p$ can be an odd integer while $(\mathcal{A},H,D)$ has a nontrivial grading, and similarly
    $p$ can be even while $(\mathcal{A},H,D)$ has no grading).

    Our second main result proves the the analytic continuation of the $\zeta-$function associated with the operator $(1+D^2)^{-\frac12}.$ 
    This result recovers all previous results concerning the residue of the $\zeta$ function on a Hochschild cycle.
    \begin{thm}\label{zeta thm} 
        Let $p\in\mathbb{N}$ and let $(\mathcal{A},H,D)$ be a spectral triple satisfying Hypothesis \ref{main assumption}. If $c\in\mathcal{A}^{\otimes (p+1)}$ is a Hochschild cycle, then the function
        \begin{equation}\label{zeta eq}
           \zeta_{c,D}(z) := {\rm Tr}(\Omega(c)(1+D^2)^{-\frac{z}{2}}),\quad\Re(z)>p,
        \end{equation}
        is holomorphic, and has analytic continuation to the set $\{\Re(z)>p-1\}\setminus\{p\}.$ The point $z = p$ is a simple pole of the analytic continuation of $\zeta_{c,D}$, with corresponding residue equal to $p\mathrm{Ch}(c)$.
    \end{thm}
    
    To prove our analogue of the Character Theorem in the unital setting, we require an additional \emph{locality} assumption on the Hochschild cycle $c$. The use of locality in noncommutative geometry was pioneered by Rennie in \cite{Rennie-2004}.
    \begin{defi}
        A Hochschild cycle $c=\sum_{j=1}^m a_0^j\otimes\cdots\otimes a_p^j \in \mathcal{A}^{\otimes(p+1)}$ is said to be local if there exists a positive element $\phi\in\mathcal{A}$ such that $\phi a_0^j=a_0^j$ for all $1\leq j\leq m$.
    \end{defi}
    For example, if $X$ is a manifold and $\mathcal{A} = C^\infty_c(X)$ is the algebra of smooth compactly supported functions on $X$, then every Hochschild cycle is local since we may choose $\phi$ to be smooth and equal to $1$ on the union
    of the supports of $\{a_0^j\}_{j=1}^m$.
    
    Our final result is the Connes Character Formula for locally compact spectral triples. In the compact case, our result recovers all previous results of this type 
    (e.g. \cite[Theorem 10.32]{GVF}, \cite[Theorem 6]{BF}, \cite[Theorem 10]{CPRS1} and \cite[Theorem 16]{CRSZ}).
    \begin{thm}\label{main thm} 
        Let $p\in\mathbb{N}$ and let $(\mathcal{A},H,D)$ be a spectral triple satisfying Hypothesis \ref{main assumption}. If $c\in\mathcal{A}^{\otimes (p+1)}$ is a local Hochschild cycle, then
        \begin{equation}\label{super main eq}
            \varphi(\Omega(c)(1+D^2)^{-\frac{p}{2}})={\rm Ch}(c).
        \end{equation}
        for every normalised trace $\varphi$ on $\mathcal{L}_{1,\infty}$.
    \end{thm}
    The notion of a normalised trace on $\mathcal{L}_{1,\infty}$ is recalled in Subsection \ref{trace subsection}. 
    { The purpose of the Connes Character Formula is to compute the Hochschild class of the Chern character by a ``local" formula,
    here stated in terms of singular traces.}
    
    { A consequence of Theorem \ref{main thm} being stated for arbitrary normalised traces on $\mathcal{L}_{1,\infty}$ is that we can deduce precise behaviour of the distribution
    of eigenvalues of the operator $\Omega(c)(1+D^2)^{-p/2}$:
    \begin{cor}
        Let $(\mathcal{A},H,D)$ satisfy Hypothesis \ref{main assumption}, and let $c \in \mathcal{A}^{\otimes (p+1)}$ be a local Hochschild cycle. Then the sequence
        $\{\lambda(k,\Omega(c)(1+D^2)^{-p/2})\}_{k=0}^\infty$ of eigenvalues of the operator $\Omega(c)(1+D^2)^{-p/2}$ arranged in non-increasing absolute value satisfies:
        \begin{equation*}
            \sum_{k=0}^n \lambda(k,\Omega(c)(1+D^2)^{-p/2}) = \mathrm{Ch}(c)\log(n)+O(1),\quad n\to\infty.
        \end{equation*}
    \end{cor}
    The above corollary is an immediate consequence of Theorem \ref{main thm} and Theorem \ref{universal measurability criterion}.
    }
        
    The main technical innovation of this paper concerns a certain integral representation for the difference of complex powers of positive operators, which originally
    appeared in \cite{CSZ} and which is reproduced here as Theorem \ref{csz key lemma}. 
    
    An operator $T \in \mathcal{L}_{1,\infty}$ is called universally measurable if all normalised traces on $\mathcal{L}_{1,\infty}$ take the same value on $T$.
    A new result of this paper, and a crucial component of our proof of Theorem \ref{main thm}, is the following:    
    \begin{thm}\label{zeta measurability theorem} 
        Let $0\leq V\in\mathcal{L}_{1,\infty}$ and let $A\in\mathcal{L}_{\infty}.$ 
        Define the $\zeta$-function: 
        \begin{equation*}
            \zeta_{A,V}(z) := \mathrm{Tr}(AV^{1+z}),\quad \Re(z) > 0.
        \end{equation*}
        If there exists $\varepsilon > 0$ such that $\zeta_{A,V}$ admits an analytic continuation to the set $\{z\;:\; \Re(z) > -\varepsilon\}\setminus \{0\}$
        with a simple pole at $0$, then for every normalised trace $\varphi$ on $\mathcal{L}_{1,\infty}$ we have:
        \begin{equation*}
            \varphi(AV)= \mathrm{Res}_{z=0}\zeta_{A,V}(z).
        \end{equation*}
        In particular, $AV$ is universally measurable.
    \end{thm}
    Theorem \ref{zeta measurability theorem} is a strengthening of an earlier result \cite[Theorem 4.13]{SUZ-indiana}, and a complete proof is given in Section \ref{subhankulov section}.
   
\section{Context of this paper}

    Connes' Character Formula dates back to Connes' 1995 paper \cite{Connes-original-spectral-1995}. There the character theorem was discovered in order to \lq\lq compute by a local formula the cyclic cohomology Chern character of $(\mathcal{A},H,D)$\rq\rq. Connes' work initiated a lengthy and ongoing program to strengthen, generalise and better understand the Character Formula.
    
    Closely linked to the Character Formula is the Local Index Theorem of Connes and Moscovici \cite{Connes-Moscovici}, and much of the work in this field was from the point of view of index theory. Among the approaches to generalising Connes character theorem, there is \cite{BF} by Benamuer and Fack, and \cite{CPRS1} by Carey, Philips, Rennie and the first named author.
    
    Instead of considering traces on $\mathcal{L}_{1,\infty},$ \cite{CPRS1} deals with Dixmier traces on the Lorentz space $\mathcal{M}_{1,\infty}.$ 
    Due to an error in the statement of Lemma 14 of \cite{CPRS1} which invalidates the proof in the $p=1$ case, a followup paper \cite{CRSZ} was
    written. In \cite{CRSZ}, the Character Formula is proved in the compact case for arbitrary normalised traces (rather than Dixmier traces).
    
    During the creation of the present manuscript an oversight was located in \cite{CRSZ}: in that paper the case where $D$ has a nontrivial kernel
    and $(\mathcal{A},H,D)$ is even was not handled correctly. It was incorrectly assumed in \cite[Case 3, page 20]{CRSZ} that if $(\mathcal{A},H,D)$ is an even spectral triple
    with grading $\Gamma$, then so is
    $$(\mathcal{A},H,(\chi_{[0,\infty)}(D)-\chi_{(-\infty,0)}(D))(1+|D|^2)^{1/2}).$$
    This is false if the kernel of $D$ is nontrivial, since then it is not necessarily the case that $\chi_{[0,\infty)}(D)-\chi_{(-\infty,0)}(D)$ anticommutes with $\Gamma$. The outcome of this oversight is that the proof of the Character Theorem as given in \cite{CRSZ} is incomplete. This oversight can be corrected by using the well known "doubling trick" that was already present in \cite[Definition 6]{CPRS1}. The present work supersedes that of \cite{CRSZ}, and so rather than submit an erratum we have decided to instead supply a complete
    proof here, in a more general setting.

    All of the work mentioned so far in this section applies exclusively in the compact case. Adapting the tools of noncommutative geometry to the locally compact case involves substantial difficulties and this task has been heavily studied by multiple authors over the past few decades: as a small sample of this body of work we mention \cite{Rennie-2003, Rennie-2004, gayral-moyal, GIV, CGRS1, CGRS2} and more recently work by Marius Junge and Li Gao concerning noncommutative planes.
    
    In 2000 Professor Nigel Higson published \cite{higson}: a detailed exposition of the local index theorem, including in the final appendix a claimed proof 
    of the Connes Character Formula in the non-unital setting. Higson's work was a major inspiration for the present paper, 
    since it is now understood and acknowledged by Higson that the claimed proof of the Character Formula 
    \cite[Theorem C.3]{higson} has a gap. This paper arose from our efforts to produce a correct statement and complete proof of the Character Formula in the non-unital setting using
    recently developed methods of Double Operator Integration theory.
    
    The nature of the gap in \cite{higson} is subtle, and concerns the relationship between Dixmier traces and zeta-function residues. To be precise: the proof
     relied on an equality between
     \begin{equation*}
	\lim_{s\downarrow 0} \mathrm{Tr}(Z|D|^{-n-s})
     \end{equation*}
     and
     \begin{equation*}
	\mathrm{Tr}_\omega(Z|D|^{-n})
     \end{equation*}
     (in the notation of \cite{higson}). In the case where $|D|^{-1}$ is compact this result can be attained using
     existing techniques from \cite[Theorem 8.6.4, Theorem 8.6.5 and Theorem 9.3.1]{LSZ}. In the case where $|D|^{-1}$ is not compact the situation
     is less well understood. The present text was motivated by an effort to understand the equality above in the non-compact case.
    
    {
    After circulating a draft of our manuscript Carey and Rennie pointed out that there was a different way to obtain a similar result on the Hochschild class using \cite{CGRS2} (which is based on \cite{CPRS4}). It is proved in these papers that the \lq\lq resolvent cocycle\rq\rq introduced there represents the cohomology class of the Chern character. From that point of view one may obtain a different representative of the Hochschild class of the Chern character using  residues of zeta functions under weaker hypotheses on the Hochschild chains and substantially stronger summability conditions on the spectral triple. For Hochschild chains satisfying some additional conditions, but not requiring locality as employed here, Carey and Rennie also have a Dixmier trace formula for the Hochschild class of the Chern character evaluated on such Hochschild chains. }

\section{Structure of the paper}
    This paper is structured as follows:
    \begin{itemize}
        \item{} Chapter \ref{preliminaries chapter} is devoted to preliminary definitions and concepts: we introduce the relevant definitions for operator ideals, traces, spectral
                triples, operator valued integrals and double operator integrals.
        \item{} Chapter \ref{examples chapter} provides important technical properties of spectral triples. In Section \ref{nc section} we prove that Hypothesis \ref{main assumption}
                is satisfied for the canonical spectral triple associated to noncommutative Euclidean spaces $\mathbb{R}^p_\theta$, and in Section \ref{manifold section} we show
                that the Hypothesis is satisfied for Hodge-Dirac spectral triples associated to arbitrary complete Riemannian manifolds.
        \item{} Chapter \ref{heat chapter} contains the proof of Theorem \ref{heat thm}.
        \item{} Chapter \ref{zeta chapter} contains the proofs of Theorems \ref{zeta thm}, \ref{zeta measurability theorem} and \ref{main thm}.
        \item{} Finally, an appendix is included to collect some of the lengthier computations.
    \end{itemize}

\chapter{Preliminaries}\label{preliminaries chapter}

\section{Operators, ideals and traces}

\subsection{General notation}\label{general notations subsection}
    Fix throughout a separable, infinite dimensional complex Hilbert space $H$. We denote by $\mathcal{L}_{\infty}$ the algebra of all bounded operators on $H$, with operator norm denoted $\|\cdot\|_\infty.$ For a compact operator $T$ on $H$,
    let $\lambda(T) := \{\lambda(k,T)\}_{k=0}^\infty$ denote the sequence of eigenvalues of $T$ arranged in order of non-increasing magnitude and with multiplicities. Similarly, let $\mu(T) := \{\mu(k,T)\}_{k=0}^\infty$ denote the sequence of singular
    values of $T$, also arranged in non-increasing order with multiplicities. The $k$th singular value may be described equivalently as either $\mu(k,T) := \lambda(k,|T|)$ or 
    \begin{equation*}
        \mu(k,T) = \inf\{\|T-R\|_\infty\;:\;\mathrm{rank}(R)\leq k\}.
    \end{equation*}
    
    The standard trace on $\mathcal{L}_\infty$ (more precisely on the trace-class ideal) is denoted $\mathrm{Tr}$.

    Fix an orthonormal basis $\{e_k\}_{k=0}^\infty$ on $H$ (the particular choice of basis is inessential). We identify the algebra $\ell_\infty$ of all bounded sequences with the subalgebra of diagonal operators on $H$ with respect
    to the chosen basis. For a given $\alpha \in \ell_\infty$, we denote the corresponding diagonal operator by $\mathrm{diag}(\alpha)$.

    For $A,B \in \mathcal{L}_{\infty}$, we say that $B$ is submajorized by $A$ in the sense of Hardy-Littlewood, written as $B \prec\prec A$, if
    \begin{equation*}
        \sum_{k=0}^n \mu(k,B) \leq \sum_{k=0}^n \mu(k,A), \quad n\geq 0.
    \end{equation*}
    
    We say that $B$ is logarithmically submajorized by $A$, written as $B \prec\prec_{\log} A$ if 
    \begin{equation*}
        \prod_{k=0}^n \mu(k,B)\leq \prod_{k=0}^n \mu(k,A),\quad n\geq 0.
    \end{equation*}
    
    An important result concerning logarithmic submajorisation is the Araki-Lieb-Thirring inequality \cite[Theorem 2]{Kosaki-alt-1992}, which states that
    for all positive bounded operators $A$ and $B$ and all $r \geq 1$,
    \begin{equation}\label{ALT inequality}
        |AB|^r \prec\prec_{\log} A^rB^r.
    \end{equation}
    
    We make frequent use of the following commutator identity: if $A$ and $B$ are operators with $B$ invertible, then
    \begin{equation}\label{favourite commutator identity}
        [B^{-1},A] = -B^{-1}[B,A]B^{-1}.
    \end{equation}
    We must take care to ensure that \eqref{favourite commutator identity} remains valid when $A$ and $B$ are potentially unbounded. If $A$ is bounded, then it is enough
    that $A:\mathrm{dom}(B)\to \mathrm{dom}(B)$.

\subsection{Ideals in $\mathcal{L}_{\infty}$ and related inequalities}
    For $p \in (0,\infty)$, we let $\mathcal{L}_{p}$ denote the Schatten-von Neumann ideal of $\mathcal{L}_{\infty}$,
    \begin{equation*}
        \mathcal{L}_p := \{T \in \mathcal{L}_{\infty}\;:\; \mu(T) \in \ell_p\}
    \end{equation*}
    where $\ell_p$ is the space of $p$-summable sequences. As usual, for $p \geq 1$ the ideal $\mathcal{L}_{p}$ is equipped
    with the norm
    \begin{equation*}
        \|T\|_{p} := \left(\sum_{k=0}^\infty \mu(k,T)^p\right)^{1/p}.
    \end{equation*}
    
    Similarly, given $0 < p < \infty$, we denote by $\mathcal{L}_{p,\infty}$ the ideal in $\mathcal{L}_{\infty}$ defined by
    \begin{equation*}
        \mathcal{L}_{p,\infty} := \{T \in \mathcal{L}_{\infty}\;:\; \sup_{k\geq 0}\, (1+k)^{1/p}\mu(k,T) < \infty\}.
    \end{equation*}
    Equivalently, 
    \begin{equation*}
        \mathcal{L}_{p,\infty} := \{T \in \mathcal{L}_{\infty}\;:\; \sup_{n\geq 0}\, n^{-p}\mathrm{Tr}(\chi_{(\frac{1}{n},\infty)}(|T|)) < \infty\}.
    \end{equation*}
    It is well known that the ideal $\mathcal{L}_{p,\infty}$ may be equipped with a quasi-norm given by the formula
    \begin{equation*}
        \|T\|_{p,\infty} := \sup_{k\geq 0} (k+1)^{1/p}\mu(k,T), \quad T \in \mathcal{L}_{p,\infty}.
    \end{equation*}
    As is conventional, $\mathcal{L}_{\infty,\infty} := \mathcal{L}_{\infty}$. 
    
    We make use of the following H\"older inequality: let $p,p_1,p_2,\ldots,p_n \in (0,\infty]$ satisfy $\frac{1}{p} = \sum_{k=1}^n \frac{1}{p_k}$. If $A_k \in \mathcal{L}_{p_k,\infty}$ for all $k = 1,\ldots, n$, then
    $A_1A_2\cdots A_n \in \mathcal{L}_{p,\infty}$, with an inequality of norms:
    \begin{equation}\label{weak-type Holder}
        \|A_1A_2\cdots A_n\|_{p,\infty} \leq c_{p_1,p_2,\ldots,p_n}\|A_1\|_{p_{1},\infty}\|A_2\|_{p_2,\infty}\cdots \|A_n\|_{p_n,\infty}
    \end{equation}
    where $c_{p_1,p_2,\ldots,p_n} > 0$.
    
    The quasi-norm $\|\cdot\|_{1,\infty}$ is not monotone with respect to Hardy-Littlewood submajorisation. It is, however, monotone under
    logarithmic submajorisation. To be precise, we have that for all $A,B \in \mathcal{L}_{1,\infty}$ if $B \prec\prec_{\log} A$
    then
    \begin{equation}\label{log majorization monotone}
        \|B\|_{1,\infty} \leq e\|A\|_{1,\infty}.
    \end{equation}
    
    Indeed, since the sequence $\{\mu(k,B)\}_{k=0}^\infty$ is nonincreasing, for all $n \geq 0$ we have:
    \begin{equation*}
        \mu(n,B)^{n+1} \leq \prod_{k=0}^n \mu(k,B).
    \end{equation*}
    So if $B \prec\prec_{\log} A$,
    \begin{equation}\label{geometric mean bound}
        \mu(n,B)^{n+1} \leq \prod_{k=0}^n \mu(k,A).
    \end{equation}
    However by definition, $\mu(k,A) \leq \frac{\|A\|_{1,\infty}}{k+1}$ for all $k$, so
    \begin{equation}\label{factorial bound}
        \prod_{k=0}^n \mu(k,A) \leq \frac{\|A\|_{1,\infty}^{n+1}}{(n+1)!}.
    \end{equation}
    Now combining \eqref{geometric mean bound} and \eqref{factorial bound}, 
    \begin{equation*}
        \mu(n,B)^{n+1} \leq \frac{\|A\|_{1,\infty}^{n+1}}{(n+1)!}.
    \end{equation*}
    Now using the Stirling approximation
    $$(n+1)!\geq\big(\frac{n+1}{e}\big)^{n+1},$$
    we arrive at
    \begin{equation*}
        \mu(n,B)^{n+1} \leq \Big(\frac{e\|A\|_{1,\infty}}{n+1}\Big)^{n+1}.
    \end{equation*}
    Hence, for all $n\geq 0$,
    \begin{equation*}
        \mu(n,B) \leq\frac{e\|A\|_{1,\infty}}{n+1}.
    \end{equation*}
    Multiplying by $n+1$, and then taking the supremum over $n$ yields $\|B\|_{1,\infty} \leq e\|A\|_{1,\infty}$ as desired.
    
    Another ideal to which we will refer is the Schatten-Lorentz ideal $\mathcal{L}_{q,1}$ for $q > 1$, defined by
    \begin{equation*}
        \mathcal{L}_{q,1} := \{T \in \mathcal{L}_\infty\;:\;\sum_{k=0}^\infty \mu(k,T)(k+1)^{\frac{1}{q}-1} < \infty\}
    \end{equation*}
    and equipped with the quasi-norm
    \begin{equation*}
        \|A\|_{q,1} := \sum_{k=0}^\infty \mu(k,A)(1+k)^{\frac{1}{q}-1}.
    \end{equation*}

    If $\frac{1}{p}+\frac{1}{q} = 1$, then we have the following H\"older-type inequality:
    \begin{equation}\label{another holder}
        \|AB\|_1 \leq \|A\|_{p,\infty}\|B\|_{q,1},\quad A \in \mathcal{L}_{p,\infty}, B \in \mathcal{L}_{q,1}.
    \end{equation}

\subsection{Traces on $\mathcal{L}_{1,\infty}$}\label{trace subsection}

    \begin{defi}\label{trace def} 
        If $\mathcal{I}$ is an ideal in $\mathcal{L}_{\infty},$ then a unitarily invariant
        linear functional $\varphi:\mathcal{I}\to\mathbb{C}$ is said to be a trace.
    \end{defi}
    Here $\varphi$ being ``unitarily invariant" means that $\varphi(U^*TU) = \varphi(T)$ for all $T \in \mathcal{I}$ and unitary operators $U$. 
    Equivalently, $\varphi(UT) = \varphi(TU)$ for all unitary operators $U$ and $T \in \mathcal{I}$. Since every bounded linear operator can be written as a linear combination
    of at most four unitary operators \cite[Page 209]{Reed-Simon-I-1980}, one may equivalently say that $\varphi(AT) = \varphi(TA)$ for all $A \in \mathcal{L}_{\infty}$
    and $T \in \mathcal{I}$.
    
    The most well-known example of a trace is the classical trace $\mathrm{Tr}$ on the ideal $\mathcal{L}_1$, however we will be primarily concerned with traces on the ideal $\mathcal{L}_{1,\infty}$. 
    There exist many traces on $\mathcal{L}_{1,\infty}$, of which the earliest discovered class of examples are the Dixmier traces which we now describe.
    
    Recall that an extended limit is a continuous linear functional $\omega \in \ell_\infty^*$ from the set of bounded sequences $\ell_\infty$
    which extends the limit functional on the subspace $c$ of convergent sequences. Readers who are more familiar with ultrafilters may consider
    the special case where $\omega$ is the limit along a non-principal (free) ultrafilter on $\mathbb{Z}_+$.
    
    \begin{ex} 
        Let $\omega$ be an extended limit. Then the functional $\mathrm{Tr}_\omega$ is defined on a positive operator $T \in \mathcal{L}_{1,\infty}$ by
        \begin{equation*}
            \mathrm{Tr}_\omega(T) := \omega\left(\left\{\frac{1}{\log(2+n)}\sum_{k=0}^n \mu(k,T)\right\}_{n=0}^\infty\right).
        \end{equation*}
        The functional $\mathrm{Tr}_\omega$ is additive on the cone of positive elements of $\mathcal{L}_{1,\infty}$, and therefore extends by linearity to a a functional on $\mathcal{L}_{1,\infty}$. The thus defined functional $\mathrm{Tr}_\omega:\mathcal{L}_{1,\infty}\to \mathbb{C}$
        is a trace, and we call such traces Dixmier traces.
    \end{ex}
    \begin{proof}
        Let $A$ and $B$ be positive operators. Combining \cite[Theorem 3.3.3, Theorem 3.3.4]{LSZ}, for all $n \geq 0$ we have:
        \begin{equation*}
            \sum_{k=0}^n \mu(k,A+B) \leq \sum_{k=0}^n \mu(k,A)+\mu(k,B) \leq \sum_{k=0}^{2n+1} \mu(k,A+B).
        \end{equation*}
        Hence,
        \begin{equation*}
            0 \leq \sum_{k=0}^n \mu(k,A)+\mu(k,B)-\mu(k,A+B) \leq \sum_{k=n+1}^{2n+1} \mu(k,A+B).
        \end{equation*}
        However $A+B \in \mathcal{L}_{1,\infty}$, so there is a constant $C > 0$ such that for all $k \geq 0$ we have $\mu(k,A+B) \leq \frac{C}{k+1}$ and therefore
        \begin{equation*}
            0 \leq \sum_{k=0}^n \mu(k,A)+\mu(k,B)-\mu(k,A+B) \leq C,\quad n\geq 0.
        \end{equation*}
        Dividing by $\log(2+n)$:
        \begin{align*}
            0 &\leq \frac{1}{\log(2+n)}\sum_{k=0}^n \mu(k,A) + \frac{1}{\log(2+n)}\sum_{k=0}^n \mu(k,B) - \frac{1}{\log(2+n)}\sum_{k=0}^n \mu(k,A+B)\\
              &\leq O(\frac{1}{\log(2+n)}), \quad n\to\infty.
        \end{align*} 
        Then applying $\omega$:
        \begin{equation*}
            0 \leq \mathrm{Tr}_\omega(A)+\mathrm{Tr}_\omega(B)-\mathrm{Tr}_\omega(A+B) \leq 0.
        \end{equation*}
        So indeed $\mathrm{Tr}_\omega(A+B) = \mathrm{Tr}_\omega(A)+\mathrm{Tr}_\omega(B)$ for any two positive operators $A$ and $B$.
    \end{proof}
    
    \begin{rem}
        Dixmier traces were first defined by J. Dixmier in \cite{Dixmier}, albeit with some important differences to $\mathrm{Tr}_\omega$ as given in the above example.
        \begin{enumerate}[{\rm (i)}]
            \item Originally Dixmier traces were defined on the ideal $\mathcal{M}_{1,\infty}$ which is strictly larger than $\mathcal{L}_{1,\infty}.$
            \item $\mathrm{Tr}_\omega$ was originally shown to be additive only for those extended limits which are translation and dilation invariant.
        \end{enumerate}
        For more technical details and historical information we refer the reader to \cite[Chapter 6]{LSZ}.
                
        As the preceding example shows, $\mathrm{Tr}_\omega$ is additive on $\mathcal{L}_{1,\infty}$ for an arbitrary extended limit.
    \end{rem}    

    An extensive discussion of traces, and more recent developments in the theory, may be found in \cite{LSZ} including a discussion of the following facts:
    \begin{enumerate}
        \item All Dixmier traces on $\mathcal{L}_{1,\infty}$ are positive.
        \item All positive traces on $\mathcal{L}_{1,\infty}$ are continuous in the quasi-norm topology.
        \item Every continuous trace is a linear combination of positive traces.
        \item There exist positive traces on $\mathcal{L}_{1,\infty}$ which are not Dixmier traces (see \cite{SSUZ-pietsch}).
        \item There exist traces on $\mathcal{L}_{1,\infty}$ which fail to be continuous (see \cite{DFWW}).
        \item Every trace on $\mathcal{L}_{1,\infty}$ vanishes on $\mathcal{L}_1$ (see \cite{DFWW}).
    \end{enumerate}

    We are mostly interested in {\it normalised traces} $\varphi:\mathcal{L}_{1,\infty}\to\mathbb{C},$ that is, satisfying $\varphi({\rm diag}(\{\frac1{k+1}\}_{k\geq0}))=1.$

    The following definition, extending that originally introduced in \cite[Definition 2.$\beta$.7]{NCG-book}, plays an important role here.
    \begin{defi}\label{def:uni-meas}
        An operator $T\in\mathcal{L}_{1,\infty}$ is called universally measurable if all normalised traces take the same value on $T.$
    \end{defi}
    
    The following result characterises universally measurable operators in terms of their eigenvalues, and a detailed proof may be found in \cite[Theorem 10.1.3(g)]{LSZ}
    \begin{thm}\label{universal measurability criterion} 
        An operator $T\in\mathcal{L}_{1,\infty}$ is universally measurable if and only if there exists a constant $c$ such that
        $$\sum_{k=0}^n\lambda(k,T)=c\cdot\log(n)+O(1), \quad n\to\infty$$
        In this case, we have $\varphi(T)=c$ for every normalised trace $\varphi$ on $\mathcal{L}_{1,\infty}.$
    \end{thm}

\section{Spectral triples}\label{spectral triple subsection}
    A spectral triple is an algebraic model for a Riemannian manifold, defined as follows:
    \begin{defi}\label{spectral triple definition}
        A spectral triple $(\mathcal{A},H,D)$ consists of the following data:
        \begin{enumerate}[{\rm (a)}]
            \item{} a separable Hilbert space $H$.
            \item{} a (possibly unbounded) self-adjoint operator $D$ on $H$ with a dense domain $\mathrm{dom}(D)\subseteq H$.
            \item{} a $*$-subalgebra $\mathcal{A}$ of the algebra of bounded linear operators on $H$.
        \end{enumerate}
        Such that for all $a \in \mathcal{A}$ we have:
        \begin{enumerate}
            \item{} $a\cdot \mathrm{dom}(D) \subseteq \mathrm{dom}(D)$,
            \item{} The commutator $[D,a]:\mathrm{dom}(D)\to H$ extends to a bounded linear operator on $H$, which we denote $\partial(a)$,
            \item{} $a(D+i)^{-1}$ is a compact operator.
        \end{enumerate}
    \end{defi}
        
    \begin{rem}
        Definition \ref{spectral triple definition} should be compared to \cite[Definition 3.1]{CGRS2}, of which it is a special case (when the underlying von Neumann algebra is $\mathcal{L}_{\infty}(H)$). 
        Within the literature there is some variation in the definition of a spectral triple. In many sources (such as \cite[Definition 9.16]{GVF}) it is assumed that the resolvent $(D+i)^{-1}$ is compact.
        We will refer to spectral triples where $(D+i)^{-1}$ is compact as compact spectral triples. In particular a spectral triple where $\mathcal{A}$ contains the identity operator is compact. 
        If $(\mathcal{A},H,D)$ is not necessarily compact, we will say that is it locally compact.
    \end{rem}
    
    \begin{defi}
    Given a spectral triple $(\mathcal{A},H,D)$ let $F_D$ denote the partial isometry defined via functional calculus as
    $$F_D := \chi_{(0,\infty)}(D)-\chi_{(-\infty,0)}(D).$$
    Where there is no ambiguity, we will frequently denote $F_D$ as $F$.        
    \end{defi}
    If $D$ has trivial kernel, then $F_D^2 = 1.$
    
    We may define the operator $|D|:\mathrm{dom}(D)\to H$ by functional calculus.
    Since $D$ is self-adjoint, for all $n\geq 1$ we have $|D|^n = |D^n|$, and so $\mathrm{dom}(|D|^n) = \mathrm{dom}(D^n)$.
    We have $F|D| = D$ as an equality of operators on $\mathrm{dom}(D)$, and on $\mathrm{dom}(D^2)$:
    \begin{equation*}
        |D|D = D|D|.
    \end{equation*}
    We also have $|D|=FD$.
    
    Note that $F_D^* = F_{D^*} = F_D$. Hence, we also have $D = D^*=|D|^*F^* = |D|F$.
    Since $|D|F = D$, it follows that $F:\mathrm{dom}(D)\to\mathrm{dom}(D)$.
    
    By similar reasoning, we also have that for all $n\geq 1$ that $D^{n}F = FD^{n}$ and hence that $F:\mathrm{dom}(D^n)\to \mathrm{dom}(D^n)$.

    Consequently, for $n,m \geq 1$, the operators $F$, $D^n$ and $|D|^m$ all mutually commute on $\mathrm{dom}(D^{n+m})$.
    
\subsection{Properties of spectral triples}
    Smoothness of a spectral triple is defined in terms of boundedness of commutators with $|D|$ (see Subsection \ref{smoothness discussion} for discussion of this issue).
    The following results will be known to the expert reader. The notion of smoothness defined in terms of domains of commutators with $|D|$ originates with Connes \cite[Section 1]{Connes-original-spectral-1995} and
    is also used in \cite{Connes-Moscovici} and \cite[Section 1.3]{CGRS2}. We provide detailed proofs here for convenience.
    
    If $T$ is a bounded operator with $T:\mathrm{dom}(D)\to \mathrm{dom}(D)$, then the commutator
    $|D|T-T|D|:\mathrm{dom}(D)\to H$ is meaningful. More generally, if there is some $n$ such that for all $0 \leq k \leq n$ we
    have $T:\mathrm{dom}(D^k)\to \mathrm{dom}(D^k)$ then we may consider the higher iterated commutator:
    \begin{equation}\label{nth iterated commutator}
        [|D|,[|D|,[\cdots [|D|,T]\cdots]]] = \sum_{k=0}^n (-1)^k \binom{n}{k}|D|^ka|D|^{n-k}.
    \end{equation}
    This is a well defined operator on $\mathrm{dom}(D^n)$ in the following sense: for each $k$ we have $|D|^{n-k}:\mathrm{dom}(D^n)\to\mathrm{dom}(D^k)$,
    $a:\mathrm{dom}(D^k)\to\mathrm{dom}(D^k)$ and $|D|^k:\mathrm{dom}(D^k)\to H$.
    
    We wish to define $\delta^n(T)$ as the bounded extension of the $n$th iterated commutator $[|D|,[|D|,[\cdots,[|D|,T]\cdots]]]$
    when such an extension exists. This motivates the following definition:
    \begin{defi}
        For $n\geq 1$, we define $\mathrm{dom}(\delta^n)$ to be the set of bounded linear operators $T$ such that for all $0 < k \leq n$ we
        have $T:\mathrm{dom}(D^k)\to \mathrm{dom}(D^k)$ and the $n$th iterated commutator in \eqref{nth iterated commutator} has bounded extension.
        
        For $T \in \mathrm{dom}(\delta^n)$, we let $\delta^n(T)$ be the bounded extension of the $n$th iterated commutator \eqref{nth iterated commutator}.
        
        The $n=0$ case is defined by $\mathrm{dom}(\delta^0) := \mathcal{L}_\infty(H)$ and $\delta^0(T) := T$.
        
        We define 
        \begin{equation*}
            \mathrm{dom}_\infty(\delta) := \bigcap_{k\geq 0} \mathrm{dom}(\delta^k).
        \end{equation*}
    \end{defi}
    
    \begin{lem}\label{dom delta infty is an algebra}
        The set $\mathrm{dom}_\infty(\delta)$ is closed under multiplication.
    \end{lem}
    \begin{proof}
        Let $T,S \in \mathrm{dom}_\infty(\delta)$. Then by definition, for all $k \geq 1$ we have that $T,S:\mathrm{dom}(D^k)\to \mathrm{dom}(D^k)$,
        and hence $TS:\mathrm{dom}(D^k)\to \mathrm{dom}(D^k)$. The $k$th iterated commutator $\delta^k(TS)|\mathrm{dom}(D^k)$ is given by:
        \begin{equation*}
            \delta^k(TS) = \sum_{j=0}^k \binom{k}{j} \delta^{k-j}(T)\delta^j(S).
        \end{equation*}
        Since for all $j$ we have $\delta^j(S) \in \mathrm{dom}_\infty(\delta)$ and $\delta^{k-j}(T) \in \mathrm{dom}_\infty(\delta)$, the above expression
        is well defined as an operator on $\mathrm{dom}(D^k)$ and has bounded extension.
    \end{proof}

    It is clear that if $k\geq 0$ and $T \in \mathrm{dom}(\delta^{k+1})$, then $\delta(T) \in \mathrm{dom}(\delta^k)$ and $\delta^k(T) \in \mathrm{dom}(\delta)$. Moreover $\delta^{k+1}(T) = \delta^k(\delta(T)) = \delta(\delta^k(T))$.
    
    We may also define $\mathrm{dom}(\partial)$ to be the set of bounded operators $T$ such that $T:\mathrm{dom}(D)\to \mathrm{dom}(D)$
    and $[D,T]:\mathrm{dom}(D)\to H$ has a bounded extension, which we denote $\partial(T)$. 
    
    The relevance of $\mathrm{dom}(\partial)$ is the following:
    \begin{lem}\label{delta and partial commute}
        Suppose that $T \in \mathrm{dom}(\delta)\cap\mathrm{dom}(\partial)$ is such that $\partial(T) \in \mathrm{dom}(\delta)$ and $\delta(T):\mathrm{dom}(D)\to \mathrm{dom}(D)$. Then $\delta(T) \in \mathrm{dom}(\partial)$ and
        \begin{equation*}
            \partial(\delta(T)) = \delta(\partial(T)).
        \end{equation*}
    \end{lem}
    \begin{proof}
        Since $T \in \mathrm{dom}(\delta)\cap\mathrm{dom}(\partial)$, we have in particular that $T:\mathrm{dom}(D)\to\mathrm{dom}(D)$. Since $\partial(T) \in \mathrm{dom}(\delta)$,
        we also have that $\partial(T):\mathrm{dom}(D)\to \mathrm{dom}(D)$. Let $\xi \in \mathrm{dom}(D^2)$. Then,
        \begin{equation*}
            DT\xi = \partial(T)\xi + TD\xi.
        \end{equation*}
        Since $T:\mathrm{dom}(D)\to \mathrm{dom}(D)$ and $\partial(T):\mathrm{dom}(D^2)\to\mathrm{dom}(D)$, it follows that $DT\xi \in \mathrm{dom}(D)$ and therefore $T:\mathrm{dom}(D^2)\to \mathrm{dom}(D^2)$. 
        
        Now since the operators $D$ and $|D|$ commute on $\mathrm{dom}(D^2)$, we have that for all $\xi \in \mathrm{dom}(D^2)$:
        \begin{equation*}
            [D,[|D],T]]\xi = [|D|,[D,T]]\xi.
        \end{equation*}
        Since by assumption $T \in \mathrm{dom}(\delta)$, $\partial T \in \mathrm{dom}(\delta)$ and $\delta(T):\mathrm{dom}(D)\to\mathrm{dom}(D)$, we may further write:
        \begin{equation*}
            [D,\delta(T)]\xi = \delta(\partial(T))\xi.
        \end{equation*}
        Since the operator on the right hand side by assumption has bounded extension, and using the fact that $\mathrm{dom}(D^2)$ is dense { in $H$} it follows that $[D,\delta(T)]$ has bounded extension and therefore
        $\delta(T) \in \mathrm{dom}(\partial)$. Thus, $\partial(\delta(T)) = \delta(\partial(T))$.
    \end{proof}

    Next we define the notion of a smooth spectral triple. Some sources (such as \cite[Definition 3.18]{CGRS2}) use the term "$QC^{\infty}$ spectral triple", and others (such as \cite[Definition 4.25]{higson} use the term "regular" spectral triple).
    \begin{defi}\label{smoothness definition}
        A spectral triple $(\mathcal{A},H,D)$ is called smooth if for all $a \in \mathcal{A}$, we have
        \begin{equation*}
            a, \partial(a) \in \mathrm{dom}_\infty(\delta).
        \end{equation*}
        If $(\mathcal{A},H,D)$ is smooth, we let $\mathcal{B}$ be the $*-$subalgebra of $\mathcal{L}_\infty(H)$ generated by all elements of the form $\delta^k(a)$ or $\delta^k(\partial(a)),$ $k\geq0,$ $a\in\mathcal{A}.$
    \end{defi}
    By Lemma \ref{dom delta infty is an algebra} and since $\delta^k(a)^* = (-1)^k\delta^k(a^*)$, we automatically have that $\mathcal{B} \subseteq \mathrm{dom}_\infty(\delta)$.

    \begin{cor}
        Let $(\mathcal{A},H,D)$ be smooth, and $a \in \mathcal{A}$. Then for all $k\geq 1$ we have $\delta^k(a) \in \mathrm{dom}(\partial)$ and
        \begin{equation*}
            \partial(\delta^k(a)) = \delta^k(\partial(a)).
        \end{equation*}
    \end{cor}
    \begin{proof}
        This proof proceeds by induction on $k$. 
        For $k=1$, by the definition of smoothness we have $\partial(a) \in \mathrm{dom}(\delta)$ and $a \in \mathrm{dom}(\delta)\cap\mathrm{dom}(\partial)$, and by definition $a:\mathrm{dom}(D)\to\mathrm{dom}(D)$. So by Lemma \ref{delta and partial commute} it follows that
        $\delta(a) \in \mathrm{dom}(\partial)$ and 
        \begin{equation*}
            \partial(\delta(a)) = \delta(\partial(a)).
        \end{equation*}
        
        Now we suppose that the claim is proved for $k-1$, $k\geq 2$ and we prove the claim for $k$. Since $(\mathcal{A},H,D)$ is smooth, $\delta^{k-1}(a):\mathrm{dom}(D)\to \mathrm{dom}(D)$ and by the inductive hypothesis, $\delta^{k-1}(a) \in \mathrm{dom}(\partial)$ and
        \begin{equation*}
            \delta^{k-1}(\partial(a)) = \partial(\delta^{k-1}(a)).
        \end{equation*}
        However since $\delta^{k-1}(a) \in \mathrm{dom}(\delta)\cap\mathrm{dom}(\partial)$ and $\delta^{k-1}(a):\mathrm{dom}(D)\to\mathrm{dom}(D)$ we may apply Lemma \ref{delta and partial commute} with $T = \delta^{k-1}(a)$ to conclude
        that $\delta^{k}(a) \in \mathrm{dom}(\partial)$ and
        \begin{align*}
            \delta(\partial(\delta^{k-1}(a))) &= \partial(\delta(\delta^{k-1}(a)))\\
                                              &= \partial(\delta^k(a)).
        \end{align*}
        By the inductive hypothesis, $\delta(\partial(\delta^{k-1}(a))) = \delta(\delta^{k-1}(\partial(a))) = \delta^k(\partial(a))$, and so this proves the result for $k$.
    \end{proof}

    \begin{defi}
        Let $T \in \mathrm{dom}(\partial)\cap \mathrm{dom}(\delta)$. Define
        \begin{equation*}
            L(T) := \partial(T)-F\delta(T).
        \end{equation*}
    \end{defi}
    Note that by definition $L(T)$ is bounded. On $\mathrm{dom}(D)$ we have:
    \begin{equation*}
        L(T) = [F,T]|D|.
    \end{equation*}
    The boundedness of $L(T)$ on $\mathrm{dom}(D)$ was already implicitly noted in the proof of \cite[Lemma 2]{CPRS1}.
    
    Our computations are greatly simplified by introducing a common dense subspace $H_\infty \subseteq H$ on which all powers $D^k$
    are defined:
    \begin{defi}\label{H_infty definition}
        Let $H_\infty := \bigcap_{n\geq 0} \mathrm{dom}(D^n)$.
    \end{defi}
    The subspace $H_\infty$ is a well known object in noncommutative geometry, appearing in \cite[Section 1]{Connes-original-spectral-1995} and more recently in \cite[Equation 10.64]{GVF} and \cite[Definition 1.20]{CGRS2}.
    One way to see that $H_\infty$ is dense in $H$ (and in particular non-zero) is to note that $\mathrm{dom}(e^{D^2}) \subseteq H_\infty$.
    If $T \in \mathrm{dom}_\infty(\delta)$, then $T:H_\infty\to H_\infty$ since by definition if $T \in \mathrm{dom}(\delta^n)$ then $T:\mathrm{dom}(D^k)\to\mathrm{dom}(D^k)$ for all $0 \leq k \leq n$. Moreover since $F:\mathrm{dom}(D^n)\to \mathrm{dom}(D^n)$ for all $n$, we also
    have $F:H_\infty\to H_\infty$. It is useful to note that for each $k$ the unbounded operators $D^k$ and $|D|^k$ map $H_\infty$ to $H_\infty$. This observation is justified in the next lemma.
    
    \begin{lem}
        Let $f:\mathbb{R}\to \mathbb{R}$ be a Borel function which has "polynomial growth at infinity" in the sense that there exists $n \geq 0$ such that
        $t\mapsto (1+t^2)^{-n/2}f(t)$ is bounded on $\mathbb{R}$. Then $f(D)$ (defined by Borel functional calculus) maps $H_\infty$ to $H_\infty$.
    \end{lem}
    \begin{proof}
        Let $k > n$, and $\xi \in \mathrm{dom}(D^k)$. By assumption $(1+D^2)^{-n/2}f(D)$ defines a bounded operator,
        \begin{equation*}
            (1+D^2)^{-\frac{n}{2}}f(D)D^k:\mathrm{dom}(D^k)\to H.
        \end{equation*}
        However $(1+D^2)^{-n/2}f(D)D^k = D^k(1+D^2)^{-n/2}f(D)$ on a dense domain. Since $D^n(1+D^2)^{-n/2}$ defines a bounded operator, we get that $D^{k-n}f(D):\mathrm{dom}(D^k)\to H$. Therefore
        $f(D):\mathrm{dom}(D^k)\to \mathrm{dom}(D^{k-n})$. Since $k > n$ is arbitrary, we have that $f(D):H_\infty\to H_\infty$.
    \end{proof}
    
    \begin{lem}
        If $T \in \mathrm{dom}(\delta)\cap \mathrm{dom}(\partial)$ is such that $T:H_\infty\to H_\infty$, then $L(T):H_\infty \to H_\infty$.
    \end{lem}
    \begin{proof}
        For $\xi \in H_\infty$,
        \begin{equation*}
            L(T)\xi = [F,T]|D|\xi.
        \end{equation*}
        However $|D|\xi \in H_\infty$, and $F:H_\infty\to H_\infty$. Thus $[F,T]|D|\xi \in H_\infty$.
    \end{proof}
    
    \begin{lem}
        Let $T\in \mathrm{dom}(\delta^2)\cap \mathrm{dom}(\partial)$ be such that $\partial(T) \in \mathrm{dom}(\delta)$. Then $L(T) \in \mathrm{dom}(\delta)$ and
        \begin{equation*}
            \delta(L(T)) = L(\delta(T)).
        \end{equation*}
    \end{lem}
    \begin{proof}
        Since $\partial(T) \in \mathrm{dom}(\delta)$, we have from Lemma \ref{delta and partial commute} that $\delta(T) \in \mathrm{dom}(\partial)\cap \mathrm{dom}(\delta)$ and hence $L(\delta(T))$ is defined and bounded.
    
        The required identity can be checked on $\mathrm{dom}(D^2)$, since $T:\mathrm{dom}(D^2)\to \mathrm{dom}(D^2)$. 
        For $\xi \in \mathrm{dom}(D^2)$, we have:
        \begin{align*}
            \delta(L(T))\xi &= [|D|,[F,T]|D|]\xi\\
                            &= [F,[|D|,T]]|D|\xi\\
                            &= L(\delta(T))\xi.
        \end{align*}
    \end{proof}
    
    Spectral triples are often classed as \emph{even} or \emph{odd}:
    \begin{defi}\label{oddeven} 
        A spectral triple $(\mathcal{A},H,D)$ is said to be
        \begin{enumerate}[{\rm (a)}]
        \item even if equipped with $\Gamma\in\mathcal{L}_{\infty}$ such that $\Gamma=\Gamma^*,$ $\Gamma^2=1$ and such that $[\Gamma,a]=0$ for all $a\in\mathcal{A},$ $\{D,\Gamma\}=0.$ Here $\{\cdot,\cdot\}$ denotes anticommutator.
        \item odd if not equipped with such $\Gamma.$ In this case, we set $\Gamma=1.$
        \item $p-$dimensional if for all $a \in \mathcal{A}$ we have $a(D+i)^{-p} \in \mathcal{L}_{1,\infty}$ and $\partial(a)(D+i)^{-p}\in\mathcal{L}_{1,\infty}$, and for all $q < p$ there exists $a_0\in \mathcal{A}$ such that $a_0(D+i)^{-q}\notin\mathcal{L}_{1,\infty}$.
        \end{enumerate}
    \end{defi}
    For an even spectral triple, we have $D^2\Gamma=\Gamma D^2.$ Therefore, $|D|\Gamma=\Gamma|D|$. We furthermore have that $F\Gamma+\Gamma F = 0$.
    
    We follow the convention of \cite{CGRS2}, where we write $\Gamma$ in all formulae referring to spectral triples, with the understanding that if the spectral triple is odd then $\Gamma=1$ and the assumption that $\{D,\Gamma\}=0$ is dropped.
    
    For an arbitrary spectral triple, we have $|D|^k\Gamma = \Gamma|D|^k$ for all $k$, and therefore $\Gamma:\mathrm{dom}(D^k)\to \mathrm{dom}(D^k)$ for all $k$. Hence $\Gamma:H_\infty\to H_\infty$.
    
    The following assertion is well-known in the compact case (see e.g. \cite{CPRS1} and \cite{PS-crelle}). To the best of our knowledge, no proof has been published in the locally compact case. 
    We supply a proof in Section \ref{fredholm section}, Proposition \ref{f der def}.
    \begin{prop*}
        {Let $p \in \mathbb{N}$.} If $(\mathcal{A},H,D)$ is a $p-$dimensional spectral triple satisfying Hypothesis \ref{main assumption}, then $[F,a]\in\mathcal{L}_{p,\infty}$ for all $a\in\mathcal{A}.$
    \end{prop*}
        
    Let $\mathcal{A}^{\otimes(p+1)}$ denote the $(p+1)$-fold algebraic tensor power of $\mathcal{A}$. We now define the two important mappings $\mathrm{ch}$ and $\Omega$.
    \begin{defi}\label{ch omega def} 
        Suppose that $D$ has a spectral gap at $0.$
        Define the multilinear mappings ${\rm ch}:\mathcal{A}^{\otimes (p+1)}\to\mathcal{L}_{\infty}$ and $\Omega:\mathcal{A}^{\otimes (p+1)}\to\mathcal{L}_{\infty}$ 
        on elementary tensors $a_0\otimes a_1\otimes\cdots\otimes a_p \in \mathcal{A}^{\otimes(p+1)}$ by
        \begin{equation*}
            \mathrm{ch}(a_0\otimes a_1\otimes\cdots \otimes a_p) = \Gamma F\prod_{k=0}^p[F,a_k]
        \end{equation*}
        and
        \begin{equation*}
            \Omega(a_0\otimes a_1\otimes \cdots \otimes a_p) = \Gamma a_0\prod_{k=1}^p \partial (a_k).
        \end{equation*}
    \end{defi}

    If $(\mathcal{A},H,D)$ is $p$-dimensional, then it follows from Proposition \ref{f der def} and the H\"older inequality \eqref{weak-type Holder} that ${\rm ch}(c)\in\mathcal{L}_{\frac{p}{p+1},\infty}\subset\mathcal{L}_1$
    for all $c\in\mathcal{A}^{\otimes (p+1)}.$ This permits the following definition:

    \begin{defi}\label{chern character zero kernel def}
        If $(\mathcal{A},H,D)$ is $p-$dimensional spectral triple satisfying Hypothesis \ref{main assumption} and if kernel of $D$ is trivial, then Connes-Chern character $\mathrm{Ch}:\mathcal{A}^{\otimes (p+1)}\to\mathbb{C}$ is defined by setting
        \begin{equation*}
            {\rm Ch}(c)=\frac12{\rm Tr}({\rm ch}(c)),\quad c\in\mathcal{A}^{\otimes (p+1)}.
        \end{equation*}
    \end{defi}
    
    In general, however, Chern character cannot be defined in terms of $F$ because $F^2\neq 1$ when $D$ has non-trivial kernel. 
    In order to ensure that ${\rm Ch}$ is a cyclic cocycle (in the sense of \cite[2.1.4]{Loday-cyclic-homology}), we require that $F^2 = 1$.
    Hence we define the Chern character of a general spectral triple in terms of another $F_0$ such that $F_0=F_0^*=F_0^2.$ For this purpose, we use a doubling trick.
    
    \begin{defi}\label{doubling definition}
        Let $(\mathcal{A},H,D)$ be a spectral triple with grading $\Gamma$, and let $P$ be the projection onto $\ker(D)$. Consider the following unitary self-adjoint operators on the Hilbert space $H_0=\mathbb{C}^2\otimes H$ defined by:
        \begin{align*}
                F_0 := \begin{pmatrix} F & P \\ P & -F\end{pmatrix}\\
            \Gamma_0 := \begin{pmatrix} \Gamma & 0 \\ 0 & (-1)^{\rm deg}\Gamma\end{pmatrix}.
        \end{align*}
        Here, ${\rm deg}=1$ for even triples and ${\rm deg}=0$ for odd triples. The algebra $\mathcal{A}$ is represented on $H_0$ by:
        \begin{equation*}
            \pi(a) = \begin{pmatrix}
                        a & 0 \\ 0 & 0
                     \end{pmatrix}.
        \end{equation*}
        For an elementary tensor $a_0\otimes\cdots a_p \in \mathcal{A}^{\otimes (p+1)}$ we set:
        \begin{equation*}
            {\rm ch}_0(a_0\otimes \cdots\otimes a_p) = \Gamma_0F_0\prod_{k=0}^p [F_0,\pi(a_k)].
        \end{equation*}
    \end{defi}

    
    \begin{defi}\label{chern character def}
    If $(\mathcal{A},H,D)$ is $p-$dimensional spectral triple satisfying Hypothesis \ref{main assumption}, then Connes-Chern character $\mathrm{Ch}:\mathcal{A}^{\otimes (p+1)}\to\mathbb{C}$ is defined by setting
        \begin{equation*}
            {\rm Ch}(c)=\frac12({\rm Tr_2}\otimes{\rm Tr})({\rm ch}_0(c)),\quad c\in\mathcal{A}^{\otimes (p+1)}.
        \end{equation*}
        Here, $\mathrm{Tr}_2$ denotes the $2\times 2$ trace on matrices.
    \end{defi}
    
    Note that Definition \ref{doubling definition} does not conflict with Definition \ref{ch omega def}: if $\ker(D)$ is trivial (i.e., $P=0$) then both definitions of $\mathrm{Ch}$ coincide.

    Strictly speaking, the Connes-Chern character is conventionally defined to be the class of $\mathrm{Ch}$ in periodic cyclic cohomology. This distinction is not relevant to the results of this paper, so in the sequel
    we will consider $\mathrm{Ch}$ merely as a multilinear functional as above.
   
    \begin{rem}
        We have opted to define the Chern character of a spectral triple $(\mathcal{A},H,D)$ in terms of the "doubled" operator $F_0$ in Definition \ref{doubling definition}. This definition is different
	from earlier work such as \cite[Definition 6]{CPRS1} and \cite[Definition 2.23]{CGRS2}. In those papers the Chern character of a spectral triple is defined to be the Chern character of any Fredholm module
        equivalent to the pre-Fredholm module $(\mathcal{A},H,F)$. It is known that the class in periodic cyclic cohomology of the chern character defined in that way is independent of the choice of Fredholm module equivalent to $(\mathcal{A},H,F)$ \cite[Section 5, Lemma 1]{Connes-differential-geometry}.
        
        In order to avoid technicalities, we have defined the Chern character in terms of a specific Fredholm module $(\pi(\mathcal{A}),H_0,F_0)$. This has the advantage of simplicity of presentation, and makes
        no difference in regards to the character formula. A reader interested in a more refined definition of the Chern character in periodic cyclic cohomology may wish to consult \cite{Connes-differential-geometry, CPRS1, CGRS2}.
    \end{rem}
    
\subsection{Discussion of smoothness}\label{smoothness discussion}
    It is tempting to define smoothness only in terms of $\partial$, without reference to $\delta$. One might naively suggest that $(\mathcal{A},H,D)$ is smooth if for all $n \geq 0$
    we have $a\cdot \mathrm{dom}(D^n)\to \mathrm{dom}(D^n)$ and the $n$th iterated commutator $[D,[D,[\cdots,[D,a]\cdots]$ extends to a bounded operator on $H$.
    
    However this condition does not hold for even the simplest spectral triples. A standard spectral triple associated to the $2$-torus $\mathbb{T}^2$ is
    $(1\otimes C^\infty(\mathbb{T}^2), L_2(\mathbb{T}^2,\mathbb{C}^2),D)$, where the $L_2$ space is defined with respect to the Haar measure, the algebra $C^\infty(\mathbb{T}^2)$
    {  of smooth complex valued functions on $\mathbb{T}^2$}
    acts on $L_2(\mathbb{T}^2,\mathbb{C}^2)$ by pointwise multiplication, and the Dirac operator $D$ is defined by:
    \begin{equation*}
        D = -i\gamma_1\otimes \partial_1-i\gamma_2\otimes \partial_2,
    \end{equation*}
    where $\partial_1$ and $\partial_2$ are differentiation with respect to the first and second coordinates on $\mathbb{T}^2$
    and $\gamma_1,\gamma_2$ are $2\times 2$ complex matrices satisfying $\gamma_j\gamma_k+\gamma_k\gamma_j = 2\delta_{j,k}1$, $j,k = 1,2$.
    
    Then if $f \in C^\infty(\mathbb{T}^2)$,
    \begin{align*}
        [D,[D,1\otimes M_f]] &= -[\gamma_1\otimes\partial_1+\gamma_2\otimes \partial_2,\gamma_1\otimes M_{\partial_1 f}+\gamma_2\otimes M_{\partial_2 f}]\\
                             &= -1\otimes M_{\partial_1^2f+\partial_2^2f}+2\gamma_1\gamma_2\otimes (M_{\partial_2f} \partial_1+M_{\partial_1f} \partial_2)
    \end{align*}
    However this operator is typically unbounded: if we choose $f(z_1,z_2) = z_1$, then $[D,[D,1\otimes M_f]] = 2\gamma_1\gamma_2\otimes(\partial_2)$ which is unbounded.

    This example breaks the implication: ``if $f \in C^\infty(\mathbb{T}^2)$, then $[D,[D,1\otimes M_f]]$ extends to a bounded linear operator".

\subsection{Discussion of dimension}\label{dimension discussion}
    As we have defined it, we say that a spectral triple $(\mathcal{A},H,D)$ is $p$-dimensional if for all $a \in \mathcal{A}$ the operators $a(D+i)^{-p}$ and $\partial(a)(D+i)^{-p}$ are in $\mathcal{L}_{1,\infty}$.
    
    An alternative definition, also used in the literature, is to say that $(\mathcal{A},H,D)$ is $p$-dimensional if $a(D+i)^{-1}$ and $\partial(a)(D+i)^{-1}$ are in $\mathcal{L}_{p,\infty}$. 
    An example of a definition along these lines is \cite[Definition 3.1]{gayral-moyal}.
    Clearly in the case where $\mathcal{A}$ is unital these definitions are equivalent, since $(D+i)^{-1} \in \mathcal{L}_{p,\infty}$ if and only if $(D+i)^{-p} \in \mathcal{L}_{1,\infty}$. 
    However in the non-unital case, the distinction may be important. 
    

\subsection{Hochschild (co)homology}\label{hochschild subsection}
    Hochschild homology and cohomology provide noncommutative generalisations of the notion of differential forms and de Rham currents
    respectively. A detailed exposition of the theory of Hochschild (co)homology and its relationship with noncommutative geometry may be found in \cite{Quillen,Loday-cyclic-homology}.
    
    Let $A$ be a (possibly non-unital) algebra. The Hochschild complex is a chain complex:
    \begin{equation*}
        \cdots \xrightarrow{b} A\otimes A\otimes A\otimes A \xrightarrow{b} A\otimes A\otimes A \xrightarrow{b} A \otimes A\xrightarrow{b} A.
    \end{equation*}
    For $n \geq 1$, the $n$th entry in the Hochschild chain complex is the $n$th tensor power $A^{\otimes n}$. The Hochschild boundary operator $b:A^{\otimes (n+1)}\to A^{\otimes n}$
    is defined on elementary tensors $a_0\otimes a_1\otimes\cdots \otimes a_n$ by:
    \begin{align*}
        b(a_0\otimes a_1\otimes \cdots \otimes a_n) &= a_0a_1\otimes a_2\otimes \cdots\otimes a_n + \sum_{k=1}^{n-1} (-1)^k a_0\otimes \cdots \otimes a_ka_{k+1}\otimes\cdots \otimes a_n\\
                                                    &+ (-1)^na_na_0\otimes a_1\cdots\otimes a_{n-1}.
    \end{align*}
    It can be verified that $b^2 = 0$, so the Hochschild complex is indeed a chain complex. An element $c \in A^{\otimes (n+1)}$ such that $bc = 0$
    is called a Hochschild cycle. For example, when $n = 1$, an elementary tensor $a_0\otimes a_1$ is a Hochschild cycle if and only if $b(a_0\otimes a_1) = a_0a_1-a_1a_0 = 0$, i.e. if $a_0$ and
    $a_1$ commute.
    
    The Hochschild cochain complex is defined in a similar way: Let $C_n(A)$ denote the space of continuous multilinear functionals from $A^{\otimes n} \to \mathbb{C}$. The Hochschild cochain complex is,
    \begin{equation*}
        C_1(A)\xrightarrow{b} C_2(A)\xrightarrow{b} C_3(A) \xrightarrow{b} \cdots
    \end{equation*}
    where the Hochschild coboundary operator $b$ is defined as follows: if $\theta:A^{\otimes n}\to \mathbb{C}$, then $b\theta:A^{\otimes (n+1)}\to \mathbb{C}$
    is given on an elementary tensor $a_0\otimes a_1\otimes\cdots \otimes a_n$ by
    \begin{align*}
        (b\theta)(a_0\otimes\cdots\otimes a_n)&=\theta(a_0a_1\otimes a_2\otimes\cdots\otimes a_n)\\
        &+\sum_{k=1}^{n-1}(-1)^k\theta(a_0\otimes a_1\otimes\cdots \otimes a_{k-1}\otimes a_ka_{k+1}\otimes a_{k+2}\otimes\cdots\otimes a_n)\\
        &+(-1)^n\theta(a_na_0\otimes a_1\otimes a_2\otimes\cdots\otimes a_{n-1}).
    \end{align*}
    Put simply, for $c \in A^{\otimes (n+1)}$ and $\theta \in C_n(A)$, the Hochschild boundary and coboundary operators are linked by,
    \begin{equation}\label{stokes formula}
        (b\theta)(c) = \theta(bc).
    \end{equation}
    A cochain $\theta \in C^n(A)$ is called a Hochschild cocycle if $b\theta = 0$. Due to \eqref{stokes formula}, a Hochschild coboundary
    vanishes on every Hochschild cycle.    
%
%
%
    \begin{rem}
        One must distinguish between Hochschild (co)homology as we have just defined it, and the analogous continuous Hochschild homology \cite[Section 8.5]{GVF}, \cite{Pflaum-continuous-hh-1998}. Continuous Hochschild (co)homology is defined with topological tensor products in place of algebraic tensor products. In this text we are only concerned with algebraic tensor products.
    \end{rem}
    
\section{Weak integrals and double operator integrals}
    
\subsection{Weak integration in $\mathcal{L}_{\infty}$}\label{weak int}
    This section concerns the theory of ``weak operator topology integrals" of operator valued functions. The following definitions, and the subsequent construction of weak integrals, are folklore. We provide suitable references whenever they
    exist, otherwise we supply a proof. For example, one can look at \cite[Definition 3.26]{rudin}, and consider the example where the topological vector space $X$ there is $\mathcal{L}_{\infty}$ equipped with the strong operator topology. Every
    continuous linear functional on $X$ can be written as a linear combination of $x\to\langle x\xi,\eta\rangle,$ $\xi,\eta\in H.$

    \begin{defi}\label{weak meas def} 
        A function $f:\ \mathbb{R}\to \mathcal{L}_{\infty}$ is measurable in the weak operator topology if, for every pair of vectors $\xi,\eta\in H,$ the function
        $$s\to\langle f(s)\xi,\eta\rangle,\quad s\in\mathbb{R},$$
        is (Lebesgue) measurable.
    \end{defi}
    
    For a function $f$, measurable in the weak operator topology, there is a notion of ``pointwise norm". 
    Namely, the scalar-valued mapping
    \begin{equation*}
        s\mapsto \|f(s)\|_\infty := \sup_{\|\xi\|,\|\eta\| \leq 1} |\langle f(s)\xi,\eta\rangle|, \quad s \in \mathbb{R}
    \end{equation*}
    is Lebesgue measurable. { Here it is crucial that we work with separable Hilbert spaces, as otherwise it is not clear whether the function $s\mapsto \|f(s)\|_\infty$ is measurable.}

    Suppose that a function $f:\mathbb{R}\to\mathcal{L}_{\infty}$ is measurable in the weak operator topology. We say that $f$ is integrable in the weak operator topology if
    \begin{equation}\label{necessary-condition}
        \int_{\mathbb{R}}\|f(s)\|_{\infty}ds < \infty.
    \end{equation}
    In particular, for all $\xi,\eta \in H$, we have
    \begin{equation*}
        \int_{\mathbb{R}} |\langle f(s)\xi,\eta\rangle|\,ds < \infty.
    \end{equation*}
    
    Hence for a function $f$ satisfying \eqref{necessary-condition}, we may therefore define the sesquilinear form
    \begin{equation*}
        (\xi,\eta)_f := \int_{\mathbb{R}} \langle f(s)\xi,\eta\rangle\,ds,\quad \xi,\eta \in H.
    \end{equation*}
    It then follows that:
    \begin{align*}
        |(\xi,\eta)_f|  &\leq \int_{\mathbb{R}} |\langle f(s)\xi,\eta\rangle|\,ds\\
                        &\leq \int_{\mathbb{R}} \|f(s)\|_{\infty}\|\xi\|\|\eta\|\,ds\\
                        &= \left(\int_{\mathbb{R}} \|f(s)\|_{\infty}\,ds\right)\|\xi\|\|\eta\|.
    \end{align*}
    Thus for a fixed $\xi \in H$, the mapping $\eta \mapsto (\xi,\eta)_f$ defines a bounded linear functional on $H$. Hence
    there is a unique $x_\xi \in H$ such that $(\xi,\eta)_f = \langle x_\xi,\eta\rangle $ for all $\eta \in H$.
    
    One can easily verify that the map $\xi\mapsto x_\xi$ is linear, and furthermore
    \begin{align*}
        \|x_\xi\|^2 &= \langle x_\xi,x_\xi\rangle\\
                    &= |(\xi,x_\xi)_f|\\
                    &\leq \left(\int_{\mathbb{R}} \|f(s)\|_{\infty}\,ds\right)\|\xi\|\|x_\xi\|.
    \end{align*}
    So the mapping $\xi\to x_\xi$ is bounded. Let $T$ be the unique bounded linear operator such that $x_\xi = T\xi$, we now define
    \begin{equation}\label{wot integral definition}
        \int_{\mathbb{R}} f(s)\,ds := T.
    \end{equation}
    
    Due to the above computation, we have that
    \begin{equation*}
        \left\|\int_{\mathbb{R}} f(s)\,ds\right\|_\infty\leq \int_{\mathbb{R}} \|f(s)\|_{\infty}\,ds.
    \end{equation*}
    Furthermore, we have that if $A \in \mathcal{L}_{\infty}$, and $f$ is integrable in the weak operator topology, then 
    $s\mapsto Af(s)$ is also integrable in the weak operator topology, and
    \begin{equation*}
        \int_{\mathbb{R}} Af(s)\,ds = A\int_{\mathbb{R}} f(s)\,ds.
    \end{equation*}
    
    Closely related to the weak integral is the Bochner integral: indeed, if $(\mathcal{E},\|\cdot\|_{\mathcal{E}})$ is a normed ideal in $\mathcal{L}_\infty$
    and $f:\mathbb{R}\to \mathcal{E}$ is Bochner integrable, then it is integrable in the weak operator topology and the weak integrals
    and Bochner integrals coincide, since if $f$ is weakly $\mathcal{E}$-valued measurable, then it is weak operator topology measurable, and if $\|f\|_{\mathcal{E}}$
    is integrable then $\|f\|_\infty$ is integrable.
    
\subsection{Properties of the weak integral}\label{peter subsection}

    The authors thank Professor Peter Dodds for his assistance with the arguments in this subsection.

    \begin{lem}\label{peter lemma} Let $s\to a(s),$ $s\in\mathbb{R},$ be continuous in the weak operator topology. If $a(s)\in\mathcal{L}_1$ for every $s\in\mathbb{R}$ and if
        $$\int_{\mathbb{R}}\|a(s)\|_1ds<\infty,$$
        then $a(s)$ is integrable in the weak operator topology, $\int_{\mathbb{R}}a(s)ds \in \mathcal{L}_1$ and
        $$\Big\|\int_{\mathbb{R}}a(s)ds\Big\|_1\leq\int_{\mathbb{R}}\|a(s)\|_1ds,\quad {\rm Tr}\Big(\int_{\mathbb{R}}a(s)ds\Big)=\int_{\mathbb{R}}{\rm Tr}(a(s))ds.$$
    \end{lem}
    \begin{proof}
        Since $\|a(s)\|_\infty \leq \|a(s)\|_{1}$ for all $s \in \mathbb{R}$, we have
        \begin{align*}
            \int_{\mathbb{R}} \|a(s)\|_\infty \,ds &\leq \int_{\mathbb{R}} \|a(s)\|_{1}\,ds\\
                                            &< \infty
        \end{align*}
        so that condition \eqref{necessary-condition} holds, so $s\mapsto a(s)$ is integrable in the weak operator topology. Thus, { let $A$ be the bounded linear
        operator on $H$ given by
        $$A : =\int_{\mathbb{R}}a(s)ds$$
        in the sense of \eqref{wot integral definition}. Next, we shall show that $A \in \mathcal{L}_1$.}

        Let $A=U|A|$ be a polar decomposition of $A$. 
        For an arbitrary finite rank projection $p,$ we have
        $$p|A|p=\int_{\mathbb{R}}pU^*a(s)pds$$
        Since $p$ is finite rank, the algebra $p\mathcal{L}_{\infty}p,$ is finite dimensional, and so here the weak operator topology coincides with the norm topology. 
        Hence, the mapping $s\mapsto pU^*a(s)p$ is continuous in the norm topology. Since on the algebra $p\mathcal{L}_{\infty}p$ the classical trace is a continuous functional with respect to the uniform norm, it follows that
        $${\rm Tr}(p|A|p)=\int_{\mathbb{R}}{\rm Tr}(pU^*a(s)p)ds.$$
        Thus,
        \begin{align*}
            {\rm Tr}(p|A|p) &\leq \int_{\mathbb{R}}|{\rm Tr}(pU^*a(s)p)|ds\\
                            &\leq \int_{\mathbb{R}}\|a(s)\|_1ds.
        \end{align*}
        Taking the supremum over all finite rank projections $p,$ we arrive at
        $$\|A\|_1\leq\int_{\mathbb{R}}\|a(s)\|_1ds.$$
        This proves the first assertion.

        Choose now a sequence $\{p_n\}_{n\geq1}$ of finite rank projections such that $p_n\uparrow 1.$ 
        We have
        $${\rm Tr}(p_nAp_n)=\int_{\mathbb{R}}{\rm Tr}(p_na(s)p_n)ds.$$
        Clearly, ${\rm Tr}(p_na(s)p_n)\to{\rm Tr}(a(s))$ as $n\to\infty$ for every $s\in\mathbb{R}.$ Since the function $s \mapsto \mathrm{Tr}(a(s))$ is integrable, we can apply the dominated convergence theorem to obtain
        $$\int_{\mathbb{R}}{\rm Tr}(p_na(s)p_n)ds\to\int_{\mathbb{R}}{\rm Tr}(a(s))ds.$$
        On the other hand, we have ${\rm Tr}(p_nAp_n)\to{\rm Tr}(A)$ as $n\to\infty.$ This proves the second assertion.
    \end{proof}
    
    According to the preceding lemma, if $a:\mathbb{R}\to \mathcal{L}_1$ is continuous and $\int_{\mathbb{R}} \|a(s)\|_1 \,ds < \infty$, then we
    have that $\int_{\mathbb{R}} a(s)\,ds \in \mathcal{L}_1$. The following lemma shows that the same implication holds when $\mathcal{L}_1$ is replaced
    by $\mathcal{L}_{r,\infty}$ for any $r > 1$.
    
    \begin{prop}\label{peter norm lemma} 
        Let $s\to a(s),$ $s\in\mathbb{R},$ be continuous in the weak operator topology. Fix $r > 1$, and suppose that for all $s$ we have $a(s)\in\mathcal{L}_{r,\infty}$.
        If $\int_{\mathbb{R}} \|a(s)\|_{r,\infty}\,ds < \infty$ then $\int_{\mathbb{R}} a(s)\,ds \in \mathcal{L}_{r,\infty}$, where the integral is understood in a weak sense, and we have a bound:
        \begin{equation*}
            \left\|\int_{\mathbb{R}} a(s)\,ds\right\|_{r,\infty} \leq \frac{r}{r-1}\int_{\mathbb{R}} \|a(s)\|_{r,\infty}\,ds.
        \end{equation*}
    \end{prop}
    \begin{proof}
        Similar to the $\mathcal{L}_{1}$ case, since $\|a(s)\|_{\infty} \leq \|a(s)\|_{r,\infty}$, we have that $\int_{\mathbb{R}} \|a(s)\|_{\infty}\,ds < \infty$,
        and so condition \eqref{necessary-condition} holds. Hence, $s\mapsto a(s)$ is integrable in the weak operator topology. Let $A := \int_{\mathbb{R}} a(s)\,ds$ { in the sense of \eqref{wot integral definition}}. Let $A = U|A|$ be a polar decomposition of $A$, and let $p$ be a rank $n$ projection, $n\geq 1$. Then,
        \begin{equation*}
            p|A|p = \int_{\mathbb{R}} pU^*a(s)p\,ds.
        \end{equation*}
        Thus,
        \begin{align*}
            \mathrm{Tr}(p|A|p) &\leq \int_{\mathbb{R}} |\mathrm{Tr}(pU^*a(s)p)|\,ds\\
                       &\leq \int_{\mathbb{R}} \|pU^*a(s)p\|_1\,ds.
        \end{align*}
        The latter integral converges because $\|pU^*a(s)p\|_1 \leq n\|a(s)\|_\infty$. Now,
        \begin{align*}
            \|pU^*a(s)p\|_1 &\leq \sum_{k=0}^{n-1} \mu(k,a(s))\\
                            &\leq \|a(s)\|_{r,\infty} \sum_{k=0}^{n-1} (k+1)^{-1/r}.
        \end{align*}
        The right hand side depends only on $n$ and not $p$, so we may take the supremum
        over all projections of rank $n$ to obtain:
        \begin{equation*}
            \sum_{k=0}^{n-1} \mu(k,A) \leq \int_{\mathbb{R}} \|a(s)\|_{r,\infty}\,ds\cdot \sum_{k=0}^{n-1} (k+1)^{-1/r}.
        \end{equation*}
        We can bound the latter sum as:
        \begin{align*}
            \sum_{k=0}^{n-1} (k+1)^{-1/r} &\leq 1+\int_1^n t^{-1/r}\,dt\\
                                          &\leq \frac{r}{r-1} n^{1-\frac{1}{r}}.
        \end{align*}
        Therefore:
        \begin{equation*}
            \sum_{k=0}^{n-1} \mu(k,A) \leq \int_{\mathbb{R}} \|a(s)\|_{r,\infty}\,ds \frac{r}{r-1}n^{1-\frac{1}{r}}.
        \end{equation*}
        Hence,
        \begin{equation*}
            n\mu(n-1,A) \leq \frac{r}{r-1}n^{1-\frac{1}{r}}\int_{\mathbb{R}} \|a(s)\|_{r,\infty}\,ds.
        \end{equation*}
        Multiplying through by $n^{\frac{1}{r}-1}$, and taking the supremum over $n$, it follows that
        \begin{equation*}
            \sup_{n\geq 1} n^{\frac{1}{r}}\mu(n-1,A) \leq \frac{r}{r-1}\int_{\mathbb{R}} \|a(s)\|_{r,\infty}\,ds.
        \end{equation*}
        So by the definition of the quasinorm on $\mathcal{L}_{r,\infty}$, the assertion follows.
\end{proof}

\section{Double operator integrals}\label{doi} 
    Here, we state the definition and basic properties of double operator integrals. This theory was initiated by the work
    of Birman and Solomyak \cite{Birman-Solomyak-I,Birman-Solomyak-II,Birman-Solomyak-III}, and more recent summaries
    of the theory may be found in \cite{Birman-Solomyak-2003, Peller-doi-2016}.
    

    Heuristically, given self-adjoint operators $X$ and $Y$ with spectra $\sigma(X)$ and $\sigma(Y)$, spectral resolutions $E_X$ and $E_Y$ and a bounded measurable function $\phi$ on $\sigma(X)\times \sigma(Y)$, the double operator integral $T^{X,Y}_{\phi}$ 
    applied to an operator $A \in \mathcal{L}_{\infty}$ is given by the formula:
    $$T_{\phi}^{X,Y}(A)=\iint_{\sigma(X)\times \sigma(Y)} \phi(\lambda,\mu)dE_X(\lambda)AdE_Y(\mu).$$
    The formal expression for $T_{\phi}^{X,Y}$ is well defined as a bounded operator on the Hilbert-Schmidt class $\mathcal{L}_2$. The theory of double operator integrals is primarily concerned
    with defining $T_\phi^{X,Y}$ on other ideals.
    This is not possible for arbitrary bounded measurable functions $\phi$, so we must restrict attention to the following class of ``good" functions. 
    
    That is, we assume that $\phi$ admits a representation
    \begin{equation}\label{integral tensor product}
        \phi(\lambda,\mu) = \int_{\Omega} a(\lambda,s)b(\mu,s)\,d\kappa(s),\quad \lambda \in \sigma(X), \mu \in \sigma(Y)
    \end{equation}
    where $(\Omega,\kappa)$ is a measure space, and where
    \begin{equation}\label{doi sufficient condition}
        \int_{\Omega} \sup_{\lambda \in \sigma(X)}|a(\lambda,s)|\sup_{\mu \in \sigma(Y)} |b(\mu,s)|\,d\kappa(s) < \infty.
    \end{equation}
    For such functions $\phi$, we may define
    \begin{equation}\label{doi definition}
        T_{\phi}^{X,Y}(A) := \int_{\Omega} a(X,s)Ab(Y,s)\,d\kappa(s)
    \end{equation}
    where the operators $a(X,s)$ and $b(Y,s)$ are defined by Borel functional calculus, and the integral can be understood in the weak operator topology.
    
    The following is proved in \cite[Theorem 4]{PS-crelle}:
    \begin{thm}
        If $\phi$ admits a decomposition as in \eqref{integral tensor product}, then the operator $T_{\phi}^{X,Y}$ is a bounded linear map from:
        \begin{enumerate}[{\rm (a)}]
            \item{} $\mathcal{L}_{\infty}$ to $\mathcal{L}_{\infty}$;
            \item{} $\mathcal{L}_{1}$ to $\mathcal{L}_{1}$;
            \item{} $\mathcal{L}_{r}$ to $\mathcal{L}_r$, for all $r \in (1,\infty)$;
            \item{} $\mathcal{L}_{r,\infty}$ to $\mathcal{L}_{r,\infty}$ for all $r \in (1,\infty)$.
        \end{enumerate}
    \end{thm}

    One of the key properties of double operator integrals is that they respect algebraic operations (see e.g.\cite[Proposition 2.8]{PSW}). Namely,
    \begin{equation}\label{doi algebraic}
        T_{\phi_1+\phi_2}^{X,Y}=T_{\phi_1}^{X,Y}+T_{\phi_2}^{X,Y},\quad T_{\phi_1\cdot \phi_2}^{X,Y}=T_{\phi_1}^{X,Y}\circ T_{\phi_2}^{X,Y}.
    \end{equation}

    If, in \eqref{integral tensor product} we take $\Omega$ to be a one-point set, then $\phi(\lambda,\mu) = a(\lambda)b(\mu)$ and
    \begin{equation}\label{separated variables in doi}
        T_{\phi}^{X,Y}(A)=a(X)Ab(Y).
    \end{equation}

\section{Fourier transform conventions}

    We follow the convention that the Fourier transform of a function $g \in L_1(\mathbb{R})$ is defined by the formula
    \begin{equation*}
        \mathcal{F}(g)(t) := (2\pi)^{-1/2}\int_{\mathbb{R}} g(s)e^{-its}\,ds
    \end{equation*}
    So that the inverse Fourier transform is given for $h \in L_1(\mathbb{R})$ by,
    \begin{equation*}
        \mathcal{F}^{-1}(h)(s) := (2\pi)^{-1/2}\int_{\mathbb{R}} h(t)e^{its}\,dt
    \end{equation*}
    and so that $\mathcal{F}$ extends to a unitary operator on $L_2(\mathbb{R})$.
    
    We often make use of the fact that if $g \in L_1(\mathbb{R})$ satisfies $g,g' \in L_2(\mathbb{R})$ then $\mathcal{F}g \in L_1(\mathbb{R})$ \cite[Lemma 7]{PS-crelle}.
    Here, the derivative $g'$ may be defined in a distributional sense. 
    
    Everywhere in the text, the symbol $\hat{g}$ denotes $(2\pi)^{-1/2}\mathcal{F}(g)$. This allows us to write for $g \in L_1(\mathbb{R})$ with $\mathcal{F}(g) \in L_1(\mathbb{R})$:
    \begin{equation*}
        g(t) = \int_{\mathbb{R}} \hat{g}(s)e^{its}\,ds.
    \end{equation*}
    We caution the reader that $\hat{g}$ does not denote the Fourier transform, but its rescaling by $(2\pi)^{-1/2}$.

\chapter{Spectral Triples: Basic properties and examples}\label{examples chapter}
    
    This chapter is primarily concerned with Hypothesis \ref{main assumption}. We study the consequences of this hypothesis, and also show that it
    is satisfied for two important classes of examples.
    
    We begin with the proof of Proposition \ref{f der def}, an important prerequisite to the definition
    of the Chern character (Definition \ref{chern character def}).     
    Next, we show that Hypothesis \ref{main assumption} is equivalent to a modified set of assumptions, Hypothesis \ref{replacement assumption}.
    { Hypothesis \ref{replacement assumption} is stated in terms of an operator $\Lambda$ (given in Definition \ref{lambda def}) rather than $\delta$. This has
    the advantage of making Hypothesis \ref{replacement assumption} more easily verified in the classes of examples studied in this chapter.}
    
    The remainder of the chapter is devoted to demonstrating that the assumptions made in Hypothesis \ref{main assumption} are satisfied for spectral triples associated to the following classes of examples:
    \begin{enumerate}[{\rm (a)}]
        \item{}\label{manifolds} Complete Riemannian manifolds.
        \item{}\label{moyal planes} Noncommutative Euclidean spaces (also known as Moyal planes or Moyal-Groenwald planes in the $2$-dimensional case).
    \end{enumerate}
    
    We re-emphasise that Hypothesis \ref{main assumption} is automatically satisfied for smooth $p$-dimensional unital spectral triples, and therefore 
    we concern ourselves with showing that it is satisfied for non-unital algebras.
    
    The first class of examples \eqref{manifolds} is purely commutative. For the Dirac operator in these examples, we use the Hodge-Dirac operator (see \cite{krym} or \cite{rosenberg}).
    In \cite{CGRS2}, spectral triples for noncompact Riemannian manifolds were studied under the significant restriction that they have bounded geometry: this is a global
    geometric property which we are able to avoid by working in local coordinates. Earlier, Rennie had studied noncompact Riemannian spin manifolds which are not necessarily of bounded geometry by similar methods \cite[Section 5]{Rennie-2004}. It is hoped that by including such a wide class of manifolds we may demonstrate the applicability of noncommutative methods in "classical" (commutative) geometry.
    
    The second example \eqref{moyal planes} is one of the most heavily studied classes of non-unital and strictly noncommutative spectral triples. A detailed exposition
    of the noncommutative Euclidean spaces may be found in \cite{gayral-moyal}.
    
\section{A spectral triple defines a Fredholm module}\label{fredholm section}

    This section is devoted to the proof of Proposition \ref{f der def}. We prove this in several steps, initially working with the assumption that $D$ has a spectral
    gap at $0$ (i.e., that $D$ has bounded inverse). We later show how this assumption can be removed. 
    
    Note that if $D$ has a spectral gap at $0$, then $F = D|D|^{-1} = |D|^{-1}D$.

    \begin{rem}\label{sp gap fact} 
        Let $(\mathcal{A},H,D)$ be a spectral triple satisfying Hypothesis \ref{main assumption}. Suppose $D$ has a spectral gap at $0.$ For every $a\in\mathcal{A},$ and all $k\geq 1$ we have
        the following four inclusions:
        \begin{align*}
            a|D|^{-p}\in \mathcal{L}_{1,\infty},&\quad \partial(a)|D|^{-p} \in \mathcal{L}_{1,\infty},\\
            \delta^k(a)|D|^{-p-1} \in \mathcal{L}_{1},&\quad \partial(\delta^k(a))|D|^{-p-1}\in \mathcal{L}_1.
        \end{align*}
    \end{rem}
    \begin{proof} 
        All four inclusions follow from the observation that since by assumption $|D|$ is invertible, the operator $\frac{D+i}{|D|}:\mathrm{dom}(D)\to H$ has bounded extension. 
        Since $(\mathcal{A},H,D)$ is $p$-dimensional, we have
        \begin{equation*}
            a(D+i)^{-p},\;\partial(a)(D+i)^{-p} \in \mathcal{L}_{1,\infty}
        \end{equation*}
        and multiplying by (the bounded extension of) $\left(\frac{D+i}{|D|}\right)^p$ yields the first two inclusions.
        
        The second pair of inclusions follow from Hypothesis \ref{main assumption}.\eqref{ass2}: we have $\delta^k(a)(D+i)^{-p-1} \in \mathcal{L}_1$ and $\partial(\delta^k(a))(D+i)^{-p-1} \in \mathcal{L}_1$.
        Then simply multiplying by $\left(\frac{D+i}{|D|}\right)^{p+1}$ again yields the result.
    \end{proof}

    Note that the preceding lemma showed that $\partial(a)|D|^{-p} \in \mathcal{L}_{1,\infty}$. We require a little more effort to show that $\delta(a)|D|^{-p} \in \mathcal{L}_{1,\infty}$.
    \begin{lem}\label{lambda lemma} 
        Let $(\mathcal{A},H,D)$ be a spectral triple satisfying Hypothesis \ref{main assumption}. Suppose $D$ has a spectral gap at $0.$ Then for all $a\in\mathcal{A},$ we have $\delta(a)|D|^{-p}\in\mathcal{L}_{1,\infty}.$
    \end{lem}
    \begin{proof} 
        Using \eqref{favourite commutator identity}, we have the following equality of operators on $H_\infty$:
        \begin{align}\label{lambda 11}
            [|D|^{-1},\partial(a)] &= -|D|^{-1}[|D|,\partial(a)]|D|^{-1}\nonumber\\
                                   &= -|D|^{-1}\partial(\delta(a))|D|^{-1}.
        \end{align}
        Where the last equality uses the fact that $\delta$ and $\partial$ commute (from Lemma \ref{delta and partial commute}). Similarly,
        \begin{align}\label{lambda 12}
            [|D|^{-1},\partial(\delta(a))] &= -|D|^{-1}[|D|,\partial(\delta(a))]|D|^{-1} \nonumber\\
                                           &= -|D|^{-1}\partial(\delta^2(a))|D|^{-1}.
        \end{align}
        Additionally, working with operators on $H_\infty$:
        \begin{align*}
            |D|^{-1}[D^2,a] &= |D|^{-1}\cdot(D\partial(a)+\partial(a)D)\\
                            &= F\partial(a)+|D|^{-1}\partial(a)D\\
                            &= F\partial(a)+\partial(a)|D|^{-1}D+[|D|^{-1},\partial(a)]D.
        \end{align*}
        Now applying \eqref{lambda 11}:    
        \begin{align*}
            |D|^{-1}[D^2,a] &= F\partial(a)+\partial(a)F-|D|^{-1}\partial(\delta(a))F\\
                            &= F\partial(a)+\partial(a)F-\partial(\delta(a))|D|^{-1}F-[|D|^{-1},\partial(\delta(a))]F
        \end{align*}
        then using \eqref{lambda 12}:
        \begin{equation*}
            |D|^{-1}[D^2,a] = F\partial(a)+\partial(a)F-\partial(\delta(a))D^{-1}+|D|^{-1}\partial(\delta^2(a))D^{-1}.
        \end{equation*}
        So multiplying on the right by $|D|^{-p}$:
        \begin{align*}
            |D|^{-1}[D^2,a]|D|^{-p} &= \Big(F\cdot\partial(a)|D|^{-p}+\partial(a)|D|^{-p}\cdot F\Big)\\
                                    &+ \Big(-\partial(\delta(a))|D|^{-p-1}\cdot F+|D|^{-1}\cdot\partial(\delta^2(a))|D|^{-p-1}\cdot F\Big).
        \end{align*}
        From Remark \ref{sp gap fact}, the first summand extends to an operator in $\mathcal{L}_{1,\infty}$ and the second summand extends to an operator in $\mathcal{L}_1.$
        Hence, the operator $|D|^{-1}[D^2,a]|D|^{-p}$ has extension to an operator in $\mathcal{L}_{1,\infty}$.

        On the other hand since $|D|^2 = D^2$, we have (again, as operators on $H_\infty$)
        \begin{align*}
            |D|^{-1}[D^2,a] &= |D|^{-1}[|D|^2,a]\\
                            &= |D|^{-1}\cdot(|D|\delta(a)+\delta(a)|D|)\\
                            &= \delta(a)+|D|^{-1}\delta(a)|D|\\
                            &= \delta(a)+\delta(a)|D|^{-1}|D|+[|D|^{-1},\delta(a)]|D|\\
                            &= 2\delta(a)-|D|^{-1}\delta^2(a)\\
                            &= 2\delta(a)-\delta^2(a)|D|^{-1}-[|D|^{-1},\delta^2(a)]\\
                            &= 2\delta(a)-\delta^2(a)|D|^{-1}+|D|^{-1}\delta^3(a)|D|^{-1}.
        \end{align*}
        
        So multiplying by $|D|^{-p}$:
        \begin{equation*}
            |D|^{-1}[D^2,a]|D|^{-p} = 2\delta(a)|D|^{-p}-\delta^2(a)|D|^{-p-1}+|D|^{-1}\delta^3(a)|D|^{-p-1}.
        \end{equation*}
        By Remark \ref{sp gap fact}, the operators $\delta^2(a)|D|^{-p-1}$ and $\delta^3(a)|D|^{-p-1}$ are in $\mathcal{L}_1$. 
        Since $|D|^{-1}[D^2,a]|D|^{-p}$ has extension to an operator in $\mathcal{L}_{1,\infty}$, it follows that $2\delta(a)|D|^{-p} \in \mathcal{L}_{1,\infty}$.
    \end{proof}

    Still working with the assumption that $D$ has a spectral gap at $0$, the following lemma is a refinement of the $\mathcal{L}_{1,\infty}$ inclusions in Remark \ref{sp gap fact} and the result of Lemma \ref{lambda lemma}. The following result should be compared with \cite[Lemma 1.37]{CGRS2}, which is of a similar nature but is stated in terms of Schatten ideals rather than weak Schatten ideals. There is a substantial difference between Schatten ideals and weak Schatten ideals, necessitating the introduction of new tools: here we use logarithmic submajorisation and the Araki-Lieb-Thirring inequality.
    
    A related antecedent to the following lemma is also \cite[Proposition 10]{Rennie-2004}, which proved a result similar to the first assertion in the setting of local spectral triples.
    
    \begin{lem}\label{z<p lemma} 
        Let $(\mathcal{A},H,D)$ be a smooth $p-$dimensional spectral triple satisfying Hypothesis \ref{main assumption}. 
        Suppose $D$ has a spectral gap at $0.$ For every $a\in\mathcal{A}$ and for every $0 < s\leq p,$ we have 
        $a|D|^{-s}\in\mathcal{L}_{\frac{p}{s},\infty},$ $\partial(a)|D|^{-s}\in\mathcal{L}_{\frac{p}{s},\infty},$ and $\delta(a)|D|^{-s}\in\mathcal{L}_{\frac{p}{s},\infty}.$
    \end{lem}
    \begin{proof} 
        We prove here only the third statement: that $\delta(a)|D|^{-s} \in \mathcal{L}_{p/s,\infty}$, the other results can be proved similarly.
        
        Let $r = \frac{p}{s} \geq 1$. By the Araki-Lieb-Thirring inequality \eqref{ALT inequality},
        \begin{equation*}
            |\delta(a)|D|^{-s}|^{r} \prec\prec_{\log} |\delta(a)|^r|D|^{-p}.
        \end{equation*}
        Due to \eqref{log majorization monotone} (the $\mathcal{L}_{1,\infty}$ quasi-norm is monotone with respect to logarithmic submajorisation)
        \begin{align*}
            \||\delta(a)|D|^{-s}|^r\|_{1,\infty} &\leq e\||\delta(a)|^r|D|^{-p}\|_{1,\infty}\\
                                                 &\leq e\|\delta(a)\|_\infty^{r-1}\|\delta(a)|D|^{-p}\|_{1,\infty}.
        \end{align*}
        Hence,
        \begin{equation*}
            \|\delta(a)|D|^{-s}\|_{r,\infty}^r \leq e\|\delta(a)\|_{\infty}^{r-1}\|\delta(a)|D|^{-p}\|_{1,\infty}.
        \end{equation*}
        By Lemma \ref{lambda lemma}, the right hand side is finite, and so $\delta(a)|D|^{-s} \in \mathcal{L}_{\frac{p}{s},\infty}$. 
        
        To prove the first two statements, one applies the same proof but with Remark \ref{sp gap fact} in place of Lemma \ref{lambda lemma}.
    \end{proof}
    
    So far the results of this section have been stated with the assumption that $D$ is invertible. The following proposition shows how we can apply these results to a spectral triple
    where $D$ may not have a spectral gap at zero, by finding a spectral triple with very similar properties but where the corresponding operator $D$ is invertible.
    A similar proposition appeared in the Remark following Definition 2.2 of \cite{CPRS2}. 
    \begin{prop}\label{pass to spectral gap 1} 
        Let $(\mathcal{A},H,D)$ be a spectral triple, and define $D_0 := F(1+D^2)^{1/2}$ with $\mathrm{dom}(D_0) = \mathrm{dom}(D)$. Then $(\mathcal{A},H,D_0)$ is spectral triple, and:
        \begin{enumerate}[{\rm (i)}]
            \item{}\label{iff p dim} $(\mathcal{A},H,D_0)$ is $p$-dimensional if and only if $(\mathcal{A},H,D)$ is $p$-dimensional;
            \item{}\label{iff dom} Let $\delta_0$ denote the bounded extension of $[|D_0|,T]$, and define $\mathrm{dom}_\infty(\delta_0)$ identically to $\mathrm{dom}_\infty(\delta)$ with $D_0$ in place of $D$. Then we have $\mathrm{dom}_\infty(\delta_0) = \mathrm{dom}_\infty(\delta)$.
            \item{}\label{iff smooth} $(\mathcal{A},H,D_0)$ is smooth if and only if $(\mathcal{A},H,D)$ is smooth;
            \item{}\label{iff main assumption} $(\mathcal{A},H,D_0)$ satisfies Hypothesis \ref{main assumption} if and only if $(\mathcal{A},H,D)$ does.
        \end{enumerate}
        Moreover, $D_0$ has a spectral gap at $0$.
    \end{prop}
    \begin{proof} 
        First, note that $\mathrm{dom}(D_0^n) = \mathrm{dom}(D^n)$ for all $n\geq 1$ and therefore that the space $H_\infty$ is identical for $D_0$ and for $D$. It is clear that $D_0$ has a spectral gap at $0$, since $|D_0| = (1+D^2)^{1/2} \geq 1$.
        Since $D_0^2 = 1+D^2$, we have $|D_0| = (1+D^2)^{1/2}$, and $F|D_0| = D_0$.
        As $|D_0| \geq 1$, the operator $|D_0|+|D|$ is invertible. Furthermore, since $|D_0|^2 = |D|^2+1$, we have:
        \begin{equation*}
            \frac{1}{|D_0|+|D|} = |D_0|-|D|.
        \end{equation*}
        Multiplying by $F$ we obtain
        \begin{equation*}
        \frac{F}{|D_0|+|D|} = D_0-D.
        \end{equation*}
        So both $|D_0|-|D|$ and $D_0-D$ extend to bounded operators. Moreover, since for all $k\geq 1$,
        \begin{equation*}
            \frac{1}{|D_0|+|D|}:\mathrm{dom}(D^k)\to \mathrm{dom}(D^{k+1})\subseteq \mathrm{dom}(D^k)
        \end{equation*}
        we have that the bounded extensions of $D_0-D$ and $|D_0|-|D|$ map $\mathrm{dom}(D^k)$ to $\mathrm{dom}(D^k)$ for all $k\geq 1$.
        
        For $T \in \mathcal{L}_\infty(H)$, let $\partial_1(T)$ denote the commutator of the bounded extension of $D_0-D$ with $T$, $\partial_1(T) := [\frac{F}{|D_0|+|D|},T]$.
        Similarly, let $\delta_1(T)$ denote the commutator of the bounded extension of $|D_0|-|D|$ with $T$, $\delta_1(T) := [\frac{1}{|D_0|+|D|},T]$.
        
        Then we have the following identity on $H_\infty$:
        \begin{equation*}
            [D_0,a] = \partial_1(a)+\partial(a).
        \end{equation*}
        Since $\partial(a)$ and $\partial_1(a)$ are bounded, it follows that $[D_0,a]$ extends to a bounded linear operator, which we denote $\partial_0(a)$.
        
        Since $D_0^2 = D^2+1$, we have that the operator $(D+i)(D_0+i)^{-1}$ has bounded extension. Hence, for all $a \in \mathcal{A}$ we have that $a(D_0+i)^{-1}$ is compact.
        This completes the proof that $(\mathcal{A},H,D_0)$ is a spectral triple. 
        
        One may similarly prove \eqref{iff p dim}: to see that $(\mathcal{A},H,D_0)$ is $p$-dimensional if $(\mathcal{A},H,D)$ is $p$-dimensional,
        we write:
        \begin{equation*}
            a(D_0+i)^{-p} = a(D+i)^{-p}\cdot \left(\frac{D+i}{D_0+i}\right)^p
        \end{equation*}
        which is in $\mathcal{L}_{1,\infty}$, because $(D+i)(D_0+i)^{-1}$ has bounded extension. Next,
        \begin{equation*}
            \partial_0(a)(D_0+i)^{-p} = (\partial(a)(D+i)^{-p}+\partial_1(a)(D+i)^{-p})\left(\frac{D+i}{D_0+i}\right)^p.
        \end{equation*}
        and $\partial_1(a)(D+i)^{-p} = (D_0-D)a(D+i)^{-p}-a(D+i)^{-p}(D_0-D)$, hence $\partial_0(a)(D_0+i)^{-p} \in \mathcal{L}_{1,\infty}$ and so $(\mathcal{A},H,D_0)$ is $p$-dimensional. The reverse
        implication may be established by an identical argument { using the fact that $(D_0+i)(D+i)^{-1}$ has bounded extension.}
        
        Next we prove \eqref{iff dom}. We have already shown that $|D_0|-|D|$ is an operator with bounded extension and which maps $\mathrm{dom}(D^k)$ to $\mathrm{dom}(D^k)$, for all $k\geq 1$. 
        By verifying the identity on $H_\infty$, we have:
        \begin{equation*}
            \delta^k(\delta_1(T)) = \delta_1(\delta^k(T)).
        \end{equation*}
        Hence if $T \in \mathrm{dom}(\delta^k)$, then $\delta_1(T) \in \mathrm{dom}(\delta^k)$.
        
        If $T \in \mathrm{dom}(\delta)$, then $[|D_0|,T] = \delta_1(T)+\delta(T)$ on $H_\infty$. So if $T \in \mathrm{dom}_\infty(\delta)$ we can compute the $k$th iterated commutator of $T$ with $|D_0|$ as:
        \begin{align}
            [|D_0|,[|D_0|,[\cdots,[|D_0|,T]\cdots]]] &= (\delta+\delta_1)^k(T)\nonumber\\
                                                     &= \sum_{j=0}^k \binom{k}{j}\delta_1^{k-j}(\delta^{j}(T))\label{binomial formula}
        \end{align}
        Thus the $k$th iterated commutator of $|D_0|$ and $T$ has bounded extension, so $T \in \mathrm{dom}_\infty(\delta_0)$. Repeating the proof using the identity $[|D|,T] = \delta_0(T)-\delta_1(T)$, 
        we also have that $\mathrm{dom}_\infty(\delta_0) \subseteq \mathrm{dom}_\infty(\delta)$. This completes the proof of \eqref{iff dom}.
        
        Now we prove \eqref{iff smooth}. Note that if $T \in \mathrm{dom}_\infty(\delta_0)$, then $\partial_1(T) \in \mathrm{dom}_\infty(\delta_0)$. Hence, if
        \begin{equation*}
            a,\, \partial(a) \in \mathrm{dom}_\infty(\delta) = \mathrm{dom}_\infty(\delta_0)
        \end{equation*}
        then
        \begin{equation*}
            a,\, \partial(a)+\partial_1(a) \in \mathrm{dom}_\infty(\delta_0).
        \end{equation*}
        Since $\partial_0(a) = \partial(a)+\partial_1(a)$, this completes the proof that if $(\mathcal{A},H,D)$ is smooth then $(\mathcal{A},H,D_0)$ is smooth.
        For the converse, we use $\partial_0(a) = \partial(a)-\partial_1(a)$.

        It now remains to show \eqref{iff main assumption}. Assume that $(\mathcal{A},H,D)$ satisfies
        Hypothesis \ref{main assumption}. From \eqref{binomial formula}, we have that
        \begin{align*}
            \delta_0^k(a)(D_0+i\lambda)^{-p-1} &= \delta_0^k(a)(D+i\lambda)^{-p-1}\left(\frac{D+i\lambda}{D_0+i\lambda}\right)^{p+1}\\
                                               &= \left(\sum_{l=0}^k \binom{k}{l} \delta_1^{k-l}(\delta^l(a))(D+i\lambda)^{-p-1}\right)\left(\frac{D+i\lambda}{D_0+i\lambda}\right)^{p+1}.
        \end{align*}
        However since $|D_0|-|D|$ commutes with functions of $D$,
        \begin{equation*}
            \delta_0^k(a)(D_0+i\lambda)^{-p-1} = \left(\sum_{l=0}^k \binom{k}{l} \delta_1^{k-l}(\delta^l(a)(D+i\lambda)^{-p-1})\right)\left(\frac{D+i\lambda}{D_0+i\lambda}\right)^{p+1}.
        \end{equation*}
        
        Now since the operator $\frac{D+i\lambda}{D_0+i\lambda}$ is bounded, and $\delta_1$ is a commutator with a bounded operator, and since $(\mathcal{A},H,D)$ satisfies Hypothesis \ref{main assumption},
        \begin{equation*}
            \|\delta^l(a)(D+i\lambda)^{-p-1}\|_1 = O(\lambda^{-1}), \quad\lambda > 0
        \end{equation*}
        it follows that
        \begin{equation*}
            \|\delta_0^l(a)(D_0+i\lambda)^{-p-1}\|_1 = O(\lambda^{-1}),\quad \lambda > 0.
        \end{equation*}
        Similarly, by writing $\partial_0 = \partial_1+\partial$, we also obtain:
        \begin{equation*}
            \|\partial_0(\delta_0^k(a))(D_0+i\lambda)^{-p-1}\|_1 = O(\lambda^{-1}).
        \end{equation*}
        To prove the converse, we write $\delta(a) = \delta_0-\delta_1$ and repeat the same argument.
    \end{proof}
    
    We are now able to prove Proposition \ref{f der def} -- without any assumptions on the invertibility of $D.$ Similar results are well
    known in the unital case (see e.g. \cite[Lemma 1]{CPRS1}, \cite[Lemma 10.18]{GVF} and \cite[Lemma 5]{BF}). In the non-unital setting, a related result is \cite[Proposition 2.14]{CGRS2} which instead proves that $[F,a] \in \mathcal{L}_{p+1}$. To the best of our knowledge, no complete proof of the following result has been published in the non-unital setting.
    \begin{prop}\label{f der def}
        If $(\mathcal{A},H,D)$ is a $p-$dimensional spectral triple satisfying Hypothesis \ref{main assumption}, then $[F,a]\in\mathcal{L}_{p,\infty}$ for all $a\in\mathcal{A}.$
    \end{prop}
    \begin{proof} 
        Let $D_0 = F(1+D^2)^{1/2}$, so that by Proposition \ref{pass to spectral gap 1}, the spectral triple $(\mathcal{A},H,D_0)$ satisfies
        Hypothesis \ref{main assumption}. As an equality of operators on $H_\infty$, we have:
        \begin{align*}
            [F,a] &= [D_0|D_0|^{-1},a]\\
                  &= [D_0,a]|D_0|^{-1}+D_0[|D_0|^{-1},a].
        \end{align*}
        Using \eqref{favourite commutator identity},
        \begin{equation*}
            [F,a] = [D_0,a]|D_0|^{-1}-F[|D_0|,a]|D_0|^{-1}.
        \end{equation*}
        Since the spectral triple $(\mathcal{A},H,D_0)$ satisfies Hypothesis \ref{main assumption} and has a spectral gap at $0$,
        we may apply Lemma \ref{z<p lemma} with $s = p$ to conclude that the operators $[D_0,a]|D_0|^{-1}$ and $[|D_0|,a]|D_0|^{-1}$ have
        extension to operators in $\mathcal{L}_{p,\infty}$. Thus, $[F,a] \in \mathcal{L}_{p,\infty}$.
    \end{proof}

\section{Restatement of Hypothesis 1.2.1}\label{replacement section}

    In this section, we introduce the operator $\Lambda$, formally defined by:
    \begin{equation*}
        \Lambda(T) = (1+D^2)^{-\frac12}[D^2,T].
    \end{equation*}
    Strictly speaking, $\Lambda(T)$ will be defined to be the bounded extension of the above operator. What is here denoted $\Lambda$ appeared in the unital settings of \cite[Appendix B]{Connes-Moscovici} 
    (there denoted $L$), \cite[Definition 6.5]{CPRS2} (there denoted $L_1$) and \cite[Equation 10.66]{GVF} (there denoted $L$).
    The mapping $\Lambda$ was also used in the non-unital setting of \cite[Definition 1.20]{CGRS2} (there called $L$).
    We undertake a self-contained development of these ideas, since our assumptions are different to those used in previous work.
    There is not any substantial conceptual difference between the proofs for the unital and nonunital cases, however there are small technical obstacles which require care to be taken when computing repeated integrals of operator valued functions (see the proof of Lemma \ref{delta to lambda power}). An expert reader familiar with this theory could skip to Hypothesis \ref{replacement assumption}.
    
    We must take care to ensure that $\Lambda(T)$ is well defined, as well as that higher powers $\Lambda^k(T)$ are defined. For this purpose we introduce the spaces $\mathrm{dom}(\Lambda^k)$.    
    \begin{defi}\label{lambda def}
        Let $k \geq 1$. We define $\mathrm{dom}(\Lambda^k)$ to be the set of bounded linear operators $T$ such that for all $1 \leq j \leq k$, we have $T:\mathrm{dom}(D^{2j})\to \mathrm{dom}(D^{2j})$ and 
        such that the $k$th iterated commutator,
        \begin{equation*}
            (1+D^2)^{-1/2}[D^2,(1+D^2)^{-1/2}[D^2,\cdots,T]]:\mathrm{dom}(D^{2k})\to H
        \end{equation*}
        has bounded extension, which we denote $\Lambda^k(T)$.
        
        Define 
        \begin{equation*}
            \mathrm{dom}_\infty(\Lambda) := \bigcap_{k\geq 0} \mathrm{dom}(\Lambda^k).
        \end{equation*}
    \end{defi}

    The mapping $\Lambda$ can be thought of as a replacement for $\delta$, and we introduce it since it is easier to work with $\Lambda$ rather than $\delta$
    in the examples covered in this chapter.
    \begin{defi}
        A spectral triple $(\mathcal{A},H,D)$ is called $\Lambda$-smooth if for all $a \in \mathcal{A}$ we have,
        \begin{equation*}
            a,\,\partial(a) \in \mathrm{dom}_\infty(\Lambda).
        \end{equation*}
    \end{defi}
    We will show that $\mathrm{dom}_\infty(\Lambda) = \mathrm{dom}_\infty(\delta_0)$, and so in view of Theorem \ref{pass to spectral gap 1}.\eqref{iff dom} the notion of
    $\Lambda$-smoothness is identical to smoothness. {  This fact is well known in the unital setting, similar results having appeared in \cite[Appendix B]{Connes-Moscovici} and \cite[Proposition 6.5]{CPRS2}. We
    provide a full proof here since to the best of our knowledge no published proof exists in the non-unital setting.}
    
    The easiest direction to establish is that $\mathrm{dom}_\infty(\delta_0)\subseteq \mathrm{dom}_\infty(\Lambda)$, as the following Lemma shows:
    \begin{lem}\label{delta smooth implies lambda smooth}
        We have $\mathrm{dom}_\infty(\delta_0) \subseteq \mathrm{dom}_\infty(\Lambda)$.
    \end{lem}
    \begin{proof}
        Let $T \in \mathrm{dom}_\infty(\delta_0)$. We have $T:\mathrm{dom}(D^k)\to \mathrm{dom}(D^k)$ for all $k\geq 1$, and so working on $H_\infty$, we can write,
        \begin{align*}
            (1+D^2)^{-1/2}[D^2,T]  &= |D_0|^{-1}[|D_0|^2,T]\\
                                   &= 2[|D_0|,T]-|D_0|^{-1}[|D_0|,[|D_0|,T]]\\
                                   &= 2\delta_0(T)-|D_0|^{-1}\delta_0^2(T).
        \end{align*}
        By assumption $T \in \mathrm{dom}_{\infty}(\delta)$, hence, $\Lambda(T)$ has bounded extension
        and so $T \in \mathrm{dom}(\Lambda)$, and on all $H$ we have:
        \begin{equation*}
            \Lambda(T) = 2\delta_0(T)-|D_0|^{-1}\delta_0^2(T).
        \end{equation*}
        However since $\delta_0(T), \delta_0^2(T)$ and $|D_0|^{-1}$ are in $\mathrm{dom}_\infty(\delta_0)$, it follows that $\Lambda(T) \in \mathrm{dom}_\infty(\delta_0)$.
        
        Hence, $\Lambda(T) \in \mathrm{dom}(\Lambda)$, and continuing by induction we get that $T \in \mathrm{dom}(\Lambda^k)$ for all $k\geq 1$.
    \end{proof}
    
    It takes some more work to prove that $\mathrm{dom}_\infty(\Lambda)\subseteq \mathrm{dom}_\infty(\delta_0)$. We achieve this by an integral representation of $\delta_0(T)$
    in terms of $\Lambda(T)$ and $\Lambda^2(T)$. We make use of the dense subspace $H_\infty$ from Definition \ref{H_infty definition}. 
    The following Lemma should be compared with the proof of \cite[Lemma B2]{Connes-Moscovici}.
    \begin{lem}\label{delta to lambda}
        Let $T\in \mathrm{dom}_\infty(\Lambda)$. Then for all $\xi \in H_\infty$ we have:
        \begin{equation*}
            [|D_0|,T]\xi = \frac{1}{2}\Lambda(T)\xi + \frac{1}{\pi}\int_0^\infty \lambda^{1/2} \frac{D_0^2}{(\lambda+D_0^2)^2}\Lambda^2(T)\frac{1}{\lambda+D_0^2}\xi\,d\lambda.
        \end{equation*}
        The integral above may be understood as a weak operator topology integral.
    \end{lem}
    \begin{proof} 
        This is essentially a combination of the following two { well known} integral formulae:
        \begin{equation}\label{little integral 1}
            (1+D^2)^{-1/2} = \frac{1}{\pi} \int_0^\infty \frac{1}{1+\lambda+D^2}\frac{d\lambda}{\lambda^{1/2}}
        \end{equation}
        and
        \begin{equation}\label{little integral 2}
            (1+D^2)^{-1/2} = \frac{2}{\pi} \int_0^\infty \frac{\lambda^{1/2}}{(1+\lambda+D^2)^2}d\lambda
        \end{equation}
        which can both be understood as integrals in the weak operator topology, since
        \begin{equation*}
            \left\|\frac{1}{1+\lambda+D^2}\right\|_\infty \leq \frac{1}{1+\lambda},\quad \lambda > 0
        \end{equation*}
        and
        \begin{equation*}
            \left\|\frac{\lambda^{1/2}}{(1+\lambda+D^2)^2}\right\|_\infty \leq \frac{\lambda^{1/2}}{(1+\lambda)^2},\quad\lambda > 0.
        \end{equation*}
        
        Let $\xi \in H_\infty$. Multiplying \eqref{little integral 1} by $(1+D^2)\xi$, we get:
        \begin{equation}\label{little integral 3}
            (1+D^2)^{1/2}\xi = \frac{1}{\pi}\int_0^\infty \frac{1+D^2}{1+\lambda+D^2}\xi\frac{d\lambda}{\lambda^{1/2}}.
        \end{equation}
        The above is a convergent Bochner integral in $H$, since
        \begin{equation*}
            \left\|\frac{1+D^2}{1+\lambda+D^2}\xi\right\|_H \leq \frac{1}{1+\lambda}\|(1+D^2)\xi\|_H.
        \end{equation*}
        Now, replacing $1+D^2 = D_0^2$ by \eqref{favourite commutator identity}:
        \begin{align}\label{adding lambda to Lambda}
            \left[\frac{1}{\lambda+D_0^2},T\right] &= -(\lambda+D_0^2)^{-1}[D^2,T](\lambda+D_0^2)^{-1}\nonumber\\
                                                   &= -\frac{D_0^2}{\lambda+D_0^2}\Lambda(T)(\lambda+D_0^2)^{-1}.
        \end{align}
        Hence,
        \begin{align*}
            \left[\frac{D_0^2}{\lambda+D_0^2},T\right] &= \left[1-\frac{\lambda}{\lambda+D_0^2},T\right]\\
                                                       &= \frac{\lambda(D_0^2)}{\lambda+D_0^2}\Lambda(T)(\lambda+D_0^2)^{-1}\\
                                                       &= \frac{\lambda(D_0^2)}{\lambda+D_0^2}([\Lambda(T),(\lambda+D_0^2)^{-1}]+(\lambda+D_0^2)^{-1}\Lambda(T)).
        \end{align*}
        Applying \eqref{adding lambda to Lambda} a second time:
        \begin{align*}
            \left[\frac{D_0^2}{\lambda+D_0^2},T\right] &= \frac{\lambda D_0^2}{\lambda+D_0^2}\Big(\frac{D_0^2}{\lambda+D_0^2}\Lambda^2(T)(\lambda+D_0^2)^{-1} +(\lambda+D_0^2)^{-1}\Lambda(T)\Big)\\
                                                       &= \frac{\lambda|D_0|}{(\lambda+D_0^2)^2}\Lambda(T)+ \frac{\lambda D_0^2}{(\lambda+D_0^2)^2}\Lambda^2(T)(\lambda+D_0^2)^{-1}.
        \end{align*}
        Now we apply the integral formula \eqref{little integral 3} to obtain:
        \begin{align*}
            [|D_0|,T]         &= \frac{1}{\pi}\int_0^\infty \left[\frac{D_0^2}{\lambda+D_0^2},T\right]\frac{d\lambda}{\lambda^{1/2}}\\
                              &= \frac{1}{\pi}\int_0^\infty \frac{\lambda|D_0|}{(\lambda+D_0^2)^2}\Lambda(T)\frac{d\lambda}{\lambda^{1/2}}\\
                              &\quad+\frac{1}{\pi}\int_0^\infty \frac{\lambda D_0^2}{(\lambda+D_0^2)^2}\Lambda^2(T)(\lambda+D_0^2)^{-1}\frac{d\lambda}{\lambda^{1/2}}.
        \end{align*}
        Now applying \eqref{little integral 2} again, we have:
        \begin{equation*}
            \int_0^\infty \lambda^{1/2}\frac{|D_0|}{(\lambda+D_0^2)^2}d\lambda = \frac{\pi}{2}.
        \end{equation*}
        Hence,
        \begin{equation*}
            [|D_0|,T]\xi = \frac{1}{2}\Lambda(T) + \frac{1}{\pi}\int_0^\infty \lambda^{1/2}\frac{D_0^2}{(\lambda+D_0^2)^2}\Lambda^2(T)\frac{1}{\lambda+D_0^2}\xi\,d\lambda.
        \end{equation*}
    \end{proof}

    The following lemma provides an integral representation of the $n$th iterated commutator $\delta_0^n(T)$. This will allow
    us to relate $\mathrm{dom}_\infty(\delta_0)$ to $\mathrm{dom}_\infty(\Lambda)$. 
    We need to take care to ensure that the relevant version of a Fubini's theorem applies.
    \begin{lem}\label{delta to lambda power} 
        For all $m\geq1,$ and $T \in \mathrm{dom}_\infty(\Lambda)$. Then for all $\xi \in H_\infty$ the $m$th iterated commutator of $|D_0|$ 
        \begin{align*}
            [|D_0|,[|D_0|,[\cdots[|D_0|,T]\cdots]]]\xi &= 2^{-m}\sum_{k=0}^m\binom{m}{k}\Big(\frac{2}{\pi}\Big)^k\int_{\mathbb{R}^k_+}\prod_{l=1}^k\frac{\lambda_l^{1/2}D_0^2}{(\lambda_l+D_0^2)^2}\\
                                                       &\quad \cdot \Lambda^{m+k}(T)\prod_{l=1}^k\frac{1}{\lambda_l+D_0^2}\xi d\lambda_1 d\lambda_2\cdots d\lambda_k.
        \end{align*}
    \end{lem}
    \begin{proof} 
        Let
        \begin{equation*}
            \Theta(T) := \int_0^\infty \lambda^{1/2} \frac{D_0^2}{(\lambda+D_0^2)^2} \Lambda^2(T)\frac{1}{\lambda+D_0^2}d\lambda
        \end{equation*}
        so that Lemma \ref{delta to lambda} states that $\delta_0 = \frac{1}{2}\Lambda+\frac{1}{\pi}\Theta$. 
        
        Since $\Lambda$ commutes with $\frac{D_0^2}{(\lambda+D_0^2)^2}$ and $\frac{1}{\lambda+D_0^2}$, we have $\Theta\circ\Lambda = \Lambda\circ \Theta$. Hence,
        \begin{equation}\label{binomial expression}
            \delta_0^m = \frac{1}{2^m}\sum_{k=0}^m \binom{m}{k} \left(\frac{2}{\pi}\right)^k\Theta^k\circ \Lambda^{m-k}.
        \end{equation}
        By the Fubini theorem { for Hilbert space valued functions (see \cite[Theorem III.11.13]{Dunford-Schwartz-1}), for all $\xi \in H_\infty$} we have:
        \begin{equation*}
            \Theta^k(T)\xi = \int_{[0,\infty)^k} \prod_{l=1}^k \frac{\lambda_l^{1/2}D_0^2}{(\lambda_l+D_0^2)^2}\Lambda^{2k}(T)\prod_{l=1}^k \frac{1}{\lambda_l+D_0^2}\xi d\lambda_1d\lambda_2\cdots d\lambda_k.
        \end{equation*}
        Therefore,
        \begin{equation*}
            \Theta^k(\Lambda^{m-k}(T))\xi = \int_{[0,\infty)^k} \prod_{l=1}^k \frac{\lambda_l^{1/2}D_0^2}{(\lambda_l+D_0^2)^2}\Lambda^{m+k}(T)\prod_{l=1}^k \frac{1}{\lambda_l+D_0^2}\xi d\lambda_1d\lambda_2\cdots d\lambda_k.
        \end{equation*}
        Substituting into \eqref{binomial expression} yields the result.
    \end{proof}
    
    The following corollary was already noted in the unital settings of \cite[Appendix B]{Connes-Moscovici}, \cite[Proposition 6.5]{CPRS2} 
    and in the non-unital setting of \cite[Equation 1.12]{CGRS2}.
    \begin{cor}
        We have $\mathrm{dom}_\infty(\Lambda) = \mathrm{dom}_\infty(\delta_0)$, and $\Lambda$-smoothness { of a spectral triple} is equivalent to smoothness { as stated in Definition \ref{smoothness definition}}
    \end{cor}
    \begin{proof}
        From Lemma \ref{delta smooth implies lambda smooth} we already know that $\mathrm{dom}_\infty(\delta_0) \subseteq \mathrm{dom}_\infty(\Lambda)$, and so we concentrate 
        on the reverse inclusion.
        
        If $T \in \mathrm{dom}_\infty(\Lambda)$, then for each $k\geq 1$ the operator $\Lambda^k(T)$ on $H_\infty$ has bounded extension. Hence the integral
        in Lemma \ref{delta to lambda power} converges as a Bochner integral, and so the $m$th iterated commutator $\delta_0^m(T)$ is bounded, for all $m\geq 0$.
        Thus $T \in \mathrm{dom}_\infty(\delta_0)$, and this completes the proof.
    \end{proof}
    
    The remainder of this section is devoted to showing that Hypothesis \ref{main assumption} is equivalent to the following:
    \begin{hyp}\label{replacement assumption} 
        The spectral triple $(\mathcal{A},H,D)$ satisfies the following conditions:
        \begin{enumerate}[{\rm (i)}]
            \item\label{rass0} $(\mathcal{A},H,D)$ is a $\Lambda$-smooth spectral triple.
            \item\label{rass1} $(\mathcal{A},H,D)$ is $p-$dimensional, i.e., for every $a\in\mathcal{A},$
                               $$a(D+i)^{-p}\in\mathcal{L}_{1,\infty},\quad \partial(a)(D+i)^{-p}\in\mathcal{L}_{1,\infty}.$$
            \item\label{rass2} For every $a\in\mathcal{A}$ and for all $k\geq0,$ we have
                               $$\Big\|\Lambda^k(a)(D+i\lambda)^{-p-1}\Big\|_1=O(\lambda^{-1}),\quad\lambda\to\infty,$$
                               $$\Big\|\partial(\Lambda^k(a))(D+i\lambda)^{-p-1}\Big\|_1=O(\lambda^{-1}),\quad\lambda\to\infty.$$
            \end{enumerate}
    \end{hyp}
    Hypothesis \ref{replacement assumption} is precisely Hypothesis \ref{main assumption}, but with smoothness replaced by $\Lambda$-smoothness, and the occurances of $\delta$
    replaced with $\Lambda$.    

    For the next two lemmas, we borrow { techniques} that were developed in \cite{CGRS2}. The next Lemma shows that if $(\mathcal{A},H,D)$ satisfies
    Hypothesis \ref{replacement assumption} then $(\mathcal{A},H,D_0)$ satisfies Hypothesis \ref{main assumption}.
    \begin{lem}\label{main replacement lemma} 
        Let $(\mathcal{A},H,D)$ be a spectral triple satisfying Hypothesis \ref{replacement assumption}. For every $a\in\mathcal{A}$ and for all $m\geq0,$ we have 
        $$\Big\|\delta_0^m(a)(D_0+i\lambda)^{-p-1}\Big\|_1=O(\lambda^{-1}),\quad\lambda\to\infty,$$
        $$\Big\|\partial_0(\delta_0^m(a))(D_0+i\lambda)^{-p-1}\Big\|_1=O(\lambda^{-1}),\quad\lambda\to\infty.$$
    \end{lem}
    \begin{proof} 
        We prove only the first assertion. The proof of the second assertion is similar.
        
        By the spectral theorem,
        \begin{align*}
            \left\|\prod_{l=1}^k \frac{\lambda_l^{1/2}(1+D^2)}{(1+\lambda_l+D^2)^2}\right\|_\infty &\leq \prod_{l=1}^k\left\|\frac{\lambda_l^{1/2}(1+D^2)}{(1+\lambda_l+D^2)^2}\right\|_\infty\\
                                                                                                   &\leq \prod_{l=1}^k \sup_{t_l\geq 1} \frac{\lambda_l^{1/2}t_l}{(\lambda_l+t_l)^2}\\
                                                                                                   &\leq \prod_{l=1}^k \lambda_l^{-1/2}
        \end{align*}    
        and also
        \begin{equation*}
            \left\|\prod_{l=1}^k \frac{1}{1+\lambda_l+D^2}\right\|_\infty \leq \prod_{l=1}^k \frac{1}{1+\lambda_l}.
        \end{equation*}
        Hence, for all $a \in \mathcal{A}$, since $(D+i\lambda)^{-1}$ and $(1+\lambda_l+D^2)^{-1}$ commute,
        \begin{align*}
            \Big\|\left(\prod_{l=1}^k \frac{\lambda_l^{1/2}(1+D^2)}{(1+\lambda_l+D^2)^2}\right)\Lambda^{m+k}(a)&\left(\prod_{l=1}^k \frac{1}{1+\lambda_l+D^2}\right)(D+i\lambda)^{-p-1}\Big\|_1\\
                                                            &\leq \|\Lambda^{m+k}(a)(D+i\lambda)^{-p-1}\|_1\prod_{l=1}^k \frac{1}{\lambda_l^{1/2}(1+\lambda_l)}.
        \end{align*}
        Now applying Lemma \ref{delta to lambda power} with Lemma \ref{peter lemma},
        \begin{align*}
            \|&\delta_0^m(a)(D+i\lambda)^{-p-1}\|_1 \\
                            &\leq 2^{-m}\sum_{k=0}^m \binom{m}{k} \left(\frac{2}{\pi}\right)^k \|\Lambda^{m+k}(a)(D+i\lambda)^{-p-1}\|_1\int_{[0,\infty)^k} \prod_{l=1}^k \frac{d\lambda_l}{\lambda_l^{1/2}(1+\lambda_l)}.
        \end{align*}
        Since $\int_0^\infty \frac{1}{\lambda^{1/2}(1+\lambda)}\,d\lambda = \pi$, we arrive at
        \begin{equation*}
            \|\delta_0^m(a)(D+i\lambda)^{-p-1}\|_{1} \leq 2^{-m} \sum_{k=0}^m \binom{m}{k} 2^k \|\Lambda^{m+k}(a)(D+i\lambda)^{-p-1}\|_1\pi^k.
        \end{equation*}
        By Hypothesis \ref{replacement assumption}, each summand above is $O(\lambda^{-1})$. Hence $\|\delta_0^m(a)(D+i\lambda)^{-p-1}\|_1 = O(\lambda^{-1})$.

        Now using the fact that the operator $\left(\frac{D+i\lambda}{D_0+i\lambda}\right)^{p+1}$ has bounded extension, { and
        \begin{align*}
            \left\|\left(\frac{D+i\lambda}{D_0+i\lambda}\right)^{p+1}\right\|_\infty &\leq \sup_{t \in \mathbb{R}} \left(\frac{t^2+\lambda^2}{1+t^2+\lambda^2}\right)^{\frac{p+1}{2}}\\
                                                                                     &\leq 1,
        \end{align*}
        we get:}
        \begin{align*}
            \|\delta_0^m(a)(D_0+i\lambda)^{-p-1}\|_1 &\leq \|\delta_0(a)(D+i\lambda)^{-p-1}\|_1\left\|\left(\frac{D+i\lambda}{D_0+i\lambda}\right)^{p+1}\right\|_\infty\\
                                                     &= O(\lambda^{-1})
        \end{align*}
        as desired.
        
    \end{proof}
    
    We can now conclude the proof of the main result of this subsection:
    \begin{thm}\label{replacement thm} 
        A spectral triple $(\mathcal{A},H,D)$ satisfies Hypothesis \ref{main assumption} if and only if it satisfies Hypothesis \ref{replacement assumption}.
    \end{thm}
    \begin{proof}
        We have already proved that $(\mathcal{A},H,D)$ satisfies \ref{replacement assumption}.\eqref{rass0} if and only if it satisfies \ref{main assumption}.\eqref{ass0},
        and \eqref{replacement assumption}.\eqref{rass1} is identical to \eqref{main assumption}.\eqref{ass1}. We now focus on \eqref{replacement assumption}.\eqref{rass2}.
    
        Suppose that $(\mathcal{A},H,D)$ satisfies Hypothesis \ref{replacement assumption}. 
        By Lemma \ref{main replacement lemma}, we have that $(\mathcal{A},H,D_0)$ satisfies Hypothesis \ref{main assumption}. Then by Proposition \ref{pass to spectral gap 1} 
        $(\mathcal{A},H,D)$ satisfies Hypothesis \ref{main assumption}.
        
        Now we prove the converse. Suppose that $(\mathcal{A},H,D)$ satisfies Hypothesis \ref{main assumption}. For $T \in \mathrm{dom}_\infty(\delta)$, we define $\alpha(T)$ and $\beta(T)$ by:
        \begin{align*}
            \alpha(T) := \frac{|D|}{(D^2+1)^{1/2}}\delta(T),\\
            \beta(T) := \frac{1}{(D^2+1)^{1/2}}\delta^2(T).
        \end{align*}
        We can express $\Lambda$ in terms of $\alpha$ and $\beta$, by applying the Leibniz rule as follows:
        \begin{align*}
            \Lambda(T) &= (1+D^2)^{-1/2}[|D|^2,T]\\
                       &= \frac{|D|}{(1+D^2)^{1/2}}\delta(T) + (1+D^2)^{-1/2}\delta(T)|D|\\
                       &= 2\alpha(T) - (1+D^2)^{-1/2}\delta^2(T)\\
                       &= 2\alpha(T)-\beta(T).
        \end{align*}
        Since $\alpha\circ\beta = \beta\circ\alpha$,
        \begin{equation*}
            \Lambda^m = \sum_{k=0}^m \binom{m}{k}(-1)^k2^{m-k}\beta^k\circ\alpha^{m-k}.
        \end{equation*}
        
        For every $k = 0,\ldots,m$ and $T \in \mathrm{dom}_\infty(\delta)$, we have:
        \begin{equation*}
            \beta^k(\alpha^{m-k}(T)) = \frac{|D|^{m-k}}{(D^2+1)^{m/2}}\delta^{m+k}(T).
        \end{equation*}
        So for $a \in \mathcal{A}$ and $m \geq 1$,
        \begin{equation*}
            \Lambda^m(a) = \sum_{k=0}^m \binom{m}{k} (-1)^k2^{m-k} \frac{|D|^{m-k}}{(D^2+1)^{m/2}}\delta^{m+k}(a).
        \end{equation*}
        Noting that the operator $\frac{|D|^{m-k}}{(D^2+1)^{m/2}}$ is bounded, there exists a constant $C_m$ such that
        \begin{equation*}
            \|\Lambda^m(a)(D+i\lambda)^{-p-1}\|_1 \leq C_m\sum_{k=0}^m \|\delta^{m+k}(a)(D+i\lambda)^{-p-1}\|_1
        \end{equation*}
        Since we are assuming that $(\mathcal{A},H,D)$ satisfies Hypothesis \ref{main assumption}, it follows
        that $\|\Lambda^m(a)(D+i\lambda)^{-p-1}\|_1 = O(\lambda^{-1})$. 
        
        We may similarly deal with $\|\partial(\Lambda^m(a))(D+i\lambda)^{-p-1}\|_1$: since $\partial$ commutes with functions of $D$:
        \begin{equation*}
            \partial(\Lambda^m(a)) = \sum_{k=0}^m (-1)^k2^{m-k} \frac{|D|^{m-k}}{(D^2+1)^{m/2}}\partial(\delta^{m+k}(a)).
        \end{equation*} 
        Thus by the same argument, we have $\|\partial(\Lambda^m(a))(D+i\lambda)^{-p-1}\|_1 = O(\lambda^{-1})$, and so $(\mathcal{A},H,D)$
        satisfies Hypothesis \ref{replacement assumption}.
    \end{proof}
    
    Thanks to Theorem \ref{replacement thm}, we can work assuming Hypothesis \ref{replacement assumption} rather than Hypothesis \ref{main assumption}.

\section{Example: Noncommutative Euclidean space}\label{nc section}
    We now discuss the most heavily studied example of a non-unital spectral triple: noncommutative Euclidean space. Subsection \ref{ncp definitions subsection} covers
    the definitions of noncommutative Euclidean spaces and their associated spectral triples. Subsection \ref{ncp verification subsection} is devoted to the proof that 
    these spectral triples satisfy Hypothesis \ref{main assumption}.
    
    Noncommutative Euclidean spaces can be found in the literature under various names, such as canonical commutation relation (CCR) algebras (as in \cite[Section 5.2.2.2]{Bratteli-Robinson2}),
    and in the $2$-dimensional case are called Moyal planes or Moyal-Groenwald planes (as in \cite{gayral-moyal}).

\subsection{Definitions for Noncommutative Euclidean spaces}\label{ncp definitions subsection}
    Our approach to noncommutative Euclidean space is to proceed from the Weyl commutation relations, in line with \cite[Section 5.2.2.2]{Bratteli-Robinson2} and \cite{LeSZ-cwikel}. An alternative
    approach is to use the Moyal product, as in \cite{gayral-moyal} and \cite[Section 5.2]{CGRS2}. We caution the reader that the approach considered here is the "Fourier dual" of the approach in \cite{gayral-moyal}.
    We briefly cite the required facts needed for this section, and refer the reader to \cite{LeSZ-cwikel} for detailed exposition and proofs.
    
    Let $\theta$ be an antisymmetric real $p\times p$ matrix. 
    Abstractly, the von Neumann algebra $L_\infty(\mathbb{R}^p_\theta)$ is generated by a strongly continuous family $\{U(t)\}_{t \in \mathbb{R}^p}$ satisfying
    \begin{equation}\label{weyl CCR}
        U(t+s) = \exp\left(\frac{1}{2}i(t,\theta s)\right)U(t)U(s),\quad t,s \in  \mathbb{R}^p.
    \end{equation}
    Here we avoid technicalities by defining $L_\infty(\mathbb{R}^p_\theta)$ to be a von Neumann algebra generated by a specific family of unitary operators on $L_2(\mathbb{R}^p)$.
    
    \begin{defi}
        Let $\theta$ be an antisymmetric real matrix. For $t \in \mathbb{R}^p$, let $U(t)$ be the linear operator on $L_2(\mathbb{R}^p)$ given by:
        \begin{equation*}
            (U(t)\xi)(r) = \exp\left(-i(t,\theta r)\right)\xi(r-t),\quad r \in \mathbb{R}^p,\,\xi \in L_2(\mathbb{R}^p).
        \end{equation*}
        Define $L_\infty(\mathbb{R}^p_\theta)$ to be the von Neumann subalgebra of $\mathcal{L}_\infty(L_2(\mathbb{R}^p))$ generated by the family $\{U(t)\}_{t \in \mathbb{R}^p}$.
    \end{defi}
    
    \begin{rem}
        It can easily be shown that $U(t)$ satisfies \eqref{weyl CCR}. Since $U(t)$ is a composition of a translation, and pointwise multiplication by $\exp\left(-i\frac{1}{2}(t,\theta s)\right)$, it is clear that each $U(t)$ is unitary, and that $t\mapsto U(t)$ is strongly continuous. Since $\theta$ is antisymmetric, $U(-t) = U(t)^{-1} = U(t)^*$.
        
        The map $t\mapsto U(t)$ is a twisted left-regular representation of $\mathbb{R}^p$ on $L_2(\mathbb{R}^p)$, in the sense of \cite{Echterhoff-2008}.
        
        Note that if $\theta = 0$, then the family $\{U(t)\}_{t \in \mathbb{R}^p}$ is simply the semigroup of translations on $\mathbb{R}^p$, and so generates the von Neumann algebra $L_\infty(\mathbb{R}^p)$. 
    \end{rem}
    
    If $\theta$ is nondegenerate (that is, $\det(\theta) \neq 0$) then $p$ is even and the algebra $L_\infty(\mathbb{R}^p_\theta)$ is isomorphic to $\mathcal{L}_\infty(L_2(\mathbb{R}^{p/2}))$. 
    This is proved in \cite{LeSZ-cwikel}, where a spatial isomorphism is constructed.    
    \begin{thm}\label{spatial}
        If $\det(\theta) \neq 0$, then there is a spatial isomorphism
        \begin{equation*}
            L_\infty(\mathbb{R}^p_\theta)\cong \mathcal{L}_\infty(L_2(\mathbb{R}^{p/2})).
        \end{equation*}
    \end{thm}
    
    We now focus exclusively on the case where $\det(\theta) \neq 0$.
    \begin{defi}
        The semifinite trace $\tau_\theta$ on $L_\infty(\mathbb{R}^p)$ is defined via the isomorphism in Theorem \ref{spatial} to be simply the classical
        trace $\mathrm{Tr}$ on $\mathcal{L}_\infty(L_2(\mathbb{R}^{p/2}))$. For $r \in [1,\infty)$, the space $L_r(\mathbb{R}^p_\theta)$ is defined by:
        \begin{equation*}
            L_r(\mathbb{R}^p_\theta) := \{x \in L_\infty(\mathbb{R}^p_\theta)\;:\;\tau_\theta(|x|^r) < \infty\}.
        \end{equation*}
        The space $L_{r}(\mathbb{R}^p_\theta)$ is equipped with the norm $\|x\|_{L_r} = \tau_\theta(|x|^r)^{1/r}$.
    \end{defi}
    Note that $L_r(\mathbb{R}^p_\theta)$ is identical to the Schatten-von Neumann class $\mathcal{L}_r(L_2(\mathbb{R}^{p/2}))$, since $\tau_\theta$ is simply the classical trace. 
    
    \begin{defi}
        For $k = 1,\ldots,p$, we define the operator $D_k^\theta$ on $L_2(\mathbb{R}^p)$ by
        \begin{equation*}
            (D_k^\theta\xi)(t) = t_k\xi(t),\quad t \in \mathbb{R}^p.
        \end{equation*}
        
        The Dirac operator $D^\theta$ is defined on the Hilbert space $L_2(\mathbb{R}^d,\mathbb{C}^{2^{p/2}})$ by 
        $D^\theta = \gamma_1\otimes D^\theta_1+\cdots +\gamma_p\otimes D^\theta_p$, where $\gamma_1,\gamma_2,\ldots\gamma_p$
        are complex $2^{p/2}\times 2^{p/2}$ matrices satisfying $\gamma_j\gamma_k+\gamma_k\gamma_j = 2\delta_{j,k}$ and $\gamma_j = \gamma_j^*$ for $1\leq j,k \leq p$.
    \end{defi}
    Evidently, the operators $D_j$ are unbounded, but may be initially defined on the dense subspace of compactly supported functions.
    
    It follows readily from the definitions of $D_k^\theta$ and $U(t)$ that
    \begin{equation*}
        [D_k^\theta,U(t)] = t_kU(t),\quad t \in \mathbb{R}^p.
    \end{equation*}
    
    Since the operators $D_1^\theta,\ldots D_p^\theta$ form a family of mutually commuting self-adjoint operators, we may apply functional calculus to define $e^{i(s,\nabla)}$, $s \in \mathbb{R}^p$.
    where $\nabla^\theta = (D_1^\theta,D_2^\theta,\ldots,D_p^\theta)$, given by:
    \begin{equation*}
        (e^{i(s,\nabla^\theta)}\xi)(r) = \exp(i(s,r))\xi(r),\quad r \in \mathbb{R}^p.
    \end{equation*}
    Hence,
    \begin{equation*}
        e^{i(s,\nabla^\theta)}U(s)e^{-i(s,\nabla^\theta)} = e^{i(s,t)}U(s),\quad s,t \in \mathbb{R}^p.
    \end{equation*}
    For convenience we also introduce the notation $\Delta^\theta := \sum_{k=1}^p (D_k^\theta)^2$.
    
    The following is \cite[Proposition 6.12]{LeSZ-cwikel}:
    \begin{lem}\label{nc poincare}
        Let $k = 1,\ldots,p$. If $x \in L_\infty(\mathbb{R}^p_\theta)$, and the operator $[D_k^\theta,x]$ has bounded extension, then its extension is an element
        of $L_\infty(\mathbb{R}^d_\theta)$.
    \end{lem}
    
    \begin{defi}
        If $x \in L_\infty(\mathbb{R}^d_\theta)$ is such that $[D_k^\theta,x]$ has bounded extension, then we denote $\partial_kx$ for the extension.
        
        We denote $\partial_j^0x := x$, for all $x \in L_\infty(\mathbb{R}^p_\theta)$ and $j$.
        
        Generally, let $\alpha = (\alpha_1,\alpha_2,\ldots,\alpha_p)$ be a multi-index. If for all $1 \leq j \leq p$ the operator
        \begin{equation*}
            \partial_j^{\alpha_j}(\partial_{j+1}^{\alpha_{j+1}}(\cdots(\partial_p^{\alpha_p}(x))\cdots))
        \end{equation*}
        has bounded extension, then the mixed partial derivative $\partial^{\alpha}x$ is defined as the operator:
        \begin{equation*}
            \partial_1^{\alpha_1}(\partial_2^{\alpha_2}(\cdots(\partial_p^{\alpha_p}(x))\cdots)).
        \end{equation*}
    \end{defi}
    By Lemma \ref{nc poincare}, we always have $\partial^\alpha x \in L_\infty(\mathbb{R}^p_\theta)$ if it is well defined.
    
    \begin{defi}
        Let $m \geq 0$ and $r \geq 1$. The space $W^{m,r}(\mathbb{R}^p_\theta)$ is defined to be the set of $x \in L_\infty(\mathbb{R}^p_\theta)$ such that 
        $\partial^{\alpha}x \in L_r(\mathbb{R}^p_\theta)$ for every $|\alpha| \leq m$, equipped with the norm:
        \begin{equation*}
            \|x\|_{W^{m,r}} := \sum_{|\alpha|\leq m} \|\partial^{\alpha}x\|_r.
        \end{equation*}
        
        We define $W^{\infty,r}(\mathbb{R}^p_\theta) := \bigcap_{m\geq 0} W^{m,r}(\mathbb{R}^p_\theta)$.
    \end{defi}  
    
    As suggested by the notation, the spaces $W^{m,r}(\mathbb{R}^p_\theta)$ are the analogues of Sobolev spaces for noncommutative Euclidean spaces. The space $W^{\infty,1}(\mathbb{R}^p_\theta)$
    is important because it forms a part of our spectral triple for noncommutative Euclidean space.
    
    The remainder of this section is devoted to showing that the triple
    \begin{equation}\label{ncp spectral triple}
        (1_{2^{p/2}}\otimes W^{\infty,1}(\mathbb{R}^p_\theta),L_2(\mathbb{R}^{p},\mathbb{C}^{2^{p/2}}),D^\theta)
    \end{equation}
    is a spectral triple satisfying Hypothesis \ref{main assumption}.

    We wish to verify Hypothesis \ref{main assumption} for noncommutative spaces in order to support our claim that \ref{main assumption} is a reasonable
    assumption to make. However, in the nondegenerate case $\det(\theta) \neq 0$, the Character Theorem \ref{main thm} is trivial for at least a dense
    subalgebra of $W^{\infty,1}(\mathbb{R}^d_\theta$).
    
    The reason for this is that due to \cite[Proposition 2.5]{gayral-moyal},
    there is a dense subalgebra of $W^{\infty,1}(\mathbb{R}^d_\theta)$ isomorphic to the 
    algebra $M_\infty(\mathbb{C})$ of finitely supported infinite matrices. However
    due to \cite[Theorem 1.4.14]{Loday-cyclic-homology}, if $n\geq 0$ then the
    $n$th Hochschild homology of $M_\infty(\mathbb{C})$ is computed by:
    \begin{equation*}
        HH_{n}(M_\infty(\mathbb{C})) = HH_n(\mathbb{C}).
    \end{equation*}
    For $n> 0$, the $n$th Hochschild homology of $\mathbb{C}$ is trivial \cite[Lemma 8.9]{GVF}.
    Hence for $n> 0$, the $n$th Hochschild homology of $M_\infty(\mathbb{C})$ is trivial.
    
    This entails that every degree $p+1$ Hochschild cycle of $M_\infty(\mathbb{C})$ is a Hochschild boundary. However the left and right hand sides of the Character Theorem,
    $c \mapsto \mathrm{Ch}(c)$ and $c\mapsto \varphi(\Omega(c)(1+D^2)^{-p/2})$, are
    both Hochschild cocycles and hence vanish on any Hochschild boundary.       
    
\subsection{Verification of Hypothesis \ref{main assumption} for Noncommutative Euclidean spaces}\label{ncp verification subsection}
    Now we prove that the triple \eqref{ncp spectral triple} is a spectral triple satisfying Hypothesis \ref{main assumption}. In fact
    it is easier to use Hypothesis \ref{replacement assumption}.

    Our main reference for this section is \cite[Section 7]{LeSZ-cwikel}. As in that reference, the spaces $\ell_1(L_\infty)$ and $\ell_{1,\infty}(L_\infty)$ are defined as follows:
    Let $K = [0,1]^p$ be the unit $p$-cube. Then $\ell_1(L_\infty)$ and $\ell_{1,\infty}(L_\infty)$ are the subspaces of $L_\infty(\mathbb{R}^d)$ such that the following norms are finite:
    \begin{align*}
        \|g\|_{\ell_1(L_\infty)} &:= \|\{\|g\|_{L_{\infty}(m+K)}\}_{m \in \mathbb{Z}^p}\|_{\ell_1(\mathbb{Z}^p)},\\
        \|g\|_{\ell_{1,\infty}(L_\infty)} &:= \|\{\|g\|_{L_{\infty}(m+K)}\}_{m \in \mathbb{Z}^p}\|_{\ell_{1,\infty}(\mathbb{Z}^p)}.
    \end{align*}
    
    The following is a special case of \cite[Theorem 7.6, Theorem 7.7]{LeSZ-cwikel}:    
    \begin{thm}\label{nc cwikel theorem}
        Let $p \geq 1$. There are constants $c_p > 0$ and $c_p' > 0$ such that for all $x \in W^{p,1}(\mathbb{R}^p_\theta)$ we have:
        \begin{enumerate}[{\rm (a)}]
            \item{}\label{first nc cwikel} If $g \in \ell_1(L_\infty)$, then $xg(\nabla^\theta) \in \mathcal{L}_1$, and $$\|xg(\nabla^\theta)\|_1 \leq c_p\|x\|_{W^{p,1}}\|g\|_{\ell_1(L_\infty)}.$$
            \item{}\label{second nc cwikel} If $g \in \ell_{1,\infty}(L_\infty)$, then $xg(\nabla^\theta) \in \mathcal{L}_{1,\infty}$ and $$\|xg(\nabla^\theta)\|_{1,\infty} \leq c_p'\|x\|_{W^{p,1}}\|g\|_{\ell_{1,\infty}(L_\infty)}.$$
        \end{enumerate}
    \end{thm}
    
    With Theorem \ref{nc cwikel theorem} at hand we can prove the following:
    \begin{thm}\label{nc dirac cwikel}
        Let $p\geq 1$. Then there exist constants $c_p > 0$ and $c_p' > 0$ such that for all $x \in W^{p,1}(\mathbb{R}^p_\theta)$ we have:
        \begin{enumerate}[{\rm (a)}]
            \item{}\label{first cwikel} $(1\otimes x)(D^\theta+i\lambda)^{-p-1} \in \mathcal{L}_1$ and $$\|(1\otimes x)(D^\theta+i\lambda)^{-p-1}\|_1 \leq c_p\frac{\|x\|_{W^{p,1}}}{\lambda},$$
            \item{}\label{second cwikel} $(1\otimes x)(D^\theta+i)^{-p} \in \mathcal{L}_{1,\infty}$ and $$\|(1\otimes x)(D^\theta+i)^{-p}\|_{1,\infty} \leq c_p'\|x\|_{W^{p,1}}.$$
        \end{enumerate}
    \end{thm}
    \begin{proof}
        Let $g(t) := (\lambda^2+\sum_{k=1}^p t_k^2)^{-(p+1)/2}$. Since $\|ab\|_1 = \|a|b^*|\|_1$,
        \begin{align*}
            \|(1\otimes x)(D^\theta + /i\lambda)^{-p-1}\|_1 &= \|(1\otimes x)|(D^\theta-i\lambda)^{-p-1}|\|_1\\
                                                          &= \|(1\otimes x)((D^\theta)^2+\lambda^2)^{-\frac{p+1}{2}}\|_1.
        \end{align*}
        So we have
        \begin{equation*}
            \|(1\otimes x)(D^\theta+i\lambda)^{-p-1}\|_1 = c_p\|xg(\nabla^\theta)\|_1.
        \end{equation*}
        It can be directly verified that $\|g\|_{\ell_1(L_\infty)} = O(\lambda^{-1})$, we can immediately apply Theorem \ref{nc cwikel theorem} to obtain \eqref{first cwikel}.
        
        To obtain \eqref{second cwikel}, we instead consider the function $g(t) = (1+\sum_{k=1}^p t_k^2)^{-p/2}$ and apply Theorem \ref{nc cwikel theorem}.\eqref{second nc cwikel}.
    \end{proof}
    
    Recall the operator $\Lambda$ from Section \ref{replacement section}, defined formally as $\Lambda(T) = (1+D^2)^{-1/2}[D^2,T]$.

    \begin{lem}\label{ncplane main assumption for lambda} 
    If $x\in W^{\infty,1}(\mathbb{R}^p_{\theta})$, then for all $m\geq 0$:
        \begin{equation*}
            \left\|\Lambda^m(1\otimes x)(D^\theta+i\lambda)^{-p-1}\right\|_1 = O(\lambda^{-1}),\quad\lambda\to\infty.
        \end{equation*}
    \end{lem}
    \begin{proof}
        We prove the assertion by induction on $m.$ Since by definition $\Lambda^0$ is the identity, the $m=0$ case is handled by Theorem \ref{nc dirac cwikel}.\eqref{first cwikel}.

        Now suppose that $m\geq 1$ and the assertion holds for $m-1.$ 
        Since $(D^\theta)^2 = 1\otimes \sum_{k=1}^p (D_k^\theta)^2$, we have
        \begin{equation*}
            \Lambda(1\otimes x) = (1+(D^\theta)^2)^{-\frac{1}{2}}\cdot(\sum_{k=1}^p1\otimes[(D_k^\theta)^2,x]).
        \end{equation*}
        Applying the Leibniz rule,
        \begin{align*} 
            [(D_k^\theta)^2,x] &= [D_k^\theta,x]D_k^\theta+D_k^\theta[D_k^\theta,x]\\
                               &= 2D_k^\theta[D_k^\theta,x]-[D_k^\theta,[D_k^\theta,x]].
        \end{align*}
        By assumption, (the bounded extensions of) $[D_k^\theta,x]$ and $[D_k^\theta,[D_k^\theta,x]]$ are in $W^{m,1}(\mathbb{R}^p_{\theta})$ for all $m\geq0.$

        Hence since $\Lambda$ commutes with $\partial$,
        \begin{align*}
            \Lambda^m(1\otimes x) = \sum_{k=1}^p&\Big(1\otimes \frac{2D_k^\theta}{(1-\Delta^\theta)^{\frac12}}\Big)\cdot \Lambda^{m-1}(1\otimes [D_k^\theta,x])\\
                                  &-\sum_{k=1}^p\Big(1\otimes\frac1{(1-\Delta^\theta)^{\frac12}}\Big)\cdot \Lambda^{m-1}(1\otimes [D_k^\theta,[D_k^\theta,x]]).
        \end{align*}
        So by the triangle inequality, we have
        \begin{align*}
            \left\|\Lambda^m(1\otimes x)(D^\theta+i\lambda)^{-p-1}\right\|_1 &\leq 2\sum_{k=1}^p\left\|\Lambda^{m-1}(1\otimes [D_k^\theta,x])(D^\theta+i\lambda)^{-p-1}\right\|_1\\
                                                                             &\quad +\sum_{k=1}^p\|\Lambda^{m-1}(1\otimes [D_k^\theta,[D_k^\theta,x]])(D^\theta+i\lambda)^{-p-1}\Big\|_1.
        \end{align*}
        The right hand side is $O(\lambda^{-1})$ as $\lambda\to\infty$ by the inductive assumption. Hence, so is the left hand side.
    \end{proof}

    We can now conclude with the main result of this subsection:
    \begin{thm}
        The triple
        \begin{equation*}
            (1\otimes W^{\infty,1}(\mathbb{R}^p,\theta),L_2(\mathbb{R}^d,\mathbb{C}^{2^{p/2}}),D^\theta)
        \end{equation*}
        is a spectral triple satisfying Hypothesis \ref{main assumption}.
    \end{thm}
    \begin{proof}
        We establish Hypothesis \ref{replacement assumption} instead, as permitted by Theorem \ref{replacement thm}. 
        First we prove that we indeed have a spectral triple. 
        
        By the definition of $W^{\infty,1}(\mathbb{R}^p)$, if $x \in W^{\infty,1}(\mathbb{R}^p)$ then $[D^\theta,1\otimes x]$
        has bounded extension, and therefore $1\otimes x:\mathrm{dom}(D^\theta)\to \mathrm{dom}(D^\theta)$.
        
        
        If $x \in W^{\infty,1}(\mathbb{R}^p_\theta)$ then:
        \begin{equation*}
            \partial(1\otimes x) = \sum_{j=1}^p \gamma_j\otimes (\partial_jx)
        \end{equation*}
        and this is bounded, by the definition of $W^{\infty,1}$.        
        
        Now we show that Hypothesis \ref{replacement assumption}.\eqref{rass0} holds.
        If $x \in W^{\infty,1}(\mathbb{R}^p_\theta)$, we show that $x,\partial(x) \in \mathrm{dom}(\Lambda^m)$ for all $m\geq 0$ by induction. We have automatically that $x,\partial(x) \in \mathrm{dom}(\Lambda^0)$. Now if we assume that $x,\partial(x) \in \mathrm{dom}(\Lambda^{m-1})$, for $m \geq 1$, we apply the Leibniz rule to obtain:
        \begin{align*}
            \Lambda^m(1\otimes x) &= \sum_{k=1}^p\Big(1\otimes \frac{2D_k^\theta}{(1-\Delta^\theta)^{\frac12}}\Big)\cdot \Lambda^{m-1}(1\otimes [D_k^\theta,x])\\
                                  &\quad-\sum_{k=1}^p\Big(1\otimes\frac1{(1-\Delta^\theta)^{\frac12}}\Big)\cdot \Lambda^{m-1}(1\otimes [D_k^\theta,[D_k^\theta,x]]).
        \end{align*}
        By the definition of $W^{\infty,1}(\mathbb{R}^p_\theta)$, the operators $[D_k^\theta,x]$ and $[D_k^\theta,[D_k^\theta,x]]$ have bounded extension, and by Lemma \ref{nc poincare}, the extensions of $[D_k^\theta,x]$ and $[D_k^\theta,[D_k^\theta,x]]$ are elements of $W^{\infty,1}(\mathbb{R}^p_\theta)$, and therefore by the inductive hypothesis are in $\mathrm{dom}(\Lambda^{m-1})$. Hence, $1\otimes x \in \mathrm{dom}(\Lambda^m)$ and so by induction $1\otimes x \in \mathrm{dom}_\infty(\Lambda)$. Applying an identical argument to $\partial(1\otimes x)$ yields $\partial(1\otimes x) \in \mathrm{dom}_\infty(\Lambda)$, and so $(1\otimes W^{\infty,1}(\mathbb{R}^p_\theta),L_2(\mathbb{R}^d,\mathbb{C}^{2^{p/2}}),D^\theta)$ is $\Lambda$-smooth.        
        
        We now show that Hypothesis \ref{replacement assumption}.\eqref{rass1} holds. 
        Let $x \in W^{\infty,1}(\mathbb{R}^p_\theta)$. By Lemma \ref{nc dirac cwikel}.\eqref{first cwikel}, the first inclusion in Hypothesis \ref{replacement assumption}.\eqref{rass1} follows. 
        
        To see the second inclusion in Hypothesis \ref{replacement assumption}.\eqref{rass1}, write
        \begin{equation*}
            \partial(1\otimes x)(D^\theta+i)^{-p} = \sum_{k=1}^p \left(\gamma_k\otimes 1\right)\cdot\left((1\otimes [D_k^\theta,x])(D^\theta+i)^{-p}\right).
        \end{equation*}
        Using the quasi-triangle inequality for $\mathcal{L}_{1,\infty}$, there is a constant $C_p$ such that
        \begin{equation*}
            \|\partial(1\otimes x)(D^\theta+i)^{-p}\|_{1,\infty} \leq C_p\sum_{k=1}^p \|(1\otimes \partial_k x)(D^\theta+i)^{-p}\|_{1,\infty}.
        \end{equation*}
        By the definition of $W^{\infty,1}(\mathbb{R}^p_\theta)$, for all $1 \leq k \leq p$ we have $\partial_k x \in W^{\infty,1}(\mathbb{R}^p_\theta)$, so we may apply Theorem \ref{nc dirac cwikel}.\eqref{second cwikel} to each summand to deduce the second inclusion in Hypothesis \ref{replacement assumption}.\eqref{rass1}.
        
        Now we discuss Hypothesis \ref{replacement assumption}.\eqref{rass2}. By Lemma \ref{ncplane main assumption for lambda}, the first inequality in Hypothesis \ref{replacement assumption}.\eqref{rass2} holds.
        
        To deduce the second inequality, we may commute $\partial$ with $\Lambda^{m}$ to obtain:
        \begin{equation*}
            \partial(\Lambda^m(1\otimes x))(D^\theta+i\lambda)^{-p-1} = \sum_{k=1}^p \left(\gamma_k\otimes 1\right)\cdot\left(\Lambda^m(1\otimes [D_k^\theta,x]\right)(D^\theta+i\lambda)^{-p-1}.
        \end{equation*}
        Note that here $\partial(T)$ denotes $[D^\theta,T]$. Using the $\mathcal{L}_1$-norm triangle inequality,
        \begin{equation*}
            \|\partial(\Lambda^m(1\otimes x))(D^\theta+i\lambda)^{-p-1}\|_1 \leq C_p\sum_{k=1}^p \|\Lambda^m(1\otimes \partial_k x)(D^\theta+i\lambda)^{-p-1}\|_1.
        \end{equation*}
        By assumption, each $\partial_k x$ is in $W^{\infty,1}(\mathbb{R}^p_\theta)$, and so by Lemma \ref{ncplane main assumption for lambda}, each summand above is $O(\lambda^{-1})$
        as $\lambda \to \infty$.
    \end{proof}
    
    \begin{rem}
        We have worked exclusively with the case that $\det(\theta) \neq 0$. Of course this excludes the fundamental $\theta = 0$ case of Euclidean space $\mathbb{R}^d$. One
        may verify directly that the standard spectral triple for $\mathbb{R}^d$ satisfies Hypothesis \ref{main assumption}, by using classical Cwikel theory, or alternatively
        $\mathbb{R}^d$ may be considered as a special case of the complete Riemannian manifolds considered in the following section.
    \end{rem}

%
%

\section{Example: Riemannian manifolds}\label{manifold section}

    The authors wish to thank Professor Yuri Kordyukov for significant contributions to this section, including providing many of the proofs.

\subsection{Basic notions about manifolds}
    
    We briefly recall the relevant definitions for Riemannian manifolds. The material in this subsection is standard,
    and may be found in for example \cite[Chapter 2]{rosenberg} or \cite{Lawson-Michelsohn-1989}. 
    Let $X$ be a second countable $p$-dimensional complete smooth Riemannian manifold with metric tensor $g$.
    Recall that $g$ defines a canonical measure $\nu_g$ on $X$. The notation $L_r(X,g)$ denotes $L_r(X,\nu_g)$. The assumption that $X$ is second countable
    ensures that $L_2(X,g)$ is separable.
    
    We denote the space of smooth compactly supported differential $k$-forms as $\Omega_c^k(X)$, and define $\Omega_c(X) := \bigoplus_{k=0}^p \Omega_c^k(X)$.
    Associated to the metric $g$ is an inner product $(\cdot,\cdot)_g$ defined on $\Omega_c(X)$.
    Let $H_k$ denote the completion of $\Omega^k_c(X)$ with respect to this inner product, and define $L_2\Omega(X,g) := \bigoplus_{k=0}^p H_k$.
    There is a grading on $L_2\Omega(X,g)$ with grading operator $\Gamma$ defined by $\Gamma|_{H_k} = (-1)^k$.
    
    For $f \in C^\infty_c(X)$, let $M_f$ denote the operator of pointwise multiplication by $f$ on $L_2\Omega(X,g)$.
    
    The exterior differential $d$ is a linear map $d:\Omega_c(X)\to\Omega_c(X)$ such that for all $k = 0,1,\ldots,p-1$
    we have
    $d:\Omega^{k}_c(X)\to \Omega^{k+1}_c(X)$, and $d|\Omega^p_c(X) = 0$. The linear operator $d$ has a formal
    adjoint $d^*$ with respect to the inner product on $\Omega_c(X)$.
    
    The Hodge-Dirac operator $D_g$ is defined by $D_g := d+d^*$. Since $X$ is complete, the operator $D_g$ uniquely extends to a self-adjoint unbounded operator on $L_2\Omega(X,g)$ (see \cite{chernov}). The Hodge-Laplace operator is defined as $\Delta_g := -D_g^2 = -dd^*-d^*d$, and each subspace $H_k$
    is invariant under $\Delta_g$. The restriction of $\Delta_g$ to $H_0 = L_2(X,g)$ coincides with the Laplace-Beltrami operator.
    
    The main focus of this section is the following:
    \begin{thm}\label{manifold theorem} 
        Let $(X,g)$ be a second countable $p$-dimensional complete Riemannian manifold. The algebra $C^\infty_c(X)$ acts on $L_2\Omega(X,g)$ by pointwise multiplication, and $D_g$ denotes
        the Hodge-Dirac operator. Then
        \begin{equation*}
            (C^\infty_c(X),L_2\Omega(X,g),D_g)
        \end{equation*}
        is an even spectral triple satisfying Hypothesis \ref{main assumption}, where the grading $\Gamma$ is defined by $\Gamma|_{H_k} = (-1)^k$.
    \end{thm}
    
    Note that the spectral triple is always even regardless of $p$. The use of the Hodge-Dirac operator to define spectral triples for arbitrary Riemannian manifolds has previously been studied in \cite{Lord-Rennie-Varilly}, and for related work see \cite{Frohlich-Grandjean-Recknagel}.
    
    To the best of our knowledge, the main results of this paper, Theorems \ref{heat thm}, \ref{zeta thm} and \ref{main thm} are new in the setting of the Hodge-Dirac operator on arbitrary complete manifolds. Most previous work on geometric applications of noncommutative geometry, such as \cite{Rennie-2004} and \cite{CGRS2} are applied to a spin Dirac operator. 
        
    The Cwikel-type estimates we establish in this section: Lemma \ref{kord first lemma} and \ref{kord second lemma}, are of interest in their own right. A predecessor to this work may be found in \cite{Rennie-2004}.
    
\subsection{Proof of Theorem \ref{manifold theorem}}

    The proof proceeds by showing the required Cwikel-type estimates for the case of a torus: $X = \mathbb{T}^p$ with the flat metric. We then
    deduce the general case by an argument involving local coordinates.
    

    We define the $p$-torus as $\mathbb{T}^p := \mathbb{R}^p/\mathbb{Z}^p$. The space $\mathbb{T}^p$ is a smooth $p$-dimensional manifold,
    and we may select local coordinates $x_1,\ldots,x_p \in (0,1]$ defined by considering the image of $(x_1,\ldots,x_d)$ in $\mathbb{R}^p/\mathbb{Z}^p$.
%
    We equip $\mathbb{T}^p$ with the flat metric $g_0$ defined locally by $g_0 = dx_1^2+\cdots+dx_p^2$.

    First, we describe the Cwikel-type estimates for $\mathbb{T}^p$.
    \begin{lem}\label{cwikel}
        Let $g_0$ denote the flat metric on $\mathbb{T}^p$, with corresponding Hodge-Dirac operator denoted $D_0$. Then:
        \begin{enumerate}[{\rm (i)}]
            \item{}\label{flat cwikel 1} We have 
                \begin{equation*}
                    (D_0+i)^{-1} \in \mathcal{L}_{p,\infty}(L_2\Omega(\mathbb{T}^p,g_0)).
                \end{equation*}
            \item{}\label{flat cwikel 2} For $\lambda \to\infty$, we have:
                \begin{equation*}
                    \|(D_0+i\lambda)^{-1}\|_{\mathcal{L}_{p+1}(L_2\Omega(\mathbb{T}^p,g_0))} = O(\lambda^{-\frac{1}{p+1}}).
                \end{equation*}
        \end{enumerate}
    \end{lem}
        
    Since the manifold $\mathbb{T}^p$ is compact, there is essentially no difference between the spaces $L_2\Omega(\mathbb{T}^p,g)$ for different metrics $g$ as the following lemma shows:
    \begin{lem}\label{L_2 is invariant under changing metric}
        Let $g_0$ be the flat metric on $\mathbb{T}^p$, and let $g$ be an arbitrary metric on $\mathbb{T}^p$. Then the Hilbert spaces $L_2\Omega(\mathbb{T}^p,g)$ and $L_2\Omega(\mathbb{T}^p,g_0)$
        coincide with an equivalence of norms. To be precise, there exist constants $0 < c_g < C_g < \infty$ with $v \in L_2\Omega(\mathbb{T}^p,g)$ we have:
        \begin{equation*}
            c_g\|v\|_{L_2\Omega(\mathbb{T}^p,g_0)} \leq \|v\|_{L_2\Omega(\mathbb{T}^p,g)} \leq C_g\|v\|_{L_2\Omega(\mathbb{T}^p,g_0)}.
        \end{equation*}
    \end{lem}
    \begin{proof}
        The metric $\nu_g$ corresponding to $g$ has Radon-Nikodym derivative $\sqrt{|\det(g)|}$ with respect to $\nu_{g_0}$. Since $\mathbb{T}^p$
        is compact and $g$ is positive definite, the Radon-Nikodym derivative $\sqrt{|\det(g)|}$ is bounded above and bounded away from zero. Hence,
        the $L_2$-norms corresponding to $\nu_g$ and $\nu_{g_0}$ are equivalent.
    \end{proof}
    
    Sobolev spaces on $\mathbb{T}^p$ {  -- and more generally on a compact Riemannian manifold -- are defined following \cite{Lawson-Michelsohn-1989}:}
    \begin{defi}
        Let $g$ be a metric on $\mathbb{T}^p$, and for $j=1,\ldots,p$ we let $\frac{\partial}{\partial x_j}$ denote the differentiation
        with respect to the $j$th coordinate of $\mathbb{T}^p$. The Sobolev space $H^1\Omega(\mathbb{T}^p)$ is defined to be the set of $v \in L_2\Omega(\mathbb{T}^p,g)$ such that the Sobolev norm:
        \begin{equation*}
            \|v\|_{H^1\Omega(\mathbb{T}^p,g)}^2 := \|v\|_{L_2\Omega(\mathbb{T}^p,g)}^2 + \sum_{j=1}^p \left\|\frac{\partial v}{\partial x_j}\right\|_{L_2\Omega(\mathbb{T}^p,g)}^2
        \end{equation*}
        is finite.
    \end{defi}
    Lemma \ref{L_2 is invariant under changing metric} shows that the space $H^1\Omega(\mathbb{T}^p)$ is independent of the choice of metric used to define the Sobolev norm,
    since different metrics will define equivalent norms.
        
    The following result is the well known G\aa rding's inequality. A proof for general compact manifolds may be found in \cite[Theorem 2.44]{rosenberg} {  and for general elliptic operators on compact manifolds in \cite[Theorem 5.2]{Lawson-Michelsohn-1989}. }
    \begin{lem}\label{gardings inequality}
        Let $g$ be a metric on $\mathbb{T}^p$, and let $D_g$ denote the corresponding Dirac operator. Then there is a constant $C_g >0$ such that for all $v \in L_2\Omega(\mathbb{T}^p,g)$ such
        that $D_gv \in L_2\Omega(\mathbb{T}_p,g)$, we have
        \begin{equation*}
            \|v\|_{H^1\Omega(\mathbb{T}^p,g)} \leq C_g(\|v\|_{L_2\Omega(\mathbb{T}^p)}+\|D_gv\|_{L_2\Omega(\mathbb{T}^p,g)}).
        \end{equation*}
    \end{lem}
    
    The following lemma is the essential technical result allowing us to transfer the Cwikel-type estimates for $\mathbb{T}^p$ with the flat metric
    to $\mathbb{T}^p$ with an arbitrary metric.
    \begin{lem}\label{main transition lemma} 
        Let $g$ be an arbitrary metric on $\mathbb{T}^p$, and let $g_0$ be the flat metric on $\mathbb{T}^p$. Denote by $D_g$ and $D_0$ the Hodge-Dirac
        operators corresponding to $g$ and $g_0$ respectively. Then the operator
        \begin{equation*}
            D_0(D_g+i)^{-1}
        \end{equation*}
        defined initially on $\Omega(\mathbb{T}^p)$ has bounded extension to $L_2\Omega(\mathbb{T}^p,g)$ (or equivalently $L_2\Omega(\mathbb{T}^p,g_0)$).
    \end{lem}
    \begin{proof} 
        By the definition of the Sobolev norm $\|\cdot \|_{H^1\Omega(\mathbb{T}^p,g)}$, for all $v \in H^1\Omega(\mathbb{T}^p,g)$ we have:
        \begin{equation*}
            \|D_0v\|_{L_2\Omega(\mathbb{T}^p,g)}^2 \leq \|v\|_{H^1\Omega(\mathbb{T}^p,g)}^2.
        \end{equation*}
        Thus by Lemma \ref{gardings inequality}, there is a constant $C$ such that,
        \begin{equation*}
            \|D_0v\|_{L_2\Omega(\mathbb{T}^p,g)} \leq C(\|D_gv\|_{L_2\Omega(\mathbb{T}^p,g)}+\|v\|_{L_2\Omega(\mathbb{T}^p,g)}).
        \end{equation*}
        Hence, if $v \in \mathrm{dom}(D_g)$ then $v \in \mathrm{dom}(D_0)$, and so if $u \in L_2\Omega(\mathbb{T}^p,g)$ and $v = (D_g+i)^{-1}u$ then $v \in \mathrm{dom}(D_0)$. 
        So substituting $v = (D_g+i)^{-1}u$ we obtain:
        \begin{equation*}
            \|D_0(D_g+i)^{-1}u\|_{L_2\Omega(\mathbb{T}^p,g)} \leq C(1+\|(D_g+i)^{-1}\|_{\mathcal{L}_{\infty}(L_2\Omega(\mathbb{T}^p,g)})\|u\|_{L_2\Omega(\mathbb{T}^p,g)}.
        \end{equation*}
        However since $(D_g+i)^{-1}$ is bounded, we get:
        \begin{equation*}
            \|D_0(D_g+i)^{-1}u\|_{L_2\Omega(\mathbb{T}^p,g)} \leq C\|u\|_{L_2\Omega(\mathbb{T}^p,g)}.
        \end{equation*}
    \end{proof}
    
    As a consequence of Lemmas \ref{cwikel} and \ref{main transition lemma}, we get:
    \begin{cor}\label{non flat cwikel}
        Let $g$ be an arbitrary metric on $\mathbb{T}^p$ and let $D_g$ be the corresponding Hodge-Dirac operator. Then,
        \begin{enumerate}[{\rm (i)}]
            \item{}\label{non flat cwikel 1} We have,
                \begin{equation*}
                    (D_g+i)^{-1} \in \mathcal{L}_{p,\infty}(L_2\Omega(\mathbb{T}^p,g)).
                \end{equation*}
            \item{}\label{non flat cwikel 2} For $\lambda\to\infty$, we have
                \begin{equation*}
                    \|(D_g+i\lambda)^{-1}\|_{\mathcal{L}_{p+1}(L_2\Omega(\mathbb{T}^p,g))} = O(\lambda^{-\frac{1}{p+1}}).
                \end{equation*}
        \end{enumerate}
    \end{cor}
    \begin{proof}
        Part \eqref{non flat cwikel 1} follows immediately from Lemma \ref{cwikel} and Lemma \ref{main transition lemma}.\ref{flat cwikel 1}.
        
        Now we prove \eqref{non flat cwikel 2}. First we compute $(D_0+i\lambda)(D_g+i\lambda)^{-1}$ (working on the dense domain $\Omega(\mathbb{T}^p)$):
        \begin{align*}
            (D_0+i\lambda)(D_g+i\lambda)^{-1} &= (D_0-D_g+D_g+i\lambda)(D_g+i\lambda)^{-1}\\
                                              &= 1 + (D_0-D_g)(D_g+i\lambda)^{-1}\\
                                              &= 1+D_0(D_g+i)^{-1}\frac{D_g+i}{D_g+i\lambda} - \frac{D_g}{D_g+i\lambda}.
        \end{align*}
        By Lemma \ref{main transition lemma}, the operator $D_0(D_g+i)^{-1}$ has bounded extension. Moreover, by functional calculus
        the operators $\frac{D_g+i}{D_g+i\lambda}$ and $\frac{D_g}{D_g+i\lambda}$ have bounded extension with norm bounded by a constant independent of $\lambda$.
        Thus,
        \begin{equation*}
            \|(D_0+i\lambda)(D_g+i\lambda)^{-1}\|_{\mathcal{L}_\infty(L_2\Omega(\mathbb{T}^p,g)} = O(1).
        \end{equation*}
        Now Lemma \ref{cwikel}.\eqref{flat cwikel 2} yields the result.
    \end{proof}
        
    With the above results in hand we are able to establish our first Cwikel-type result for $(X,g).$ A similar result to Lemma \ref{kord first lemma} can also be found in \cite[Proposition 13]{Rennie-2004}. The result in \cite{Rennie-2004} is very similar in nature (despite applying
    to the somewhat different situation of Riemannian spin manifolds). The method of proof here is different: we use reduce the problem to $\mathbb{T}^p$ by using local coordinates, rather than a doubling construction as employed in \cite{Rennie-2004}.
    
    \begin{lem}\label{kord first lemma} 
        Let $f\in C^{\infty}_c(X).$ We have
        \begin{equation*}
            M_f(D_g+i)^{-1}\in\mathcal{L}_{p,\infty}(L_2\Omega(X,g)).
        \end{equation*}
    \end{lem}
    \begin{proof}
        By using a partition of unity if necessary, we may assume without loss of generality that $f$ is supported in a single chart $(U,h)$ where $h:U\to \mathbb{R}^p$ is a homeomorphism
        onto its image, and since $f$ has compact support we may further assume without loss of generality that $h(U)$ is bounded. Since $h(U)$ is bounded, there is a sufficiently
        large box $[-N,N]^p$ with $h(U)$ in the interior of $[-N,N]^p$. By applying a translation and dilation if necessary, we may assume without loss of generality that $h(U)$
        is contained within the interior of the box $[0,1]^p$. By identifying the edges of $[0,1]^p$, we may view $h$ as a continuous function $h:U\to\mathbb{T}^p$.
        
        We define three smooth ``cut-off" functions $\phi_1,\phi_2,\phi_3$ compactly supported in $h(U)$, defined so that for each $j = 1,2,3$ we have $0\leq \phi_j \leq 1$,
        and
        \begin{enumerate}[{\rm (a)}]
            \item{} for all $x \in U$, $\phi_1(h(x))f(x) = f(x)$,
            \item{} we have $\phi_2\phi_1 = \phi_1$,
            \item{} we have $\phi_3\phi_2 = \phi_2$.
        \end{enumerate}
        In other words, on $\mathrm{supp}(f\circ h^{-1})$ we have $\phi_1=1$, and on $\mathrm{supp}(\phi_1)$ we have $\phi_2 = 1$ and on $\mathrm{supp}(\phi_2)$ we have $\phi_3 = 1$.
        
        For $j = 1,2,3$, we also define the function $\psi_j$ by pulling back $\phi_j$ to $X$:
        \begin{equation*}
            \psi_j(x) = \begin{cases}
                                    (\phi_j\circ h)(x),\quad x\in U.\\
                                    0,\quad x\notin U.
                                \end{cases}
        \end{equation*}
        Since $\phi_j$ is compactly supported in $h(U)$, the function $\psi_j$ is smooth and compactly supported in $U$.
        
        Let $g_0$ denote the flat metric on $\mathbb{T}^p$.
        The metric $g$ can be pushed forward by $h$ to a metric $h^*g$ on $h(U)$. We then define a new metric $g_1$ on $\mathbb{T}^p$ by:
        \begin{equation*}
            g_1 := (h^*g)\phi_3+g_0(1-\phi_3).
        \end{equation*}        
        Since $\phi_3$ is compactly supported in $h(U)$, the metric $g_1$ is well defined. Moreover, on $\mathrm{supp}(\phi_2)$ we have $\phi_3 = 1$ so on $\mathrm{supp}(\phi_2)$ the metric
        $g_1$ is identical to $h^*g$.
        
        We define a partial isometry $V:L_2\Omega(X,g)\to L_2\Omega(\mathbb{T}^p,g_1)$ on $\xi \in L_2\Omega(X,g)$, $z \in \mathbb{T}^p$ by:
        \begin{equation*}
            V\xi(z) = \begin{cases}
                            \xi\circ h^{-1}(z),\quad z \in \mathrm{supp}(\phi_2),\\
                            0,\quad z\notin \mathrm{supp}(\phi_2).
                      \end{cases}
        \end{equation*}
        By construction, $V$ induces an isometry from $L_2\Omega(\mathrm{supp}(\psi_2),g)\to L_2\Omega(\mathrm{supp}(\phi_2),g_1)$, and
        \begin{align*}
            VV^* &= M_{\chi_{\mathrm{supp}{\phi_2}}},\\
            V^*V &= M_{\chi_{\mathrm{supp}{\psi_2}}}.
        \end{align*}
        
        We also have that if $j = 1,2$ then
        \begin{equation*}
            M_{\phi_j}V = VM_{\psi_{j}}.
        \end{equation*}
        
        We use the important fact that for $j = 1,2$, we have an equality on $\Omega_c(X)$:
        \begin{equation*}
            V^*D_{g_1}M_{\phi_j} = D_{g}M_{\psi_j}V^*.
        \end{equation*}
        
        Next we consider the following two operators on $L_2\Omega(\mathbb{T}^p,g_1)$.
        \begin{align*}
            P := M_{\phi_1}(D_{g_1}+i)^{-1}M_{\phi_1}\\
            Q := M_{\phi_2}(D_{g_1}+i)^{-1}M_{\phi_2}.
        \end{align*}
        By Lemma \ref{non flat cwikel}, the operators $P$ and $Q$ are in $\mathcal{L}_{p,\infty}(L_2\Omega(\mathbb{T}^p,g_1))$.
        
        We now consider the operator $(D_g+i)M_{\psi_2}V^*PV$. We note that this operator is well defined on $\Omega_c(X)$, since if $u \in \Omega_c(X)$, then $M_{\psi_2}V^*PV$ is 
        smooth and supported in $\mathrm{supp}{\psi_2}$.
        Hence (working on $\Omega_c(X)$):
        \begin{align*}
            (D_g+i)M_{\psi_2}V^*PV &= V^*(D_{g_1}+i)M_{\phi_2}PV\\
                                    &= V^*(D_{g_1}+i)M_{\phi_1}(D_{g_1}+i)^{-1}M_{\phi_1}V\\
                                    &= V^*([D_{g_1},M_{\phi_1}](D_{g_1}+i)^{-1}M_{\phi_1}+M_{\phi_1}^2)V.
        \end{align*}
        Now recalling that $D_{g_1}$ is a local operator, we have $[D_{g_1},M_{\phi_1}] = [D_{g_1},M_{\phi_1}]M_{\phi_2}$.
        
        Moreover, $V^*M_{\phi_1}^2V = M_{\psi_1}^2$ and so
        \begin{equation*}
            (D_g+i)M_{\psi_2}V^*PV = V^*[D_{g_1},M_{\phi_1}]QM_{\phi_1}V+M_{\psi_1}^2.
        \end{equation*}
        Now multiplying on the left by $(D_g+i)^{-1}$, we arrive at:
        \begin{equation*}
            M_{\psi_2}V^*PV = (D_g+i)^{-1}V^*[D_{g_1},M_{\phi_1}]QM_{\phi_1}V+(D_g+i)^{-1}M_{\psi_1}^2.
        \end{equation*}
        The final step is to use the fact that $\psi_2=1$ on the support of $f$, so we may use $M_{\psi_2}^2M_f = M_f$, and multiply on the right by $M_f$
        to obtain:
        \begin{equation*}
            M_{\psi_2}V^*PVM_f = (D_g+i)^{-1}V^*[D_1,M_{\phi_1}]QM_{\phi_1}VM_f+(D_g+i)^{-1}M_f.
        \end{equation*}
        Since both $P$ and $Q$ are in $\mathcal{L}_{p,\infty}(L_2\Omega(\mathbb{T}^p,g_1))$, we finally obtain that $(D_g+i)^{-1}M_f \in \mathcal{L}_{p,\infty}(L_2\Omega(X,g))$.
    \end{proof}
    
    \begin{lem}\label{kord second lemma} 
        Let $f\in C^{\infty}_c(X).$
        Then:
        \begin{equation*}
            \|M_f(D_g+i\lambda)^{-1}\|_{\mathcal{L}_{p+1}(L_2\Omega(X,g))} = O(\lambda^{-\frac{1}{p+1}}),\quad \lambda\to\infty.
        \end{equation*}
    \end{lem}
    \begin{proof} 
        This proof proceeds along similar lines to Lemma \ref{kord first lemma}. We again assume without loss of generality
        that $f$ is supported in a single chart $(U,h)$, and construct the metric $g_1$ on $\mathbb{T}^p$ and the partial isometry $V$ identically to the proof of Lemma \ref{kord first lemma}. We
        also use the same cut-off functions $\phi_1$, $\phi_2$ and $\phi_3$, and $\psi_1,\psi_2,\psi_3$.
        
        In place of the operators $P$ and $Q$, we introduce $P_{\lambda}$ and $Q_{\lambda}$ given by:
        \begin{align*}
            P_{\lambda} &= M_{\phi_1}(D_{g_1}+i\lambda)^{-1}M_{\phi_1}\\
            Q_{\lambda} &= M_{\phi_2}(D_{g_1}+i\lambda)^{-1}M_{\phi_2}.
        \end{align*}
        
        Following the argument of Lemma \ref{kord first lemma} with $P_{\lambda}$ and $Q_\lambda$ in place of $P$ and $Q$, we arrive at:
        \begin{equation*}
            M_{\psi_2}V^*P_{\lambda}VM_f = (D_g+i\lambda)^{-1}V^*[D_1,M_{\phi_1}]Q_{\lambda}M_{\phi_1}VM_f + (D_g+i\lambda)^{-1}M_f.
        \end{equation*}
        
        Due to Lemma \ref{non flat cwikel}.\eqref{non flat cwikel 1}, we have $\|P_\lambda\|_{p+1} = O(\lambda^{-\frac{1}{p+1}})$ and similarly for $Q_\lambda$. Thus,
        \begin{equation*}
            \|(D_g+i\lambda)^{-1}\|_{\mathcal{L}_{p+1}(L_2\Omega(X,g))} = O(\lambda^{-\frac{1}{p+1}}).
        \end{equation*}
    \end{proof}

    We now finally have the results necessary to prove Theorem \ref{manifold theorem}.
    \begin{proof}[Proof of Theorem \ref{manifold theorem}] 
        We will instead work with Hypothesis \ref{replacement assumption}, as justified by Theorem \ref{replacement thm}. 
                
        First, we show that \ref{replacement assumption}.\eqref{rass1} holds
        for the triple $(C^\infty_c(X),L_2\Omega(X,g),D_g)$.
        
        To this end let $f \in C^\infty_c(X)$.
        We will prove by induction that for all $j\geq 1$, we have
        \begin{equation}\label{manifold inductive eq1}
            M_f(D_g+i)^{-j} \in \mathcal{L}_{\frac{p}{j},\infty}(L_2\Omega(X,g)).
        \end{equation}
        The case $j=1$ is already established by Lemma \ref{kord first lemma}. 
        
        Suppose now that \eqref{manifold inductive eq1} holds for $j \geq 1$. Choose $\phi \in C^\infty_c(X)$ such that $f\phi = f$. Then,
        \begin{align*}
            M_f(D_g+i)^{-j-1} &= M_{\phi}M_f(D_g+i)^{-1}\cdot (D_g+i)^{-j}\\
                              &= -M_{\phi}[(D_g+i)^{-1},M_f](D_g+i)^{-j}+M_{\phi}(D_g+i)^{-1}M_f(D_g+i)^{-j}\\
                              &= M_{\phi}(D_g+i)^{-1}[D,M_f](D_g+i)^{-j-1}+M_{\phi}(D_g+i)^{-1}M_f(D_g+i)^{-j}\\
                              &= M_{\phi}(D_g+i)^{-1}[D,M_f]M_{\phi}(D_g+i)^{-j-1}+M_{\phi}(D_g+i)^{-1}M_f(D_g+i)^{-j}.
        \end{align*}
        Due to Lemma \ref{kord first lemma}, we have $M_{\phi}(D_g+i)^{-1} \in \mathcal{L}_{p,\infty}$, and by the inductive assumption
        we also have $M_{\phi}(D_g+i)^{-j} \in \mathcal{L}_{\frac{p}{j},\infty}$. Then $M_f(D_g+i)^{-j-1} \in \mathcal{L}_{p,\infty}\cdot \mathcal{L}_{\frac{p}{j},\infty}$,
        so applying the H\"older inequality we arrive at $M_f(D_g+i)^{-j-1} \in \mathcal{L}_{\frac{p}{j+1},\infty}$. Taking $j = p$, we get that $M_f(D_g+i)^{-p} \in \mathcal{L}_{1,\infty}(L_2\Omega(X,g))$.
        
        Similarly, since $D_g$ is a local operator, we have that $[D_g,M_f] = [D_g,M_f]M_{\phi}$. Hence, $[D_g,M_f](D_g+i)^{-p} \in \mathcal{L}_{1,\infty}(L_2\Omega(X,g))$. This completes the proof of Hypothesis \ref{replacement assumption}.\eqref{rass1}
        in the case $k=0$.
        
        What remains is to show that Hypothesis \ref{replacement assumption}.\eqref{ass2} holds. We will first deal with the $k=0$ case. To that end, we will show by induction that for all $j\geq 1$:
        \begin{equation}\label{manifold inductive eq2}
            \|M_f(D_g+i\lambda)^{-j}\|_{\frac{p+1}{j}} = O(\lambda^{-\frac{j}{p+1}}),\lambda\to\infty.
        \end{equation}
        The base case $j = 1$ is the result of Lemma \ref{kord second lemma}, and the case $j=p+1$ is what is required for Hypothesis \ref{replacement assumption}.
        
        Suppose now that \eqref{manifold inductive eq2} holds for $j\geq 1$, and again choose $\phi\in C^\infty_c(X)$ such that $f\phi = f$. Then,
        \begin{align*}
            M_f(D_g+i\lambda)^{-j-1}  &= M_{\phi}\cdot M_f(D_g+i\lambda)^{-1}(D_g+i\lambda)^{-j}\\
                                      &= -M_{\phi}[(D_g+i\lambda)^{-1},M_f](D_g+i\lambda)^{-j}\\
                                      &\quad+M_{\phi}(D_g+i\lambda)^{-1}M_f(D_g+i\lambda)^{-j}\\
                                      &= M_{\phi}(D_g+i\lambda)^{-1}[D,M_f]M_{\phi}(D_g+i\lambda)^{-j-1}\\
                                      &\quad+M_{\phi}(D_g+i\lambda)^{-1}M_f(D_g+i\lambda)^{-j}.
        \end{align*}    
        By Lemma \ref{kord second lemma}, we have $\|M_{\phi}(D_g+i\lambda)^{-1}\|_{p+1} = O(\lambda^{-\frac{1}{p+1}})$, and by the inductive assumption
        we also have $\|M_{\phi}(D_g+i\lambda)^{-k}\|_{\frac{p+1}{j}} = O(\lambda^{-\frac{j}{p+1}})$. Then by the Holder inequality,
        \begin{equation*}
            \|M_f(D_g+i\lambda)^{-j-1}\|_{\frac{p+1}{j+1}} \leq O(\lambda^{-\frac{j}{p+1}})\cdot O(\lambda^{-\frac{1}{p+1}}).
        \end{equation*}
        So $\|M_f(D_g+i\lambda)^{-j-1}\|_{\frac{p+1}{j+1}} = O(\lambda^{-\frac{j+1}{p+1}})$. To conclude the same for $\partial(f)$ in place of $f$, we once more use the fact that $[D_g,M_f]M_{\phi} = [D_g,M_f]$.
        This completes the proof of the $k=0$ case of Hypothesis \ref{replacement assumption}.\eqref{rass2}.
        
        For $k > 0$, if $\phi f = f$, then we have
        \begin{equation*}
            \Lambda^k(M_f) = \Lambda^k(M_f)M_{\phi}
        \end{equation*}
        so we may apply the $k=0$ case to $\phi$ to deduce the result. 
        
        That $(C_c^\infty(X),L_2\Omega(X,g),D_g)$ satisfies
        \ref{replacement assumption}.\eqref{rass0} follows from similar reasoning.        
    \end{proof}

\chapter{Asymptotic of the heat trace}\label{heat chapter}
    
    In this chapter we complete the proof of Theorem \ref{heat thm}. This will require some delicate computations 
    exploiting Hochschild homology.
    
    For the remainder of this chapter, we assume that $(\mathcal{A},H,D)$ is a spectral triple satisfying Hypothesis \ref{main assumption}.
    Furthermore we will need the following auxiliary assumption:
    \begin{hyp}\label{auxiliary assumption}
        The spectral triple $(\mathcal{A},H,D)$ satisfies the following:
        \begin{enumerate}[{\rm (i)}]
            \item{}\label{aux ass0} $(\mathcal{A},H,D)$ has the same parity as $p$: this means that $(\mathcal{A},H,D)$ is even when $p$ is even (with grading $\Gamma$) and odd when $p$ is odd.
            \item{}\label{aux ass1} $D$ has a spectral gap at $0$.
        \end{enumerate}
    \end{hyp}
    We will show at the end of this chapter how Hypothesis \ref{auxiliary assumption} can be removed.
    
    Recall from Definition \ref{ch omega def} the fundamental mappings $\mathrm{ch}$ and $\Omega$, given by:
    \begin{align*}
        \mathrm{ch}(a_0\otimes a_1\otimes\cdots\otimes a_p) &:= \Gamma F[F,a_0][F,a_1]\cdots[F,a_p],\\
        \Omega(a_0\otimes a_1\otimes\cdots\otimes a_p) &:= \Gamma a_0\partial(a_1)\partial(a_2)\cdots \partial(a_p)
    \end{align*}
    and that Theorem \ref{heat thm} states that for all Hochschild cycles $c \in \mathcal{A}^{\otimes(p+1)}$,
    \begin{equation*}
        \mathrm{Tr}(\Omega(c)(1+D^2)^{1-\frac{p}{2}}e^{-s^2D^2}) = \frac{p}{2}\mathrm{Tr}(\mathrm{ch}(c))s^{-2}+O(s^{-1}),\quad s\to 0.
    \end{equation*}
    
    The computations in Sections \ref{combinatorial section} and \ref{commutator section} are {inspired by those in \cite{CPRS1,CRSZ}. Since
    we do not assume that the algebra $\mathcal{A}$ is unital, the computations are more delicate than those of \cite{CPRS1,CRSZ}.}

\section[Combinatorial expression]{Combinatorial expression for $\mathrm{Tr}(\Omega(c)|D|^{2-p}e^{-s^2D^2})$}\label{combinatorial section}
    We begin this section with the introduction of a new set of multilinear maps:
    \begin{defi}\label{W definition}
        Let $\mathscr{A} \subseteq \{1,2,\ldots,p\}$. We define the multilinear map $\mathcal{W}_\mathscr{A}:\mathcal{A}^{\otimes(p+1)}\to \mathcal{L}_{\infty}$ by:
        \begin{equation*}
            \mathcal{W}_{\mathscr{A}}(a_0\otimes a_1\otimes\cdots\otimes a_p) = \Gamma a_0\prod_{k=1}^p b_k(a_k)
        \end{equation*}
        where for $a \in \mathcal{A}$ and each $k$ we define:
        \begin{equation*}
            b_k(a) = \begin{cases}
                            \delta(a),\,k \in \mathscr{A},\\
                            [F,a],\,k \notin \mathscr{A}.
                       \end{cases}
        \end{equation*} 
        In the case where $\mathscr{A} = \{m\}$, a single number, $1 \leq m \leq p$, then write $\mathcal{W}_m$ for $\mathcal{W}_{\{m\}}$.
    \end{defi}
    Since by assumption $(\mathcal{A},H,D)$ is smooth, the operators $\delta(a_k)$ are defined and bounded, so that $\mathcal{W}_{\mathscr{A}}$ is well defined as a bounded operator.
    
    The two extreme cases, $\mathcal{W}_{\emptyset}$ and $\mathcal{W}_{\{1,2,\ldots,p\}}$ are easily described as:
    \begin{equation*}
        \mathcal{W}_{\emptyset}(a_0\otimes a_1\otimes\cdots\otimes a_p) = \Gamma a_0\prod_{k=1}^p [F,a_k]
    \end{equation*}
    It will be important to observe that if $(\mathcal{A},H,D)$ has the same parity as $p$, then $\mathrm{ch}(c) = \mathcal{W}_{\emptyset}(c)+F\mathcal{W}_{\emptyset}(c)F$ (this is Lemma \ref{W and ch link}).
    
    On the other extreme.
    \begin{equation*}
        \mathcal{W}_{\{1,2,\ldots,p\}}(a_0\otimes a_1\otimes\cdots\otimes a_p) = \Gamma a_0\prod_{k=1}^p\delta(a_k).
    \end{equation*}
        
    Associated to a subset $\mathscr{A} \subseteq \{1,2,\ldots,p\}$ we have the number,
    \begin{equation*}
        n_\mathscr{A} = |\{(j,k) \in \{1,2,\ldots,p\}^2\;:\; j < k\text{ and }j \in \mathscr{A}, k\neq \mathscr{A}\}|,
    \end{equation*}
    where $|\cdot|$ denotes the cardinality of a set.

    Theorem \ref{combinatorial theorem} (to be stated below) is the main result of this section. Roughly speaking, it shows that one can replace $\Omega(c)$ in $\mathrm{Tr}(\Omega(c)|D|^{2-p}e^{-s^2D^2})$ with a sum of $\mathcal{W}_\mathscr{A}(c)$
    over all subsets $\mathscr{A} \subseteq \{1,\ldots,p\}$.
    However first we need a Lemma which constitutes the core of the proof of Theorem \ref{combinatorial theorem}. Most of this section is devoted to the proof of the following lemma, which is split into various parts.
    \begin{lem}\label{combinatorial lemma} 
        Let $(\mathcal{A},H,D)$ be a smooth spectral triple where $D$ has a spectral gap at $0.$ For all $c\in(\mathcal{A}+\mathbb{C})^{\otimes (p+1)},$ the operator
        \begin{equation*}
            \left(\Omega(c)|D|^{2-p}-\sum_{\mathscr{A}\subseteq \{1,\cdots,p\}}(-1)^{n_{\mathscr{A}}}\mathcal{W}_{\mathscr{A}}(c)D^{2-|\mathscr{A}|}\right)\cdot|D|^{p-1}.
        \end{equation*}
        has bounded extension.
    \end{lem}
        
    Before proving Lemma \ref{combinatorial lemma} above we initially need:

    \begin{lem}\label{w bounded lemma} 
        Let $(\mathcal{A},H,D)$ be a smooth spectral triple, where $D$ has a spectral gap at $0$. For all $c\in(\mathcal{A}+\mathbb{C})^{\otimes (p+1)},$ the operator
        $$\mathcal{W}_{\mathscr{A}}(c)\cdot |D|^{p-|\mathscr{A}|}$$
        has bounded extension.
    \end{lem}
    \begin{proof} 
        By linearity it suffices to prove the assertion for elementary tensors. Let $c = a_0\otimes a_1\otimes\cdots\otimes a_p\in(\mathcal{A}+\mathbb{C})^{\otimes (p+1)}$ 
        and let $c'=a_0\otimes a_1\otimes\cdots\otimes a_{p-1}\in(\mathcal{A}+\mathbb{C})^{\otimes p}.$

        We will prove this assertion by induction on $p.$ If $p=1$ and $\mathscr{A}=\emptyset,$ then
        $$\mathcal{W}_{\mathscr{A}}(c)|D|^{p-|\mathscr{A}|}=\Gamma a_0[F,a_p]|D|$$
        and recall that $[F,a_1]|D|$ has a bounded extension $L(a_p)$.
        On the other hand, if $p=1$ and $\mathscr{A}=\{1\},$ then
        $$\mathcal{W}_{\mathscr{A}}(c)|D|^{p-|\mathscr{A}|}=\Gamma a_0\delta(a_1)\in\mathcal{L}_{\infty}.$$
        This proves the base of induction. 
        
        Next, let $p \geq 2$ and assume that the statement is true for $p-1$. To this end, 
        let $\mathscr{B}\subseteq\{1,\cdots,p-1\}$ be defined by the formula $\mathscr{B}=\mathscr{A}\setminus\{p\}.$ There are two distinct cases: when $p \in \mathscr{A}$ and $p \notin \mathscr{A}$.

        First suppose $p\in\mathscr{A}.$ Here we have $|\mathscr{B}| = |\mathscr{A}|-1$, and
        \begin{align*}
            \mathcal{W}_{\mathscr{A}}(c)|D|^{p-|\mathscr{A}|} &= \mathcal{W}_{\mathscr{B}}(c')\delta(a_p)|D|^{p-1-|\mathscr{B}|}\\
                                      &= \left(\mathcal{W}_{\mathscr{B}}(c')|D|^{p-1-|\mathscr{B}|}\right)\left(|D|^{-p+1+|\mathscr{B}|}\delta(a_p)|D|^{p-1-|\mathscr{B}|}\right)
        \end{align*}
        The first factor in the right hand side has bounded extension by the inductive assumption. Moreover, the second factor on the right hand side has bounded extension by Lemma \ref{left to right corollary}. This proves step of induction for the case $p\in\mathscr{A}.$
        
        Now we deal with the case where $p\notin \mathscr{A}$. Then $\mathscr{B} = \mathscr{A}$, and
        \begin{align*}
            \mathcal{W}_{\mathscr{A}}(c)|D|^{p-|\mathscr{A}|} &= \mathcal{W}_{\mathscr{B}}(c')L(a_p)|D|^{p-1-|\mathscr{B}|}\\
                                      &= \left(\mathcal{W}_{\mathscr{B}}(c')|D|^{p-1-|\mathscr{B}|}\right)\left(|D|^{-p+1+|\mathscr{B}|}L(a_p)|D|^{p-1-|\mathscr{B}|}\right)
        \end{align*}
        The first factor in the above right hand side is bounded by the inductive assumption. For the second factor, we use the expression $L(a_p) = \partial(a_p)-F\delta(a_p)$.
        Since $\partial(a_p), \delta(a_p) \in \mathcal{B}$ then it follows from Lemma \ref{left to right corollary} that the second factor on the right hand side has bounded extension. Hence, $\mathcal{W}_{\mathscr{A}}(c)|D|^{p-|\mathscr{A}|}$
        extends to a bounded linear operator. 
        
    \end{proof}

    In order to prove Lemma \ref{combinatorial lemma}, we introduce two more classes of multilinear functionals.
    \begin{defi}
        Let $\mathscr{A} \subseteq \{1,2,\ldots,p\}$. We define the multilinear map $\mathcal{R}_{\mathscr{A}}:\mathcal{A}^{\otimes(p+1)}\to \mathcal{L}_{\infty}$ by
        \begin{equation*}
            \mathcal{R}_{\mathscr{A}}(a_0\otimes\cdots\otimes a_p) := \Gamma a_0\prod_{k=1}^px_k(a_k),    
        \end{equation*}
        where for each $1\leq k \leq p$ and $a \in \mathcal{A}$,
        \begin{equation*}
            x_k(a) := \begin{cases}
                        F\delta(a),\,k \in \mathscr{A},\\
                        L(a), k \notin \mathscr{A}.
                   \end{cases}
        \end{equation*}
        
        We also define the multilinear map $\mathcal{P}_{\mathscr{A}}:\mathcal{A}^{\otimes(p+1)}\to \mathcal{L}_{\infty}$ by
        \begin{equation*}
            \mathcal{P}_{\mathscr{A}}(a_0\otimes \cdots\otimes a_p) := \Gamma a_0\prod_{k=1}^p y_k(a_k)
        \end{equation*}
        where for each $1\leq k \leq p$ and $a \in \mathcal{A}$,
        \begin{equation*}
            y_k(a) := \begin{cases}
                        \delta(a),\,k \in \mathscr{A},\\
                        L(a),\, k\notin \mathscr{A}.
                   \end{cases}
        \end{equation*}
    \end{defi}

    \begin{lem}\label{first combinatorial lemma} 
        Let $(\mathcal{A},H,D)$ be a smooth spectral triple, where $D$ has a spectral gap at $0$.
        For all $c \in(\mathcal{A}+\mathbb{C})^{\otimes(p+1)}$ and all $\mathscr{A} \subseteq \{1,2,\ldots,p\}$, the operator
        \begin{equation*}
            \left(\mathcal{R}_{\mathscr{A}}(c)-(-1)^{n_{\mathscr{A}}}\mathcal{P}_{\mathscr{A}}(c)\cdot F^{|\mathscr{A}|}\right)\cdot |D|
        \end{equation*}
        has bounded extension.
    \end{lem}
    \begin{proof} 
        This proof is similar to that of \ref{w bounded lemma}. Once again it suffices to prove the result for an elementary tensor $c = a_0\otimes a_1\otimes\cdots \otimes a_p$,
        and we prove the statement by induction on $p$. Denote, for brevity, $c'=a_0\otimes a_1\otimes\cdots\otimes a_{p-1}\in(\mathcal{A}+\mathbb{C})^{\otimes p}.$

        For the base of induction, when $p = 1$, we deal with the two possibilities $\mathscr{A} = \{1\}$ and $\mathscr{A} = \emptyset$. 
        If $\mathscr{A} = \{1\}$, then
        \begin{align*}
            \mathcal{R}_{\mathscr{A}}(c) &= \Gamma a_0F\delta(a_1)\\
            \mathcal{P}_{\mathscr{A}}(c)F^{|\mathscr{A}|} &= \Gamma a_0\delta(a_1)F.
        \end{align*}
        So,
        \begin{align*}
            \left(\mathcal{R}_{\mathscr{A}}(c)-(-1)^{n_{\mathscr{A}}}\mathcal{P}_{\mathscr{A}}(c)\cdot F^{|\mathscr{A}|}\right)\cdot |D| &= \Gamma a_0(F\delta(a_1)-\delta(a_1)F)|D|\\
                                                                                         &= \Gamma a_0[F,\delta(a_1)]|D|\\
                                                                                         &= \Gamma a_0L(\delta(a_1)).
        \end{align*}
        since $L(\delta(a_1))$ has bounded extension, so does the above left hand side.
        
        Now in the case that $p = 1$ and $\mathscr{A} = \emptyset$,
        \begin{align*}
            \mathcal{R}_{\mathscr{A}}(c) &= \Gamma a_0 L(a_1)\\
            \mathcal{P}_{\mathscr{A}}(c) &= \Gamma a_0 L(a_1)
        \end{align*}
        and $|\mathscr{A}| = 0$, so $\mathcal{R}_{\mathscr{A}}(c) - \mathcal{P}_{\mathscr{A}}(c)F^{|\mathscr{A}|} = 0$. This proves the $p = 1$ case.
        
        Now suppose that $p > 1$ and the assertion is proved for $p-1$.
        For this purpose, let $\mathscr{B}\subset\{1,\cdots,p-1\}$ be defined by 
        $\mathscr{B}=\mathscr{A}\backslash\{p\}.$ 
        
        If $p \in \mathscr{A}$, then $n_{\mathscr{A}} = n_{\mathscr{B}}$, and if $p \notin \mathscr{A}$ then $n_{\mathscr{A}} = n_{\mathscr{B}} + |\mathscr{B}|$.
        
        
        Now we consider separately the cases $p \in \mathscr{A}$ and $p\notin \mathscr{A}$.
        
        First, if $p \in \mathscr{A}$, then,
        \begin{align*}
            \mathcal{R}_{\mathscr{A}}(c) &= \mathcal{R}_{\mathscr{B}}(c')F\delta(a_p),\\
            \mathcal{P}_{\mathscr{A}}(c) &= \mathcal{P}_{\mathscr{B}}(c')\delta(a_p).
        \end{align*}   
        Hence, since $|\mathscr{A}| = |\mathscr{B}|+1$ and $n_{\mathscr{A}} = n_{\mathscr{B}}$ in this case:
        $$(\mathcal{R}_{\mathscr{A}}(c) - (-1)^{n_{\mathscr{A}}}\mathcal{P}_{\mathscr{A}}(c)F^{|\mathscr{A}|})|D| =(\mathcal{R}_{\mathscr{B}}(c')F\delta(a_p)-(-1)^{n_{\mathscr{A}}}\mathcal{P}_{\mathscr{B}}(c')\delta(a_p)F^{|\mathscr{B}|+1})|D|=$$
        $$= \mathcal{R}_{\mathscr{B}}(c')[F,\delta(a_p)]|D|+(\mathcal{R}_{\mathscr{B}}(c')\delta(a_p)F-(-1)^{n_{\mathscr{B}}}\mathcal{P}_{\mathscr{B}}(c')\delta(a_p)F\cdot F^{|\mathscr{B}|})|D|.$$
        $$= \mathcal{R}_{\mathscr{B}}(c')L(a_p)+(\mathcal{R}_{\mathscr{B}}(c')\delta(a_p)-(-1)^{n_{\mathscr{B}}}\mathcal{P}_{\mathscr{B}}(c')\delta(a_p)\cdot F^{|\mathscr{B}|})D.$$

        In the case that $|\mathscr{B}|$ is even, we have $F^{|\mathscr{B}|} = 1$ and so:
        $$(\mathcal{R}_{\mathscr{A}}(c) - (-1)^{n_{\mathscr{A}}}\mathcal{P}_{\mathscr{A}}(c)F^{|\mathscr{A}|})|D|=$$
        $$=\mathcal{R}_{\mathscr{B}}(c')L(a_p)+(\mathcal{R}_{\mathscr{B}}(c')-(-1)^{n_{\mathscr{B}}}\mathcal{P}_{\mathscr{B}}(c')\cdot F^{|\mathscr{B}|})\cdot \delta(a_p)D.$$
        $$=\mathcal{R}_{\mathscr{B}}(c')L(a_p)+(\mathcal{R}_{\mathscr{B}}(c')-(-1)^{n_{\mathscr{B}}}\mathcal{P}_{\mathscr{B}}(c')\cdot F^{|\mathscr{B}|})D\cdot \delta(a_p)-$$
        $$-(\mathcal{R}_{\mathscr{B}}(c')-(-1)^{n_{\mathscr{B}}}\mathcal{P}_{\mathscr{B}}(c')\cdot F^{|\mathscr{B}|})\cdot\partial(\delta(a_p)).$$
        Since $\delta(a_p),$ $\partial(\delta(a_p))$ and $L(a_p)$ are bounded, by the inductive hypothesis the above has bounded extension, completing the proof of the case where $p \in \mathscr{A}$ and $|\mathscr{B}|$ is even.
        
        On the other hand, if $p \in \mathscr{A}$ and $|\mathscr{B}|$ is odd, then:
        $$(\mathcal{R}_{\mathscr{A}}(c) - (-1)^{n_{\mathscr{A}}}\mathcal{P}_{\mathscr{A}}(c)F^{|\mathscr{A}|})|D|=$$
        $$= \mathcal{R}_{\mathscr{B}}(c')L(a_p)+\mathcal{R}_{\mathscr{B}}(c')\delta(a_p)D-(-1)^{n_{\mathscr{B}}}\mathcal{P}_{\mathscr{B}}(c')\delta(a_p)|D|=$$
        $$= \mathcal{R}_{\mathscr{B}}(c')L(a_p)+(\mathcal{R}_{\mathscr{B}}(c')-(-1)^{n_{\mathscr{B}}}\mathcal{P}_{\mathscr{B}}(c')F^{|\mathscr{B}|})D\cdot\delta(a_p)-$$
        $$-\mathcal{R}_{\mathscr{B}}(c')\delta(a_p)D+(-1)^{n_{\mathscr{B}}}\mathcal{P}_{\mathscr{B}}(c')\delta^2(a_p).$$
        Since $\delta(a_p),$ $\delta^2(a_p)$ and $L(a_p)$ are bounded, by the inductive hypothesis the above has bounded extension, completing the proof of the case where $p \in \mathscr{A}$ and $|\mathscr{B}|$ is odd.
        
        Now assume that $p\notin \mathscr{A}$. Then,
        \begin{align*}
            \mathcal{R}_{\mathscr{A}}(c) &= \mathcal{R}_{\mathscr{B}}(c')L(a_p)\\
            \mathcal{P}_{\mathscr{A}}(c) &= \mathcal{P}_{\mathscr{B}}(c')L(a_p).
        \end{align*}
        Focusing on $\mathcal{P}_{\mathscr{A}}(c)$:
        \begin{align*}
            \mathcal{P}_{\mathscr{A}}(c)F^{|\mathscr{A}|} &= \mathcal{P}_{\mathscr{B}}(c')L(a_p)F^{|\mathscr{B}|}.
        \end{align*}
        So:
        \begin{equation*}
             (\mathcal{R}_{\mathscr{A}}(c)-(-1)^{n_{\mathscr{A}}}\mathcal{P}_{\mathscr{A}}(c)F^{|\mathscr{A}|})|D| = (\mathcal{R}_{\mathscr{B}}(c')L(a_p)-(-1)^{n_{\mathscr{B}}+|\mathscr{B}|}\mathcal{P}_{\mathscr{B}}(c')L(a_p)F^{|\mathscr{B}|})|D|.
        \end{equation*}
        
        Note that since $F$ anticommutes with $[F,a_p]$, $F$ also anticommutes with $L(a_p)$. Hence,
        \begin{equation*}
            L(a_p)F^{|\mathscr{B}|} = (-1)^{|\mathscr{B}|}F^{|\mathscr{B}|}L(a_p)
        \end{equation*}
        and so
        \begin{align*}
            (\mathcal{R}_{\mathscr{A}}(c)-(-1)^{n_{\mathscr{A}}}\mathcal{P}_{\mathscr{A}}(c)F^{|\mathscr{A}|})|D| &= (\mathcal{R}_{\mathscr{B}}(c')-(-1)^{n_{\mathscr{B}}}\mathcal{P}_{\mathscr{B}}(c')F^{|\mathscr{B}|})L(a_p)|D|\\
                                                                  &= (\mathcal{R}_{\mathscr{B}}(c')-(-1)^{n_{\mathscr{B}}}\mathcal{P}_{\mathscr{B}}(c')F^{|\mathscr{B}|})(|D|L(a_p)+L(\delta(a_p))).
        \end{align*}
        By the inductive assumption, the operator $(\mathcal{R}_{\mathscr{B}}(c')-(-1)^{n_{\mathscr{B}}}\mathcal{P}_{\mathscr{B}}(c')F^{|\mathscr{B}|})|D|$ has bounded extension. This completes the proof of the $p \notin \mathscr{A}$ case.
        
        Hence, the statement is true for $p$ and this completes the induction.
    \end{proof}

    \begin{lem}\label{second combinatorial lemma} 
        Let $(\mathcal{A},H,D)$ be a smooth spectral triple. Suppose $D$ has a spectral gap at $0.$ For all $c\in(\mathcal{A}+\mathbb{C})^{\otimes (p+1)}$ and for all $\mathscr{A}\subseteq \{1,\ldots,p\}$ the operator
        \begin{equation*}
            \left(\mathcal{P}_{\mathscr{A}}(c)-\mathcal{W}_{\mathscr{A}}(c)\cdot |D|^{p-|\mathscr{A}|}\right)\cdot|D|.
        \end{equation*}
        has bounded extension.
    \end{lem}
    \begin{proof} 
        This proof is again similar to the proofs of Lemmas \ref{w bounded lemma} and \ref{first combinatorial lemma}.
        
        Once more, it suffices to prove the assertion for an elementary tensor $c = a_0\otimes a_1\otimes \cdots \otimes a_p \in (\mathcal{A}+\mathbb{C})^{\otimes(p+1)}$. Let $c' := a_0\otimes\cdots\otimes a_{p-1} \in (\mathcal{A}+\mathbb{C})^{\otimes p}$.
        The proof proceeds by induction on $p$.
        
        First, if $p = 1$, then either $\mathscr{A} = \{1\}$ or $\mathscr{A} = \emptyset$. If $\mathscr{A} = \{1\}$, then
        \begin{align*}
                         \mathcal{P}_{\mathscr{A}}(c) &= \Gamma a_0\delta(a_1),\\
            \mathcal{W}_{\mathscr{A}}(c)|D|^{p-|\mathscr{A}|} &= \Gamma a_0\delta(a_1).
        \end{align*}
        and if $\mathscr{A} = \emptyset$, then
        \begin{align*}
                         \mathcal{P}_{\mathscr{A}}(c) &= \Gamma a_0L(a_1),\\
            \mathcal{W}_{\mathscr{A}}(c)|D|^{p-|\mathscr{A}|} &= \Gamma a_0[F,a_1]|D|.
        \end{align*}
        Since $L(a_1) = [F,a_1]|D|$ on the dense subspace $H_\infty$, it follows that in all cases with $p = 1$ the difference $\mathcal{P}_{\mathscr{A}}(c)-\mathcal{W}_{\mathscr{A}}(c)|D|^{p-|\mathscr{A}|}$ is identically
        zero on $H_\infty$ and therefore has trivial bounded extension to $H$. This establishes the $p=1$ case.
        
        Now suppose that $p > 1$ and the assertion has been proved for $p-1$. Let $\mathscr{B} = \mathscr{A} \setminus \{p\}$, and we consider
        the two cases of $p \in \mathscr{A}$ and $p \notin \mathscr{A}$. 
        
        Suppose that $p \in \mathscr{A}$. Then,
        \begin{align*}
                         \mathcal{P}_{\mathscr{A}}(c) &= \mathcal{P}_{\mathscr{B}}(c')\delta(a_p),\\
            \mathcal{W}_{\mathscr{A}}(c)|D|^{p-|\mathscr{A}|} &= \mathcal{W}_{\mathscr{B}}(c')\delta(a_p)|D|^{p-1-|\mathscr{B}|}.
        \end{align*}
        
        So,
        \begin{align}\label{simplified pc expression} 
            \mathcal{P}_{\mathscr{A}}(c)|D| &= \mathcal{P}_{\mathscr{B}}(c')\delta(a_p)|D|\nonumber\\
                            &= \mathcal{P}_{\mathscr{B}}(c')|D|-\mathcal{P}_{\mathscr{B}}(c')\delta^2(a_p).
        \end{align}
        and
        \begin{align}\label{simplified wc expression}
            \mathcal{W}_{\mathscr{A}}(c)|D|^{p-|\mathscr{A}|+1} &= \mathcal{W}_{\mathscr{B}}(c')\delta(a_p)|D|^{p-|\mathscr{B}|}\nonumber\\
                                        &= \mathcal{W}_{\mathscr{B}}(c')|D|^{p-|\mathscr{B}|}\delta(a_p)+\mathcal{W}_{\mathscr{B}}(c')[\delta(a_p),|D|^{p-|\mathscr{B}|}]\nonumber\\
                                        &= \mathcal{W}_{\mathscr{B}}(c')|D|^{p-|\mathscr{B}|}\delta(a_p)-\mathcal{W}_{\mathscr{B}}(c')|D|^{p-|\mathscr{B}|-1}|D|^{|\mathscr{B}|-p+1}[|D|^{p-|\mathscr{B}|},\delta(a_p)].
        \end{align}
        Applying Lemma \ref{left to right corollary}, the operator $|D|^{|\mathscr{B}|-p+1}[|D|^{p-|\mathscr{B}|},\delta(a_p)]$ has bounded extension. 
        So combining \eqref{simplified pc expression} and \eqref{simplified wc expression}:
        \begin{align*}
            (\mathcal{P}_{\mathscr{A}}(c)-\mathcal{W}_{\mathscr{A}}(c)|D|^{p-|\mathscr{A}|})|D| - (\mathcal{P}_{\mathscr{B}}(c')-\mathcal{W}_{\mathscr{B}}(c')|D|^{p-1-|\mathscr{B}|})|D|
        \end{align*}
        has bounded extension. So by the inductive hypothesis, $(\mathcal{P}_{\mathscr{A}}(c)-\mathcal{W}_{\mathscr{A}}(c)|D|^{p-|\mathscr{A}|})|D|$ has bounded extension in the case that $p \in \mathscr{A}$.
        
        Now suppose that $p\notin \mathscr{A}$. In this case, we have:
        \begin{align*}
            \mathcal{P}_{\mathscr{A}}(c) &= \mathcal{P}_{\mathscr{B}}(c')L(a_p)\\
            \mathcal{W}_{\mathscr{A}}(c) &= \mathcal{W}_{\mathscr{B}}(c')[F,a_p].
        \end{align*}
        
        Multiplying by $|D|$, we have
        \begin{equation}\label{simplified pc expression 2}
            \mathcal{P}_{\mathscr{A}}(c)|D| = \mathcal{P}_{\mathscr{B}}(c')|D|L(a_p)-\mathcal{P}_{\mathscr{B}}(c')L(\delta(a_p)).
        \end{equation}
        Note that $\mathcal{P}_{\mathscr{B}}(c')L(\delta(a_p))$ is bounded.
        
        Also,
        \begin{align}\label{simplified wc expression 2}
            \mathcal{W}_{\mathscr{A}}(c)|D|^{p+1-|\mathscr{A}|} &= \mathcal{W}_{\mathscr{B}}(c')L(a_p)|D|^{p-|\mathscr{A}|}\nonumber\\
                                        &= \mathcal{W}_{\mathscr{B}}(c')|D|^{p-|\mathscr{A}|}L(a_p)-\mathcal{W}_{\mathscr{B}}(c')[|D|^{p-|\mathscr{A}|},L(a_p)]\nonumber\\
                                        &= \mathcal{W}_{\mathscr{B}}(c')|D|^{p-|\mathscr{A}|}L(a_p)-\mathcal{W}_{\mathscr{B}}(c')|D|^{p-1-|\mathscr{A}|}|D|^{-p+1+|\mathscr{A}|}[|D|^{p-|\mathscr{A}|},L(a_p)].
        \end{align}
        By Lemma \ref{left to right corollary}, the operator $|D|^{-p+1+|\mathscr{A}|}[|D|^{p-|\mathscr{A}|},L(a_p)]$ has bounded extension. So combining \eqref{simplified pc expression 2}
        and \eqref{simplified wc expression 2}, it follows that
        \begin{equation*}
            (\mathcal{P}_{\mathscr{A}}(c)-\mathcal{W}_{\mathscr{A}}(c)|D|^{p-|\mathscr{A}|})|D|-(\mathcal{P}_{\mathscr{B}}(c')-\mathcal{W}_{\mathscr{B}}(c')|D|^{p-1-|\mathscr{B}|})|D|L(a_p)
        \end{equation*} 
        has bounded extension. So by the inductive hypothesis, it follows that $(\mathcal{P}_{\mathscr{A}}(c)-\mathcal{W}_{\mathscr{A}}(c)|D|^{p-|\mathscr{A}|})|D|$ has bounded extension
        in the case $p \notin \mathscr{A}$.
    \end{proof}
    
    The main idea used in the proof of Lemma \ref{combinatorial lemma} is the algebraic identity,
    \begin{equation}\label{combinatorial fact}
        \prod_{k=1}^p (x_k+y_k) = \sum_{\mathscr{A} \subseteq \{1,\cdots,p\}} z_{\mathscr{\mathscr{A}}}
    \end{equation}
    where $z_{\mathscr{\mathscr{A}}}$ is given by the product $z_1z_2\cdots z_p$, where $z_k = x_k$ for $k \in \mathscr{A}$
    and $z_k = y_k$ for $k\notin \mathscr{A}$.
    
    Now we are ready to complete the proof of Lemma \ref{combinatorial lemma}:

    \begin{proof}[Proof of Lemma \ref{combinatorial lemma}] 
        Since
        \begin{equation*}
            [D,a] = F[|D|,a]+[F,a]|D|,
        \end{equation*}
        as an equality of operators on $H_\infty$, it follows that we have $\partial(a) = F\delta(a)+L(a)$.
        Now using \eqref{combinatorial fact}:
        \begin{align*}
            \Omega(c) &= \Gamma a_0\prod_{k=1}^p (F\delta(a_k)+L(a_k))\\
                      &= \sum_{\mathscr{A}\subseteq \{1,\cdots,p\}} \mathcal{R}_{\mathscr{A}}(c).
        \end{align*}
        So on $H_\infty$:
        \begin{equation*}
            \Omega(c)|D| = \sum_{\mathscr{A}\subseteq \{1,\cdots,p\}} \mathcal{R}_{\mathscr{A}}(c)|D|.
        \end{equation*}
        We may now apply Lemma \ref{first combinatorial lemma} to each summand to conclude that
        \begin{equation*}
            \Omega(c)|D| - \sum_{\mathscr{A}\subseteq \{1,\cdots,p\}} (-1)^{n_{\mathscr{A}}}\mathcal{P}_{\mathscr{A}}(c)F^{|\mathscr{A}|}|D|
        \end{equation*}
        has bounded extension.
        Now applying Lemma \ref{second combinatorial lemma} to each summand, we have that the operator
        \begin{equation*}
            \Omega(c)|D|-\sum_{\mathscr{A}\subseteq \{1,\cdots,p\}}(-1)^{n_{\mathscr{A}}}\mathcal{W}_{\mathscr{A}}(c)|D|^{p-|\mathscr{A}|+1}F^{|\mathscr{A}|}
        \end{equation*}
        has bounded extension.
        
        Equivalently,
        \begin{equation*}
            \left(\Omega(c)|D|^{2-p}-\sum_{\mathscr{A}\subseteq \{1,\ldots,p\}}(-1)^{n_{\mathscr{A}}}\mathcal{W}_{\mathscr{A}}(c)D^{2-|\mathscr{A}|}\right)|D|^{p-1}
        \end{equation*}
        has bounded extension.
    \end{proof}
    
    We now prove the main result of this section.
    \begin{thm}\label{combinatorial theorem} 
        Let $(\mathcal{A},H,D)$ be a spectral triple satisfying Hypothesis \ref{main assumption} and Hypothesis \ref{auxiliary assumption}. 
        For all $c\in\mathcal{A}^{\otimes (p+1)},$ we have
        \begin{align*}
            \mathrm{Tr}(\Omega(c)|D|^{2-p}e^{-s^2D^2}) = \sum_{\mathscr{A} \subseteq \{1,2,\ldots,p\}} (-1)^{n_{\mathscr{A}}}\mathrm{Tr}(\mathcal{W}_{\mathscr{A}}(c)D^{2-|\mathscr{A}|}e^{-s^2D^2})+O(s^{-1})
        \end{align*}
        as $s\to 0$.
    \end{thm}
    \begin{proof}
        As in the preceding lemmas it suffices to prove the result for an elementary tensor $c = a_0\otimes\cdots \otimes a_p \in \mathcal{A}^{\otimes(p+1)}$. Let $c' = 1\otimes a_1\otimes\cdots\otimes a_p$. 
        By the cyclicity of the trace and the fact that $\mathcal{A}$ commutes with $\Gamma$, we have:
        \begin{equation*}
            \mathrm{Tr}(\Omega(c)|D|^{2-p}e^{-s^2D^2}) = \mathrm{Tr}(\Omega(c')|D|^{2-p}e^{-s^2D^2}a_0)
        \end{equation*}
        and for all $\mathscr{A} \subseteq \{1,2,\ldots,p\}$:
        \begin{equation*}
            \mathrm{Tr}(\mathcal{W}_{\mathscr{A}}(c)D^{2-|\mathscr{A}|}e^{-s^2D^2}) = \mathrm{Tr}(\mathcal{W}_{\mathscr{A}}(c')D^{2-|\mathscr{A}|}e^{-s^2D^2}a_0).
        \end{equation*}
        Thus,
        \begin{align*}
            &\mathrm{Tr}(\Omega(c)|D|^{2-p}e^{-s^2D^2})-\sum_{\mathscr{A}\subseteq \{1,\ldots,p\}}(-1)^{n_{\mathscr{A}}}\mathrm{Tr}(\mathcal{W}_{\mathscr{A}}(c)D^{2-|\mathscr{A}|}e^{-s^2D^2})\\
                                                    &=\mathrm{Tr}\left(\left(\Omega(c')|D|^{2-p}-\sum_{\mathscr{A}\subseteq \{1,\ldots,p\}}(-1)^{n_{\mathscr{A}}}\mathcal{W}_{\mathscr{A}}(c')D^{2-|\mathscr{A}|}\right)|D|^{p-1}|D|^{1-p}e^{-s^2D^2}a_0\right)
        \end{align*}
            Hence,
        \begin{align*}
            &\left|\mathrm{Tr}(\Omega(c)|D|^{2-p}e^{-s^2D^2}-\sum_{\mathscr{A}\subseteq \{1,\ldots,p\}}(-1)^{n_{\mathscr{A}}}\mathrm{Tr}(\mathcal{W}_{\mathscr{A}}(c)D^{2-|\mathscr{A}|}e^{-s^2D^2})\right|\\
            &\leq \left\|\left(\Omega(c')D^{2-p}-\sum_{\mathscr{A}\subseteq \{1,\ldots,p\}}(-1)^{n_{\mathscr{A}}}\mathcal{W}_{\mathscr{A}}(c')D^{2-|\mathscr{A}|}\right)|D|^{p-1}\right\|_\infty\||D|^{1-p}e^{-s^2D^2}a_0\|_1
        \end{align*}
        The first factor is finite, by Lemma \ref{combinatorial lemma}, and the second factor is $O(s^{-1})$, by Lemma \ref{first decay lemma}.
        This completes the proof.
    \end{proof}

\section{Auxiliary commutator estimates}\label{commutator section}
    This section is a slight detour from the main task of this chapter. Here we establish bounds on the $\mathcal{L}_1$-norm of commutators
    of the form $[f(s|D|),x]$, where $x \in \mathcal{B}$, $s > 0$ and $f$ is the square of a Schwartz class function on $\mathbb{R}.$ These bounds are used everywhere in the subsequent sections of this chapter.
    
    Recall that the algebra $\mathcal{B}$ is defined in Definition \ref{smoothness definition}.
    
    The following lemma serves the same purpose as \cite[Lemma 18]{CPRS1}, but the right hand sides are different here due to the fact that we deal with non-unital spectral triples.
    \begin{lem}\label{first commutator lemma} 
        Let $(\mathcal{A}, H, D)$ be a smooth spectral triple. Let $h$ be a Schwartz class function on $\mathbb{R}$ and let $f = h^2$. Then for all $x \in \mathcal{B}$, we have
        \begin{equation*}
            \left\|[f(s|D|),x]-\frac{s}{2}\{f'(s|D|),\delta(x)\}\right\|_1 \leq \frac{1}{2}s^2\|\widehat{h''}\|_1\cdot \left(\|\delta^2(x)h(s|D|)\|_1+\|h(s|D|)\delta^2(x)\|_1\right)
        \end{equation*}
        Here, $\{\cdot,\cdot\}$ denotes the anti-commutator.
    \end{lem}
    \begin{proof}
        Since $f'(s|D|) = 2h'(s|D|)h(s|D|)$, by the Leibniz rule we have:
        $$[f(s|D|),x]-\frac{s}{2}\{f'(s|D|),\delta(x)\} = [h(s|D|)^2,x]-s\{h'(s|D|)h(s|D|),\delta(x)\}=$$
        $$= h(s|D|)\big([h(s|D|),x]-sh'(s|D|)\delta(x)\big)+\big([h(s|D|),x]-s\delta(x)h'(s|D|)\big)h(s|D|).$$
        Applying Lemma \ref{second commutator rep lemma}, we have
        $$[h(s|D|),x]-sh'(s|D|)\delta(x) = -s^2\int_{-\infty}^\infty \int_0^1 \widehat{h''}(u)(1-v)e^{ius(1-v)|D|}\delta^2(x)e^{iusv|D|}\,dvdu$$
        $$[h(s|D|),x]-s\delta(x)h'(s|D|) = -s^2\int_{-\infty}^\infty \int_0^1 \widehat{h''}(u)(1-v)e^{iusv|D|}\delta^2(x)e^{ius(1-v)|D|}\,dvdu.$$
        
        Therefore,
        \begin{align*}
            [f(s|D|),x] - \frac{s}{2}&\{f'(s|D|),\delta(x)\}= \\
            & -s^2\int_{-\infty}^\infty \int_0^1 \widehat{h''}(u)(1-v)e^{iu(1-v)s|D|}h(s|D|)\delta^2(x)e^{iuvs|D|}\,dvdu\\
            & -s^2\int_{-\infty}^\infty \int_0^1 \widehat{h''}(u)(1-v)e^{iuvs|D|}\delta^2(x)h(s|D|)e^{iu(1-v)s|D|}\,dvdu.
        \end{align*}
        Now applying Lemma \ref{peter lemma} to each integral, we have
        \begin{align*} 
             \|[f(s|D|),x]&-\frac{s}{2}\{f'(s|D|),\delta(x)\}\|_1 \\ 
                               &\leq s^2\int_{-\infty}^\infty \int_0^1 \left\|\widehat{h''}(u)(1-v)e^{iu(1-v)s|D|}h(s|D|)\delta^2(x)e^{iuvs|D|}\right\|_1dvdu\\
                               &+s^2\int_{-\infty}^\infty \int_0^1 \left\|\widehat{h''}(u)(1-v)e^{iuvs|D|}\delta^2(x)h(s|D|)e^{iu(1-v)s|D|}\right\|_1dvdu\\ 
                               &= s^2\|\widehat{h''}\|_1(\|\delta^2(x)h(s|D|)\|_1+\|h(s|D|)\delta^2(x)\|_1)\int_0^1(1-v)dv\\ 
                               &= \frac{1}{2}s^2\|\widehat{h''}\|_1(\|\delta^2(x)h(s|D|)\|_1+\|h(s|D|)\delta^2(x)\|_1). 
        \end{align*}
%
    \end{proof}

    \begin{lem}\label{second commutator lemma} 
        Let $(\mathcal{A},H,D)$ be a smooth spectral triple and assume that $D$ has a spectral gap at $0$. Let $h$ be a Schwartz function on $\mathbb{R}$ and let $f=h^2.$ Then for every $x\in\mathcal{B},$ we have
        \begin{equation*}
            \left\||D|^m[f(s|D|),x]\right\|_1 \leq s\|\widehat{h'}\|_1\left(\left\||D|^mh(s|D|)\delta(x)\right\|_1+\left\||D|^m\delta(x)h(s|D|)\right\|_1\right).
        \end{equation*}
    \end{lem}
    \begin{proof} 
        Since $f = h^2$, by the Leibniz rule, we have
        \begin{equation*}
            [f(s|D|),x] = h(s|D|)[h(s|D|),x]+[h(s|D|),x]h(s|D|).
        \end{equation*}
        Using Lemma \ref{first commutator rep lemma}, we have
        \begin{equation*}
            [h(s|D|),x] = s\int_{\mathbb{R}} \int_0^1 \widehat{h'}(u)e^{ius(1-v)|D|}\delta(x)e^{iusv|D|}dvdu.
        \end{equation*}
        
        Thus,
        \begin{align*}
            |D|^m[f(s|D|),x] &= s\int_{\mathbb{R}} \int_0^1 \widehat{h'}(u)e^{ius(1-v)|D|}|D|^mh(s|D|)\delta(x)e^{iusv|D|}\,dvdu\\
                             &\quad\quad +s\int_{\mathbb{R}} \int_0^1 \widehat{h'}(u)e^{ius(1-v)|D|}|D|^m\delta(x)h(s|D|)e^{iusv|D|}\,dvdu.
        \end{align*}
        
        Bounding the $\mathcal{L}_1$ norm using Lemma \ref{peter lemma}, we have
        \begin{align*}
            \left\||D|^m[f(s|D|),x]\right\|_1 &\leq s\int_{\mathbb{R}}\int_0^1 |\widehat{h'}(u)|\||D|^mh(s|D|)\delta(x)\|_1\,dvdu\\
                                              &\quad\quad + s\int_\mathbb{R} \int_0^1 |\widehat{h'}(u)|\||D|^m\delta(x)h(s|D|)\|_1\,dvdu\\
                                              &= s\|\widehat{h'}\|_{L_1(\mathbb{R})}(\||D|^mh(s|D|)\delta(x)\|_1+\||D|^m\delta(x)h(s|D|)\|_1).
        \end{align*}
    \end{proof}
    
    \begin{lem}\label{commutator 6} 
        Let $(\mathcal{A},H,D)$ be a smooth spectral triple satisfying Hypothesis \ref{main assumption} and Hypothesis \ref{auxiliary assumption}. For all $x\in\mathcal{B}$ and for all integers $m>-p,$ we have
        \begin{equation*}
            \||D|^m[e^{-s^2D^2},x]\|_1=O(s^{1-p-m}),\quad s\downarrow0.
        \end{equation*}
    \end{lem}
    \begin{proof} 
        Let $h(t)=e^{-\frac{1}{2}t^2},$ $t\in\mathbb{R}.$ By Lemma \ref{second commutator lemma}, we have
        \begin{equation*}
            \||D|^m[e^{-s^2D^2},x]\|_1 \leq s\|\widehat{h'}\|_1(\||D|^me^{-\frac12s^2D^2}\delta(x)\|_1+\||D|^m\delta(x)e^{-\frac12s^2D^2}\|_1).
        \end{equation*}
        
        If $m \leq 0$, then the assertion now follows from applying Lemma \ref{first decay lemma} to the terms $\||D|^me^{-\frac{1}{2}s^2D^2}\delta(x)\|_1$ and $\||D|^m\delta(x)e^{-\frac{1}{2}s^2D^2}\|_1$.        
        Assume now $m > 0$. Using Lemma \ref{schwartz lemma} with the Schwartz function $t\mapsto t^me^{-\frac{1}{2}t^2}$, we obtain
        \begin{equation*}
            s\left\||D|^me^{-\frac{1}{2}s^2D^2}\delta(x)\right\|_1 = O(s^{1-p-m}),\quad s\to 0.
        \end{equation*}
        By Lemma \ref{left to right lemma}, we have
        \begin{equation*}
            |D|^m\delta(x)e^{-\frac{1}{2}s^2D^2} = \sum_{k=0}^m \binom{m}{k}\delta^{m+1-k}(x)|D|^ke^{-\frac{1}{2}s^2D^2}.
        \end{equation*}
        Now we apply Lemma \ref{schwartz lemma} to each summand, using the function $t\mapsto t^ke^{-\frac{1}{2}s^2D^2}$ for the $k$th summand.
        So,
        \begin{align*}
            s\||D|^m\delta(x)e^{-\frac{1}{2}s^2D^2}\|_1 &\leq s\sum_{k=0}^m \binom{m}{k}O(s^{1-p-k})\\
                                                        &= O(s^{1-p-m}), \quad s\to 0.
        \end{align*}
    \end{proof}
    
    The following lemma is used in the proof of Theorem \ref{first cycle thm}.
    \begin{lem}\label{third commutator lemma} 
        Let $(\mathcal{A},H,D)$ be a spectral triple satisfying Hypothesis \ref{main assumption} and Hypothesis \ref{auxiliary assumption}. Let $f(t)=e^{-t^2},$ $t\in\mathbb{R}.$
        \begin{enumerate}[{\rm (i)}]
            \item\label{3com1} for every $a\in\mathcal{A},$ we have
                $$\Big\|[f(s|D|),a]-s\delta(a)f'(s|D|)\Big\|_{\infty}=O(s^2),\quad s\downarrow0.$$
            \item\label{3com2} for every $a\in\mathcal{A},$ we have
                $$\Big\|[f(s|D|),a]-s\delta(a)f'(s|D|)\Big\|_1=O(s^{2-p}),\quad s\downarrow0.$$
            \item\label{3com3} for every $a\in\mathcal{A},$ we have
                $$\Big\|[f(s|D|),a]-s\delta(a)f'(s|D|)\Big\|_{p,1}=O(s),\quad s\downarrow0.$$
        \end{enumerate}
    \end{lem}
    \begin{proof} 
        First we prove \eqref{3com1}: this is a simple combination of Lemma \ref{second commutator rep lemma} and the triangle inequality:
        \begin{equation*}
            \|[f(s|D|),a]-sf'(s|D|)\delta(a)\|_\infty \leq s^2\|\widehat{f''}\|_1\|\delta^2(a)\|_\infty.
        \end{equation*}
        
        Now we prove \eqref{3com2}. Let $h(t) = e^{-t^2/2}$, $t \in \mathbb{R}$, so that $f = h^2$. By Lemma \ref{first commutator lemma}, for all $a \in \mathcal{A}$ we have
        \begin{equation}\label{big commutator difference bound}
            \left\|[f(s|D|),a]-\frac{s}{2}\{f'(s|D|),\delta(a)\}\right\|_1 \leq \frac{1}{2}s^2\|\hat{h''}\|_1(\|\delta^2(a)h(s|D|)\|_1+\|h(s|D|)\delta^2(a)\|_1).
        \end{equation} 
        Using Lemma \ref{schwartz lemma}, we have
        \begin{align}\label{schwartz type bounds}
            \|\delta^2(a)h(s|D|)\|_1 &= O(s^{-p})\text{ and, }\nonumber\\
            \|h(s|D|)\delta^2(a)\|_1 &= O(s^{-p}).
        \end{align}
        Combining \eqref{big commutator difference bound} and \eqref{schwartz type bounds}, we arrive at:
        \begin{equation}\label{com31}
            \|[f(s|D|),a]-\frac{s}{2}\{f'(s|D|),\delta(a)\}\|_1 = O(s^{2-p}).
        \end{equation}
        On the other hand,
        \begin{align*}
            \|[f'(s|D|),\delta(a)]\|_1 &= 2s\|[|D|e^{-s^2D^2},\delta(a)]\|_1\\
                                       &\leq 2s\|\delta^2(a)e^{-s^2D^2}\|_1+2s\||D|[e^{-s^2D^2},\delta(a)]\|_1.
        \end{align*}
        Due to Lemma \ref{schwartz lemma}, we have $2s\|\delta^2(a)e^{-s^2D^2}\|_1 = O(s^{1-p})$, and by Lemma \ref{commutator 6}
        we also have $2s\||D|[e^{-s^2D^2},\delta(a)]\|_1 = O(s^{1-p})$. Therefore,
        \begin{equation}\label{com32}
            \|[f'(s|D|),\delta(a)]\|_1 = O(s^{1-p}),\quad s \downarrow 0.
        \end{equation}
        
        By combining \eqref{com31} and \eqref{com32}, we obtain \eqref{3com2}.
        
        Finally, to prove \eqref{3com3}, we use the inequality
        $$\|T\|_{p,1}\leq\|T\|_1^{\frac1p}\|T\|_{\infty}^{1-\frac1p}$$
        and write
        $$\|[f(s|D|),a]-s\delta(a)f'(s|D|)\|_{p,1} \leq$$
        $$\leq\|[f(s|D|),a]-s\delta(a)f'(s|D|)\|_1^{\frac1p}\|[f(s|D|,a)]-s\delta(a)f'(s|D|)\|_{\infty}^{1-\frac1p}=$$
        $$= O(s^{2-p})^{\frac1p}\cdot O(s^{2})^{1-\frac1p}=O(s).$$
        \end{proof} 

    The following Lemma is used in Lemma \ref{kogom2}, Lemma \ref{kogom3} and Lemma \ref{second opposite lemma}.
    \begin{lem}\label{commutator 7} 
        Let $(\mathcal{A},H,D)$ be a smooth $p-$dimensional spectral triple satisfying Hypothesis \ref{main assumption} and Hypothesis \ref{auxiliary assumption}. For every $0\leq m\leq p,$ and $x \in \mathcal{B}$ we have
        \begin{equation*}
            \|D^{m-p}[D^{2-m}e^{-s^2D^2},x]\|_1 = O(s^{-1}),\quad s\downarrow0.
        \end{equation*}
    \end{lem}
    \begin{proof} 
        By the Leibniz rule,
        \begin{equation*}
            [D^{2-m}e^{-s^2D^2},x] = [D^{2-m},a]e^{-s^2D^2}+D^{2-m}[e^{-s^2D^2},x].
        \end{equation*}
        Thus,
        \begin{equation*}
            \|D^{m-p}[D^{2-m}e^{-s^2D^2},x]\|_1 \leq \|D^{2-p}[e^{-s^2D^2},x]\|_1+\|D^{m-p}[D^{2-m},x]e^{-s^2D^2}\|_1.
        \end{equation*}
        
        By Lemma \ref{commutator 6}, we have $\|D^{2-p}[e^{-s^2D^2},x]\|_1 = O(s^{-1})$, so we now focus on the second summand.
        First, for the case when $m > 2$ we apply the Leibniz rule
        \begin{align*}
            [D^{2-m},x] &= -D^{2-m}[D^{m-2},x]D^{2-m}\\
                        &= -\sum_{k+l=m-3} D^{k+2-m}\partial(x)D^{l+2-m}.
        \end{align*}
        Now using the triangle inequality:
        \begin{equation*}
            \|D^{m-p}[D^{2-m},x]e^{-s^2D^2}\|_1 \leq \sum_{k+l=m-3}\||D|^{k+2-p}\partial(x)|D|^{l+2-m}e^{-s^2|D|^2}\|_1.
        \end{equation*}
        Applying Lemma \ref{first decay lemma} to each summand, we then conclude that $${\|D^{m-p}[D^{2-m},x]e^{-s^2D^2}\|_1 = O(s^{-1})},$$
        thus proving the claim for $m > 2$.
        
        We now deal with the remaining cases $m = 0,1,2$ individually.        
        In the case $m=2$, we have $D^{m-p}[D^{2-m},x]e^{-s^2D^2} = 0$, and so the claim follows trivially in this case.
        
        For $m = 1$, we have
        \begin{equation*}
            \|D^{m-p}[D^{2-m},x]e^{-s^2D^2}\|_1 = \||D|^{1-p}\partial(x)e^{-s^2D^2}\|_1
        \end{equation*}
        So by Lemma \ref{first decay lemma}, we also have in this case that the above is $O(s^{-1})$.
        
        Finally, for $m=0$,
        \begin{align*}
            [D^2,x] &= [|D|^2,x]\\ 
                    &= |D|\delta(x)+\delta(x)|D|\\
                    &= 2|D|\delta(x)-\delta^2(x).
        \end{align*}
        So by the triangle inequality:
        \begin{equation*}
            \|D^{m-p}[D^{2-m},x]e^{-s^2D^2}\|_1 \leq 2\||D|^{1-p}\delta(x)e^{-s^2D^2}\|_1+\||D|^{-p}\delta^2(x)e^{-s^2D^2}\|_1
        \end{equation*}
        so an application of Lemma \ref{first decay lemma} to each of the above summands yields the result.
    \end{proof}

\section{Exploiting Hochschild homology}\label{cohomology section}
    
    Recall the multilinear mapping $\mathcal{W}_p$ from Definition \ref{W definition}. In this section, we prove the following:
    \begin{thm}\label{reduction} 
        Let $(\mathcal{A},H,D)$ be a spectral triple satisfying Hypothesis \ref{main assumption} and Hypothesis \ref{auxiliary assumption}. For every Hochschild cycle $c\in\mathcal{A}^{\otimes (p+1)}$ we have:
        \begin{equation*}
            \mathrm{Tr}(\Omega(c)|D|^{2-p}e^{-s^2D^2})-p\mathrm{Tr}(\mathcal{W}_p(c)De^{-s^2D^2}) = O(s^{-1}),\quad s\downarrow 0.
        \end{equation*}
    \end{thm}
    We achieve this using the commutator estimates of the preceding section.
    
    Our strategy to prove Theorem \ref{reduction} is to start from Theorem \ref{combinatorial theorem} and then show that:
    \begin{enumerate}
        \item{} all of the terms with $|\mathscr{A}| \geq 2$ are $O(s^{-1})$ (see Lemma \ref{kogom4})
        \item{} the $\mathscr{A} = \emptyset$ term is $O(s^{-1})$ (see Lemma \ref{kogom5}
        \item{} finally we complete the proof by showing that the terms with $|\mathscr{A}|=1$ are all equal to the $\mathscr{A}=\{p\}$ term up to terms of size $O(s^{-1})$.
    \end{enumerate}
    The proofs in this section rely crucially on the assumption that $c$ is a Hochschild cycle.
    
    First, we show that terms in Theorem \ref{combinatorial theorem} such that there is some $m$ with $m-1,m \in \mathscr{A}$ are $O(s^{-1})$.
    \begin{lem}\label{kogom2} 
        Let $(\mathcal{A},H,D)$ be a spectral triple satisfying Hypothesis \ref{main assumption} and Hypothesis \ref{auxiliary assumption}. Let $m \in \{1,\ldots,p\}$ and suppose that $m-1,m\in\mathscr{A}$ (so necessarily we have $|\mathscr{A}|\geq 2$). 
        For every Hochschild cycle $c\in\mathscr{A}^{\otimes (p+1)},$ we have
        \begin{equation*}
            \mathrm{Tr}(\mathcal{W}_{\mathscr{A}}(c)D^{2-|\mathscr{A}|}e^{-s^2D^2}) = O(s^{-1}),\quad s\downarrow 0.
        \end{equation*}
    \end{lem}
    \begin{proof} 
        For $s > 0$, consider the multilinear mapping $\theta_s:\mathcal{A}^{\otimes p}\to \mathbb{C}$ defined by:
        \begin{equation*}
            \theta_s(a_0\otimes \cdots \otimes a_{p-1}) = \mathrm{Tr}\left(\Gamma a_0\left(\prod_{k=1}^{m-2}b_k(a_k)\right)\delta^2(a_{m-1})\left(\prod_{k=m}^{p-1}b_{k+1}(a_k)\right)D^{2-|\mathscr{A}|}e^{-s^2D^2}\right)
        \end{equation*}
        where $b_k$ is as in Definition \ref{W definition}. Then, from the computation in Appendix \ref{coboundary app}, we have that the Hochschild coboundary is:
        \begin{align*}
            (b\theta_s)(a_0\otimes &\cdots \otimes a_p)\\
                                   &= (-1)^p\mathrm{Tr}\left(\Gamma a_0\left(\prod_{k=1}^{m-2}[b_k,a_k]\right)\delta^2(a_{m-1})\left(\prod_{k=m}^{p-1}[b_{k+1},a_k]\right)[D^{2-|\mathscr{A}|}e^{-s^2D^2},a_p]\right)\\
                                   &+ 2(-1)^{m-1}\mathrm{Tr}(\mathcal{W}_{\mathscr{A}}(a_0\otimes \cdots \otimes a_p)D^{2-|\mathscr{A}|}e^{-s^2D^2}). 
        \end{align*}
        
        We now claim that the first summand is $O(s^{-1})$ as $s\downarrow 0$. Indeed, dividing and multiplying by $D^{|\mathscr{A}|-p}$:
        \begin{align*}
            \Big\|\Gamma a_0\left(\prod_{k=1}^{m-2}[b_k,a_k]\right)&\delta^2(a_{m-1})\left(\prod_{k=m}^{p-1}[b_{k+1},a_k]\right)[D^{2-|\mathscr{A}|}e^{-s^2D^2},a_p]\Big\|_1\\
                                                                   &\leq \left\|D^{|\mathscr{A}|-p}[D^{2-|\mathscr{A}|}e^{-s^2D^2},a_p]\right\|_1\\
                                                                   &\times \left\|\Gamma a_0\left(\prod_{k=1}^{m-2}[b_k,a_k]\right)\delta^2(a_{m-1})\left(\prod_{k=m}^{p-1}[b_{k+1},a_k]\right)|D|^{p-|\mathscr{A}|}\right\|_\infty 
        \end{align*}
        The first factor is $O(s^{-1})$ by Lemma \ref{commutator 7}, and the second factor is finite by Lemma \ref{c first alert lemma} and has no dependence on $s$.
        
        To summarise, so far we that if $c \in \mathcal{A}^{\otimes (p+1)}$: 
        \begin{equation*}
            (b\theta_s)(c) = 2(-1)^{m-1}\mathrm{Tr}(\mathcal{W}_{\mathscr{A}}(c)D^{2-|\mathscr{A}|}e^{-s^2D^2})+O(s^{-1}),\quad s\downarrow 0.
        \end{equation*}
        If $c$ is a Hochschild cycle, then $(b\theta_s)(c) = \theta_s(bc) = 0$, and so
        \begin{equation*}
            2(-1)^{m-1}\mathrm{Tr}(\mathcal{W}_{\mathscr{A}}(c)D^{2-|\mathscr{A}|}e^{-s^2D^2}) = O(s^{-1})
        \end{equation*}
        as required.
    \end{proof}

    \begin{lem}\label{kogom3} 
        Let $(\mathcal{A},H,D)$ be a spectral triple satisfying Hypothesis \ref{main assumption} and Hypothesis \ref{auxiliary assumption}. 
        Let $\mathscr{A}_1,\mathscr{A}_2 \subseteq \{1,\ldots,p\}$, with $|\mathscr{A}_1|=|\mathscr{A}_2|$ and that the symmetric difference $\mathscr{A}_1\Delta\mathscr{A}_2=\{m-1,m\}$ for some $m.$ 
        Then for every Hochschild cycle $c\in\mathcal{A}^{\otimes (p+1)},$ we have
        \begin{equation*}
            \mathrm{Tr}(\mathcal{W}_{\mathscr{A}_1}(c)D^{2-|\mathscr{A}_1|}e^{-s^2D^2})+\mathrm{Tr}(\mathcal{W}_{\mathscr{A}_2}(c)D^{2-|\mathscr{A}_2|}e^{-s^2D^2}) = O(s^{-1}),\quad s\downarrow0.
        \end{equation*}
    \end{lem}
    \begin{proof} 
        This proof is similar to that of Lemma \ref{kogom2}. For $s > 0$ we consider the multilinear mapping $\theta_s:\mathcal{A}^{\otimes p}\to \mathbb{C}$ given by
        \begin{equation*}
            \theta_s(a_0\otimes\cdots\otimes a_{p-1}) = \mathrm{Tr}\left(\Gamma a_0\left(\prod_{k=1}^{m-2}b_k(a_k)\right)[F,\delta(a_{m-1})]\left(\prod_{k=m}^{p-1}b_{k+1}(a_k)\right)D^{2-|\mathscr{A}_1|}e^{-s^2D^2}\right).
        \end{equation*}
        Here, as in Lemma \ref{kogom2}, the operators $b_k$ are defined as in Definition \ref{W definition}, relative to the set $\mathscr{A} = \mathscr{A}_1$. From the computation in Appendix \ref{coboundary app},
        \begin{align*}
            (b\theta_s)(a_0&\otimes\cdots\otimes a_p)\\
                           &= (-1)^p\mathrm{Tr}\left(\Gamma a_0\left(\prod_{k=1}^{m-2}[b_k,a_k]\right)[F,\delta(a_{m-1})]\left(\prod_{k=m}^{p-1}[b_{k+1},a_k]\right)[D^{2-|\mathscr{A}_1|}e^{-s^2D^2},a_p]\right)\\
                           &+ (-1)^{m-1}\mathrm{Tr}(\mathcal{W}_{\mathscr{A}_1}(a_0\otimes\cdots\otimes a_p)D^{2-|\mathscr{A}_1|}e^{-s^2D^2})\\
                           &+ (-1)^{m-1}\mathrm{Tr}(\mathcal{W}_{\mathscr{A}_2}(a_0\otimes\cdots\otimes a_p)D^{2-|\mathscr{A}_2|}e^{-s^2D^2}).
        \end{align*}
        We first show that the first summand above is $O(s^{-1})$. Indeed,
        \begin{align*}       
            \Big\|\Gamma a_0\left(\prod_{k=1}^{m-2}[b_k,a_k]\right)&[F,\delta(a_{m-1})]\left(\prod_{k=m}^{p-1}[b_{k+1},a_k]\right)[D^{2-|\mathscr{A}_1|}e^{-s^2D^2},a_p]\Big\|_1\\
                                                                   &\leq \left\|D^{|\mathscr{A}_1|-p}[D^{2-|\mathscr{A}_1|}e^{-s^2D^2},a_p]\right\|_1\\
                                                                   &\times\left\|\Gamma a_0\prod_{k=1}^{m-2}[b_k,a_k][F,\delta(a_{m-1})]\prod_{k=m}^{p-1}[b_{k+1},a_k]\cdot|D|^{p-|\mathscr{A}_1|}\right\|_{\infty}.
        \end{align*}
        The first factor above is $O(s^{-1})$ due to Lemma \ref{commutator 7}, and the second factor is finite by Lemma \ref{c second alert lemma} and has no dependence on $s$.
        
        Summarising the above, if $c \in \mathcal{A}^{\otimes (p+1)}$ we have
        \begin{align*}
            (b\theta_s)(c) &= (-1)^{m-1}\mathrm{Tr}(\mathcal{W}_{\mathscr{A}_1}(c)D^{2-|\mathscr{A}_1|}e^{-s^2D^2})\\
                           &+(-1)^{m-1}\mathrm{Tr}(\mathcal{W}_{\mathscr{A}_2}(c)D^{2-|\mathscr{A}_2|}e^{-s^2D^2})+ O(s^{-1})
        \end{align*}
        as $s\downarrow 0$. Hence, if $c$ is a Hochschild cycle then $(b\theta_s)(c) = \theta_s(bc) = 0$, and this completes the proof.
    \end{proof}

    \begin{lem}\label{kogom4} 
        Let $(\mathcal{A},H,D)$ be a spectral triple satisfying Hypothesis \ref{main assumption} and Hypothesis \ref{auxiliary assumption}. 
        For every Hochschild cycle $c\in\mathcal{A}^{\otimes (p+1)}$ and for every $\mathscr{A}\subset\{1,\cdots,p\}$ with $|\mathscr{A}|\geq 2,$ we have
        \begin{equation*}
            \mathrm{Tr}(\mathcal{W}_{\mathscr{A}}(c)D^{2-|\mathscr{A}|}e^{-s^2D^2}) = O(s^{-1}),\quad s\downarrow0.
        \end{equation*}
    \end{lem}
    \begin{proof} 
        Let $m$ be the maximum element in $\mathscr{A}$ and let $n$ be the maximal element in $\mathscr{A}\backslash\{m\}.$ 
        If $n=m-1,$ then the assertion is already proved in Lemma \ref{kogom2}. If $n < m-1,$ then $m-n > 1$, and hence $n+j \notin \mathscr{A}$
        for all $1 \leq j < m-n$. Now for each $0 \leq j < m-n$ we define $\mathscr{A}_j$ to be $\mathscr{A}$ with $n$ replaced with $n+j$. That is:
        \begin{equation*}
            \mathscr{A}_j := (\mathscr{A}\backslash\{n\})\cup\{j+n\},\quad 0 \leq j < m-n.
        \end{equation*}
        Then by construction:
        \begin{enumerate}
            \item $|\mathscr{A}_j|=|\mathscr{A}|$ and $\mathscr{A}_j\Delta\mathscr{A}_{j-1} = \{n+j,n+j-1\}$ for all $1 \leq j < m-n.$
            \item $\mathscr{A}_0=\mathscr{A}$ and $m-1,m\in\mathscr{A}_{m-n-1}.$
        \end{enumerate}
        Hence if $1 \leq j < m-n$ the subsets $\mathscr{A}_j$ and $\mathscr{A}_{j-1}$ satisfy the conditions of Lemma \ref{kogom3}. So for all Hochschild cycles $c \in \mathcal{A}^{\otimes (p+1)}$:
        \begin{equation*}
            \mathrm{Tr}(\mathcal{W}_{\mathscr{A}_{j-1}}(c)D^{2-|\mathscr{A}_{i-1}|}e^{-s^2D^2})=-\mathrm{Tr}(\mathcal{W}_{\mathscr{A}_j}(c)D^{2-|\mathscr{A}_{j}|}e^{-s^2D^2}) + O(s^{-1}),\quad s \downarrow 0.
        \end{equation*}
        So by induction, we have:
        \begin{align}\label{beginning and end}
            \mathrm{Tr}(\mathcal{W}_{\mathscr{A}_0}(c)D^{2-|\mathscr{A}_0|}e^{-s^2D^2}) &= (-1)^{m-n-1}\mathrm{Tr}(\mathcal{W}_{\mathscr{A}_{m-n-1}}(c)D^{2-|\mathscr{A}_{m-n-1}|}e^{-s^2D^2})\nonumber\\
                                                        &\quad+O(s^{-1}),\quad s\downarrow0.
        \end{align}
        On the other hand, since $m-1,m \in \mathscr{A}_{m-n-1}$, we may apply Lemma \ref{kogom2} to $\mathscr{A}_{m-n-1}$ to obtain:
        \begin{equation*}\label{terminal}
            \mathrm{Tr}(\mathcal{W}_{\mathscr{A}_{m-n-1}}(c)D^{2-|\mathscr{A}_{m-n-1}|}e^{-s^2D^2}) = O(s^{-1}),\quad s\downarrow0.
        \end{equation*}
        Combining \eqref{beginning and end} and \eqref{terminal}, we get
        \begin{equation*}
            \mathrm{Tr}(\mathcal{W}_{\mathscr{A}_0}(c)D^{2-|\mathscr{A}_0|}e^{-s^2D^2}) = O(s^{-1}).
        \end{equation*}
        since $\mathscr{A}_0 = \mathscr{A}$, the proof is complete.
    \end{proof}
    
    Recall the mapping $\mathrm{ch}$ from Definition \ref{ch omega def}.
    \begin{lem}\label{kogom pre 5} 
        Let $(\mathcal{A},H,D)$ be a spectral triple satisfying Hypothesis \ref{main assumption} and Hypothesis \ref{auxiliary assumption}. For every $c\in\mathcal{A}^{\otimes (p+1)},$ we have
        \begin{align*}
            \|\mathrm{ch}(c)D^2e^{-s^2D^2}\|_1 = O(s^{-1}),\quad s \downarrow 0.
        \end{align*}
    \end{lem}
    \begin{proof} 
        Recall that on $H_\infty$ we have $[F,a_p]|D| = L(a_p)$.
        So on $H_\infty$:
        \begin{align*}
            [F,a_{p-1}][F,a_p]|D|^2 &= [F,a_{p-1}]\cdot L(a_p)\cdot|D|\\
                                    &= [F,a_{p-1}]\cdot |D|L(a_p)-[F,a_{p-1}]\cdot [|D|,L(a_p)]\\
                                    &= [F,a_{p-1}]|D|\cdot L(a_p)-[F,a_{p-1}]\cdot\delta(L(a_p))\\
                                    &= L(a_{p-1})\cdot L(a_p)-[F,a_{p-1}]\cdot L(\delta(a_p)).
        \end{align*}
        So for $c=a_0\otimes\cdots\otimes a_p \in \mathcal{A}^{\otimes(p+1)}$ we have
        \begin{align*}
            \mathrm{ch}(c)\cdot D^2e^{-s^2D^2} &= \Gamma\left(\prod_{k=0}^{p-2}[F,a_k]\right) |D|^{p-1}\cdot |D|^{1-p}L(a_{p-1})L(a_p)e^{-s^2D^2}\\
                                       &\quad-\Gamma\left(\prod_{k=0}^{p-1}[F,a_k]\right)|D|^{p-1}\cdot|D|^{1-p}L(\delta(a_p))e^{-s^2D^2}.
        \end{align*}
        Using Lemma \ref{c third alert lemma}, the operators $\left(\prod_{k=0}^{p-2}[F,a_k]\right)|D|^{p-1}$ and $\left(\prod_{k=0}^{p-1}[F,a_k]\right)|D|^{p-1}$ both
        have bounded extension and no dependence on $s$.
        From Lemma \ref{first decay lemma}, we have that:
        \begin{align*}
            \||D|^{1-p}L(a_{p-1})L(a_p)e^{-s^2D^2}\|_1 &= O(s^{-1})\\
              \||D|^{1-p}L(\delta(a_p))e^{-s^2D^2}\|_1 &= O(s^{-1}).
        \end{align*}
        So by the triangle inequality: $\|\mathrm{ch}(c)D^2e^{-s^2D^2}\|_1 = O(s^{-1})$ as $s\downarrow 0$.
    \end{proof}
            
    \begin{lem}\label{W and ch link}
        Let $(\mathcal{A},H,D)$ be a spectral triple of dimension $p$, where $p$ has the same parity as $(\mathcal{A},H,D)$. If $c \in \mathcal{A}^{\otimes(p+1)}$ then
        \begin{equation*}
            \mathrm{ch}(c) = \mathcal{W}_{\emptyset}(c)+F\mathcal{W}_{\emptyset}(c)F.
        \end{equation*}
    \end{lem}
    \begin{proof}
        Let $c = a_0\otimes \cdots \otimes a_p \in \mathcal{A}^{\otimes (p+1)}$. Recall that
        \begin{equation*}
            \mathcal{W}_{\emptyset}(c) = \Gamma a_0\prod_{k=1}^p [F,a_k].
        \end{equation*}
        Using the fact that $F$ anticommutes with $[F,a_k]$ for all $k$, we have:
        \begin{align*}
            \mathrm{ch}(c) &= \Gamma F[F,a_0]\prod_{k=1}^p [F,a_k]\\
                   &= \Gamma a_0\prod_{k=1}^p [F,a_k]-\Gamma Fa_0F\prod_{k=1}^p [F,a_k].\\
                   &= \mathcal{W}_{\emptyset}(c)+(-1)^{p+1}\Gamma Fa_0\left(\prod_{k=1}^p [F,a_k]\right)F.
        \end{align*}
        Since $\Gamma^2 = 1$,
        \begin{equation*}
            \mathrm{ch}(c) = \mathcal{W}_{\emptyset}(c) + (-1)^{p+1}\Gamma F\Gamma \mathcal{W}_{\emptyset}(c)F.
        \end{equation*}
        Since the parity of $p$ matches the parity of $\Gamma$, we have $\Gamma F = (-1)^{p+1}F\Gamma$. This completes the proof.
    \end{proof}

    \begin{lem}\label{kogom5}
        Let $(\mathcal{A},H,D)$ be a spectral triple satisfying Hypothesis \ref{main assumption} and Hypothesis \ref{auxiliary assumption}. For every $c \in \mathcal{A}^{\otimes(p+1)}$, we have
        \begin{equation*}
            \mathrm{Tr}(\mathcal{W}_{\emptyset}(c)D^2e^{-s^2D^2}) = O(s^{-1}),\quad s\downarrow 0.
        \end{equation*}
    \end{lem}
    \begin{proof} 
        By Lemma \ref{W and ch link}, we have:
        \begin{equation*}
            \mathrm{Tr}(\mathrm{ch}(c)D^2e^{-s^2D^2}) = \mathrm{Tr}(\mathcal{W}_{\emptyset}(c)D^2e^{-s^2D^2})+\mathrm{Tr}(F\mathcal{W}_{\emptyset}(c)FD^2e^{-s^2D^2}).
        \end{equation*}
        However since $F$ commutes with $D^2e^{-s^2D^2}$ and $F^2=1$ we have:
        \begin{equation*}
            2\mathrm{Tr}(\mathcal{W}_{\emptyset}(c)D^2e^{-s^2D^2}) = \mathrm{Tr}(\mathrm{ch}(c)D^2e^{-s^2D^2}).
        \end{equation*}
        However by Lemma \ref{kogom pre 5},
        \begin{equation*}
            |\mathrm{Tr}(\mathrm{ch}(c)D^2e^{-s^2D^2})| = O(s^{-1}).
        \end{equation*}
        Hence $\mathrm{Tr}(\mathcal{W}_{\emptyset}(c)D^2e^{-s^2D^2}) = O(s^{-1})$.
    \end{proof}

    We are now ready to prove the main result of this section.

    \begin{proof}[Proof of Theorem \ref{reduction}] 
        Let $c  \in \mathcal{A}^{\otimes (p+1)}$. Then using Theorem \ref{combinatorial theorem}:
        \begin{equation*}
            \mathrm{Tr}(\Omega(c)|D|^{2-p}e^{-s^2D^2}) = \sum_{\mathscr{A} \subseteq \{1,\ldots,p\}}(-1)^{n_{\mathscr{A}}}\mathrm{Tr}(\mathcal{W}_{\mathscr{A}}(c)D^{2-|\mathscr{A}|}e^{-s^2D^2})+O(s^{-1}),\quad s\downarrow 0.
        \end{equation*}
        
        Applying Lemma \ref{kogom4} to every summand with $|\mathscr{A}| \geq 2$, and Lemma \ref{kogom5} to the summand $\mathscr{A} = \emptyset$, it follows that:
        \begin{equation*}
            \mathrm{Tr}(\Omega(c)|D|^{2-p}e^{-s^2D^2}) = \sum_{k=1}^p (-1)^{n_{\{k\}}} \mathrm{Tr}(\mathcal{W}_k(c)De^{-s^2D^2})+O(s^{-1}),\quad s\downarrow 0.
        \end{equation*}
        Recall that $n_{\mathscr{A}} = |\{(j,k)\in \{1,\ldots,p\}^2\;:\;j \in \mathscr{A}, k\notin \mathscr{A}\}|$.
        So in particular, $n_{\{k\}} = p-k$. Hence:
        \begin{equation}\label{single element expansion}
            \mathrm{Tr}(\Omega(c)|D|^{2-p}e^{-s^2D^2}) = \sum_{k=1}^p (-1)^{p-k}\mathrm{Tr}(\mathcal{W}_{k}(c)De^{-s^2D^2})+O(s^{-1}),\quad s\downarrow 0.
        \end{equation}
        For any $1 \leq k \leq p-1$, the sets $\mathscr{A}_1 = \{k\}$ and $\mathscr{A}_2 = \{k+1\}$ satisfy the conditions of Lemma \ref{kogom3} with $m = k+1$. So we have
        \begin{equation*}
            \mathrm{Tr}(\mathcal{W}_k(c)De^{-s^2D^2}) = -\mathrm{Tr}(\mathcal{W}_{k+1}(c)De^{-s^2D^2}) + O(s^{-1}),\quad s\downarrow0,
        \end{equation*}
        Hence by induction, for any $1 \leq k \leq p$:
        \begin{equation}\label{all singletons are p}
            \mathrm{Tr}(\mathcal{W}_k(c)De^{-s^2D^2}) = (-1)^{p-k}\mathrm{Tr}(\mathcal{W}_{p}De^{-s^2D^2}) + O(s^{-1}),\quad s\downarrow 0.
        \end{equation}        
        Substituting \eqref{all singletons are p} into each summand of \eqref{single element expansion} we finally get:
        \begin{equation*}
            \mathrm{Tr}(\Omega(c)|D|^{2-p}e^{-s^2D^2}) = p\mathrm{Tr}(\mathcal{W}_p(c)De^{-s^2D^2})+O(s^{-1}).
        \end{equation*}
    \end{proof}

\section{Preliminary heat semigroup asymptotic}\label{preliminary heat section}

    In this section, we move closer to proving Theorem \ref{heat thm}.
    We will show that if $(\mathcal{A},H,D)$ satisfies Hypothesis \ref{main assumption} and Hypothesis \ref{auxiliary assumption} {(in particular, $D$ has a spectral gap at $0$)}, then for a Hochschild cycle $c \in \mathcal{A}^{\otimes(p+1)}$
    we have 
    \begin{equation}\label{prelim heat asymptotic}
        \mathrm{Tr}(\mathcal{W}_p(c)|D|^{2-p}e^{-s^2D^2}) = \frac14\mathrm{Tr}(\mathrm{ch}(c))s^{-2}+O(s^{-1}), s\downarrow 0.
    \end{equation}
    By the Theorem \ref{reduction}, this is a formula very close to Theorem \ref{heat thm}: the only difference is the assumption of Hypothesis \ref{auxiliary assumption} and that the occurance of $|D|$ should be replaced with $(1+D^2)^{1/2}.$ 
    
    We start with the following asymptotic result.
    \begin{lem}\label{initial convergence theorem} 
        Let $(\mathcal{A},H,D)$ be a spectral triple satisfying Hypothesis \ref{main assumption} and Hypothesis \ref{auxiliary assumption}. For all $c\in\mathcal{A}^{\otimes (p+1)},$ we have
        \begin{equation*}
            \mathrm{Tr}(\mathrm{ch}(c)e^{-s^2D^2}) = \mathrm{Tr}(\mathrm{ch}(c)) + O(s),\quad s\downarrow0.
        \end{equation*}
    \end{lem}
    \begin{proof} 
        We wish to show that
        \begin{equation*}
            \mathrm{Tr}(\mathrm{ch}(c)(1-e^{-s^2D^2})) = O(s)
        \end{equation*}
        so it suffices to prove
        \begin{equation*}
            \|\mathrm{ch}(c)(1-e^{-s^2D^2})\|_1 = O(s).
        \end{equation*}
        
        Let $c = a_0\otimes\cdots \otimes a_p \in \mathcal{A}^{\otimes (p+1)}$. Then using $[F,a_p] = |D|^{-1}\partial(a_p)-|D|^{-1}\delta(a_p)F$, we have (on $H_\infty$)
        \begin{align*}
            \mathrm{ch}(c)(1-e^{-s^2D^2}) &= \Gamma F \left(\prod_{k=0}^{p-1}[F,a_k]\right)(|D|^p|D|^{-p})[F,a_p](1-e^{-s^2D^2})\\
                                  &= \Gamma F \left(\prod_{k=0}^{p-1}[F,a_k]\right)|D|^p\cdot |D|^{-p-1}\partial(a_p)(1-e^{-s^2D^2})\\
                                  &\quad-\Gamma F \left(\prod_{k=0}^{p-1}[F,a_k]\right)|D|^p\cdot |D|^{-p-1}\delta(a_p)(1-e^{-s^2D^2})F.
        \end{align*}
        
        Thus,
        \begin{align*}
                 \|\mathrm{ch}(c)(1-e^{-s^2D^2})\|_1 &\leq  \left\|\left(\prod_{k=0}^{p-1}[F,a_k]\right)|D|^p\right\|_{\infty}\cdot\left\||D|^{-p-1}\partial(a_p)(1-e^{-s^2D^2})\right\|_1\\
                                             &\quad+\left\|\left(\prod_{k=0}^{p-1}[F,a_k]\right)|D|^p\right\|_{\infty}\cdot\left\||D|^{-p-1}\delta(a_p)(1-e^{-s^2D^2})\right\|_1.
        \end{align*}
        In both summands, the first factor is finite by Lemma \ref{c third alert lemma}, the second factor is $O(s)$ by Lemma \ref{right convergence estimate}.
    \end{proof}

    \begin{thm}\label{first cycle thm} 
        Let $(\mathcal{A},H,D)$ be a spectral triple satisfying Hypothesis \ref{main assumption} and Hypothesis \ref{auxiliary assumption}. For every Hochschild cycle $c\in\mathcal{A}^{\otimes (p+1)},$ we have
        \begin{equation}\label{wp heat estimate}
            \mathrm{Tr}(\mathcal{W}_p(c)De^{-s^2D^2}) = \frac{\mathrm{Ch}(c)}{2}s^{-2}+O(s^{-1}),\quad s\downarrow0.
        \end{equation}
    \end{thm}
    \begin{proof}
        Let $c \in \mathcal{A}^{\otimes (p+1)}$ be a Hochschild cycle.
        By Lemma \ref{initial convergence theorem}, we have
        \begin{equation*}
            \mathrm{Tr}(\mathrm{ch}(c)e^{-s^2D^2}) = \mathrm{Tr}(\mathrm{ch}(c))+O(s),\quad s\downarrow0.
        \end{equation*}
        Since the spectral triple and $p$ both have the same parity, we may apply Lemma \ref{W and ch link} to get:
        \begin{equation}\label{restated initial convergence}
            2\mathrm{Tr}(\mathcal{W}_{\emptyset}(c)e^{-s^2D^2}) = \mathrm{Tr}(\mathrm{ch}(c))+O(s),\quad s\downarrow0,
        \end{equation}
        for all $c \in \mathcal{A}^{\otimes (p+1)}$

        Define the multilinear mappings $\mathcal{K}_s,\,\mathcal{H}_s:\mathcal{A}^{\otimes (p+1)}\to\mathbb{C}$ by setting
        \begin{align*}
            \mathcal{K}_s(a_0\otimes\cdots\otimes a_p) &= \mathrm{Tr}(\Gamma a_0\left(\prod_{k=1}^{p-1}[F,a_k]\right)[Fe^{-s^2D^2},a_p]),\\
            \mathcal{H}_s(a_0\otimes\cdots\otimes a_p) &= \mathrm{Tr}(\Gamma a_0\left(\prod_{k=1}^{p-1}[F,a_k]\right)F[e^{-s^2D^2},a_p]).
        \end{align*}

        By the Leibniz rule, we have
        \begin{equation*}
            [F,a_p]e^{-s^2D^2} = [Fe^{-s^2D^2},a_p]-F[e^{-s^2D^2},a_p].
        \end{equation*}
        Therefore:
        \begin{equation}\label{k minus h}
            \mathrm{Tr}(\mathcal{W}_{\emptyset}(c)e^{-s^2D^2}) = \mathcal{K}_s(c)-\mathcal{H}_s(c).
        \end{equation}

        Now combining \eqref{restated initial convergence} and \eqref{k minus h}, we arrive at
        \begin{equation}\label{k minus h convergence}
            2\mathcal{K}_s(c)-2\mathcal{H}_s(c) = \mathrm{Tr}(\mathrm{ch}(c))+O(s),\quad s\downarrow0,
        \end{equation}

        However it is shown in Appendix \ref{coboundary app} that $\mathcal{K}_s$ is a Hochschild coboundary, and thus since $c$ is a Hochschild cycle we have $\mathcal{K}_s(c) = 0$.
        
        Using \eqref{k minus h convergence}, we obtain
        \begin{equation}\label{minus h convergence}
            -2\mathcal{H}_s(c) = \mathrm{Tr}(\mathrm{ch}(c))+O(s),\quad s\downarrow0,
        \end{equation}

        Define the multilinear mapping $\mathcal{V}_s:\mathcal{A}^{\otimes (p+1)}\to\mathbb{C}$ by setting
        \begin{equation*}
            \mathcal{V}_s(a_0\otimes\cdots\otimes a_p) = \mathrm{Tr}(\Gamma a_0\left(\prod_{k=1}^{p-1}[F,a_k]\right)F\delta(a_p)|D|e^{-s^2D^2}).
        \end{equation*}

        Let $\frac{1}{q} = 1-\frac{1}{p}$. By the H\"older inequality in the form \eqref{another holder}:
        \begin{align*}
            |(\mathcal{H}_s+2s^2\mathcal{V}_s)&(a_0\otimes\cdots\otimes a_p)| \\
                    &= \left|\mathrm{Tr}\Big(\Gamma a_0\Big(\prod_{k=1}^{p-1}[F,a_k]\Big)F\cdot\Big([e^{-s^2D^2},a_p]+2s^2\delta(a_p)|D|e^{-s^2D^2}\Big)\Big)\right|\\
                    &\leq \Big\|\Gamma a_0\prod_{k=1}^{p-1}[F,a_k]F\Big\|_{q,\infty}\Big\|[e^{-s^2D^2},a_p]+2s^2\delta(a_p)|D|e^{-s^2D^2}\Big\|_{p,1}.
        \end{align*}

        The first factor on the above right hand side is finite by Proposition \ref{f der def}, and the second factor above is $O(s)$ by Lemma \ref{third commutator lemma}.\eqref{3com3}. Therefore, we have
        \begin{equation}\label{h minus v}
            (\mathcal{H}_s+2s^2\mathcal{V}_s)(c) = O(s),\quad s\downarrow0,
        \end{equation}

        Combining \eqref{h minus v} and \eqref{minus h convergence}, we arrive at
        \begin{equation}\label{minus v convergence}
            4s^2\mathcal{V}_s(c) = \mathrm{Tr}(\mathrm{ch}(c))+O(s),\quad s\downarrow0,
        \end{equation}
        for all Hochschild cycles $c\in\mathcal{A}^{\otimes (p+1)}.$

        From the definition of $\mathcal{W}_p$, if $a_0\otimes \cdots \otimes a_p \in \mathcal{A}^{\otimes (p+1)}$:
        \begin{align*}
            \mathrm{Tr}(\mathcal{W}_p(a_0\otimes\cdots\otimes a_p)De^{-s^2D^2}) &= \mathrm{Tr}(\Gamma a_0\left(\prod_{k=1}^{p-1}[F,a_k]\right)\delta(a_p)De^{-s^2D^2})\\
                                                                &= \mathrm{Tr}(\Gamma a_0\left(\prod_{k=1}^{p-1}[F,a_k]\right)\delta(a_p)F|D|e^{-s^2D^2}).
        \end{align*}
        Thus,
        \begin{align*}
            \Big|\mathcal{V}_s(a_0\otimes\cdots\otimes a_p)&-\mathrm{Tr}(\mathcal{W}_p(a_0\otimes\cdots\otimes a_p)De^{-s^2D^2})\Big| \\
                                                   &= \Big|\mathrm{Tr}(\Gamma a_0\left(\prod_{k=1}^{p-1}[F,a_k]\right)[F,\delta(a_p)]|D|e^{-s^2D^2})\Big|\\
                                                   &= \Big|\mathrm{Tr}(\Gamma \left(\prod_{k=1}^{p-1}[F,a_k]\right)[F,\delta(a_p)]|D|e^{-s^2D^2}a_0)\Big|\\
                                                   &\leq \Big\|\Gamma \left(\prod_{k=1}^{p-1}[F,a_k]\right)[F,\delta(a_p)]|D|^p\Big\|_{\infty}\cdot\Big\||D|^{1-p}e^{-s^2D^2}a_0\Big\|_1.
        \end{align*}
        By Lemma \ref{first decay lemma}, we have that this is $O(s^{-1})$ and therefore, we have
        \begin{equation}\label{v minus answer}
        \mathcal{V}_s(c)=\mathrm{Tr}(\mathcal{W}_p(c)De^{-s^2D^2})+O(s^{-1}),\quad s\downarrow0,
        \end{equation}
        for all $c\in\mathcal{A}^{\otimes (p+1)}.$

        Combining \eqref{minus v convergence} and \eqref{v minus answer}, we arrive at
        \begin{equation*}
            4s^2\mathrm{Tr}(\mathcal{W}_p(c)De^{-s^2D^2})=\mathrm{Tr}(\mathrm{ch}(c))+O(s),\quad s\downarrow0
        \end{equation*}
        for all Hochschild cycles $c\in\mathcal{A}^{\otimes (p+1)}.$ Dividing by $4s^2$,
        \begin{equation*}
            \mathrm{Tr}(\mathcal{W}_p(c)De^{-s^2D^2}) = \frac{1}{4}s^{-2}\mathrm{Tr}(\mathrm{ch}(c))+O(s^{-1}).
        \end{equation*}        
        Since $\mathrm{Ch}(c) = \frac{1}{2}\mathrm{Tr}(\mathrm{ch}(c))$, this formula coincides with \eqref{wp heat estimate}.
    \end{proof}
    
    We remark that \eqref{prelim heat asymptotic} follows as a simple combination of Theorems \ref{reduction} and \ref{first cycle thm}.

\section{Heat semigroup asymptotic: the proof of the first main result}\label{heat section}

    In this section, we finally complete the proof of Theorem \ref{heat thm}. We start by removing the assumption of Hypothesis \ref{auxiliary assumption}.
    
    The following two lemmas show that if the parity of $p$ does not match $(\mathcal{A},H,D)$, then the statement of \eqref{prelim heat asymptotic} becomes trivial.
    \begin{lem}\label{first opposite lemma}
        Let $(\mathcal{A},H,D)$ be a spectral triple satisfying Hypothesis \ref{main assumption}. Suppose that $D$ has a spectral gap at $0$. Suppose that the dimension $p$ is odd but $(\mathcal{A},H,D)$ is even. Then:
        \begin{enumerate}[{\rm (i)}]
            \item\label{first opposite 1} for every $c\in\mathcal{A}^{\otimes (p+1)},$ we have
                \begin{equation*}
                    \mathrm{Tr}(\Omega(c)|D|^{2-p}e^{-s^2D^2})=0,\quad s>0.
                \end{equation*}
            \item\label{first opposite 2} for every $c\in\mathcal{A}^{\otimes (p+1)},$ we have
                \begin{equation*}
                    \mathrm{Tr}(\mathrm{ch}(c))=0.
                \end{equation*}
        \end{enumerate}
    \end{lem}
    \begin{proof}
        Let us prove \eqref{first opposite 1}. Since $\Gamma D = -D\Gamma$ and $\Gamma$ commutes with $a \in \mathcal{A}$, we have $\Gamma [D,a]=-[D,a]\Gamma$ on $H_\infty$. Hence $\Gamma \partial(a) = -\partial(a)\Gamma$
        for all $a \in \mathcal{A}$.
        Thus for $c = a_0\otimes \cdots\otimes a_p \in \mathcal{A}^{\otimes (p+1)}$ since $p$ is odd we have:
        $$\Omega(c)=\Gamma a_0\prod_{k=1}^p\partial(a_k)=-a_0\left(\prod_{k=1}^p\partial(a_k)\right)\Gamma.$$
        However $\Gamma D^2 = D^2\Gamma$, so by the spectral theorem we have $\Gamma e^{-s^2D^2} = e^{-s^2D^2}$. So multiplying by $e^{-s^2D^2}$ on the right and taking the trace,
        \begin{align*}
        \mathrm{Tr}(\Omega(c)|D|^{2-p}e^{-s^2D^2}) &= -\mathrm{Tr}(a_0\left(\prod_{k=1}^p\partial(a_k)\right)|D|^{2-p}e^{-s^2D^2}\Gamma)\\
                                           &= -\mathrm{Tr}(\Gamma a_0\left(\prod_{k=1}^p\partial(a_k)\right)|D|^{2-p}e^{-s^2D^2}).
        \end{align*}
        This proves \eqref{first opposite 1}

        The argument for \eqref{first opposite 2} is similar. Note that we have $\Gamma[F,a]=-[F,a]\Gamma$ for every $a\in\mathcal{A}.$ Thus since $p+1$ is even:
        \begin{align*}
            \mathrm{ch}(c) &= \Gamma F\prod_{k=0}^p[F,a_k]\\
                   &= -F\Gamma \prod_{k=0}^p[F,a_k]\\
                   &= -F\cdot\prod_{k=0}^p[F,a_k]\Gamma.
        \end{align*}
        Thus,
        \begin{align*}
            \mathrm{Tr}(\mathrm{ch}(c)) &= -\mathrm{Tr}(\Gamma F\prod_{k=0}^p[F,a_k])\\
                        &= -\mathrm{Tr}(\mathrm{ch}(c)).
        \end{align*}
        This proves the second assertion.
    \end{proof}
    
    Now, we deal with the other case where the parity of $(\mathcal{A},H,D)$ does not match $p$.
    \begin{lem}\label{second opposite lemma} 
        Let $(\mathcal{A},H,D)$ be a spectral triple satisfying Hypothesis \ref{main assumption}. Suppose $D$ has a spectral gap at $0$ and $(\mathcal{A},H,D)$ is odd but $p$ is even.
        \begin{enumerate}[{\rm (i)}]
            \item\label{second opposite 1} for every Hochschild cycle $c\in\mathcal{A}^{\otimes (p+1)},$ we have
                \begin{equation*}
                    \mathrm{Tr}(\Omega(c)|D|^{2-p}e^{-s^2D^2}) = O(s^{-1}),\quad s\downarrow0.
                \end{equation*}
            \item\label{second opposite 2} for every $c\in\mathcal{A}^{\otimes (p+1)},$ we have
                \begin{equation*}
                    \mathrm{Tr}(\mathrm{ch}(c))=0.
                \end{equation*}
        \end{enumerate}
    \end{lem}    
    \begin{proof} 
        First, we prove \eqref{second opposite 1}. Consider the multilinear mapping $\theta_s:\mathcal{A}^{\otimes p}\to\mathbb{C}$ defined by the formula
        \begin{equation*}
            \theta_s(a_0\otimes\cdots\otimes a_{p-1}) = \mathrm{Tr}(\left(\prod_{k=0}^{p-1}\partial(a_k)\right) |D|^{2-p}e^{-s^2D^2}).
        \end{equation*}
        The Hochschild coboundary $b\theta_s$ is computed in Section \ref{coboundary app} by the formula:
        \begin{align*}
            (b\theta_s)(a_0\otimes\cdots\otimes a_p) &= 2\mathrm{Tr}(a_0\left(\prod_{k=1}^p\partial(a_k)\right) |D|^{2-p}e^{-s^2D^2})\\
                                                     &\quad +\mathrm{Tr}(a_0\left(\prod_{k=1}^{p-1}\partial(a_k)\right)[|D|^{2-p}e^{-s^2D^2},\partial(a_p)])\\
                                                     &\quad +\mathrm{Tr}(\left(\prod_{k=0}^{p-1}\partial(a_k)\right) [|D|^{2-p}e^{-s^2D^2},a_p]).
        \end{align*}

        Using the H\"older inequality, we have
        \begin{align*}
            \Big|\mathrm{Tr}(a_0\left(\prod_{k=1}^{p-1}\partial(a_k)\right)&[|D|^{2-p}e^{-s^2D^2},\partial(a_p)])\Big|\\
                                                                   &\leq\|a_0\|_{\infty}\prod_{k=1}^{p-1}\|\partial(a_k)\|_{\infty}\cdot\Big\|[|D|^{2-p}e^{-s^2D^2},\partial a_p])\Big\|_1.
        \end{align*}
        By Lemma \ref{commutator 7}, we have
        \begin{equation}\label{sec op eq1}
            \mathrm{Tr}(a_0\left(\prod_{k=1}^{p-1}\partial(a_k)\right) [|D|^{2-p}e^{-s^2D^2},\partial(a_p)]) = O(s^{-1}),\quad s\downarrow0.
        \end{equation}
        Similarly, we have
        \begin{equation*}
            \Big|\mathrm{Tr}(\left(\prod_{k=0}^{p-1}\partial(a_k)\right) [|D|^{2-p}e^{-s^2D^2},a_p])\Big| \leq \prod_{k=0}^{p-1}\|\partial(a_k)\|_{\infty}\cdot\Big\|[|D|^{2-p}e^{-s^2D^2},a_p])\Big\|_1.
        \end{equation*}
        By Lemma \ref{commutator 7}, we have
        \begin{equation}\label{sec op eq2}
            \mathrm{Tr}(\left(\prod_{k=0}^{p-1}\partial(a_k)\right) [|D|^{2-p}e^{-s^2D^2},a_p])=O(s^{-1}),\quad s\downarrow0.
        \end{equation}

        For every $c\in\mathcal{A}^{\otimes (p+1)},$ it follows from \eqref{sec op eq1} and \eqref{sec op eq2} that
        \begin{equation*}
            (b\theta_s)(c) = 2\mathrm{Tr}(\Omega(c)|D|^{2-p}e^{-s^2D^2}) + O(s^{-1}),\quad s\downarrow0.
        \end{equation*}
        If $c$ is a Hochschild cycle, then $(b\theta_s)(c)=0.$ Thus,
        \begin{equation*}
            \mathrm{Tr}(\Omega(c)|D|^{2-p}e^{-s^2D^2}) = O(s^{-1}),\quad s\downarrow0.
        \end{equation*}
        for every Hochschild cycle $c.$ 
        This completes the proof of \eqref{second opposite 1}.

        The proof of \eqref{second opposite 2} is similar to Lemma \ref{first opposite lemma}.\eqref{first opposite 2}. 
        For all $a \in \mathcal{A}$, we have $F[F,a]=-[F,a]F.$ Since $p+1$ is odd,
        \begin{equation*}
            F\cdot \prod_{k=0}^p[F,a_k]=-\left(\prod_{k=0}^p[F,a_k]\right) F.
        \end{equation*}
        Hence,
        \begin{equation*}
            \mathrm{Tr}(F\prod_{k=0}^p[F,a_k])=-\mathrm{Tr}(F\prod_{k=0}^p[F,a_k]).
        \end{equation*}
        This proves \eqref{second opposite 2}.
    \end{proof}

    The preceding two lemmas show how to remove the assumption that the parities of $p$ and $(\mathcal{A},H,D)$ match. It remains to remove the assumption that $D$ has a spectral gap at $0.$ For this purpose, we use the doubling trick.
    The "doubling trick" in this form follows \cite[Definition 6]{CPRS1}.
    
    Let $\mu >0.$ We define another spectral triple $(\pi(\mathcal{A}),H_0,D_{\mu}),$ where
    \begin{align*}
          H_0 &= \mathbb{C}^2\otimes H,\\
        D_\mu &= \begin{pmatrix} D & \mu \\ \mu & -D\end{pmatrix}.
    \end{align*}
    and $\pi$ is the same representation of $\mathcal{A}$ as in Definition \ref{doubling definition}. That is:
    \begin{equation*}
        \pi(a) := \begin{pmatrix} a & 0 \\ 0 & 0\end{pmatrix}.
    \end{equation*}
    For a tensor $c \in \mathcal{A}^{\otimes (p+1)}$, we denote $\pi(c)$ for the corresponding element of $(\pi(\mathcal{A}))^{\otimes (p+1)}$ obtained by applying the map $\pi^{\otimes (p+1)}$ to $c$.
    The spectral triple $(\pi(\mathcal{A}),H_0,D_{\mu})$ is equipped with grading operator
    $$\Gamma_0=\begin{pmatrix} \Gamma & 0 \\ 0 & (-1)^{\rm deg}\Gamma\end{pmatrix}$$
    where $\Gamma$ is the (possibly trivial) grading of $(\mathcal{A},H,D)$ (see Definition \ref{doubling definition}).
        
    Let $\Omega_{\mu}$ and ${\rm ch}_{\mu}$ be the multilinear mappings $\Omega$ and ${\rm ch}$ (just as in Definition \ref{ch omega def}) as applied to the spectral triple $(\pi(\mathcal{A}),H_0,D_{\mu}).$
    
    \begin{lem}\label{first doubling lemma} 
        Let $(\mathcal{A},H,D)$ be a spectral triple satisfying Hypothesis \ref{main assumption}. Let $F_0$ be as in Definition \ref{doubling definition}. If $a\in\mathcal{A},$ then
        as $\mu\downarrow 0$ we have:
        $$[{\rm sgn}(D_{\mu}),\pi(a)]\to[F_0,\pi(a)]$$
        in $\mathcal{L}_{p+1}$.
    \end{lem}
    \begin{proof} 
        We have
        \begin{equation*}
            \mathrm{sgn}(D_{\mu}) = \begin{pmatrix} 
                                            \frac{D}{(D^2+\mu^2)^{1/2}}   & \frac{\mu}{(D^2+\mu^2)^{1/2}} \\ 
                                            \frac{\mu}{(D^2+\mu^2)^{1/2}} & -\frac{D}{(D^2+\mu^2)^{1/2}}
                                    \end{pmatrix}.
        \end{equation*}
    Hence,
    \begin{equation*}
        [\mathrm{sgn}(D_\mu),\pi(a)]  = \begin{pmatrix}
                                            \left[\frac{D}{(D^2+\mu^2)^{1/2}},a\right] & -a\frac{\mu}{(D^2+\mu^2)^{1/2}}\\ \frac{\mu}{(D^2+\mu^2)^{1/2}}a & 0.
                                        \end{pmatrix}
    \end{equation*}
    On the other hand, we have
    \begin{equation*}
        [F_0,\pi(a)] = \begin{pmatrix} 
                          [\mathrm{sgn}(D),a] & -aP\\ Pa & 0
                       \end{pmatrix}.
    \end{equation*}
    Therefore:
    \begin{align*}
        \Big\|[{\rm sgn}(D_{\mu}),\pi(a)&]-[F_0,\pi(a)]\Big\|_{p+1}\\
                                        &\leq \Big\|\Big({\rm sgn}(D)-\frac{D}{(D^2+\mu^2)^{\frac12}}\Big)a\Big\|_{p+1}+\Big\|a\Big({\rm sgn}(D)-\frac{D}{(D^2+\mu^2)^{\frac12}}\Big)\Big\|_{p+1}\\
                                        &\quad + \Big\|\Big(P-\frac{\mu}{(D^2+\mu^2)^{\frac12}}\Big)a\Big\|_{p+1}+\Big\|a\Big(P-\frac{\mu}{(D^2+\mu^2)^{\frac12}}\Big)\Big\|_{p+1}\\
                                        &= \Big\|a^*\Big({\rm sgn}(D)-\frac{D}{(D^2+\mu^2)^{\frac12}}\Big)^2a\Big\|_{\frac{p+1}{2}}^{\frac12}+\Big\|a\Big({\rm sgn}(D)-\frac{D}{(D^2+\mu^2)^{\frac12}}\Big)^2a^*\Big\|_{\frac{p+1}{2}}^{\frac12}\\
                                        &\quad + \Big\|a^*\Big(P-\frac{\mu}{(D^2+\mu^2)^{\frac12}}\Big)^2a\Big\|_{\frac{p+1}{2}}^{\frac12}+\Big\|a\Big(P-\frac{\mu}{(D^2+\mu^2)^{\frac12}}\Big)^2a^*\Big\|_{\frac{p+1}{2}}^{\frac12}.
    \end{align*}
    By functional calculus, as $\mu\downarrow  0$ we have:
    \begin{align*}
        \Big({\rm sgn}(D)-\frac{D}{(D^2+\mu^2)^{\frac12}}\Big)^2 &\downarrow 0,\\
                 \Big(P-\frac{\mu}{(D^2+\mu^2)^{\frac12}}\Big)^2 &\downarrow 0
    \end{align*}
    in the weak operator topology.
    Hence,
    $$a^*\Big({\rm sgn}(D)-\frac{D}{(D^2+\mu^2)^{\frac12}}\Big)^2a\downarrow 0,\quad a\Big({\rm sgn}(D)-\frac{D}{(D^2+\mu^2)^{\frac12}}\Big)^2a^*\downarrow 0,\quad\mu\downarrow0,$$
    $$a^*\Big(P-\frac{\mu}{(D^2+\mu^2)^{\frac12}}\Big)^2a\downarrow 0,\quad a\Big(P-\frac{\mu}{(D^2+\mu^2)^{\frac12}}\Big)^2a^*\downarrow0,\quad\mu\downarrow0.$$
    also in the weak operator topology.
    The assertion follows now from the order continuity of the $\mathcal{L}_{\frac{p+1}{2}}$ norm.
    \end{proof}

    \begin{lem}\label{second doubling lemma} 
        Let $(\mathcal{A},H,D)$ be a spectral triple satisfying Hypothesis \ref{main assumption}. If $c\in\mathcal{A}^{\otimes (p+1)},$ then
        \begin{equation*}
            \lim_{\mu\to 0} {\rm ch}_{\mu}(\pi(c)) = {\rm ch}_0(c)
        \end{equation*}
        in the $\mathcal{L}_1$-norm.
    \end{lem}
    \begin{proof} 
        It suffices to prove the assertion for $c=a_0\otimes\cdots\otimes a_p.$ Then we have
        \begin{align*}
                 {\rm ch}_{0}(c)-{\rm ch}_\mu(\pi(c)) &= \Gamma_0(F_0-{\rm sgn}(D_{\mu}))\prod_{k=0}^p[F_0,\pi(a_k)]\\
                                                           & - \Gamma_0\mathrm{sgn}(D_\mu)\Big(\prod_{k=0}^p[{\rm sgn}(D_{\mu}),\pi(a_k)]-\prod_{k=0}^p[F_0,\pi(a_k)]\Big).
        \end{align*}
        
        Next we use the fact that if $q \geq 1$ and if $A_\mu\to A$ in the strong operator topology, $B_\mu\to B\in \mathcal{L}_q$ in the $\mathcal{L}_q$ norm, and $\sup_{\mu\downarrow 0}\|A_\mu\|_\infty < \infty$ then $A_\mu B_\mu\to AB$
        in $\mathcal{L}_1$ (see \cite[Chapter 2, Example 3]{Simon}).
        
        We have that $\mathrm{sgn}(D_\mu)- F_0\to 0$ in the strong operator topology, and since $\prod_{k=0}^p [F_0,\pi(a_k)] \in \mathcal{L}_1$, it follows that the first
        summand converges to $0$ in the $\mathcal{L}_1$ norm.
                
        For the second summand, we apply Lemma \ref{first doubling lemma}. Since for each $k$ we have that $[\mathrm{sgn}(D_\mu),\pi(a_k)] \to [F_0,\pi(a_k)]$
        in $\mathcal{L}_{p+1}$, it follows that the second summand converges to $0$ in $\mathcal{L}_1$.
    \end{proof}
    
    \begin{lem}\label{third doubling lemma} 
        Let $(\mathcal{A},H,D)$ be a spectral triple satisfying Hypothesis \ref{main assumption}. 
        For every $c\in\mathcal{A}^{\otimes (p+1)},$ we have
        $$\Big\|\Omega_{\mu}(\pi(c))(1\otimes (1+D^2)^{-\frac{p}{2}}e^{-s^2D^2})\Big\|_1=O(s^{-1}),\quad s\downarrow0.$$
    \end{lem}
    \begin{proof} 
        Once more it suffices to prove the assertion for an elementary tensor $c=a_0\otimes\cdots\otimes a_p.$ We have
        \begin{align*}
              \Big\|\Omega_{\mu}&(\pi(c))(1\otimes (1+D^2)^{-\frac{p}{2}}e^{-s^2D^2})\Big\|_1\\
                                &\leq \Big\|\Gamma_0 \pi(a_0)\prod_{k=1}^{p-1}[D_{\mu},\pi(a_k)]\Big\|_{\infty}\Big\|[D_{\mu},\pi(a_p)](1\otimes (1+D^2)^{-\frac{p}{2}}e^{-s^2D^2})\Big\|_1\\
                                &\leq \|a_0\|_{\infty}\prod_{k=1}^{p-1}\|[D_{\mu},\pi(a_k)]\|_{\infty}\Big\|[D,a_p](1+D^2)^{-\frac{p}{2}}e^{-s^2D^2}\Big\|_1\\
                                &\quad +2\mu\|a_0\|_{\infty}\prod_{k=1}^{p-1}\|[D_{\mu},\pi(a_k)]\|_{\infty}\Big\|a_p(1+D^2)^{-\frac{p}{2}}e^{-s^2D^2}\Big\|_1.
        \end{align*}
        The assertion follows by applying Lemma \ref{first decay lemma} (with $m_1=0$ and $m_2=p-1$) to the odd spectral triple $(\mathcal{A},H,F(1+D^2)^{\frac12}).$    
    \end{proof}

    We note that since $\pi$ is an algebra homomorphism, if $c \in \mathcal{A}^{\otimes (p+1)}$ is a Hochschild cycle then so is $\pi(c)$.
    \begin{lem}\label{fourth doubling lemma} 
        Let $(\mathcal{A},H,D)$ be a spectral triple satisfying Hypothesis \ref{main assumption}. Let $c\in\mathcal{A}^{\otimes (p+1)}$ be a Hochschild cycle. We have
        $$(\mathrm{Tr}_2\otimes\mathrm{Tr})(\Omega_{\mu}(\pi(c))(1\otimes (1+D^2)^{1-\frac{p}{2}}e^{-s^2D^2}))= \frac{p}{2}\mathrm{Ch}_{\mu}(\pi(c))s^{-2} + O(s^{-1}),\quad s\downarrow 0.$$
    \end{lem}
    \begin{proof}     
        For $\mu>0,$ the spectral triple $(\pi(\mathcal{A}),H_0,D_{\mu})$ satisfies Hypothesis \ref{main assumption} and the spectral gap assumption. This allows us to apply Theorems \ref{reduction} and \ref{first cycle thm} (or Lemmas \ref{first opposite lemma} and \ref{second opposite lemma}) to the Hochschild cycle $\pi(c)\in (\pi(\mathcal{A}))^{\otimes (p+1)}.$
    
        A combination of Theorems \ref{reduction} and \ref{first cycle thm} (if the parities of $p$ and $(\pi(\mathcal{A}),H_0,D_{\mu})$ match) or one of Lemmas \ref{first opposite lemma} and \ref{second opposite lemma} (if parities of $p$ and 
        $(\pi(A),H_0,D_{\mu})$ do not match) yields
        $$(\mathrm{Tr}_2\otimes\mathrm{Tr})(\Omega_{\mu}(\pi(c))|D_{\mu}|^{2-p}e^{-s^2D_{\mu}^2}) = \frac{p}{2}\mathrm{Ch}_{\mu}(\pi(c))s^{-2} + O(s^{-1}),\quad s\downarrow 0.$$
        Noting that $D_{\mu}^2=D_0^2+\mu^2,$ and $e^{-s^2\mu^2} = O(1)$ we obtain
        $$(\mathrm{Tr}_2\otimes\mathrm{Tr})(\Omega_{\mu}(\pi(c))|D_{\mu}|^{2-p}e^{-s^2D_0^2}) = \frac{p}{2}\mathrm{Ch}_{\mu}(\pi(c))s^{-2} + O(s^{-1}),\quad s\downarrow 0.$$
        By Lemma \ref{third doubling lemma}, we have
        \begin{align*}
            \Big\|\Omega_{\mu}(\pi(c))&|D_{\mu}|^{2-p}e^{-s^2D_0^2}-\Omega_{\mu}(\pi(c))|D_1|^{2-p}e^{-s^2D_0^2}\Big\|_1\\
                                      &\leq\Big\|\Omega_{\mu}(\pi(c))|D_1|^{-p}e^{-s^2D_0^2}\Big\|_1\Big\|(|D_{\mu}|^{2-p}-|D_1|^{2-p})\cdot |D_1|^{p-2}\Big\|_{\infty}\\
                                      &=O(s^{-1}).
        \end{align*}
    \end{proof}

    We are now ready to prove the main result of this chapter. 

    \begin{proof}[Proof of Theorem \ref{heat thm}] 
        For $\mathscr{A}\subseteq\{1,\cdots,p\},$ we define the multilinear functional on $a_0\otimes\cdots a_p\in \mathcal{A}^{\otimes(p+1)}$ by:
        $$\mathcal{T}_{\mathscr{A}}(a_0\otimes\cdots\otimes a_p)=\Gamma_0\pi(a_0)\prod_{k=1}^p y_k(a_k).$$
        Here,
        \begin{equation*}
            y_k(a) := \begin{cases}
                        \begin{pmatrix} \partial(a) & 0 \\ 0 & 0\end{pmatrix}, k\notin\mathscr{A}\\
                        \begin{pmatrix} 0 & -a \\ a& 0\end{pmatrix}         , k\in\mathscr{A}.
                      \end{cases}
        \end{equation*}
        In particular,
        \begin{equation*}
            \mathcal{T}_{\emptyset}(a_0\otimes\cdots a_p) = \begin{pmatrix} 
                                                            \Omega(a_0\otimes\cdots\otimes a_p) & 0\\
                                                                0 & 0
                                                            \end{pmatrix}.
        \end{equation*}       
        For $c\in\mathcal{A}^{\otimes(p+1)},$ we apply \eqref{combinatorial fact} to get
        $$\Omega_{\mu}(\pi(c)) = \sum_{\mathscr{A}\subseteq\{1,\cdots,p\}}\mu^{|\mathscr{A}|}\mathcal{T}_{\mathscr{A}}(c).$$
        For each $0\leq k\leq p,$ we set
        $$f_k(s)=\sum_{|\mathscr{A}|=k}(\mathrm{Tr}_2\otimes\mathrm{Tr})(\mathcal{T}_{\mathscr{A}}(c)\cdot (1\otimes (1+D^2)^{1-\frac{p}{2}}e^{-s^2D^2})).$$
        That is, $f_k(s)$ is the coefficient of $\mu^k$ in $(\mathrm{Tr}_2\otimes \mathrm{\mathrm{Tr}})(\Omega_\mu(\pi(c))(1\otimes (1+D^2)^{1-\frac{p}{2}}e^{-s^2D^2}))$.
        
        Now if $c\in\mathcal{A}^{\otimes (p+1)}$ is a Hochschild cycle, then Lemma \ref{fourth doubling lemma} yields
        \begin{equation}\label{heat main epsilon}
            \sum_{k=0}^p\mu^kf_k(s) = \frac{p}{2}\mathrm{Ch}_{\mu}(\pi(c))s^{-2} + O(s^{-1}),\quad s\downarrow 0.
        \end{equation}
        
        Select a set $\{\mu_0,\ldots,\mu_p\}$ of distinct positive numbers, and for each $0 \leq l \leq p$ we may take $\mu = \mu_l$ in \eqref{heat main epsilon}
        to arrive at:
        $$\sum_{k=0}^p\mu_l^kf_k(s) = \frac{p}{2}\mathrm{Ch}_{\mu_l}(\pi(c))s^{-2} + O(s^{-1}),\quad s\downarrow 0,\quad 0\leq l\leq p.$$
        Since the Vandermonde matrix $\{\mu_l^k\}_{0\leq l,k\leq p}$ is invertible, it follows that there exist $\{\alpha_0,\ldots, \alpha_k\}$
        such that:
        $$f_k(s)=\frac{\alpha_k}{s^2}+O(s^{-1}),\quad s\downarrow0,\quad 0\leq k\leq p.$$
        Substituting this back to \eqref{heat main epsilon}, we obtain
        $$\sum_{k=0}^p\mu^k\alpha_k=\frac{p}{2}\mathrm{Ch}_{\mu}(\pi(c)).$$
        In particular
        $$\alpha_0=\frac{p}{2}\lim_{\mu\to0}\mathrm{Ch}_{\mu}(\pi(c))$$
        So by Lemma \ref{second doubling lemma},
        \begin{equation*}
            \alpha_0 = \frac{p}{2}{\rm Ch}(c).
        \end{equation*}
        Hence,
        $$f_0(s)=\frac{p}{2}\mathrm{Ch}(c)s^{-2} + O(s^{-1}),\quad s\downarrow 0.$$
        The assertion follows now from the definition of $f_0.$
    \end{proof}

\chapter{Residue of the $\zeta-$function and the Connes character formula}\label{zeta chapter}
    In this chapter we complete the proofs of Theorem \ref{zeta thm} and Theorem \ref{main thm}. 
    
    For a spectral triple $(\mathcal{A},H,D)$ satisfying Hypothesis \ref{main assumption}, we define the zeta function of a Hochschild cycle $c \in \mathcal{A}^{\otimes (p+1)}$ by the formula
    \begin{equation*}
        \zeta_{c,D}(z) := \mathrm{Tr}(\Omega(c)(1+D^2)^{-z/2}),\quad \Re(z) > p+1.
    \end{equation*}
    Indeed, by Hypothesis \ref{main assumption}.\eqref{ass2} if $\Re(z) \geq p+1$ then the operator $\Omega(c)(1+D^2)^{-z/2}$ is trace class, and so $\zeta_{c,D}$
    is well defined when $\Re(z) > p+1$. In Section \ref{zeta section} we prove that $\zeta_{c,D}$ is holomorphic and has analytic continuation to the set
    \begin{equation*}
        \{z \in \mathbb{C}\;:\; \Re(z) > p-1\}\setminus \{p\}.
    \end{equation*}
    We also show that the point $p$ is a simple pole for $\zeta_{c,D}$, and that the corresponding residue of $\zeta_{c,D}$ at $p$ is equal to $p\mathrm{Ch}(c)$, thus completing the proof
    of Theorem \ref{zeta thm}.
    
    Then we undertake the more difficult task of proving Theorem \ref{main thm}. We achieve this by a new characterisation
    of universal measurability in Section \ref{subhankulov section}, which allows us to deduce Theorem \ref{main thm} as a corollary of Theorem \ref{zeta thm}.
    
    The most novel feature of this chapter, and of this manuscript as a whole, is a certain integral representation
    of the difference $B^zA^z-(A^{\frac{1}{2}}BA^{\frac{1}{2}})^z$ for two positive bounded operators $A$ and $B$
    and $z \in \mathbb{C}$ with $\Re(z) > 0$ (see Lemma \ref{csz key lemma}). This result first appeared in \cite[Lemma 5.2]{CSZ} {for the special case where $z$ is real and positive. }
    With this integral representation we are able to prove the analyticity
    of the function
    \begin{equation*}
        z \mapsto \mathrm{Tr}(XB^zA^z)-\mathrm{Tr}(X(A^{\frac{1}{2}}BA^{\frac{1}{2}})^z)
    \end{equation*}
    for a bounded operator $X$ and for $z$ in a certain domain in the complex plane, under certain assumptions on $A$ and $B$. This result is
    stated in full in Section \ref{difference section}.
    
    In sections \ref{main thm p>2} and \ref{main thm p=1,2} we complete the proof of Theorem \ref{main thm}.

\section{Analyticity of the $\zeta$-function for $\Re(z)>p-1$, $z \neq p$}\label{zeta section}
    This section contains the proof of Theorem \ref{zeta thm}. The proof is relatively short, since we are able to use Theorem \ref{heat thm}.
    
    \begin{lem}\label{simple lemma} 
        Let $h\in L_{\infty}(0,1)$ and $u \in L_\infty(1,\infty)$. Then,
        \begin{enumerate}[{\rm (i)}]
            \item\label{simple lemma 1} the function
                $$F(z) := \int_0^1s^{z-1}h(s)ds,\quad\Re(z)>0,$$
            is analytic.
            \item\label{simple lemma 2} the function
                $$G(z) := \int_1^{\infty}s^{z-1}u(s)e^{-s}ds,\quad z\in\mathbb{C},$$
            is analytic.
        \end{enumerate}
    \end{lem}
    \begin{proof} 
        Let us prove \eqref{simple lemma 1}. Define
        $$F_n(z)=\int_{\frac1n}^1s^{z-1}h(s)ds,\quad\Re(z)>0.$$
        Then for $\Re(z) > 0$ we have:
        \begin{align*}
            |F(z)-F_n(z)| &= \left|\int_0^{\frac1n}s^{z-1}h(s)ds\right|\\
                          &\leq \int_0^{\frac1n}s^{\Re(z)-1}|h(s)|ds\\
                          &\leq \|h\|_{\infty}\int_0^{\frac1n}s^{\Re(z)-1}ds\\
                          &= \frac{\|h\|_{\infty}}{\Re(z)} n^{-\Re(z)}.
        \end{align*}
        So for every $\varepsilon > 0$, we have that $F_n$ converges uniformly to $F$ on the set $\{z\;:\; \Re(z) > \varepsilon\}$.

        We now show that for each $n$ the function $F_n$ is entire. 
        Indeed, we have the power series expansion
        \begin{align*}
            s^{z-1} &= e^{(z-1)\log(s)}\\
                    &= \sum_{k\geq0}\frac{(\log(s))^k}{k!}(z-1)^k.
        \end{align*}
        which converges uniformly on compact subsets of $\mathbb{C}$.
        Therefore, interchanging the integral and summation, we have that for all $z \in \mathbb{C}$
        \begin{equation}\label{fn taylor expansion}
            F_n(z) = \sum_{k\geq0}\frac1{k!}\left(\int_{\frac1n}^1(\log(s))^kh(s)ds\right)(z-1)^k,\quad z\in\mathbb{C}.
        \end{equation}
        This power series has infinite radius of convergence, since
        \begin{equation*}
            \left|\int_{\frac1n}^1(\log(s))^kh(s)ds\right| \leq \|h\|_{\infty}(\log(n))^k.
        \end{equation*}
        So each $F_n$ is entire.        

        In summary, the sequence $\{F_n\}_{n\geq1}$ of entire functions converges to $F$ uniformly on the half-plane $\{z\;:\; \Re(z) > \varepsilon\}$. Since $\varepsilon$ is arbitrary, the sequence
        $\{F_n\}_{n\geq 0}$ converges uniformly to $F$ on compact subsets of the half plane $\{z \;:\;\Re(z) > 0\}$. Thus, $F$ is holomorphic on this half-plane.

        To prove \eqref{simple lemma 2}, we consider the functions
        \begin{equation*}
            G_n(z) := \int_1^ns^{z-1}u(s)e^{-s}ds,\quad z\in\mathbb{C}
        \end{equation*}
        
        Exactly the same argument as above shows that each $G_n$ is entire. 
        For all $n\geq 1$, we have:
        \begin{align*}
            |G(z)-G_n(z)| &\leq \int_n^\infty s^{\Re(z)-1}|u(s)|e^{-s}\,ds\\
                          &\leq \|u\|_{\infty} \int_n^\infty s^{\Re(z)-1}e^{-s}\,ds
        \end{align*}
        So for any $N > 0$, we have that $G_n$ converges uniformly to $G$ in the set $\{z \in \mathbb{C}\;:\;\Re(z) < N\}$, and therefore on compact subsets of the plane. Hence, $G$ is
        entire.
    \end{proof}
    
    We are now able to prove Theorem \ref{zeta thm}.
    \begin{thm*}
        Let $(\mathcal{A},H,D)$ be a spectral triple satisfying Hypothesis \ref{main assumption}, and let $c \in \mathcal{A}^{\otimes (p+1)}$ be a Hochschild cycle. Then
        the function
        \begin{equation*}
            \zeta_{c,D}(z) := \mathrm{Tr}(\Omega(c)(1+D^2)^{-z/2}),\quad \Re(z) > p+1
        \end{equation*}
        is analytic, and has analytic continuation to the set $\{z \in \mathbb{C}\;:\;\Re(z) > p-1\}\setminus \{p\}$ and a simple pole at $p$ with residue $p\mathrm{Ch}(c)$.
    \end{thm*}
    \begin{proof}[Proof of Theorem \ref{zeta thm}] 
        Let $z \in \mathbb{C}$ with $\Re(z) > 1$. Then for all $x>0,$ we have
        \begin{align*}
            \int_0^{\infty}s^{z-1}e^{-s^2x^2}\,ds &= x^{-z}\int_0^\infty t^{z-1}e^{-t^2}\,dt\\
                                                  &= x^{-z}\int_0^\infty u^{\frac{z-1}{2}}e^{-u}\frac{u^{-\frac{1}{2}}}{2}\,du\\
                                                  &= \frac{x^{-z}}{2}\Gamma\left(\frac{z}{2}\right).
        \end{align*}
        Thus,
        \begin{equation*}
            x^{2-z} = \frac{2}{\Gamma\left(\frac{z}{2}\right)}\int_0^\infty s^{z-1}x^2e^{-x^2s^2}\,ds
        \end{equation*}
        So by the functional calculus, for $\Re(z) > 2$ we have an integral in the weak operator topology:
        \begin{equation*}
            (1+D^2)^{1-\frac{z}{2}} = \frac{2}{\Gamma\left(\frac{z}{2}\right)}\int_0^{\infty}s^{z-1}(1+D^2)e^{-s^2(1+D^2)}\,ds.
        \end{equation*}        
        
        We now multiply on the left by the bounded operator $\Omega(c)(1+D^2)^{-p/2}$ to arrive at:
        \begin{equation*}
            \Omega(c)(1+D^2)^{1-\frac{z+p}{2}} = \frac{2}{\Gamma\left(\frac{z}{2}\right)}\int_0^\infty s^{z-1}\Omega(c)(1+D^2)^{1-\frac{p}{2}}e^{-s^2(1+D^2)}\,ds.
        \end{equation*}
              
        We claim that this integral converges in $\mathcal{L}_1$. First, if $p=1$ then consider the function $h_1(t) = te^{-t^2}$. Applying Lemma \ref{schwartz lemma} with the function $h_1$,
        we have:
        \begin{align*}
            \|s^{\Re(z)-1}\Omega(c)(1+D^2)^{1-\frac{p}{2}}e^{-s^2(1+D^2)}\|_1 &= s^{\Re(z)-2}\|\Omega(c)h_1(s(1+D^2)^{1/2})\|_1\\
                                                                              &= O(s^{\Re(z)-3}),\quad s\downarrow 0.
        \end{align*}
        On the other hand, if $p > 1$ then define $h_p(t) = (1+t^2)^{1-p/2}e^{-t^2}.$ Now applying Lemma \ref{schwartz lemma} with the function $h_p$:
        \begin{align*}
            \|s^{z-1}\Omega(c)(1+D^2)^{1-\frac{p}{2}}e^{-s^2(1+D^2)}\|_1 &\leq \left\|\left(\frac{1+D^2}{1+s^2D^2}\right)^{1-\frac{p}{2}}\right\|_{\infty}\|s^{z-1}\Omega(c)h_p(s|D|)\|_1\\
                                                                         &= s^{\Re(z)+p-3}\|\Omega(c)h_p(s|D|)\|_1\\
                                                                         &= O(s^{\Re(z)-3}),\quad s\downarrow 0.
        \end{align*}
        So in both cases, since $\Re(z) > 2$, the function $s\mapsto \|s^{z-1}\Omega(c)(1+D^2)^{1-\frac{p}{2}}e^{-s^2(1+D^2)}\|_1$ is integrable on $[0,1].$
        
        For $s>1,$ we have
        \begin{equation*}
            e^{-s^2D^2}\leq e^{-D^2}\leq(1+D^2)^{-3/2}.
        \end{equation*}
        so we have that
        \begin{equation*}
            \|s^{z-1}\Omega(c)(1+D^2)^{1-\frac{p}{2}}e^{-s^2(1+D^2)}\|_{1} \leq s^{\Re(z)-1}\|\Omega(c)(1+D^2)^{-\frac{p+1}{2}}\|_1.
        \end{equation*}
        Hence, $s\mapsto \|s^{z-1}\Omega(c)(1+D^2)^{1-\frac{p}{2}}e^{-s^2(1+D^2)}\|_1$ is integrable on the interval $(1,\infty)$.

        By Lemma \ref{peter lemma}, for $\Re(z) > 2$ we therefore have:
        \begin{align*}
            \|\Omega(c)(1+D^2)^{1-\frac{z+p}{2}}\|_1 &\leq \frac{2}{\Gamma\left(\frac{z}{2}\right)}\int_0^\infty \|s^{z-1}\Omega(c)(1+D^2)^{1-\frac{p}{2}}e^{-s^2(1+D^2)}\|_1\,ds\\
                                                     &< \infty
        \end{align*}
        and
        \begin{equation}\label{heat mellin transform}
            \mathrm{Tr}(\Omega(c)(1+D^2)^{1-\frac{z+p}{2}}) = \frac{2}{\Gamma\left(\frac{z}{2}\right)}\int_0^\infty s^{z-1}\mathrm{Tr}(\Omega(c)(1+D^2)^{1-\frac{p}{2}}e^{-s^2(1+D^2)})\,ds.
        \end{equation}        
        We will now apply the result of Theorem \ref{heat thm} to the integrand. First we define a function $h$ on $(0,\infty)$ by:
        \begin{equation*}
            h(s) := \begin{cases}
                        e^s\mathrm{Tr}(\Omega(c)(1+D^2)^{1-\frac{p}{2}}e^{-s^2(1+D^2)}), \quad s \geq 1\\
                        s\mathrm{Tr}(\Omega(c)(1+D^2)^{1-\frac{p}{2}}e^{-s^2(1+D^2)})-\frac{p}{2}\mathrm{Ch}(c)s^{-1}, \quad 0 < s < 1.
                    \end{cases}
        \end{equation*}
        By Theorem \ref{heat thm}, the function $h$ is bounded on the interval $(0,1).$ For $s > 1$, we have a constant $c$ such that.
        \begin{equation*}
            |h(s)| \leq Ce^{s-s^2}
        \end{equation*}
        Hence, $h$ is bounded on $[0,\infty)$. Substituting $h$ in \eqref{heat mellin transform}:
        \begin{align*}
            \mathrm{Tr}(\Omega(c)(1+D^2)^{1-\frac{z+p}{2}}) &= \frac{2}{\Gamma\left(\frac{z}{2}\right)}\int_0^1 s^{z-1}(s^{-1}h(s)+\frac{p}{2}\mathrm{Ch}(c)s^{-2})\,ds\\
                                                    &\quad +\frac{2}{\Gamma\left(\frac{z}{2}\right)}\int_1^\infty s^{z-1}e^{-s}h(s)\,ds.
        \end{align*}
        By Lemma \ref{simple lemma}.\eqref{simple lemma 2}, the second term in the above sum has extension to an entire function, and so we focus on the first term.
        We have,
        \begin{align*}
            \frac{2}{\Gamma\left(\frac{z}{2}\right)}\int_0^1 s^{z-1}(s^{-1}h(s)+\frac{p}{2}\mathrm{Ch}(c)s^{-2})\,ds &= \frac{2}{\Gamma\left(\frac{z}{2}\right)} \int_0^1 s^{z-2}h(s)\,ds\\
                                                                                                             &\quad + \frac{p}{\Gamma\left(\frac{z}{2}\right)}\mathrm{Ch}(c)\int_0^1 s^{z-3}\,ds.
        \end{align*}
        Due to Lemma \ref{simple lemma}.\eqref{simple lemma 1}, the first term in the above sum has extension to an analytic function for $\Re(z-1) > 0$. That is, for $\Re(z) > 1$. 
        
        As for the second term, since we are still working with $\Re(z) > 2$ we may compute:
        \begin{equation*}
            \int_0^1 s^{z-3}\,ds = (z-2)^{-1}.
        \end{equation*}
        So in summary, the function initially defined for $\Re(z) > 2$ given by:
        \begin{equation*}
            z\mapsto \mathrm{Tr}(\Omega(c)(1+D^2)^{1-\frac{p+z}{2}}) - \frac{p}{\Gamma\left(\frac{z}{2}\right)}(z-2)^{-1}
        \end{equation*}
        is analytic on the set $\Re(z) > 2$, and has analytic continuation to the set $\Re(z) > 1$. Since the function $\frac{1}{\Gamma\left(\frac{z}{2}\right)}$ is entire, and $\Gamma(1) = 1$, we may equivalently say that
        the function
        \begin{equation*}
            z\mapsto \zeta_{c,D}(z+p-2) - p\mathrm{Ch}(c)(z-2)^{-1},\quad \Re(z) > 2
        \end{equation*}
        has analytic continuation to the set $\Re(z) > 1$.
        In other words, for $\Re(z) > p$
        \begin{equation*}
            \zeta_{c,D}(z) - p\mathrm{Ch}(c)(z-p)^{-1}
        \end{equation*}
        has analytic continuation to the set $\Re(z) > p-1$. This is exactly the statement of the theorem.
    \end{proof}

\section[Integral representation]{An Integral Representation for $B^zA^z-(A^{\frac12}BA^{\frac12})^z$}\label{representation section}

    In this section, we follow the convention that for all $s \in \mathbb{R}$ we have $0^{is}=0$, so in particular we have the unusual convention that $0^{i0} = 0$. 
    This section is devoted to the proof of Theorem \ref{csz key lemma} (stated below). Theorem \ref{csz key lemma} is a strengthening of \cite[Lemma 5.2]{CSZ} (which corresponds to the special
    case where $z$ is real and $B$ is compact). Theorem \ref{csz key lemma} also substantially strengthens \cite[Proposition 4.4]{CGRS1}.
    
    Here we work with abstract operators on a separable Hilbert space $H$. Given a positive bounded operator $A$ on $H$,
    and a complex number $z$ with $\Re(z) > 0$, the operator $A^z$ may be defined by continuous functional calculus.
    \begin{thm}\label{csz key lemma} 
        Let $A$ and $B$ be bounded, positive operators on $H$, and let $z \in \mathbb{C}$ with $\Re(z) > 1$. Let $Y := A^{1/2}BA^{1/2}$.
        We define the mapping $T_z:\mathbb{R}\to \mathcal{L}_\infty$ by,
        \begin{align*}
            T_z(0) &:= B^{z-1}[BA^{\frac{1}{2}},A^{z-\frac{1}{2}}]+[BA^{\frac{1}{2}},A^{\frac{1}{2}}]Y^{z-1},\\
            T_z(s) &:= B^{z-1+is}[BA^{\frac{1}{2}},A^{z-\frac{1}{2}+is}]Y^{-is}+B^{is}[BA^{\frac{1}{2}},A^{\frac{1}{2}+is}]Y^{z-1-is},\quad s \neq 0.
        \end{align*}
        We also define the function $g_z:\mathbb{R}\to \mathbb{C}$ by:
        \begin{align*}
            g_z(0) &:= 1-\frac{z}{2},\\
            g_z(t) &:= 1-\frac{e^{\frac{z}{2}t}-e^{-\frac{z}{2}t}}{(e^{\frac{t}{2}}-e^{-\frac{t}{2}})(e^{\left(\frac{z-1}{2}\right)t}+e^{-\left(\frac{z-1}{2}\right)t})}.
        \end{align*} 
        Then:
        \begin{enumerate}[{\rm (i)}]
            \item{} The mapping $T_z:\mathbb{R}\to \mathcal{L}_\infty$ is continuous in the weak operator topology.
            \item{} We have:
                \begin{equation*}
                    B^zA^z-(A^{\frac{1}{2}}BA^{\frac{1}{2}})^z = T_z(0)-\int_{\mathbb{R}} T_z(s)\widehat{g}_z(s)\,ds. 
                \end{equation*}
        \end{enumerate}
    \end{thm}
    
    \begin{rem}
        For $\Re(z) > 1$, the function $g_z$ is Schwartz, and hence so is the (rescaled) Fourier transform $\widehat{g}_z$. 
    \end{rem}
    \begin{proof}
        For $t \neq 0$, we may rewrite $g_z(t)$ as:
        \begin{equation*}
            g_z(t) = \frac{1}{2}\left(\frac{\tanh((z-1)t/2)}{\tanh(t/2)}-1\right).
        \end{equation*}
        Letting $s = t/2$ and $w = z-1$, it then suffices to show that for $\Re(w) > 0$ the function
        \begin{equation*}
            f_w(s) = \frac{\tanh(ws)}{\tanh(s)}-1,\quad s\neq 0
        \end{equation*}
        with $f_w(0) = w-1$ is Schwartz. Since $\lim_{s\to 0} \frac{\tanh(s)}{s} = 1$, it is evident that $f_w$ is continuous at $0$
        and that $f_w$ is smooth in $[-1,1]$. 
        
        It suffices now to show that the function $\tanh(ws)-\tanh(s)$ is Schwartz, since for $|s| > 1$ the function $\frac{1}{\tanh(s)}$ is 
        smooth and bounded with all derivatives bounded. For $s > 1$, we can write,
        \begin{equation*}
            \tanh(ws)-\tanh(s) = \tanh(ws)-1+(1-\tanh(s))
        \end{equation*}
        and then note that since $\Re(w) > 0$, the functions $\tanh(ws)-1$ and $1-\tanh(s)$ have rapid decay as $s\to\infty$, with all derivatives to all orders also of rapid decay as $s\to\infty$. Similarly, for $s < -1$, we
        can write $\tanh(ws)-\tanh(s) = \tanh(ws)+1-(\tanh(s)+1)$ and then use the fact that $\tanh(s)+1$ and $\tanh(ws)+1$ have rapid decay, with all derivatives of rapid decay, as $s\to-\infty$. 
    \end{proof}

    \begin{lem}\label{measurability lemma} 
        Let $A_k,X_k\in\mathcal{L}_{\infty},$ $1\leq k\leq n,$ and let $X_k\geq0,$ $1\leq k\leq n.$ The function from $\mathbb{R}$ to $\mathcal{L}_\infty$ given by:
        \begin{equation*}
            s\mapsto \prod_{k=1}^nA_kX_k^{is}, s \in \mathbb{R}
        \end{equation*}
        is continuous in the strong operator topology (and in particular in the weak operator topology).
    \end{lem}
    \begin{proof}
        If uniformly bounded nets $\{A_i\}_{i\in\mathbb{I}}$ and $\{B_i\}_{i\in\mathbb{I}}$ converge in the strong operator topology to $A$ and $B$ respectively, then the net $\{A_iB_i\}_{i\in\mathbb{I}}$ converges to $AB$ in strong operator topology. This fact is standard and can be found e.g. in \cite[Proposition 2.4.1]{Bratteli-Robinson1}. Therefore it suffices to show that for each $k = 1,\ldots,n$ that the function $s\mapsto X_k^{is}$
        is continuous in the strong operator topology.        
        
        We note if $X_k$ has a spectral gap at $0$, then $\log(X_k)$ is well defined by continuous functional calculus and $X_k^{is} = \exp(is\log(X_k))$
        is strongly continuous by the Stone-von Neumann theorem. 
        
        If $X_k$ does not necessarily have a spectral gap, then instead we use the Borel function:
        \begin{equation*}
            \log_0(t) := \begin{cases}
                            \log(t),\quad t > 0\\
                            0,\quad t = 0.
                         \end{cases}
        \end{equation*}
        Hence, for $s \in \mathbb{R}$ and $t \geq 0$,
        \begin{equation*}
            \exp(is\log_0(t)) = \begin{cases}
                                    t^{is},\quad t > 0,\\
                                    1,\quad t = 0.
                                \end{cases}
        \end{equation*}
        {Recalling our convention stated at the start of this section,} that $0^{is} = 0$ for all $t \geq 0$, we have:
        \begin{equation*}
            t^{is} = \exp(is\log_0(t))(1-\chi_{\{0\}}(t)).
        \end{equation*}
        
        Let $P_k$ be the support projection of $X_k$ (i.e., the projection onto the orthogonal complement of the kernel of $X_k$). Then since $P_k = 1-\chi_{\{0\}}(X_k)$
        by Borel functional calculus we have:
        \begin{equation*}
            X_k^{is} = P_k\exp(is\log_0(X_k)).
        \end{equation*}
        Since the operator $\log_0(X_k)$ is self-adjoint, by the Stone-von Neumann theorem it follows that $s\mapsto \exp(is\log_0(X_k))$ is strongly
        continuous. Hence, $s\mapsto X_k^{is}$ is strongly continuous and the proof is complete.
    \end{proof}

    \begin{lem}\label{first integral formula} 
        Let $X$ and $Y$ be positive bounded operators and $z \in \mathbb{C}$ with $\Re(z) > 1$. Set $V_z := X^{z-1}(X-Y)+(X-Y)Y^{z-1}$. Then,
        \begin{equation*}
            X^z-Y^z=V_z-\int_{\mathbb{R}}X^{is}V_zY^{-is}\widehat{g}_z(s)ds.
        \end{equation*}
        The integral is understood in the weak operator topology sense. { The function $g_z$ is the same as in the statement of Theorem \ref{csz key lemma}.}
    \end{lem}
    \begin{proof}
        We define the function $\phi_{1,z}$ on $[0,\infty)\times [0,\infty)$ by:
        \begin{align*}
            \phi_{1,z}(\lambda,\mu) &:= g_z(\log(\frac{\lambda}{\mu})) \quad  \lambda,\mu>0,\\
                  \phi_{1,z}(0,\mu) &:= 0,\quad \mu\geq 0,\\
              \phi_{1,z}(\lambda,0) &:= 0,\quad \lambda\geq 0.
        \end{align*}
        We caution the reader that $\phi_{1,z}$ is not continuous at $(0,0)$ unless $z=2$ (indeed, $\phi_{1,z}(\lambda,\lambda)=g_z(0)=1-\frac{z}{2}$ for all $\lambda>0$).
        If we rewrite the definition of $g_z$ in terms of exponentials,
        then for $t \neq 0$ we get
        \begin{equation*}
            g_z(t) = 1-\frac{e^{\frac{z}{2}t}-e^{-\frac{z}{2}t}}{(e^{\frac{t}{2}}-e^{-\frac{t}{2}})(e^{\left(\frac{z-1}{2}\right)t}+e^{-\left(\frac{z-1}{2}\right)t})},
        \end{equation*}        
        and therefore:
        \begin{equation}\label{phi1 alg}
            \phi_{1,z}(\lambda,\mu) = 1-\frac{\lambda^z-\mu^z}{(\lambda-\mu)(\lambda^{z-1}+\mu^{z-1})},\quad \lambda,\mu>0, \lambda\neq \mu.
        \end{equation}
        We claim that
        \begin{equation}\label{phi1 rep}
            \phi_{1z}(\lambda,\mu) = \int_{\mathbb{R}}\widehat{g}_z(s)\lambda^{is}\mu^{-is}ds,\quad\lambda,\mu\geq0.
        \end{equation}
        Indeed, since $g_z$ is Schwartz we can use the Fourier inversion theorem:
        \begin{equation*}
            g_z(t)=\int_{\mathbb{R}}\widehat{g}_z(s)e^{ist}ds,\quad t\in\mathbb{R}.
        \end{equation*}
        If $\lambda,\mu > 0$, we simply substitute $t = \log(\lambda/\mu)$. For $\lambda = 0$ or $\mu = 0$, then the right hand side of \eqref{phi1 rep} vanishes,
        as does $\phi_{1,z}$ by definition. Hence \eqref{phi1 rep} is valid for all $\lambda,\mu \geq 0$.
        

        Thus, by the definition of the double operator integral \eqref{doi definition}, we have:
        \begin{equation}\label{doi usual}
            T_{\phi_{1,z}}^{X,Y}(A)=\int_{\mathbb{R}}\widehat{g}_z(s)X^{is}AY^{-is}ds.
        \end{equation}
        Indeed, since $g_z$ is a Schwartz function, it follows that $\widehat{g}_z\in L_1(\mathbb{R})$ and so the condition \eqref{doi sufficient condition} holds. Therefore, \eqref{doi usual} follows as a consequence of \eqref{doi definition}.
        Here, the integral on the right hand side of \eqref{doi usual} is understood in the weak operator topology sense.
        
        Measurability of the function $s\mapsto X^{is}AY^{-is}$ in the weak operator topology is guaranteed by Lemma \ref{measurability lemma} and condition \eqref{necessary-condition} follows from the inequality
        $$\|\widehat{g}_z(s)X^{is}AY^{-is}\|_{\infty}\leq|\widehat{g}_z(s)|\cdot\|A\|_{\infty},\quad s\in\mathbb{R},$$
        and from the fact that $\widehat{g}_z$ is a Schwartz (and in particular integrable) function. So it follows that $T^{X,Y}_{\phi_{1,z}}$ is bounded in the operator norm from $\mathcal{L}_{\infty}$ to $\mathcal{L}_{\infty}.$
        
        We introduce two more functions on $[0,\infty)\times [0,\infty)$. First,
        \begin{equation*}
            \phi_{2,z}(\lambda,\mu) = (\lambda^{z-1}+\mu^{z-1})(\lambda-\mu),\quad\lambda,\mu\geq0
        \end{equation*}
        and secondly,
        \begin{equation*}
            \phi_{3,z}(\lambda,\mu) = (\lambda^{z-1}+\mu^{z-1})(\lambda-\mu)-(\lambda^z-\mu^z),\quad  \lambda,\mu\geq 0.
        \end{equation*}
        Both functions are bounded on compact subsets of $[0,\infty)^2$, and so in particular on $\mathrm{Spec}(X)\times \mathrm{Spec}(Y)$, since
        by assumption both $X$ and $Y$ are bounded.
        
        The equality $\phi_{3,z}=\phi_{1,z}\phi_{2,z}$ holds on $[0,\infty)\times [0,\infty)$. 
        Indeed this follows from \eqref{phi1 alg} for $\lambda,\mu>0$, $\lambda\neq \mu.$ For $\lambda=\mu>0$ one has $\phi_{1,z}(\lambda,\lambda)=1-\frac{z}{2},$  $\phi_{2,z}(\lambda,\lambda)=0$ and $\phi_{3,z}(\lambda,\lambda)=0.$ If $\lambda=0$ or $\mu=0$ one has $\phi_{1,z}(\lambda,\mu)=0$ and $\phi_{3,z}(\lambda,\mu)=0.$

        Using formulae \eqref{separated variables in doi} and \eqref{doi definition}, we obtain that $T^{X,Y}_{\phi_{2,z}}:\mathcal{L}_{\infty}\to\mathcal{L}_{\infty}$ and
        $$T^{X,Y}_{\phi_{2,z}}(A) = X^zA-X^{z-1}AY+XAY^{z-1}-AY^z.$$
        Since $\phi_{3,z}$ bounded on $\mathrm{Spec}(X) \times \mathrm{Spec}(Y)$, we also get that $T^{X,Y}_{\phi_{3,z}}$ is bounded in the operator norm from $\mathcal{L}_\infty$ to $\mathcal{L}_\infty$, and
        $$T^{X,Y}_{\phi_{3,z}}(A)=(X^zA-X^{z-1}AY+XAY^{z-1}-AY^z)-(X^zA-AY^z).$$
        
        We note at this point that $T^{X,Y}_{\phi_{3,z}}(1) = V_z$.

        We have $\phi_{3,z}=\phi_{1,z}\phi_{2,z}$ on $\mathrm{Spec}(X)\times \mathrm{Spec}(Y)$, and thus by \eqref{doi algebraic}:
        \begin{align*}
            T^{X,Y}_{\phi_{1,z}}(V_z) &= T^{X,Y}_{\phi_{1,z}}(T^{X,Y}_{\phi_{2,z}}(1))\\
                                      &= T^{X,Y}_{\phi_{3,z}}(1)\\
                                      &= V_z-(X^z-Y^z).
        \end{align*}
        The assertion follows now from \eqref{doi usual}.
    \end{proof}

    We are now able to prove Theorem \ref{csz key lemma} in the special case where the spectrum of $B$ is a finite set.
    \begin{lem}\label{second integral formula} 
        Theorem \ref{csz key lemma} holds under the additional assumption that the spectrum of $B$ consists of a finite set of points.
    \end{lem}
    \begin{proof} 
        Suppose that $\mathrm{Spec}(B) = \{\lambda_1,\ldots,\lambda_n\}$, where each $\lambda_k\geq 0$ is distinct. By the spectral theorem
        there exists $n$ mutually orthogonal projections $\{P_k\}_{k=1}^n$ such that
        \begin{equation*}
            B = \sum_{k=1}^n \lambda_kP_k
        \end{equation*}
        and $\sum_{k=1}^n P_k = 1$. We have,
        \begin{equation*}
            B^{z} = \sum_{k=1}^n \lambda_k^z P_k.
        \end{equation*}
        Therefore,
        \begin{align}
            B^zA^z-Y^z &= \sum_{k=1}^n (P_k\lambda_k^zA^z-P_kY^z)\nonumber\\
                       &= \sum_{k=1}^n P_k((\lambda_kA)^z-Y^z)\label{spectral algebra}.
        \end{align}
        Applying Lemma \ref{first integral formula} to each term in the above sum, with $X = \lambda_kA$, if 
        \begin{equation*}
            V_{k,z} = (\lambda_kA)^{z-1}(\lambda_kA-Y)+(\lambda_k A-Y)Y^{z-1}
        \end{equation*}
        then
        \begin{equation*}
            (\lambda_kA)^z-Y^z = V_{k,z}-\int_{\mathbb{R}} (\lambda_kA)^{is}V_{k,z}Y^{-is}\widehat{g}_z(s)\,ds.
        \end{equation*}
        Now substituting into \eqref{spectral algebra}, we have:
        \begin{align*}
            B^zA^z-Y^z &= \sum_{k=1}^n P_k\left(V_{k,z}-\int_{\mathbb{R}} (\lambda_kA)^{is}V_{k,z}Y^{-is}\widehat{g}_z(s)\,ds\right)\\
                       &= \sum_{k=1}^n P_kV_{k,z}-\int_{\mathbb{R}} \left(\sum_{k=1}^n P_k(\lambda_kA)^{is}V_{k,z}\right)Y^{-is}\,\widehat{g}_z(s)\,ds.
        \end{align*}
        By the definition of $V_{k,z}$:
        \begin{equation*}
            V_{k,z} = (\lambda_k A)^z-(\lambda_kA)^{z-1}Y+\lambda_kAY^{z-1}-Y^z
        \end{equation*} 
        and so
        \begin{align*}
            \sum_{k=1}^n P_kV_{k,z} &= B^zA^z-B^{z-1}A^{z-1}Y+BAY^{z-1}-Y^z\\
                                    &= B^{z-1}(BA^z-A^{z-1}Y)+(BA-Y)Y^{z-1}\\
                                    &= B^{z-1}(BA^z-A^{z-1}A^{1/2}BA^{1/2})+(BA-A^{1/2}BA^{1/2})Y^{z-1}\\
                                    &= [BA^{1/2},A^{z-1/2}]+[BA^{1/2},A^{1/2}]Y^{z-1}\\
                                    &= T_z(0).
        \end{align*}
        We may also compute the sum in the integrand:
        \begin{align*}
            \sum_{k=1}^nP_k(\lambda_kA)^{is}V_{k,z} &= \sum_{k=1}^n P_k\left((\lambda_k A)^{z+is}-(\lambda_k A)^{z-1+is}Y+(\lambda_k A)^{1+is}Y^{z-1}-(\lambda_k A)^{is}Y^z\right)\\
                                                    &= \sum_{k=1}^njP_k\lambda_k^{z+is}\cdot A^{p+is}-\sum_{k=1}^n P_k\lambda_k^{z-1+is}\cdot A^{z-1+is}Y\\
                                                    &\quad + \sum_{k=1}^n P_k\lambda_k^{1+is} A^{1+is}Y^{z-1}-\sum_{k=1}^n P_k\lambda_k^{is}\cdot A^{is}Y^z.
        \end{align*}
        
        By functional calculus, we have
        \begin{align*}
            \sum_{k=1}^n P_k(\lambda_kA)^{is}V_{k,z} &= B^{z+is}A^{z+is}-B^{z-1+is}A^{z-1+is}Y+B^{1+is}A^{1+is}Y^{z-1}-B^{is}A^{is}Y^z\\
                                                     &= B^{z-1+is}(BA^{z+is}-A^{z-1+is}Y)+B^{is}(BA^{1+is}-A^{is}Y)Y^{z-1}\\
                                                     &= B^{z-1+is}[BA^{\frac12},A^{z-\frac12+is}]+B^{is}[BA^{\frac12},A^{\frac12+is}]Y^{z-1}.
        \end{align*}
        So multiplying on the right by $Y^{-is}$,
        \begin{equation*}
            \left(\sum_{k=1}^n P_k(\lambda_kA)^{is}V_{k,z}\right)Y^{-is} = B^{z-1+is}[BA^{\frac12},A^{z-\frac12+is}]Y^{-is}+B^{is}[BA^{\frac12},A^{\frac12+is}]Y^{z-1-is}.
        \end{equation*}
        We recognise this right hand side as exactly $T_z(s)$, and so 
        \begin{equation*}
            B^zA^z-Y^z = T_z(0)-\int_{\mathbb{R}} T_z(s)\widehat{g}_z(s)\,ds
        \end{equation*}
        as required.
    \end{proof}
    
    We now explain how to deduce the general version of Lemma \ref{csz key lemma} from the special case of Lemma \ref{second integral formula} (i.e, when $\mathrm{Spec}(B)$ is a finite set).
    To do this we will select a sequence $\{B_n\}_{n=1}^\infty$ with $B_n\to B$ in the uniform norm and such that the spectrum of each $B_n$ is finite.
    
    The following lemma shows that under certain conditions, if $B_n\to B$ in the uniform norm, then $B_n^{is} \to B^{is}$ in the weak operator topology for each fixed $s \in \mathbb{R}$.    

    { For a bounded operator $T$ we denote $\mathrm{supp}(T)$ for the projection onto the orthogonal complement of $\ker(T)$ (this is the support projection of $T$).}
    \begin{lem}\label{unitary group lemma}
        Let $C$ be a positive bounded operator, and let $\{C_n\}_{n=1}^\infty$ be a sequence of positive bounded operators such that $C_n\to C$ in the operator norm,
        and for each $n$ we have $\mathrm{supp}(C_n) = \mathrm{supp}(C)$. Then for all $s \in \mathbb{R}$, we have that $C_n^{is}\to C^{is}$ in the weak operator topology.
    \end{lem}
    \begin{proof} 
        By definition, we need to show that for all $s \in \mathbb{R}$ and $\xi,\eta \in H$ we have
        \begin{equation*}
            \lim_{n\to\infty} \langle C_n^{is}\xi,\eta\rangle = \langle C^{is}\xi,\eta\rangle.
        \end{equation*}
        By assumption, $\mathrm{supp}(C_n) = \mathrm{supp}(C)$ for every $n\geq0.$ 
        By taking a quotient by the closed subspace $\ker(C)$ if necessary, we may assume without loss of generality that $\mathrm{supp}(C_n)=\mathrm{supp}(C)=1$
        for every $n\geq0.$ Also without loss of generality, we assume that $\|C\|_{\infty}\leq 1$ and $\sup_{n\geq 0}\|C_n\|_{\infty}\leq 1$. Let $\xi ,\eta \in H$ be such
        that $\|\xi\| = \|\eta\| = 1$.

        Fix $\varepsilon>0$. { Since $\ker(C) = \{0\}$, we may} select $m > 1$ such that
        $$\|\chi_{[0,\frac1m]}(C))\xi\|<\varepsilon.$$
        Let $0\leq\phi\leq1$ be a smooth function supported on the interval $[\frac1{m+1},2]$ such that $\phi=1$ on the interval $[\frac1m,1].$
        We note that therefore 
        \begin{align*}
            \|(1-\phi(C))\xi\| &= \|(1-\phi(C))\chi_{[0,\frac{1}{m})}\xi\|\\
                               &\leq \|\chi_{[0,\frac{1}{m})}\xi\|\\
                               &< \varepsilon.
        \end{align*}
        Let $\psi(t) := t^{is}\phi(t)$. Since $\phi$ and $\psi$ are smooth and compactly supported, it follows that their first and second derivatives are in $L_2(\mathbb{R})$. 
        These conditions are sufficient for $\phi$ and $\psi$ to be operator Lipschitz (see \cite[Lemma 6, Lemma 7]{PS-crelle}): i.e.,
        there are constants $C_{\phi}$ and $C_{\psi}$ such that
        \begin{align*}
            \|\phi(C_n)-\phi(C)\|_\infty &\leq C_{\phi}\|C_n-C\|_\infty,\\
            \|\psi(C_n)-\psi(C)\|_\infty &\leq C_{\psi}\|C_n-C\|_\infty.
        \end{align*}
        
        Select $N> 0$ such that for all $n > N$ we thus have,
        \begin{align*}
                          \|\phi(C_n)-\phi(C)\|_{\infty} &\leq \epsilon \text{ and, }\\
            \|\phi(C_n)C_n^{is}-\phi(C)C^{is}\|_{\infty} &\leq \epsilon.
        \end{align*}

        For $n > N$ we have:
        \begin{align}
            \langle C_n^{is}\xi,\eta\rangle-\langle C^{is}\xi,\eta\rangle &= \langle C_n^{is}(1-\phi(C_n))\xi,\eta\rangle+\langle C^{is}(\phi(C)-1)\xi,\eta\rangle\nonumber\\
                                                                          &\quad +\langle (C_n^{is}\phi(C_n)-C^{is}\phi(C))\xi,\eta\rangle\label{big weak sum}.
        \end{align}
        
        For the first term in \eqref{big weak sum} above, we have:
        \begin{align*}
            |\langle C_n^{is}(1-\phi(C_n))\xi,\eta\rangle| &\leq \|(1-\phi(C_n))\xi\|\\
                                                           &\leq \|(1-\phi(C))\xi\|+\|\phi(C_n)-\phi(C)\|_{\infty}\\
                                                           &< 2\varepsilon.
        \end{align*}
        Next, for the second term in \eqref{big weak sum}, we have
        \begin{align*}
            |\langle C^{is}(\phi(C)-1)\xi,\eta\rangle| &\leq \|(1-\phi(C))\xi\|_\infty\\
                                                       &< \varepsilon.
        \end{align*}
        Finally, for the third term in \eqref{big weak sum},
        \begin{align*}
            |\langle (C_n^{is}\phi(C_n)-C^{is}\phi(C))\xi,\eta\rangle| &\leq \|\psi(C_n)-\psi(C)\|_\infty\leq \varepsilon.
        \end{align*}
        
        So in summary, for $n\geq N$ we have:
        $$|\langle C_n^{is}\xi,\eta\rangle-\langle C^{is}\xi,\eta\rangle|\leq 4\varepsilon.$$
        Since $\varepsilon>0$ is arbitrarily small, the assertion follows.
    \end{proof}

    We are now ready to prove Theorem \ref{csz key lemma}.
    \begin{proof}[Proof of Theorem \ref{csz key lemma}] 
        Without loss of generality, $\|B\|_{\infty}=1$ (if not, then we replace the couple $(A,B)$ with a couple $(cA,c^{-1}B)$ with a suitable constant $c > 0$). 
        In this case we have $\mathrm{Spec}(B) \subseteq [0,1]$ and $1 \in \mathrm{Spec}(B)$.
        For every $n\geq 1 ,$ set
        $$B_n=\sum_{m=1}^n\frac{m}{n}\chi_{(\frac{m-1}{n},\frac{m}{n}]}(B).$$
        Recall that $Y := A^{1/2}BA^{1/2}$, and let $Y_n := A^{1/2}B_nA^{1/2}$, and let $T_{n,z}(s)$ be defined as $T_z(s)$
        with the occurances of $B$ replaced with $B_n$ and $Y$ replaced with $Y_n$.
        
        By construction, the spectrum of $B$ consists of at most $n$ points, indeed by the spectral mapping theorem:
        \begin{equation*}
            \mathrm{Spec}(B_n) \subseteq \left\{\frac{m}{n}\right\}_{m=1}^n.
        \end{equation*}        
        Since $1 \in \mathrm{Spec}(B)$, we always have that $\chi_{(\frac{n-1}{n},1]}(B) \neq 0$, so $1 \in \mathrm{Spec}(B_n)$
        and $\|B_n\|_\infty = 1$.
        We also have that $\mathrm{supp}(B_n) = \mathrm{supp}(B)$, and $\mathrm{supp}(Y_n) = \mathrm{supp}(Y)$. 
        
        Moreover, $B_n$ converges in norm to $B$, since $\|B-B_n\|_\infty \leq \frac{1}{n}$. Thus by Lemma \ref{unitary group lemma}, 
        for any $s \in \mathbb{R}$ we also have that $B_n^{is}\to B^{is}$ in the weak operator topology. Similarly, $Y_n\to Y$ in the norm topology
        and $Y_n^{is}\to Y^{is}$ in the weak operator topology. 
        
        It follows now that for each $s \in \mathbb{R}$ and $z \in \mathbb{C}$ with $\Re(z) > 1$, we have that $T_{n,z}(s)\to T_z(s)$ in the weak operator topology. One can also see that $\sup_{s \in \mathbb{R}} \sup_{n\geq 1} \|T_{n,z}(s)\|_\infty < \infty$.
        
        In other words, for every $\xi,\eta\in H,$ we have
        $$\langle T_{n,z}(s)\xi,\eta\rangle\to\langle T_z(s)\xi,\eta\rangle.$$      
        
        Since $|\langle T_{n,z}(s)\xi,\eta\rangle| \leq \sup_{n\geq 1} \|T_{n,z}(s)\|_\infty$, and this is bounded in $s$, we may use the Dominated Convergence theorem
        to obtain
        $$\int_{\mathbb{R}}\langle T_n(s)\xi,\eta\rangle \widehat{g}_z(s)ds\to \int_{\mathbb{R}}\langle T(s)\xi,\eta\rangle \widehat{g}_z(s)ds$$
        
        By Lemma \ref{second integral formula}, for all $n\geq 1$ and $\xi,\eta \in H$ we have:
        \begin{equation}\label{weak second integral formula}
            \langle (B_n^zA^z - Y_n^z)\xi,\eta\rangle =\langle T_{n,z}(0)\xi,\eta\rangle-\int_{\mathbb{R}} \langle T_{n,z}(s)\xi,\eta\rangle \widehat{g}_z(s)\,ds.
        \end{equation}
        
        As already discussed, $B_n^z\to B^z$ in the weak operator topology, and similarly $Y_n^z\to Y^z$ in the weak operator topology. Hence
        both sides of \eqref{weak second integral formula} converge, and:
        \begin{equation*}
            \langle (B^zA^z-Y^z)\xi,\eta\rangle =\langle T_z(0)\xi,\eta\rangle-\int_{\mathbb{R}} \langle T_z(s)\xi,\eta\rangle \widehat{g}_z(s)\,ds.
        \end{equation*}
        Since $\xi$ and $\eta$ are arbitrary, this completes the proof.
    \end{proof}

\section{Analyticity of the mapping $z \mapsto g_z$}
    So far we have considered the function,
    \begin{equation*}
        g_z(t) := 1-\frac{e^{\frac{z}{2}t}-e^{-\frac{z}{2}t}}{(e^{\frac{t}{2}}-e^{-\frac{t}{2}})(e^{\left(\frac{z-1}{2}\right)t}+e^{-\left(\frac{z-1}{2}\right)t})}, t\neq 0
    \end{equation*}
    with $g_z(0) := 1-\frac{z}{2}$ as a Schwartz function of $t$ with a fixed parameter $z \in \mathbb{C}$, with $\Re(z) > 1$. 
    
    We may equally well consider $g$ as a function of $z$. That is, the mapping $z \mapsto g_z$
    defines a function:
    \begin{equation*}
        \{z \in \mathbb{C}\;:\; \Re(z)>1\} \to \mathcal{S}(\mathbb{R}).
    \end{equation*}
    
    As a matter of fact, the function $z\mapsto g_z$ is holomorphic with values in the Hilbert-Sobolev space:
    \begin{equation*}
        H^2(\mathbb{R}) := \{f \in L_2(\mathbb{R})\;:\; f',f'' \in L_2(\mathbb{R})\},
    \end{equation*}
    equipped with the Sobolev norm:
    \begin{equation*}
        \|f\|_{H^2(\mathbb{R})}^2 := \|f\|_{L_2(\mathbb{R})}^2 + \|f'\|_{L_2(\mathbb{R})}^2 + \|f''\|_{L_2(\mathbb{R})}^2.
    \end{equation*}

    We remind the reader of the meaning of Banach space valued holomorphy. If $D \subseteq \mathbb{C}$ is a domain, $X$ is a Banach space and $f:D\to X$
    then the following two conditions are equivalent:
    \begin{enumerate}[{\rm (a)}]
        \item{} For any continuous linear functional $\varpi \in X^*$, the function $\varpi\circ f:D\to \mathbb{C}$ is holomorphic.
        \item{} For any $z \in \mathbb{C}$, the limit in the norm topology of $X$
                \begin{equation*}
                    f'(z) = \lim_{\zeta\to z} \frac{f(z)-f(\zeta)}{z-\zeta}
                \end{equation*}
                exists.
    \end{enumerate}
    The equivalence of these conditions is well known, see e.g. \cite[Theorem 3.31]{rudin}.
    
    We work with the first condition. Since $H^2(\mathbb{R})$ is a Hilbert space, for any continuous linear functional $\varpi$ 
    on $H^2(\mathbb{R})$ there exists $h \in H^2(\mathbb{R})$ such that:
    \begin{equation*}
        \varpi(g_z) = \int_{\mathbb{R}} g_z(t)h(t)\,dt + \int_{\mathbb{R}} g_z'(t)h'(t)\,dt + \int_{\mathbb{R}}g_z''(t)h''(t)\,dt.
    \end{equation*}
    So we focus on proving that for each $h \in H^2(\mathbb{R})$, the mapping $z\mapsto \varpi(g_z)$ is holomorphic.

    \begin{lem}\label{computational continuity lemma} Let $G:\{z \in \mathbb{C}\;:\; \Re(z) > 1\}\to H^2(\mathbb{R})\}$ be the function given by $G(z) = g_z$. Then $G$ is continuous on its domain.    
    \end{lem}
    \begin{proof} It suffices to prove that the mappings $G:\{z \in \mathbb{C}\;:\; \Re(z) > 1\}\to H^2(-1,1)\},$ $G:\{z \in \mathbb{C}\;:\; \Re(z) > 1\}\to H^2(1,\infty)\}$ and $G:\{z \in \mathbb{C}\;:\; \Re(z) > 1\}\to H^2(-\infty,-1)\}$ are continuous on their domains.
    
    We write the first function as follows:
    $$g_z=1-a_zbc_z,\quad\Re(z)>1,$$
    where
    $$a_z(t)=\frac{\sinh(\frac{zt}{2})}{t},\quad b(t)=\frac{t}{2\sinh(\frac{t}{2})},\quad c_z(t)=\frac1{\cosh(\left(\frac{z-1}{2}\right)t)}.$$
    Since $z\to a_z$ and $z\to c_z$ are continuous $C^2[-1,1]-$valued mappings, the first assertion follows.
    
    We rewrite our function as follows.
    $$g_z=-\frac{\sinh(\frac{(z-2)t}{2})}{2\sinh(\frac{t}{2})\cosh(\frac{(z-1)t}{2})},\quad t\in\mathbb{R}.$$
    Equivalently,
    $$g_z=-a_zbc_z,$$
    where
    $$a_z=e^{-\frac{(z+1)t}{2}}-e^{-\frac{(3z-3)t}{2}},\quad b=\frac{e^{-t}}{1-e^{-t}},\quad c_z=\frac1{1+e^{-(z-1)t}}.$$
    Since $z\to a_z$ and $z\to c_z$ are continuous $C^2(1,\infty)-$valued mappings and $b\in H^2(1,\infty),$ the second assertion follows.
    \end{proof}

    \begin{thm}\label{computational analytic lemma} 
        Let $G:\{z \in \mathbb{C}\;:\; \Re(z) > 1\}\to H^2(\mathbb{R})\}$ be the function given by $G(z) = g_z$. Then $G$ is holomorphic on its domain.
    \end{thm}
    \begin{proof}
        To show that $G$ is holomorphic with values in $H^2(\mathbb{R})$, it suffices to show for all continuous linear functionals $\varpi$ on $H^2(\mathbb{R})$ that $z\mapsto \varpi(G(z))$ is holomorphic.
        
        Since $H^2(\mathbb{R})$ is a Hilbert space, it suffices to show for all $h \in H^2(\mathbb{R})$ that:
        \begin{equation}\label{companal main eq}
            z \mapsto \int_{\mathbb{R}} g_z(t)h(t)+g_z'(t)h'(t)+g_z''(t)h''(t)\,dt,\quad \Re(z) > 1
        \end{equation}
        is holomorphic. 
        
        Let $\gamma$ be a simple closed curve contained in $\{z\;:\;\Re(z) > 1\}$. By Lemma \ref{computational continuity lemma}, the function 
        \begin{equation*}
            (z,t) \mapsto g_z(t)h(t) + g_z'(t)h'(t) + g_z''(t)h''(t),\quad t \in \mathbb{R}, z \in \gamma
        \end{equation*}
        is integrable on $\gamma\times \mathbb{R}.$ Indeed,
        $$\int_{\gamma}\Big(\int_{\mathbb{R}}|g_z(t)h(t) + g_z'(t)h'(t) + g_z''(t)h''(t)|dt\Big)|dz|\leq$$
        $$\leq\int_{\gamma}\|g_z\|_{H^2(\mathbb{R})}|dz|\leq {\rm length}(\gamma)\cdot\sup_{z\in\gamma}\|g_z\|_{H^2(\mathbb{R})}<\infty.$$
        
        We may apply Fubini's theorem to conclude that:
        \begin{equation*}
            \int_{\gamma} \varpi(G_z)\,dz = \int_{\mathbb{R}}\Big(\int_{\gamma}(g_z(t)h(t)+g_z'(t)h'(t)+g_z''(t)h''(t))dz\Big)dt.
        \end{equation*}
        For each fixed $t$ it follows from the definition that the functions $g_z(t)$, $g_z'(t)$ and $g_z''(t)$ are holomorphic in $z$. Hence $\int_{\gamma} \varpi(G(z))\,dz = 0$ for all simple closed curves $\gamma$ contained in $\{z\;:\;\Re(z) > 1\}$. Since $z \mapsto \varpi(G(z))$ is continuous, by Morera's theorem
        $z\mapsto \varpi(G(z))$ is holomorphic in the domain $\{z\;:\;\Re(z) > 1\}$. Since $\varpi$ is arbitrary, $G$ is $H^2(\mathbb{R})$-valued holomorphic.
    \end{proof} 
 
\section[The difference of two $\zeta-$functions admits analytic continuation]{The function $z\to{\rm Tr}(XB^zA^z)-{\rm Tr}(X(A^{\frac12}BA^{\frac12})^z)$ admits analytic continuation to $\{\Re(z)>p-1\}$}\label{difference section}
    As indicated in the title, in this section we prove (under certain assumptions on $A$ and $B$) that for all $X \in \mathcal{L}_\infty$ the mapping
    \begin{equation*}
        z \mapsto \mathrm{Tr}(XB^zA^z)-\mathrm{Tr}(X(A^{\frac{1}{2}}BA^{\frac{1}{2}})^z)
    \end{equation*}
    defined initially on $\{z\in \mathbb{C}\;:\; \Re(z) > p\}$ is holomorphic, and admits analytic continuation to the set $\{z\;:\;\Re(z) > p-1\}$.
    The precise assumptions on $A$ and $B$ are as follows:
    \begin{cond}\label{conditions for analyticity} 
        Let $p>2$ and let $0\leq A,B\in\mathcal{L}_{\infty}$ satisfy the conditions
        \begin{enumerate}[{\rm (i)}]
            \item\label{anacond1} $B^pA\in\mathcal{L}_{1,\infty}$
            \item\label{anacond2} $B^{q-2}[B,A]\in\mathcal{L}_1$ for every $q>p$
            \item\label{anacond3} $A^{\frac12}BA^{\frac12}\in\mathcal{L}_{p,\infty}$
            \item\label{anacond4} $[B,A^{\frac12}]\in\mathcal{L}_{\frac{p}{2},\infty}$.
        \end{enumerate}
    \end{cond}
    
    The main result of this section is the following:
    \begin{thm}\label{analyticity theorem I} 
        Let $p>2$ and let $A$ and $B$ satisfy Condition \ref{conditions for analyticity}. If $X\in\mathcal{L}_{\infty},$ then the mapping
        $$z\to{\rm Tr}(XB^zA^z)-{\rm Tr}(X(A^{\frac12}BA^{\frac12})^z),\quad \Re(z)>p,$$
        admits an analytic continuation to the half-plane $\{\Re(z)>p-1\}.$
    \end{thm}

    \begin{lem}\label{1 to r lemma} 
        Assume that $p\geq 1$ and that $A,B$ satisfy Condition \ref{conditions for analyticity}. Then for all $r \geq 1$ we have $B^{\frac{p}{r}}A\in\mathcal{L}_{r,\infty}$. More precisely, we have the following norm bound:
        \begin{equation*}
            \|B^{\frac{p}{r}}A\|_{r,\infty}\leq e\|A\|_{\infty}^{1-\frac1r}\|B^pA\|_{1,\infty}^{\frac1r}.
        \end{equation*}
    \end{lem}
    \begin{proof} 
        We show this is a consequence of the Araki-Lieb-Thirring inequality \eqref{ALT inequality} and \eqref{log majorization monotone}.
        
        Fix $r \geq 1$. Then by the Araki-Lieb-Thirring inequality:
        \begin{equation*}
            |B^{p/r}A|^r \prec\prec_{\log} B^pA^r.
        \end{equation*}
        Now using \eqref{log majorization monotone}:
        \begin{align*}
            \|B^{p/r}A\|_{r,\infty} &= \||B^{p/r}A|^r\|_{1,\infty}\\
                                    &\leq e\|B^pA^r\|_{1,\infty}\\
                                    &\leq e\|A\|_\infty^{r-1}\|B^pA\|_{1,\infty}.
        \end{align*}
    \end{proof}
    
    The next lemma provides a sufficient condition for a function to be holomorphic with values in a Banach ideal of $\mathcal{L}_\infty$.
    \begin{lem}\label{banach holomorphic criteria} 
        Assume that $D \subseteq \mathbb{C}$ is a domain (i.e., a connected open set) and that $F:D\to \mathcal{L}_\infty$ is a holomorphic function.
        If $\mathcal{I}$ is a Banach-normed ideal of $\mathcal{L}_\infty$ such that $F$ takes values in $\mathcal{I}$ and $F:D\to \mathcal{I}$ is continuous, then $F$
        is also an $\mathcal{I}$-valued holomorphic function.
    \end{lem}
    \begin{proof} 
        Fix a contour $\gamma\subset D.$ such that the interior of $\gamma$ is contained in $D$. Then for all $z$ in the interior of $\gamma$, we have
        $$F(z)=\frac1{2\pi i}\int_{\gamma}\frac{F(w)}{w-z}\,dw$$
        A priori, the integral is a weak integral. However, $F:D\to\mathcal{I}$ is continuous and, therefore, the integral is Bochner in $\mathcal{I}.$

        In order to show that $F$ is holomorphic, it suffices to show that it is differentiable. Let
        $$G(z)=\frac1{2\pi i}\int_{\gamma}\frac{F(w)}{(w-z)^2}\,dw$$
        Once more, since $F:D\to \mathcal{I}$ is continuous, then this integral is defined as an $\mathcal{I}$-valued Bochner integral.
        If $F$ were holomorphic, then $G$ would be the derivative of $F$. The proof will be completed upon showing that
        $G$ is indeed the derivative of $F$ considered as an $\mathcal{I}$-valued mapping.
        
        For every $z$ and $z_0$ in the interior of $\gamma$, we have:
        $$\frac{F(z)-F(z_0)}{z-z_0}-G(z_0)=\frac{z-z_0}{2\pi i}\int_{\gamma}\frac{F(w)}{(w-z)(w-z_0)^2}\,dw$$
        Again, this integral is an $\mathcal{I}$-valued Bochner integral.

        Thus,
        \begin{align*}
            \Big\|\frac{F(z)-F(z_0)}{z-z_0}-G(z_0)\Big\|_{\mathcal{I}} &\leq \frac{|z-z_0|}{2\pi}\int_{\gamma}\|F(w)\|_{\mathcal{I}}\frac{1}{|w-z|\cdot|w-z_0|^2}\,|dw|\\
                                                                        &\leq \sup_{w\in\gamma}\|F(w)\|_{\mathcal{I}}\cdot \frac{|z-z_0|}{2\pi}\int_{\gamma}\frac{1}{|w-z|\cdot|w-z_0|^2}\,|dw|.
        \end{align*}
        Since $f$ is continuous, the right hand side tends to $0$ as $z\to z_0$. Hence, $G = F'$ and so $F$ is holomorphic.
    \end{proof}

    \begin{lem}\label{exp fact} 
        Let $0\leq A\in\mathcal{L}_{\infty}.$ The function $z\to A^z$ is $\mathcal{L}_{\infty}-$valued holomorphic on the half-plane $\{z\;:\;\Re(z)>0\}.$
    \end{lem}
    \begin{proof} 
        For $z_0,z\in\mathbb{C}$ with $\Re(z_0)>0$ and $\Re(z)> 0,$ we define the operator $A^{z_0}\log(A)$ by means of functional calculus (the convention $0^{z_0}\log(0)=0$ is used). 
        
        Hence,
        \begin{align*}
            \Big\|\frac{A^z-A^{z_0}}{z-z_0}-A^{z_0}\log(A)\Big\|_{\infty} &\leq \sup_{0\leq \lambda\leq\|A\|_{\infty}}\Big|\frac{\lambda^z-\lambda^{z_0}}{z-z_0}-\lambda^{z_0}\log(\lambda)\Big|\\
                                                                          &= O(z-z_0),\quad z_0\to z.
        \end{align*}
        Hence, for all $z$ with $\Re(z) > 0$
        \begin{equation*}
            A^z-A^{z_0} = A^{z_0}\log(A)(z-z_0)+o(1),\quad z_0\to z
        \end{equation*}
        and so $A^z$ is $\mathcal{L}_\infty$-valued holomorphic with derivative $A^{z}\log(A)$.
    \end{proof}

    \begin{lem}\label{psacta lemma} 
        Let $p>2$ and assume $A$ and $B$ satisfy condition \ref{conditions for analyticity}. The mapping
        $$z\to [B,A^z],\quad\Re(z)>1,$$
        is a holomorphic $\mathcal{L}_{\frac{p}{2},\infty}-$valued function.
    \end{lem}
    \begin{proof}
        We take care to note that since $p > 2$, the ideal $\mathcal{L}_{p/2,\infty}$ can be equipped with a norm generating the same topology as
        that of the canonical quasi-norm (see \cite[Chapter 4, Lemma 4.5]{Bennet-Sharpley-interpolation-1988}). Denote such a norm as $\|\cdot\|_{\mathcal{L}_{p/2,\infty}}'$.
        
        Denote 
        \begin{equation*}
            F(z) = [B,A^z].
        \end{equation*}
        Since $F(z) = BA^z-A^zB$ it now follows from Lemma \ref{exp fact} that $F$ is $\mathcal{L}_\infty$-valued holomorphic.
        Due to Lemma \ref{banach holomorphic criteria}, it now suffices to show that $F$ is $\mathcal{L}_{p/2,\infty}$-valued
        and $\mathcal{L}_{p/2,\infty}$-continuous.
        
        Let $\phi$ be a compactly supported smooth function on $\mathbb{R}$, such that $0 \leq \phi \leq 1$ and $\phi = 1$ on the interval
        $[0,\|A\|_\infty]$. Define
        \begin{equation*}
            \phi_z(t) = |t|^z\phi(t),\quad t \in \mathbb{R}, \Re(z) > 1.
        \end{equation*}
        Since $\phi = 1$ on the spectrum of $A$, we have that $\phi_z(A) = A^z$ and so for all $z,z_1,z_2$ with $\Re(z), \Re(z_1), \Re(z_2) > 1$,
        \begin{align*}
            F(z) &= [B,\phi_z(A)]\\
            F(z_1)-F(z_2) &= [B,(\phi_{z_1}-\phi_{z_2})(A)].
        \end{align*}
        Now we refer to \cite{PS-acta}, where it is proved that if $A$ and $B$ are self-adjoint operators and $r > 1$ is such that $[A,B] \in \mathcal{L}_{r,\infty}$,
        then for all Lipschitz functions $f$, there is a constant $c_r$ such that:
        \begin{equation*}
            \|[f(A),B]\|_{\mathcal{L}_{r,\infty}} \leq c_r\|f'\|_{L_\infty(\mathbb{R})}\|[A,B]\|_{\mathcal{L}_{r,\infty}}.
        \end{equation*}
        Since $p > 2$, we may apply this result with $r = p/2$ and since $\phi$ is smooth and compactly supported,
        we may take $f = \phi_{z}$. Thus,
        \begin{equation*}
            \|F(z)\|_{\mathcal{L}_{p/2,\infty}}' \leq c_{p/2}\|\phi_z'\|_{L_\infty(\mathbb{R})}\|[B,A]\|_{\mathcal{L}_{p/2,\infty}}'
        \end{equation*} 
        and similarly taking $f = \phi_{z_1}-\phi_{z_2}$,
        \begin{equation*}
            \|F(z_1)-F(z_2)\|_{\mathcal{L}_{p/2,\infty}}' \leq c_{p/2}\|\phi_{z_1}'-\phi_{z_2}'\|_{L_\infty(\mathbb{R})}\|[B,A]\|_{\mathcal{L}_{p/2,\infty}}'.
        \end{equation*}
        Note that for $z_1\to z_2$ we have:
        \begin{equation*}
            \|\phi_{z_1}'-\phi_{z_2}'\|_{L_\infty(\mathbb{R})}\to 0
        \end{equation*}
        and hence $F$ is $\mathcal{L}_{p/2,\infty}$-valued continuous. The result now follows.
    \end{proof}

    \begin{lem}\label{trivial analytic lemma} 
        Let $p>2$ and let $A$ and $B$ satisfy Condition \ref{conditions for analyticity}. Then:
        \begin{enumerate}[{\rm (i)}]
            \item\label{triv0} The mapping $F_0(z) := B^{z-1}[B,A^{z-\frac12}]A^{\frac12}+[B,A^{\frac12}]A^{\frac12}Y^{z-1}$ is an $\mathcal{L}_1-$valued 
                                holomorphic function for the domain $\Re(z)>p-1.$
            \item\label{triv1} The mapping $F_1(z) := B^{z-1}A^{z-1}$ is an $\mathcal{L}_{\frac{p}{p-2},1}-$valued holomorphic function for the domain $\Re(z)>p-1.$
            \item\label{triv2} The mapping $F_2(z) := B^{z-1}[BA,A^{z-1}]$ is an $\mathcal{L}_1-$valued holomorphic function for the domain $\Re(z)>p-1.$
            \item\label{triv3} The mapping $F_3(z) := Y^{z-1}$ is an $\mathcal{L}_{\frac{p}{p-2},1}-$valued holomorphic function for the domain $\Re(z)>p-1.$ (Recall that $Y = A^{1/2}BA^{1/2}$).
        \end{enumerate}
    \end{lem}
    \begin{proof}         
        We first prove \eqref{triv1}. Fix $q\in(p,p+2).$ If $\Re(z)>q-1,$ then
        $$B^{z-1}A^{z-1}=B^{z-q+1}B^{q-2}A\cdot A^{z-2}.$$
        By Lemma \ref{1 to r lemma}, we have that
        $$B^{q-2}A\in\mathcal{L}_{\frac{p}{q-2},\infty}\subset\mathcal{L}_{\frac{p}{p-2},1}$$
        and correspondingly by Lemma \ref{exp fact} the mapping $z\mapsto B^{z-1}A^{z-1}$ is continuous in the $\mathcal{L}_{p/(p-2),1}$ norm.
        
        Moreover Lemma \ref{exp fact} implies that the mappings $z\to B^{z-q+1}$ and $z\to A^{z-2}$ are $\mathcal{L}_{\infty}-$valued holomorphic for $\Re(z)>q-1.$ 
        Thus by Lemma \ref{banach holomorphic criteria} the mapping $z\to B^{z-1}A^{z-1}$ is $\mathcal{L}_{\frac{p}{p-2},1}-$valued holomorphic for $\Re(z)>q-1.$ 
        Since $q>p$ is arbitrary, \eqref{triv1} follows.

        Now we prove \eqref{triv2}. Again fix $q\in(p,p+2)$ and let $z$ satisfy $\Re(z)>q-1.$ We rewrite $F_2$ as:
        \begin{align}
            F_2(z) &= B^zA^z-B^{z-1}A^{z-1}BA\nonumber\\
                   &= B^{z-1}A^{z-1}\cdot [A,B]+[B,B^{z-1}A^z]\nonumber\\
                   &= B^{z-1}A^{z-1}\cdot[A,B]+B^{z-1}\cdot[B,A]\cdot A^{z-1}+B^{z-1}A\cdot[B,A^{z-1}].\label{F_2 expansion}
        \end{align}
        The first summand in \eqref{F_2 expansion} can be written as
        $$B^{z-q+1}\cdot B^{q-2}A\cdot A^{z-2}\cdot [A,B].$$
        Due to Lemma \ref{exp fact}, the mappings $z\to B^{z-q+1}$ and $z\to A^{z-2}$ are $\mathcal{L}_{\infty}-$valued holomorphic for $\Re(z)>q-1.$ By Lemma \ref{1 to r lemma}, we have that
        $$B^{q-2}A\in\mathcal{L}_{\frac{p}{q-2},\infty}\subset\mathcal{L}_{\frac{p}{p-2},1}.$$
        The element $[A,B] = A^{1/2}[A^{1/2},B]+[A^{1/2},B]A^{1/2}$ belongs to $\mathcal{L}_{\frac{p}{2},\infty}.$ Hence, the first summand is $\mathcal{L}_1$-valued holomorphic for $\Re(z)>q-1.$

        The second summand in \eqref{F_2 expansion} can be written as
        $$z\to B^{z-q+1}\cdot B^{q-2}[B,A]\cdot A^{z-1}.$$
        By our assumption of Condition \ref{conditions for analyticity}, the operator $B^{q-2}[B,A]$ belongs to $\mathcal{L}_1$ and accordingly the map $z\to B^{z-1}\cdot[B,A]\cdot A^{z-1}$ is $\mathcal{L}_1$-continuous. Once again due to Lemma \ref{exp fact}, the mappings $z\to B^{z-q+1}$ and $z\to A^{z-1}$ are $\mathcal{L}_{\infty}-$valued holomorphic for $\Re(z)>q-1.$ Hence, the second summand of \eqref{F_2 expansion} is $\mathcal{L}_1$-valued holomorphic for $\Re(z)>q-1.$

        We now treat the third summand of \eqref{F_2 expansion}. 
        By Lemma \ref{psacta lemma}, the mapping
        $$z\to [B,A^{z-1}],\quad\Re(z)>q-1,$$
        is a holomorphic $\mathcal{L}_{\frac{p}{2},\infty}-$valued function. Due to Condition \ref{conditions for analyticity}, we have that
        $$B^{q-2}A\in\mathcal{L}_{\frac{p}{q-2},\infty}\subset\mathcal{L}_{\frac{p}{p-2},1}.$$
        Hence
        $$z\to B^{z-q+1}\cdot B^{q-2}A\cdot [B,A^{z-1}],\quad\Re(z)>q-1,$$
        is an $\mathcal{L}_{1}$-valued holomorphic function.

        So, all three summands in \eqref{F_2 expansion} are holomorphic for $\Re(z)>q-1.$ Thus, $z\to F_2(z)$ is holomorphic for $\Re(z)>q-1.$ Since $q>p$ is arbitrary, it follows that $z\to 
        F_2(z)$ is holomorphic for $\Re(z)>p-1.$ This proves \eqref{triv2}.

        For \eqref{triv3}, we note that this is an immediate consequence of the assumption of Condition \ref{conditions for analyticity}.\eqref{anacond3}
        and Lemma \ref{exp fact}.
        
        Finally we prove \eqref{triv0}. We write $F_0(z)$ as:
        \begin{align*}
            F_0(z) &= B^{z-1}[B,A^{z-1}A^{1/2}]A^{1/2}+[B,A^{1/2}]A^{1/2}Y^{z-1}\\
                   &= F_2(z)+B^{z-1}A^{z-1}[B,A^{1/2}]A^{1/2}+[B,A^{1/2}]A^{1/2}F_3(z)\\
                   &= F_2(z)+F_1(z)[B,A^{1/2}]A^{1/2}+[B,A^{1/2}]F_3(z).
        \end{align*}
        Hence, by \eqref{triv1}, \eqref{triv2} and \eqref{triv3} and Condition \ref{conditions for analyticity}.\eqref{anacond4},
        \begin{equation*}
            F_0(z) \in \mathcal{L}_1+\mathcal{L}_{p/(p-2),1}\cdot \mathcal{L}_{p/2,\infty}+\mathcal{L}_{p/2,\infty}\mathcal{L}_{p/(p-2),\infty}.
        \end{equation*}
        so by the H\"older-type inequality \eqref{another holder}, $F_0(z) \in \mathcal{L}_1$, and it continuous in the $\mathcal{L}_1$-norm. So by Lemma \ref{banach holomorphic criteria}, $F_0$ is $\mathcal{L}_1$-valued holomorphic.
    \end{proof}

    \begin{lem}\label{g abstract analytic lemma} 
        Let $s\to H(s)$ be a bounded $\mathcal{L}_{\infty}-$valued function measurable in the weak operator topology (see Definition \ref{weak meas def}). { Let $g_z$ be as in Theorem \ref{csz key lemma}}. Define
        $$G(z) := \int_{\mathbb{R}}H(s)\widehat{g}_z(s)ds,\quad \Re(z)>1$$
        { as a weak operator topology integral.}
        \begin{enumerate}[{\rm (i)}]
            \item\label{gaba} $G$ is an $\mathcal{L}_{\infty}-$valued holomorphic function for the domain $\Re(z)>1.$
            \item\label{gabb} if there is $r > 1$ such that for all $s \in \mathbb{R}$ we have $\|H(s)\|_{r,\infty}\leq 1+|s|$, then $G$ is an $\mathcal{L}_{r,\infty}-$valued holomorphic function 
                            for the domain $\Re(z)>1.$
        \end{enumerate}
    \end{lem}
    \begin{proof} 
        Define:
        $$g_{1,z}=\frac{\partial}{\partial z}g_z,\quad z\in\mathbb{C},\quad \Re(z)>1.$$
        Define the $\mathcal{L}_\infty$-valued function $G_1$ by:
        \begin{equation*}
            G_1(z)=\int_{\mathbb{R}}H(s)\hat{g}_{1,z}(s)ds.
        \end{equation*}
        We will show that $G$ is $\mathcal{L}_\infty$-valued holomorphic by showing that $G_1$ is the derivative of $G$.

        Let $z,z_0 \in \mathbb{C}$ have real part greater than $1$. Then we have
        \begin{equation*}
            \frac{G(z)-G(z_0)}{z-z_0}-G_1(z_0) = \int_{\mathbb{R}}H(s)\Big(\frac{\hat{g}_z(s)-\hat{g}_{z_0}(s)}{z-z_0}-\hat{g}_{1,z_0}(s)\Big)ds.
        \end{equation*}
        So by the triangle inequality,
        $$\Big\|\frac{G(z)-G(z_0)}{z-z_0}-G'(z_0)\Big\|_{\infty}\leq\esssup_{s \in \mathbb{R}}\|H(s)\|_{\mathcal{L}_\infty}\int_{\mathbb{R}}\Big|\frac{\hat{g}_z(s)-\hat{g}_{z_0}(s)}{z-z_0}-\hat{g}_{1,z_0}(s)\Big|ds.$$
        By \cite[Lemma 7]{PS-crelle}, we have an absolute constant $c_{\mathrm{abs}}$ such that:
        $$\int_{\mathbb{R}}\Big|\frac{\hat{g}_z(s)-\hat{g}_{z_0}(s)}{z-z_0}-\hat{g}_{1,z_0}(s)\Big|ds\leq c_{abs}\Big(\Big\|\frac{g_z-g_{z_0}}{z-z_0}-g_{1,z_0}\Big\|_2+\Big\|\frac{g_z'-g_{z_0}'}{z-z_0}-g_{1,z_0}'\Big\|_2\Big).$$
        The assertion \eqref{gaba} follows now from Theorem \ref{computational analytic lemma}, and hence moreover we have that $G_1 = G'$.

        Let us now establish \eqref{gabb}. Indeed, by Lemma \ref{peter norm lemma}, we have
        \begin{align*}
                                                \|G(z)\|_{r,\infty}  &\leq \frac{r}{r-1}\int_{\mathbb{R}}(1+|s|)|\hat{g}_z(s)|ds,\\
                                                \|G'(z)\|_{r,\infty} &\leq \frac{r}{r-1}\int_{\mathbb{R}}(1+|s|)|\hat{g}_{1,z}(s)|ds,\\
            \Big\|\frac{G(z)-G(z_0)}{z-z_0}-G'(z_0)\Big\|_{r,\infty} &\leq \frac{r}{r-1}\int_{\mathbb{R}}(1+|s|)\Big|\frac{\hat{g}_z(s)-\hat{g}_{z_0}(s)}{z-z_0}-\hat{g}_{1,z_0}(s)\Big|ds.
        \end{align*}

        Once more using \cite[Lemma 7]{PS-crelle}, we have
        $$\int_{\mathbb{R}}(1+|s|)|\hat{g}_z(s)|ds=\|\hat{g}_z\|_1+\|\hat{g'}_z\|_1\leq c_{\mathrm{abs}}\|g_z\|_{W^{2,2}}.$$
        Similarly,
        $$\int_{\mathbb{R}}(1+|s|)|\hat{g}_{1,z}(s)|ds=\|\hat{g}_{1,z}\|_1+\|\hat{g'}_{1,z}\|_1\leq c_{abs}\|g_{1,z}\|_{W^{2,2}}$$
        and
        $$\int_{\mathbb{R}}(1+|s|)\Big|\frac{\hat{g}_z(s)-\hat{g}_{z_0}(s)}{z-z_0}-\hat{g}_{1,z_0}(s)\Big|ds=$$
        $$=\Big\|\frac{\hat{g}_z-\hat{g}_{z_0}}{z-z_0}-\hat{g}_{1,z_0}\Big\|_1+\Big\|\frac{\hat{g'}_z-\hat{g'}_{z_0}}{z-z_0}-\hat{g'}_{1,z_0}\Big\|_1\leq$$
        $$\leq c_{abs}\Big\|\frac{g_z-g_{z_0}}{z-z_0}-g_{1,z_0}\Big\|_{W^{2,2}}.$$
        We now deduce \eqref{gabb} from Theorem \ref{computational analytic lemma}.   
    \end{proof}

    \begin{lem}\label{another analytic lemma} 
        Let $p>2$ and let $A$ and $B$ satisfy Condition \ref{conditions for analyticity}. If $X\in\mathcal{L}_{\infty},$ then
        \begin{enumerate}[{\rm (i)}]
            \item\label{ano1} the mapping
                $$G_1(z) := \int_{\mathbb{R}}[BA^{\frac12},A^{\frac12+is}]Y^{-is}XB^{is}\hat{g}_z(s)ds,$$
                is an $\mathcal{L}_{\frac{p}{2},\infty}-$valued holomorphic function for $\Re(z)>1.$
            \item\label{ano2} the mapping
                $$G_2(z) := \int_{\mathbb{R}}A^{is}Y^{-is}XB^{is}\hat{g}_z(s)ds,$$
                is $\mathcal{L}_{\infty}-$valued holomorphic for $\Re(z)>1.$
            \item\label{ano3} the mapping
                $$G_3(z) := \int_{\mathbb{R}}Y^{-is}XB^{is}[BA^{\frac12},A^{\frac12+is}]\hat{g}_z(s)ds,$$
                is $\mathcal{L}_{\frac{p}{2},\infty}-$valued holomorphic for $\Re(z)>1.$
        \end{enumerate}
    \end{lem}
    \begin{proof} 
        By Lemma \ref{g abstract analytic lemma} \eqref{gaba}, the functions $G_1,$ $G_2$ and $G_3$ are $\mathcal{L}_{\infty}-$valued holomorphic. 
        In particular, this proves \eqref{ano2}. We now prove the first and third assertions.

        Set
        $$H_1(s)=[BA^{\frac12},A^{\frac12+is}]Y^{-is}XB^{is},\quad s\in\mathbb{R},$$
        and
        $$H_2(s)=Y^{-is}XB^{is}[BA^{\frac12},A^{\frac12+is}],\quad s\in\mathbb{R}.$$

        For every $s\in\mathbb{R},$ we have for $j = 1,2$,
        $$\Big\|H_j(s)\Big\|_{\frac{p}{2},\infty}\leq\|X\|_{\infty}\cdot\Big\|[BA^{\frac12},A^{\frac12+is}]\Big\|_{\frac{p}{2},\infty},$$
        Let $\phi_s(t) := |t|^{1+2is},$ $t\in\mathbb{R}.$ We now write
        $$[BA^{\frac12},A^{\frac12+is}]=[BA^{\frac12},\phi_s(A^{\frac12})].$$
        We again refer to \cite{PS-acta}. Since $p > 2$ and $\phi_s$ is a Lipschitz function, we have
        \begin{align*}
            \|[BA^{\frac12},A^{\frac12+is}]\|_{\frac{p}{2},\infty} &\leq c_p\|\phi_s'\|_{\infty}\|[BA,A^{\frac12}]\|_{\frac{p}{2},\infty}\\
                                                                   &\leq c_p(1+|s|)\|[BA,A^{\frac12}]\|_{\frac{p}{2},\infty}.
        \end{align*}
        Hence for $j = 1,2$ we have:
        $$\|H_j(s)\|_{\frac{p}{2},\infty}\leq c_p(1+|s|)\|[BA,A^{\frac12}]\|_{\frac{p}{2},\infty}\|X\|_{\infty}.$$
        The assertions \eqref{ano1} and \eqref{ano3} now follow from Lemma \ref{g abstract analytic lemma}.\eqref{gabb}
        upon taking $j=1$ for \eqref{ano1} and $j=2$ for \eqref{ano3}.
    \end{proof}

    \begin{lem}\label{integrability lemma} 
        Let $p>2$ and let $A$ and $B$ satisfy Condition \ref{conditions for analyticity}. Let $T_z(s)$, $s \in \mathbb{R}$ be defined as in Theorem \ref{csz key lemma}. Then if $\Re(z) > p-1$ we have:
        $$\int_{\mathbb{R}}\|T_z(s)\|_1\cdot |\hat{g}_z(s)|ds<\infty.$$
    \end{lem}
    \begin{proof} 
        We recall the definition of $T_z(s)$:
        $$T_z(s)= B^{z-1+is}[BA^{\frac{1}{2}},A^{z-\frac{1}{2}+is}]Y^{-is}+B^{is}[BA^{\frac{1}{2}},A^{\frac{1}{2}+is}]Y^{z-1-is},\quad s\in\mathbb{R}.$$
                
        Consider the first summand in the definition of $T_z(s)$. Using the Leibniz rule:
        \begin{align*}
            B^{z-1+is}[BA^{\frac12},A^{z-\frac12+is}]Y^{-is} &= B^{z-1+is}[BA^{\frac12},A^{z-1} A^{\frac12+is}]Y^{-is}\\
                                                             &= B^{is} B^{z-1}A^{z-1}[BA^{\frac12},A^{\frac12+is}] Y^{-is}\\
                                                             &\quad+B^{is} B^{z-1}[BA,A^{z-1}] A^{is}Y^{-is}.
        \end{align*}
        
        By the $\mathcal{L}_1$-triangle inequality, we have
        \begin{align*}
            \|B^{z-1+is}[BA^{\frac12}&,A^{z-\frac12+is}]Y^{-is}\|_1\\
                                     &\leq \|B^{z-1}A^{z-1}[BA^{\frac12},A^{\frac12+is}]\|_1+\|B^{z-1}[BA,A^{z-1}]\|_1.
        \end{align*}
        We now apply the H\"older-type inequality \eqref{another holder},
        \begin{align}
            \|B^{z-1+is}[BA^{\frac12}&,A^{z-\frac12+is}]Y^{-is}\|_1\nonumber\\
                                     &\leq\|B^{z-1}A^{z-1}\|_{\frac{p}{p-2},1}\|[BA^{\frac12},A^{\frac12+is}]\|_{\frac{p}{2},\infty}+\|B^{z-1}[BA,A^{z-1}]\|_1.\label{first big holder estimate}
        \end{align}
        
        Consider the function $\phi_s(t) = |t|^{1+2is}$, $t \in \mathbb{R}$. Immediately, $\phi_s$ is Lipschitz and $\|\phi_s'\|_{L_\infty} \leq 2(1+|s|).$ Since $p > 2$, we may apply the result of \cite{PS-acta} to obtain:
        \begin{equation*}
            \|[BA^{1/2},\phi(A^{1/2})]\|_{p/2,\infty} \leq C_p(1+|s|)\|[BA^{1/2},A^{1/2}]\|_{p/2,\infty}.
        \end{equation*}
        
        Therefore,
        \begin{equation}\label{introduction of s dependence}
            \|[BA^{\frac12},A^{\frac12+is}]\|_{\frac{p}{2},\infty} \leq C_p(1+|s|)\|[BA^{\frac12},A^{\frac12}]\|_{\frac{p}{2},\infty}.
        \end{equation}
        Combining \eqref{first big holder estimate} and \eqref{introduction of s dependence}, we have
        \begin{align}
            \|B^{z-1+is}[BA^{\frac12}&,A^{z-\frac12+is}]Y^{-is}\|_1\nonumber\\
                                     &\leq C_p(1+|s|)\|[BA^{\frac12},A^{\frac12}]\|_{\frac{p}{2},\infty}\|B^{z-1}A^{z-1}\|_{\frac{p}{p-2},1}+\|B^{z-1}[BA,A^{z-1}]\|_1.\label{step 1 result}
        \end{align}

        Let us now consider the second summand in $T_z(s).$ Using the H\"older inequality in the form of \eqref{another holder}, we obtain
        \begin{align}
            \|B^{is}[BA^{\frac12},A^{\frac12+is}]Y^{z-1-is}\|_1&\leq\|[BA^{\frac12},A^{\frac12+is}]Y^{z-1}\|_1\nonumber\\
                                                               &\leq\|[BA^{\frac12},A^{\frac12+is}]\|_{\frac{p}{2},\infty}\|Y^{z-1}\|_{\frac{p}{p-2},1}.\label{second big holder estimate}
        \end{align}
        Now combining \eqref{second big holder estimate} with \eqref{introduction of s dependence}, we arrive at:
        \begin{align}
            \|B^{is}[BA^{\frac12}&,A^{\frac12+is}]Y^{z-1-is}\|_1 \nonumber\\
                                 &\leq C_p(1+|s|)\|[BA^{\frac12},A^{\frac12}]\|_{\frac{p}{2},\infty}\|Y^{z-1}\|_{\frac{p}{p-2},1}.\label{step 2 result}
        \end{align}

        Now we may combine \eqref{step 1 result} and \eqref{step 2 result}:
        \begin{align*}
            \|T_z(s)\|_1 &\leq \|B^{z-1}[BA,A^{z-1}]\|_1\\
                         &\quad +c_{abs}(1+|s|)\|[BA^{\frac12},A^{\frac12}]\|_{\frac{p}{2},\infty}\cdot\Big(\|B^{z-1}A^{z-1}\|_{\frac{p}{p-2},1}+\|Y^{z-1}\|_{\frac{p}{p-2},1}\Big).
        \end{align*}
        By Lemma \ref{trivial analytic lemma},
        \begin{equation*}
            \|B^{z-1}[BA,A^{z-1}]\|_1,\|B^{z-1}A^{z-1}\|_{\frac{p}{p-2},1},\|Y^{z-1}\|_{\frac{p}{p-2},1}<\infty,\quad\Re(z)>p-1.
        \end{equation*}
        Hence,
        $$\|T_z(s)\|_1\leq C_{A,B,z}\cdot (1+|s|).$$
        We now have
        $$\int_{\mathbb{R}}\|T_z(s)\|_1\cdot |\hat{g}_z(s)|ds\leq C_{A,B,z}\int_{\mathbb{R}}(1+|s|)|\hat{g}_z(s)|ds.$$
        By \cite[Lemma 7]{PS-crelle}, we have
        $$\int_{\mathbb{R}}(1+|s|)|\hat{g}_z(s)|ds=\|g_z\|_2+\|g_z'\|_2\leq c_{abs}\|g_z\|_{W^{2,2}}.$$
        Thus,
        $$\int_{\mathbb{R}}\|T_z(s)\|_1 |\hat{g}_z(s)|ds\leq C_{A,B,z}\|g_z\|_{W^{2,2}}.$$
        The assertion follows now from Theorem \ref{computational analytic lemma}.
    \end{proof}
    
    \begin{cor}\label{difference is trace class}
        If $p > 2$ and $A$ and $B$ satisfy condition \ref{conditions for analyticity}, then for all $\Re(z) > p-1$ we have:
        \begin{equation*}
            B^zA^z-(A^{1/2}BA^{1/2})^z \in \mathcal{L}_1.
        \end{equation*}
    \end{cor}
    \begin{proof}
        The integral formula from Lemma \ref{csz key lemma} is valid for $\Re(z) > 1$. Since $p > 2$, we therefore have a weak operator topology integral:
        \begin{equation*}
            B^zA^z-(A^{1/2}BA^{1/2})^z = T_z(0)-\int_{\mathbb{R}} T_z(s)\widehat{g}_z(s)\,ds.
        \end{equation*}
        From Lemma \ref{integrability lemma}, we have that $$\int_{\mathbb{R}}\|T_z(s)\|_1\cdot |\hat{g}_z(s)|ds<\infty.$$
        Thus by Lemma \ref{peter lemma}, we have that $\int_{\mathbb{R}} T_z(s)\widehat{g}_z(s) \in \mathcal{L}_1$. 
        
        Recalling the definition of $T_z(0)$:
        \begin{equation*}
            T_z(0) := B^{z-1}[BA^{\frac{1}{2}},A^{z-\frac{1}{2}}]+[BA^{\frac{1}{2}},A^{\frac{1}{2}}]Y^{z-1}
        \end{equation*}
        So by \eqref{first big holder estimate} and \eqref{second big holder estimate}, $T_z(0) \in \mathcal{L}_1$. Therefore, $B^zA^z-(A^{1/2}BA^{1/2})^z \in \mathcal{L}_1$.
    \end{proof}

    \begin{lem}\label{splitting lemma} 
        Assume that $p>2$ and let $A$ and $B$ satisfy Condition \ref{conditions for analyticity}. If $X\in\mathcal{L}_{\infty},$ then for $\Re(z) > p$:
        $${\rm Tr}\Big(X\Big(B^zA^z-(A^{\frac12}BA^{\frac12})^z\Big)\Big)={\rm Tr}(XF_0(z))-\sum_{k=1}^3{\rm Tr}(G_k(z)F_k(z))$$
        Here, the functions $F_k$ are as in Lemma \ref{trivial analytic lemma} and the functions $G_k$ are as in Lemma \ref{another analytic lemma}.
    \end{lem}
    \begin{proof} 
        {By Lemma \ref{integrability lemma},
        \begin{equation*}
            \int_{\mathbb{R}} \|T_z(s)\|_1|\widehat{g}_z(s)|\,ds < \infty.
        \end{equation*}}
        So by Lemma \ref{peter lemma}, we have:
        $$\int_{\mathbb{R}}T_z(s)\widehat{g}_z(s)ds\in\mathcal{L}_1$$
        and
        $${\rm Tr}\Big(X\int_{\mathbb{R}}T_z(s)\widehat{g}_z(s)ds\Big)=\int_{\mathbb{R}}{\rm Tr}(XT_z(s))\widehat{g}_z(s)ds.$$
                
        By Theorem \ref{csz key lemma}, we have:
        \begin{equation}
            {\rm Tr}\Big(X\Big(B^zA^z-(A^{\frac12}BA^{\frac12})^z\Big)\Big) = {\rm Tr}(XT_z(0))-\int_{\mathbb{R}}{\rm Tr}(XT_z(s))\hat{g}_z(s)ds
        \end{equation}
        for $\Re(z)>p.$ 
        
        Observing that $F_0(z) = T_z(0)$, we have:
        \begin{equation}\label{csz with traces}
            {\rm Tr}\Big(X\Big(B^zA^z-(A^{\frac12}BA^{\frac12})^z\Big)\Big) = {\rm Tr}(X\cdot F_0(z))-\int_{\mathbb{R}}{\rm Tr}(XT_z(s))\hat{g}_z(s)ds.
        \end{equation}

        By \eqref{first big holder estimate} and \eqref{second big holder estimate}, we have
        \begin{align*}
             B^{z-1}[BA^{\frac{1}{2}},A^{z-\frac{1}{2}+is}] &\in \mathcal{L}_1,\\
              [BA^{\frac{1}{2}},A^{\frac{1}{2}+is}]Y^{z-1} &\in \mathcal{L}_1. 
        \end{align*}
        Now by the definition of $T_z(s),$ we have
        \begin{align*}
            \mathrm{Tr}(XT_z(s)) &= \mathrm{Tr}(Y^{-is}XB^{is}B^{z-1}[BA^{\frac12},A^{z-\frac12+is}])\\
                                 &\quad+ \mathrm{Tr}(Y^{-is}XB^{is}[BA^{\frac12},A^{\frac12+is}]Y^{z-1}).
        \end{align*}
        By an application of the Leibniz rule, we have
        $$B^{z-1}[BA^{\frac12},A^{z-\frac12+is}] = B^{z-1}A^{z-1}[BA^{\frac12},A^{\frac12+is}]+B^{z-1}[BA,A^{z-1}] A^{is}.$$
        Each of the above terms is $\mathcal{L}_1$, by \eqref{first big holder estimate} and Lemma \ref{trivial analytic lemma}.\eqref{triv2} respectively.
        Therefore,
        \begin{align*}
            {\rm Tr}(X\cdot T_z(s)) &= {\rm Tr}([BA^{\frac12},A^{\frac12+is}]Y^{-is}XB^{is} B^{z-1}A^{z-1})\\
                                    &+ {\rm Tr}(A^{is}Y^{-is}XB^{is} B^{z-1}[BA,A^{z-1}])+{\rm Tr}(Y^{-is}XB^{is}[BA^{\frac12},A^{\frac12+is}]Y^{z-1}).
        \end{align*}
        Thus,
        \begin{align*}
            \int_{\mathbb{R}}{\rm Tr}&(XT_z(s))\hat{g}_z(s)ds\\
                                     &={\rm Tr}\Big(\int_{\mathbb{R}}[BA^{\frac12},A^{\frac12+is}]Y^{-is}XB^{is}\hat{g}_z(s)ds\cdot B^{z-1}A^{z-1}\Big)\\
                                     &\quad +{\rm Tr}\Big(\int_{\mathbb{R}}A^{is}Y^{-is}XB^{is}\hat{g}_z(s)ds\cdot B^{z-1}[BA,A^{z-1}]\Big)\\
                                     &\quad +{\rm Tr}\Big(\int_{\mathbb{R}}Y^{-is}XB^{is}[BA^{\frac12},A^{\frac12+is}]\hat{g}_z(s)ds\cdot Y^{z-1}\Big).
        \end{align*}
        Using the notations from Lemma \ref{trivial analytic lemma} and Lemma \ref{another analytic lemma}, we may summarise the above equality as:
        \begin{equation}\label{integral as a sum of three traces}
            \int_{\mathbb{R}}{\rm Tr}(XT_z(s))\hat{g}_z(s)ds=\sum_{k=1}^3{\rm Tr}(G_k(z)F_k(z)).
        \end{equation}
        Combining \eqref{csz with traces} and \eqref{integral as a sum of three traces} completes the proof.
    \end{proof}
    
    We are now ready to complete the proof of Theorem \ref{analyticity theorem I}.

    \begin{proof}[Proof of Theorem \ref{analyticity theorem I}] 
        We will show that the function
        \begin{equation*}
            A(z) := {\rm Tr}(X\cdot F_0(z))-\sum_{k=1}^3{\rm Tr}(G_k(z)F_k(z))        
        \end{equation*}
        is analytic for $\Re(z) > p-1$.
        
        For the $F_0$ term, we use Lemma \ref{trivial analytic lemma}.\eqref{triv0}: the mapping $z\to XF_0(z)$ is $\mathcal{L}_1-$valued analytic for $\Re(z)>p-1.$ 
        For the $G_1F_1$ term, we use Lemma \ref{trivial analytic lemma}.\eqref{triv1} and Lemma \ref{another analytic lemma}.\eqref{ano1} to see that the mapping $z\to G_1(z)F_1(z)$ is $\mathcal{L}_1-$valued analytic for $\Re(z)>p-1.$ 
        For the $G_2F_2$ term, we use Lemma \ref{trivial analytic lemma}.\eqref{triv2} 
        and Lemma \ref{another analytic lemma}.\eqref{ano2} to see that the mapping $z\to G_2(z)F_2(z)$ is $\mathcal{L}_1-$valued analytic for $\Re(z)>p-1.$ 
        Finally, for the $G_3F_3$ term, we use Lemma \ref{trivial analytic lemma}.
        \eqref{triv3} and Lemma \ref{another analytic lemma}.\eqref{ano3} to see that mapping $z\to G_3(z)F_3(z)$ is $\mathcal{L}_1-$valued analytic for $\Re(z)>p-1.$
        
        Hence, $A$ is holomorphic in the set $\Re(z) > p-1$. By Lemma \ref{splitting lemma},
        \begin{equation*}
            A(z) = {\rm Tr}\Big(X\Big(B^zA^z-(A^{\frac12}BA^{\frac12})^z\Big)\Big)
        \end{equation*}
        and so the proof is complete.
    \end{proof}

\section{Criterion for universal measurability in terms of a $\zeta-$function}\label{subhankulov section}
    In this section we provide a sufficient condition for universal measurability of operators in $\mathcal{L}_{1,\infty}$. 
    
    We recall that a linear functional $\varphi$ on the weak Schatten ideal $\mathcal{L}_{1,\infty}$ is called
    a trace if for all unitary operators $U$ and $T \in \mathcal{L}_{1,\infty}$, we have $\varphi(U^*TU) = \varphi(T)$. Equivalently,
    for all bounded operators $A$ we have $\varphi(AT) = \varphi(TA)$. We say that $\varphi$ is normalised if 
    \begin{equation*}
        \varphi\left(\mathrm{diag}\left\{\frac{1}{n+1}\right\}_{n\geq 0}\right) = 1.
    \end{equation*}
    An operator $T \in \mathcal{L}_{1,\infty}$ is called \emph{universally measurable} if all normalised traces take the same value on $T$.
    
    In this section we prove Theorem \ref{zeta measurability theorem}, which provides a sufficient condition for operators of the form $AV$, $A \in \mathcal{L}_\infty$, $V \in \mathcal{L}_{1,\infty}$
    to be universally measurable. This result is new, and is sufficiently powerful to allow us to prove Theorem \ref{main thm}. A similar characterisation is provided in \cite[Theorem 4.13]{SUZ-indiana}, but the
    result provided here is stronger. A previously known characterisation of universal measurability
    in terms of a heat trace can be found in \cite[Proposition 6]{CRSZ}.
    
    Let $b$ be a signed Borel measure on $[0,\infty)$. Recall that $b$ can be written as a difference of two positive measures, $b = b_+-b_-$
    such that $b_+$ and $b_-$ are mutually singular to each other. Given a Borel set $S$, the total variation of $b$ on $S$ is defined to be $\mathrm{Var}_S(b) := b_+(S)+b_-(S)$.
    
    We will consider measures $b$ which satisfy,
    \begin{equation}\label{main measure condition}
        \sup_{x \geq 0} \mathrm{Var}_{[x,x+1]}(b) = c_b < \infty.
    \end{equation}
    Clearly any measure of finite total variation will satisfy this condition, as will some measures with infinite total variation such as Lebesgue measure and the measure $d\nu(t) = \sin(t)dt$.
    
    \begin{lem} 
        Let $b$ be a signed Borel measure satisfying the condition \eqref{main measure condition} and let $f$ be the Laplace transform of $b$, that is,
        \begin{equation*}
            f(z) := \int_0^\infty e^{-tz}\,db(t),\quad \Re(z) > 0.
        \end{equation*}
        The function $f$ is analytic on the half-plane $\{\Re(z)>0\}.$    
    \end{lem}
    \begin{proof} 
        For every $n\geq0,$ let
        $$f_n(z)=\int_0^ne^{-tz}db(t),\quad z\in\mathbb{C}.$$
        
        Each function $f_n,$ $n\geq0,$ is entire. Let $\varepsilon > 0.$ For $\Re(z)>\varepsilon,$ we have
        $$|f(z)-f_n(z)|=|\int_n^{\infty}e^{-tz}db(t)|\leq\sum_{k\geq n}e^{-k\varepsilon}\mathrm{Var}_{[k,k+1]}(b)\leq c_b\frac{e^{-n\varepsilon}}{1-e^{-\varepsilon}}.$$
        Therefore, $f_n\to f$ uniformly on the half-plane $\{\Re(z)>\varepsilon\}.$ Since $\varepsilon>0$ is arbitrary, it follows that $f_n\to f$ uniformly on the compact subsets of the half-plane $\{\Re(z)>0\}.$ Thus, $f$ is analytic on the half-plane $\{\Re(z)>0\}.$
    \end{proof}
    
    For $x \geq 0$, we denote
    \begin{equation*}
        b(x) := b([0,x]).
    \end{equation*}
    We caution the reader that $b(x)$ does not denote $b(\{x\})$ nor the Radon-Nikodym derivative of $b$ at $x$.
    
    The following lemma is very similar to \cite[Lemma 2.1.3]{subhankulov}. However we require a slightly different formulation of that result and we were unable to find an English-language
    version of its proof, and so we include a self-contained proof here.
    \begin{lem}\label{subhankulov key estimate} 
        Let $b$ be a signed Borel measure on $[0,\infty)$ satisfying \eqref{main measure condition}, and and let $f$ be the Laplace transform of $b.$
        For $x\geq 1$, $x \to\infty$, we have
        \begin{equation*}
            b(x) = \frac1{2\pi}\int_{-1}^1\frac{(1-t^2)^2}{\frac1x+it}f(\frac1x+it)e^{(\frac1x+it)x}dt+O(1),\quad x\geq1.
        \end{equation*}
    \end{lem}
    \begin{proof} 
        Let $x \geq 1$. By definition we have:
        \begin{equation*}
            \int_{-1}^1\frac{(1-t^2)^2}{\frac1x+it}f(\frac1x+it)e^{(\frac1x+it)x}dt = \int_{-1}^1\int_0^{\infty}\frac{(1-t^2)^2}{\frac1x+it}e^{(\frac1x+it)(x-s)}db(s)dt.
        \end{equation*}
        Examining the integrand, we see that the function
        \begin{equation*}
            (s,t) \mapsto \frac{(1-t^2)^2}{x^{-1}+t}e^{(\frac{1}{x}+it)(x-s)}
        \end{equation*}
        is bounded above in absolute value by
        \begin{equation*}
            (s,t)\mapsto exe^{-s/x}.
        \end{equation*}
        Since $x \geq 0$ the function $s\to e^{-s/x}$ is in $L_1([0,\infty),b)$ we may interchange the integrals to get:
        \begin{equation*}
            \int_{-1}^1 \frac{(1-t^2)^2}{\frac{1}{x}+it}f(\frac{1}{x}+it)e^{(\frac{1}{x}+it)x}\,dt = \int_0^\infty \left(\int_{-1}^1 \frac{(1-t^2)^2}{\frac{1}{x}+it}e^{(\frac{1}{x}+it)(x-s)}\,dt\right)\,db(s)
        \end{equation*}
        
        Now we refer to Proposition \ref{subhankulov compute}, where it is proved that:
        $$\frac1{2\pi}\int_{-1}^1\frac{(1-t^2)^2}{\frac1x+it}e^{(\frac1x+it)(x-s)}dt=(1+x^{-2})^2\chi_{[0,x]}(s)+\min\{1,(x-s)^{-2}\}\cdot O(1).$$
        Thus,
        \begin{align*}
            \frac1{2\pi}\int_{-1}^1\frac{(1-t^2)^2}{\frac1x+it}f(\frac1x+it)e^{(\frac1x+it)x}dt &= \int_0^x(1+x^{-2})^2db(s)\\
                                                                                                &+ \int_0^{\infty}\min\{1,(x-s)^{-2}\}\cdot O(1)\,db(s).
        \end{align*}
        We let $h(s)$ be the total variation of $b$ on $[0,s]$. That is, $h(s) := \mathrm{Var}_{[0,s]}(b)$. Then by the triangle inequality, there is a positive constant $c_{\mathrm{abs}}$ such that
        \begin{align*}
            \Big|\int_0^{\infty}\min\{1,(x-s)^{-2}\}\cdot O(1)\cdot db(s)\Big|
                            &\leq c_{abs}\cdot \int_0^{\infty}\min\{1,(x-s)^{-2}\}dh(s)\\
                            &=c_{abs}\cdot\sum_{k\geq0}\int_k^{k+1}\min\{1,(x-s)^{-2}\}dh(s)\\   
                            &\leq c_{abs}\cdot\sum_{k\geq0}\sup_{s\in[k,k+1]}\min\{1,(x-s)^{-2}\}\cdot\int_k^{k+1}dh(s)\\
                            &\leq c_{abs}\cdot c_b\cdot\sum_{k\geq0}\sup_{s\in[k,k+1]}\min\{1,(x-s)^{-2}\}\\
                            &= O(1)
        \end{align*}
        Thus,
        \begin{equation}\label{subhankulov partial computation}
            \frac1{2\pi}\int_{-1}^1\frac{(1-t^2)^2}{\frac1x+it}f(\frac1x+it)e^{(\frac1x+it)x}dt=\int_0^x(1+x^{-2})^2db(s)+O(1).
        \end{equation}
        
        By the definition of $b(x)$,
        $$|b(x)-b(0)|\leq\mathrm{Var}_{[0,x]}(b)\leq c_b(1+x),\quad x > 0.$$
        Therefore,
        $$\int_0^x(1+x^{-2})^2db(s)=(1+x^{-2})^2\cdot (b(x)-b(0))=b(x)+O(1),\quad x\geq1.$$
        A combination of the latter inequality with \eqref{subhankulov partial computation} completes the proof.
    \end{proof}
    
    The following Lemma is similar in spirit to (but much stronger than) the well-known Wiener-Ikehara Tauberian theorem \cite[Theorem 14.1]{Shubin-pseudo-2001}.
    \begin{lem}\label{subhankulov main lemma} 
        Let $b$ be a signed Borel measure on $[0,\infty)$ satisfying \eqref{main measure condition}, and let $f$
        be the Laplace transform of $b.$ If there exists $\epsilon>0$ such that $f$ has analytic continuation to a half-plane $\{z \;:\; \Re(z)>-\epsilon\},$ then
        for $x \geq 1$
        $$b(x)=O(1),\quad x\to\infty.$$
    \end{lem}
    \begin{proof} 
        We write $f(z)=f(0)+zf_0(z),$ where $f_0$ is an analytic function on the half-plane $\{z\;:\;\Re(z)>-\epsilon\}.$ In particular, since the (closure of the) set
        $$\Big\{\frac1x+it:\ x\geq1,\ t\in[-1,1]\Big\},$$
        is a compact subset in $\{\Re(z)>-\epsilon\},$ it follows that
        \begin{equation}\label{the integrand is bounded}
            \sup_{x\geq 1}\sup_{t\in[-1,1]}|f_0(\frac1x+it)|<\infty.
        \end{equation}
        The assertion of Lemma \ref{subhankulov key estimate} is now written as follows.
        \begin{align}
            \frac1{2\pi}\int_{-1}^1(1-t^2)^2f_0(\frac1x+it)&e^{(\frac1x+it)x}dt+\frac1{2\pi}\int_{-1}^1\frac{(1-t^2)^2}{\frac1x+it}f(0)e^{(\frac1x+it)x}dt\nonumber\\
                                                           &=b(x)+O(1),\quad x\to\infty.\label{subhankulov main lemma 1}
        \end{align}

        The first summand in the left hand side is bounded for $x\geq 1.$ due to \eqref{the integrand is bounded}. By Proposition \ref{subhankulov compute}, we have
        \begin{equation}\label{subhankulov main lemma 2}
            \frac1{2\pi}\int_{-1}^1\frac{(1-t^2)^2}{\frac1x+it}e^{(\frac1x+it)x}dt=1+O(x^{-2}),\quad x\to \infty.
        \end{equation}
        Combining \eqref{subhankulov main lemma 1} and \eqref{subhankulov main lemma 2}, we get:
        \begin{equation*}
            b(x)+O(1) = f(0)+O(1),\quad x\to\infty.
        \end{equation*}
        So $b(x) = O(1)$ as $x\to\infty$.
    \end{proof}
    
    In order to continue our discussion of measurability we refer to the concept of a modulated operator. This
    theory was introduced in \cite{KLPS} and is developed extensively in \cite[Section 11.2]{LSZ}. If $V \in \mathcal{L}_{1,\infty}$ is positive,
    and $T \in \mathcal{L}_\infty$, we say that $T$ is $V$-modulated if
    \begin{equation*}
        \sup_{t > 0} t^{1/2}\|T(1+tV)^{-1}\|_{\mathcal{L}_2} < \infty.
    \end{equation*}
    It can be easily seen that if $T$ is $V$-modulated, and $A \in \mathcal{L}_\infty$, then $AT$ is $V$-modulated (this is also \cite[Proposition 11.2.2]{LSZ}).
    It is proved in \cite[Lemma 11.2.8]{LSZ} that $V$ is $V$-modulated, and therefore that $AV$ is $V$-modulated.
    
    The relevance of the notion of a $V$-modulated operator to measurability comes from \cite[Theorem 11.2.3]{LSZ}, which states
    that if $V \geq 0$ is in $\mathcal{L}_{1,\infty}$, $\ker(V) = 0$, $T$ is $V$-modulated and $\{e_n\}_{n\geq 0}$ is an eigenbasis for $V$ ordered
    so that $Ve_n = \mu(n,V)e_n$ for $n\geq 0$, then
    \begin{enumerate}[{\rm (a)}]
        \item $T \in \mathcal{L}_{1,\infty}$ and $\mathrm{diag}\{\langle Te_n,e_n\rangle\}_{n=0}^\infty \in \mathcal{L}_{1,\infty}$
        \item We have:
            \begin{equation*}
                \sum_{k=0}^n \lambda(k,T) - \sum_{k=0}^n \langle Te_k,e_k\rangle = O(1),\quad n\to\infty.
            \end{equation*}
    \end{enumerate}
    Recall that $\{\lambda(k,T)\}_{k= 0}^\infty$ denotes an eigenvalue sequence for $T$, ordered with non-increasing absolute value.

    \begin{lem}\label{simple klps lemma} 
        Let $0\leq V\in\mathcal{L}_{1,\infty}$ satisfy $\ker(V) = 0$ and let $A\in\mathcal{L}_{\infty}.$ Let $\{e_k\}_{k=0}^\infty$ be an eigenbasis for $V$ ordered
        such that $Ve_k = \mu(k,V)e_k$. We have        
        \begin{equation*}
            \sum_{k=0}^n\lambda(k,AV) = \sum_{\mu(k,V)>\frac{\|V\|_{1,\infty}}{n}}\langle Ae_k,e_k\rangle\mu(k,V)+O(1).
        \end{equation*}
        Here, $e_k$ is the eigenvector of $V$ corresponding to the eigenvalue $\mu(k,V).$
    \end{lem}
    \begin{proof} 
        We have that $AV$ is $V-$modulated. \cite[Theorem 11.2.3]{LSZ} now states that as $n\to\infty$
        \begin{align*}
            \sum_{k=0}^n\lambda(k,AV) &= \sum_{k=0}^n\langle AVe_k,e_k\rangle+O(1)\\
                                      &=\sum_{k=0}^n\langle Ae_k,e_k\rangle\mu(k,V)+O(1).
        \end{align*}
        For $n\geq 1$, let
        $$m(n)=\max\{k\in\mathbb{N}:\ \mu(k,V)>\frac{\|V\|_{1,\infty}}{n}\}.$$
        Using this notation, we write
        $$\sum_{\mu(k,V)>\frac{\|V\|_{1,\infty}}{n}}\langle Ae_k,e_k\rangle\mu(k,V)=\sum_{k=0}^{m(n)}\langle Ae_k,e_k\rangle\mu(k,V).$$

        We have $\mu(k,V)\leq\frac{\|V\|_{1,\infty}}{k+1}$ for every $k\geq0$ and, therefore,
        $$m(n)\leq \max\Big\{k\in\mathbb{Z}_+:\ \frac{\|V\|_{1,\infty}}{k+1}>\frac{\|V\|_{1,\infty}}{n}\Big\}=n-2<n.$$
        On the other hand, we have $\mu(k,V)\leq\frac{\|V\|_{1,\infty}}n$ for all $k>m(n)$. Thus,
        \begin{align*}
            \Big|\sum_{k=m(n)+1}^n\langle Ae_k,e_k\rangle\mu(k,V)\Big| &\leq \|A\|_{\infty} \sum_{k=m(n)+1}^n\mu(k,V)\\
                                                                    &\leq \frac{\|A\|_{\infty}\|V\|_{1,\infty}}{n}\sum_{k=m(n)+1}^n1\\
                                                                    &= O(1).
        \end{align*}
        Finally, we have
        \begin{align*}
            \sum_{k=0}^n\lambda(k,AV) &= \sum_{k=0}^n\langle Ae_k,e_k\rangle\mu(k,V)+O(1)\\
                                      &= \sum_{k=0}^{m(n)}\langle Ae_k,e_k\rangle\mu(k,V)+O(1)\\
                                      &= \sum_{\mu(k,V)>\frac{\|V\|_{1,\infty}}{n}}\langle Ae_k,e_k\rangle\mu(k,V)+O(1).
        \end{align*}
    \end{proof}
    
    We now conclude with the proof of Theorem \ref{zeta measurability theorem}, a sufficient condition for universal measurability in terms of a $\zeta$-function.
    {Recall that if $f$ is a meromorphic function of a complex variable $z$ with a simple pole at $z = 0$, then $\mathrm{Res}_{z=0} f(z)$ denotes the coefficient of $z^{-1}$ in the Laurent
    expansion of $f$, or equivalently the value of $zf(z)$ at $z=0$.}
    \begin{thm*}
        Let $0\leq V\in\mathcal{L}_{1,\infty}$ and let $A\in\mathcal{L}_{\infty}.$ 
        Define the $\zeta$-function: 
        \begin{equation*}
            \zeta_{A,V}(z) := \mathrm{Tr}(AV^{1+z}),\quad \Re(z) > 0.
        \end{equation*}
        If there exists $\varepsilon > 0$ such that $\zeta_{A,V}$ admits an analytic continuation to the set $\{z\;:\; \Re(z) > -\varepsilon\}\setminus \{0\}$
        with a simple pole at $0$, then for every normalised trace $\varphi$ on $\mathcal{L}_{1,\infty}$ we have:
        \begin{equation*}
            \varphi(AV)= \mathrm{Res}_{z=0}\zeta_{A,V}(z).
        \end{equation*}
        In particular, $AV$ is universally measurable.
    \end{thm*}
    \begin{proof} 
        Assume without loss of generality that $\ker(V) = 0$. Select an orthonormal basis $\{e_k\}_{k=0}^\infty$
        such that $Ve_k = \mu(k,V)e_k$.
    
        We define:
        $$b(t):=-t\mathrm{Res}_{w=0}\zeta_{A,V}(w)+\sum_{\mu(k,V)>e^{-t}}\langle Ae_k,e_k\rangle\mu(k,V),\quad t \geq 0.$$
        Since $b$ is a linear combination of monotone functions, it is of locally bounded variation. Hence there is a signed Borel measure $b$
        such that $b([0,x]) = b(x)$, $x \geq 0$.
        
        We first prove that $b$ satisfies \eqref{main measure condition}. If $x \geq 0$, then:        
        \begin{align*}
            \mathrm{Var}_{[x,x+1]}b &\leq |\mathrm{Res}_{w=0}\zeta_{A,V}(w)|+\sum_{\mu(k,V)\in(e^{-1-x},e^{-x}]}|\langle Ae_k,e_k\rangle|\mu(k,V)\\
                            &\leq |\mathrm{Res}_{w=0}\zeta_{A,V}(w)|+2\|A\|_{\infty}e^{-x}\sum_{\mu(k,V)\in(e^{-1-x},e^{-x}]}1\\
                            &\leq |\mathrm{Res}_{w=0}\zeta_{A,V}(w)|+2\|A\|_{\infty}e^{-x}\sum_{\mu(k,V)\in(e^{-1-x},\infty]}1.
        \end{align*}
        Since $V\in\mathcal{L}_{1,\infty},$ we have that
        $$\sum_{\mu(k,V)\in(e^{-1-x},\infty]}1\leq e^{1+x}\|V\|_{1,\infty},\quad x\in\mathbb{R}.$$
        Therefore,
        $$\mathrm{Var}_{[x,x+1]}b\leq |\mathrm{Res}_{w=0}\zeta_{A,V}(w)|+2e\|A\|_{\infty}\|V\|_{1,\infty}.$$
        So $b$ indeed satisfies \eqref{main measure condition}.

        Let $\alpha_k:= \log(\frac1{\mu(k,V)})$ and $b_k:=\langle Ae_k,e_k\rangle\mu(k,V),$ then
        the function $b$ has a jump discontinuity at the point $\alpha_k$ of magnitude $b_k.$ Let $\Re(z) > 0$. Using the identity $e^{-\alpha_k z} = \mu(k,V)^z$, we have:
        \begin{align*}
        \int_0^{\infty}e^{-zt}db(t) &= \sum_{k\geq0}e^{-\alpha_k z}\cdot b_k-\mathrm{Res}_{w=0}\zeta_{A,V}(w)\int_0^{\infty}e^{-zt}dt\\
                                    &= \sum_{k\geq0}\mu(k,V)^z\cdot\langle Ae_k,e_k\rangle\mu(k,V)-\mathrm{Res}_{w=0}\zeta_{A,V}(w)\int_0^{\infty}e^{-zt}dt\\
                                    &= \sum_{k\geq0}\langle AV^{1+z}e_k,e_k\rangle-\mathrm{Res}_{w=0}\zeta_{A,V}(w)\int_0^{\infty}e^{-zt}dt\\
                                    &= \mathrm{Tr}(AV^{1+z})-\frac1z\mathrm{Res}_{w=0}\zeta_{A,V}(w).
        \end{align*}
        By assumption, the above right hand side has analytic continuation to the set $\{z\;:\;\Re(z) > -\varepsilon\}$.

        Thus, since $b$ satisfies \eqref{main measure condition} the assumptions of Lemma \ref{subhankulov main lemma} are satisfied, and we may then conclude that $b(t)=O(1),$ for $t\to\infty$. Thus
        by the definition of $b$:,
        $$\sum_{\mu(k,V)>e^{-t}}\langle Ae_k,e_k\rangle\mu(k,V)=t\cdot\mathrm{Res}_{w=0}\zeta_{A,V}(w)+O(1),\quad t \to\infty.$$
        Setting $e^{-t}=\frac{\|V\|_{1,\infty}}{n},$ we obtain
        $$\sum_{\mu(k,V)>\frac{\|V\|_{1,\infty}}{n}}\langle Ae_k,e_k\rangle\mu(k,V)=\log(n)\cdot \mathrm{Res}_{w=0}\zeta_{A,V}(w)+O(1),\quad t\to\infty.$$
        By Lemma \ref{simple klps lemma}, we have
        $$\sum_{k=0}^n\lambda(k,AV)=\log(n)\cdot \mathrm{Res}_{w=0}\zeta_{A,V}(w)+O(1),\quad t\to\infty.$$
        The assertion follows now from Theorem \ref{universal measurability criterion}.
    \end{proof}

\section[Proof of the main result, $p>2$]{Proof of Theorem \ref{main thm}, $p>2$}\label{main thm p>2}

    In this section we complete the proof of Theorem \ref{main thm} under the restriction that $p > 2$. We require this restriction
    in order to directly apply the results of Section \ref{difference section}. We will handle the $p = 1$ and $p=2$ cases separately in the next section.

    \begin{lem}\label{verification p>2} 
        Let $(\mathcal{A},H,D)$ be a spectral triple satisfying Hypothesis \ref{main assumption}. 
        Let $0\leq a\in\mathcal{A}.$ If $p>2,$ then the operators $A=a^2$ and $B=(1+D^2)^{-\frac12}$ satisfy Condition \ref{conditions for analyticity}.
    \end{lem}
    \begin{proof} 
        Let $D_0=F(1+D^2)^{\frac12}.$ By Lemma \ref{pass to spectral gap 1}, the spectral triple $(\mathcal{A},H,D_0)$ satisfies Hypothesis \ref{main assumption}. 
        Since $|D_0| \geq 1$, we have that $\||D_0|^{-1}\|_\infty \leq 1$.

        Let us establish Condition \ref{conditions for analyticity}.\eqref{anacond1}. 
        We have
        $$B^pA=|D_0|^{-p}a^2.$$
        Since $(\mathcal{A},H,D)$ is $p$-dimensional, we have that $|D_0|^{-p}a \in \mathcal{L}_{1,\infty}$, and so $B^pA \in \mathcal{L}_{1,\infty}$.

        Next let us prove Condition \ref{conditions for analyticity}.\eqref{anacond2}. Let $q>p.$ We have
        \begin{align*}
            B^{q-2}[B,A] &= |D_0|^{2-q}[|D_0|^{-1},a^2]\\
                         &= -|D_0|^{1-q}\delta_0(a^2)|D_0|^{-1}\\
                         &= -|D_0|^{1-q}\delta_0(a)a|D_0|^{-1}-|D_0|^{1-q}a\delta_0(a)|D_0|^{-1}.
        \end{align*}
        Referring to Lemma \ref{z<p lemma}, we have
        $$|D_0|^{1-q}\delta_0(a) \in \mathcal{L}_{p/(q-1),\infty},\quad |D_0|^{1-q}a \in \mathcal{L}_{p/(q-1),\infty},$$
        $$\delta_0(a)|D_0|^{-1} \in \mathcal{L}_{p,\infty},\quad a|D_0|^{-1} \in \mathcal{L}_{p,\infty}.$$
        Therefore,
        \begin{align*}
            B^{q-2}[B,A] &\in \mathcal{L}_{p/(q-1),\infty}\cdot \mathcal{L}_{p,\infty} + \mathcal{L}_{p/(q-1),\infty}\cdot \mathcal{L}_{p,\infty}
        \end{align*}
        So by the H\"older inequality, $B^{q-2}[B,A] \in \mathcal{L}_{p/q,\infty}$. Since $q > p$, it then follows that $B^{q-2}[B,A] \in \mathcal{L}_1$.
        
        Now we establish Condition \ref{conditions for analyticity}.\eqref{anacond3}. We have
        $$A^{\frac12}BA^{\frac12}=a|D_0|^{-1}a.$$
        Thus,
        $$\|A^{\frac12}BA^{\frac12}\|_{p,\infty}\leq\|a\|_{\infty}\|a|D_0|^{-1}\|_{p,\infty}.$$
        The above right hand side is finite by Lemma \ref{z<p lemma}.

        Finally we prove Condition \ref{conditions for analyticity}.\eqref{anacond4}. We recall the notation that $\delta_0(a)$ denotes the bounded extension of $[|D_0|,a]$. 
        Using \eqref{favourite commutator identity}, we have:
        $$[B,A^{\frac12}]=[|D_0|^{-1},a]=-|D_0|^{-1}\delta_0(a)|D_0|^{-1}.$$
        By Theorem 9 in \cite{sbik}, we have
        $$|D_0|^{-1}\delta_0(a)|D_0|^{-1}\prec\prec \delta_0(a)|D_0|^{-2}.$$
        By Lemma \ref{z<p lemma}, we have 
        $$\delta_0(a)|D_0|^{-2}\in\mathcal{L}_{\frac{p}{2},\infty}.$$
        Since the norm in the space $\mathcal{L}_{\frac{p}{2},\infty}$ is monotone with respect to the Hardy-Littlewood submajorisation (recall that $p>2$), it follows that also
        $$[B,A^{\frac12}]=-|D_0|^{-1}\delta_0(a)|D_0|^{-1}\in\mathcal{L}_{\frac{p}{2},\infty}.$$
    \end{proof}
    
    Now we may prove Theorem \ref{main thm} for the case $p > 2$.
    \begin{thm*}[\ref{main thm}, $p > 2$ case]
        Assume $p > 2$ and let $(\mathcal{A},H,D)$ be a spectral triple satisfying Hypothesis \ref{main assumption}. If $c\in\mathcal{A}^{\otimes (p+1)}$ is a local Hochschild cycle, then
        for every normalised trace $\varphi$ on $\mathcal{L}_{1,\infty}$ we have:
        \begin{equation*}
            \varphi(\Omega(c)(1+D^2)^{-\frac{p}{2}})=\mathrm{Ch}(c).
        \end{equation*}
    \end{thm*}
    \begin{proof}
        By Theorem \ref{zeta thm}, the function
        \begin{equation*}
            \zeta_{c,D}(z) := \mathrm{Tr}(\Omega(c)(1+D^2)^{-z/2}),\quad \Re(z) > p
        \end{equation*}
        admits an analytic continuation to the set $\{z\;:\;\Re(z) > p-1\}\setminus \{p\}$, and $p$ is a simple pole for $\zeta_{c,D}$
        with residue $p\mathrm{Ch}(c)$.

        Let $c = \sum_{j=1}^m a^{j}_0\otimes a^j_1\otimes\cdots\otimes a^j_p$. We assume that $c$ is local, i.e. that there exists $0 \leq a \in \mathcal{A}$ such that 
        for all $j$ we have $aa^j_0 = a^j_0$. Equivalently, $(1-a)a^j_0 = 0$. So $\mathrm{im}(a^j_0) \subseteq \ker(1-a)$. Since the support projection $\mathrm{supp}(1-a)$
        is the projection onto the orthogonal complement of the kernel, we have:
        \begin{equation*}
            \mathrm{supp}(1-a)a^j_0 = 0.
        \end{equation*}
        By functional calculus, $\mathrm{supp}(1-a) = 1-\chi_{\{1\}}(a)$, so moreover we have that $\chi_{\{1\}}(a)a^j_0 = a^j_0$. Therefore, for all $z \in \mathbb{C}$ with $\Re(z) > 0$:
        $$a^{2z}a_0^j=a^{2z}\chi_{\{1\}}(a)a_0^j=1^{2z}a_0^j=a_0^j.$$
        Recall that $\Omega(c) = \sum_{j=0}^m \Gamma a_0^j\partial(a_1^j)\cdots \partial(a_p^j)$. Since $a$ commutes with $\Gamma$, we have for all $\Re(z) > 0$,
        \begin{equation}\label{omega is local}
            a^{2z}\Omega(c) = \Omega(c).
        \end{equation}
        Let $A=a^2$ and $B=(1+D^2)^{-\frac12}$, as in Lemma \ref{verification p>2}. Then $B^zA^z = (1+D^2)^{-z/2}a^{2z}$, and hence
        \begin{align*}
            \mathrm{Tr}(\Omega(c)B^zA^z) &= \mathrm{Tr}(\Omega(c)(1+D^2)^{-\frac{z}{2}}a^{2z})\\
                                 &= \mathrm{Tr}(a^{2z}\Omega(c)(1+D^2)^{-\frac{z}{2}})\\
                                 &= \mathrm{Tr}(\Omega(c)(1+D^2)^{-\frac{z}{2}}).
        \end{align*}
        By Theorem \ref{zeta thm}, it then follows that $z\mapsto \mathrm{Tr}(\Omega(c)B^zA^z)$        
        admits an analytic continuation to the set $\{\Re(z)>p-1\}\setminus \{p\}$ with a simple pole at $z = p$
        and residue $p\mathrm{Ch}(c)$.
        
        Condition \ref{conditions for analyticity} holds for $A$ and $B$ by Lemma \ref{verification p>2}. Hence we may apply Theorem \ref{analyticity theorem I} to conclude that
        $$z\to{\rm Tr}\Big(\Omega(c)\Big(B^zA^z-(A^{\frac12}BA^{\frac12})^z\Big)\Big)$$
        admits an analytic continuation to the set $\{\Re(z)>p-1\}.$

        By Lemma \ref{verification p>2}, $A^{\frac12}BA^{\frac12}\in\mathcal{L}_{p,\infty}.$ Hence, the function (defined {\it a priori} for $\Re(z)>p$)
        $$z\to{\rm Tr}(\Omega(c)(A^{\frac12}BA^{\frac12})^z)$$
        admits an analytic continuation to the set $\{\Re(z)>p-1\}\setminus \{p\}$, with $z=p$ being a simple pole with residue $p\mathrm{Ch}(c).$ 
        Consider $V=(A^{\frac12}BA^{\frac12})^p\in\mathcal{L}_{1,\infty}.$ 
        It has just been demonstrated that
        $$z\to{\rm Tr}(\Omega(c)V^z)$$
        admits an analytic continuation to the set $\{\Re(z)>1-\frac1p\}\setminus \{1\},$ with a simple pole at $z=1$ and the corresponding residue being $\mathrm{Ch}(c).$

        By Theorem \ref{zeta measurability theorem}, we therefore have
        $$\varphi(\Omega(c)V)=\mathrm{Ch}(c)$$
        for every normalised trace $\varphi$ on $\mathcal{L}_{1,\infty}.$

        By Corollary \ref{difference is trace class} we have:
        $$V-B^pA^p=(A^{\frac12}BA^{\frac12})^p-B^pA^p\in\mathcal{L}_1.$$
        Since $\varphi$ vanishes on $\mathcal{L}_1,$ it follows that
        $$\varphi(\Omega(c)B^pA^p)=\varphi(\Omega(c)V)-\varphi(\Omega(c)(V-B^pA^p))=\varphi(\Omega(c)V)$$
        for every normalised trace $\varphi$ on $\mathcal{L}_{1,\infty}.$ Now using \eqref{omega is local} with $z = p$:
        $$\varphi(\Omega(c)(1+D^2)^{-\frac{p}{2}})=\varphi(a^{2p}\Omega(c)(1+D^2)^{-\frac{p}{2}})=\varphi(\Omega(c)B^pA^p)=\mathrm{Ch}(c)$$
        for every normalised trace $\varphi$ on $\mathcal{L}_{1,\infty}.$
    \end{proof}

\section[Proof of the main result, $p=1,2$]{Proof of Theorem \ref{main thm}, $p=1,2$}\label{main thm p=1,2}

    In this final section we complete the proof of Theorem \ref{main thm} by dealing with the remaining cases of $p = 1$ and $p=2$.
    { We require adjustment for these cases since Theorem \ref{analyticity theorem I} is inapplicable for $p \leq 2$.}

    \begin{lem}\label{13 lemma} 
        Let $(\mathcal{A},H,D)$ satisfy Hypothesis \ref{main assumption}. 
        Suppose that spectrum of the operator $D$ does not intersect the interval $(-1,1).$ Then for all $x \in \mathrm{dom}(\delta)$ we have an absolute constant $c_{\mathrm{abs}}$ such that
        \begin{equation*}
            \|[|D|^{\frac13},x]\|_1\leq c_{\mathrm{abs}}\|[|D|,x]\|_1,
        \end{equation*}
        and for all $r \in (1,\infty)$ a constant $c_r > 0$ such that
        \begin{equation*}
            \|[|D|^{\frac13},x]\|_{r,\infty}\leq c_r\|[|D|,x]\|_{r,\infty}.
        \end{equation*}
        These inequalities are understood to be trivially true if the right hand side is infinite.
    \end{lem}
    \begin{proof} 
        We only prove the first assertion. One can prove the second inequality by an identical argument, with the $\mathcal{L}_{r,\infty}$ quasi-norm in place of the $\mathcal{L}_1$ norm.

        Let $f$ be a smooth function on $\mathbb{R}$ such that $f(t)=|t|^{\frac13}$ for $|t|>1.$ For $\varepsilon > 0$, set $f_{\varepsilon}(t)=f(t)e^{-\varepsilon^2t^2},$
        $t\in\mathbb{R}.$ Then,
        \begin{align*}
               f_\varepsilon'(t) &= (f'(t)-2\varepsilon^2tf(t))e^{-\varepsilon^2t^2},\\
            f_{\varepsilon}''(t) &= (f''(t)-4\varepsilon^2tf'(t)+(4t\varepsilon^4-2\varepsilon^2)f(t))e^{-\varepsilon^2t^2}
        \end{align*}
        Since for $|t| > 1$ we have $f'(t) = \frac{1}{3}|t|^{-2/3}$, we have that as $\varepsilon \to 0$ the $L_2$-norm
        $\|f_\varepsilon'\|_{L_2(\mathbb{R})}$ is uniformly bounded. Similarly, since for $|t| > 1$, $f''(t) = -\frac{2}{9}|t|^{-5/3}$, we
        also have that $\|f_{\varepsilon}''(t)\|_{L_2(\mathbb{R})}$ is uniformly bounded.
        
        We have that (see e.g. \cite[Lemma 7]{PS-crelle}):
        $$\|\widehat{f_{\epsilon}'}\|_1\leq c_{abs}\Big(\|f_{\epsilon}'\|_2+\|f_{\epsilon}''\|_2\Big).$$
        and so if $\varepsilon \in(0,1)$, $\|\widehat{f_{\varepsilon}'}\|_1$ is uniformly bounded.

        By Lemma \ref{first commutator rep lemma} (taken with $s=1$), we have the identity:
        $$[f_{\epsilon}(|D|),x]=\int_{-\infty}^{\infty}\Big(\int_0^1\hat{f_{\epsilon}'}(u)e^{iu(1-v)|D|}\delta(x)e^{iuv|D|}dv\Big)du.$$
        So taking the $\mathcal{L}_1$-norm, we conclude from Lemma \ref{peter norm lemma} that
        \begin{align*}
            \Big\|[f_{\epsilon}(|D|),x]\Big\|_{1} &\leq \|\hat{f_{\epsilon}'}\|_1\|\delta(x)\|_{1}\\
                                                  &\leq c_{\mathrm{abs}}\|\delta(x)\|_{1}.
        \end{align*}

        Fix $N>0.$ We have
        $$\Big\|\chi_{[0,N]}(|D|)[f_{\epsilon}(|D|),x]\chi_{[0,N]}(|D|)\Big\|_{r,\infty}\leq\frac{c_{abs}r}{r-1}\|\delta(x)\|_{r,\infty}.$$
        Since as $\varepsilon \to 0$, we have that $f_{\varepsilon}$ converges uniformly to $f$ on the set $[0,N]$, we have that:
        As $\epsilon\to0,$ we have
        $$\chi_{[0,N]}(|D|)[f_{\epsilon}(|D|),x]\chi_{[0,N]}(|D|)\to \chi_{[0,N]}(|D|)[|D|^{\frac13},x]\chi_{[0,N]}(|D|)$$
        in the operator norm. By the Fatou property of the $\mathcal{L}_1$-norm:
        $$\Big\|\chi_{[0,N]}(|D|)\cdot[|D|^{\frac13},x]\cdot\chi_{[0,N]}(|D|)\Big\|_{1} \leq c_{\mathrm{abs}}\|\delta(x)\|_{1}.$$

        Since the above inequality is true for arbitrary $N>0,$ we may take the limit $N\to\infty$ and again using the Fatou property of the $\mathcal{L}_1$ norm, we arrive at
        $$\Big\|[|D|^{\frac13},x]\Big\|_{1}\leq c_{\mathrm{abs}}\|\delta(x)\|_{1}.$$
    \end{proof}

    As a replacement for Lemma \ref{verification p>2} in the $p=1$ case we use the following:
    \begin{lem}\label{verification p=1} 
        Let $(\mathcal{A},H,D)$ be a $1-$dimensional spectral triple satisfying Hypothesis \ref{main assumption}. 
        If $0\leq a \in\mathcal{A},$ then the operators $A=a^4$ and $B=(1+D^2)^{-\frac16}$ satisfy Condition \ref{conditions for analyticity} with $p=3.$
    \end{lem}
    \begin{proof} 
        This proof is similar to that of Lemma \ref{verification p>2}.
    
        Let $D_0=F(1+D^2)^{\frac12}.$ By Lemma \ref{pass to spectral gap 1}, the $1-$dimensional spectral triple $(\mathcal{A},H,D_0)$ satisfies Hypothesis \ref{main assumption}. Since $|D_0| \geq 1$,
        we have that $\||D_0|^{-1}\|_\infty \leq 1$.

        Let us establish Condition \ref{conditions for analyticity}.\eqref{anacond1}. We have
        $$B^pA=|D_0|^{-1}a^4\in\mathcal{L}_{1,\infty}$$
        since by assumption $(\mathcal{A},H,D)$ is $1$-dimensional.

        Next we establish Condition \ref{conditions for analyticity}.\eqref{anacond2}. Let $q\in(3,4).$ Using \eqref{favourite commutator identity}, we have on $H_\infty$:
        \begin{align*}
            B^{q-2}[B,A] &= |D_0|^{\frac{2-q}{3}}[|D_0|^{-\frac13},a^4]\\
                         &= -|D_0|^{\frac{1-q}{3}}[|D_0|^{\frac13},a^4]|D_0|^{-\frac13}\\
                         &= -[|D_0|^{\frac13},|D_0|^{\frac{1-q}3}a^4|D_0|^{-\frac13}].
        \end{align*}
        By Lemma \ref{13 lemma}, we have
        $$\|B^{q-2}[B,A]\|_1\leq c_{abs}\|[|D_0|,|D_0|^{\frac{1-q}3}a^4|D_0|^{-\frac13}]\|_1.$$
        Still working on $H_\infty$, we also have:
        \begin{align*}
            [|D_0|,|D_0|^{\frac{1-q}3}a^4|D_0|^{-\frac13}] &= |D_0|^{\frac{1-q}3}\delta_0(a^4)|D_0|^{-\frac13}\\
                                                           &= |D_0|^{\frac{1-q}3}\delta_0(a^2)a^2|D_0|^{-\frac13}+|D_0|^{\frac{1-q}3}a^2\delta_0(a^2)|D_0|^{-\frac13}.
        \end{align*}
        Applying Lemma \ref{z<p lemma}, we have
        $$[|D_0|,|D_0|^{\frac{1-q}3}a^4|D_0|^{-\frac13}]\in\mathcal{L}_{\frac{3}{q-1},\infty}\cdot\mathcal{L}_{3,\infty}\subset\mathcal{L}_{\frac{3}{q},\infty}$$
        by the Holder inequality, since $q > 3$, $\mathcal{L}_{3/q,\infty} \subset \mathcal{L}_1$, and so $B^{q-2}[B,A] \in \mathcal{L}_1$.

        Now we establish Condition \ref{conditions for analyticity}.\eqref{anacond3}. We may compute:
        $$A^{\frac12}BA^{\frac12}=a^2|D_0|^{-\frac13}a^2.$$
        Thus,
        $$\|A^{\frac12}BA^{\frac12}\|_{3,\infty}\leq\|a\|_{\infty}^3\|a|D_0|^{-\frac13}\|_{3,\infty}.$$
        The right hand side is finite by Lemma \ref{z<p lemma}.

        Finally we verify Condition \ref{conditions for analyticity}.\eqref{anacond4}. We have:
        \begin{align*}
            [B,A^{\frac12}] &= [|D_0|^{-\frac13},a^2]\\
                            &= -|D_0|^{-\frac13}[|D_0|^{\frac13},a^2]|D_0|^{-\frac13}\\
                            &= -[|D_0|^{\frac13},|D_0|^{-\frac13}a^2|D_0|^{-\frac13}].
        \end{align*}
        By Lemma \ref{13 lemma}, we have
        $$\|[B,A^{\frac12}]\|_{\frac{3}{2},\infty}\leq c_{\mathrm{abs}}\|[|D_0|,|D_0|^{-\frac13}a^2|D_0|^{-\frac13}]\|_{\frac{3}{2},\infty}.$$
        Moreover, by the Leibniz rule
        \begin{align*}
            [|D_0|,|D_0|^{-\frac13}a^2|D|^{-\frac13}] &= |D_0|^{-\frac13}\delta_0(a^2)|D_0|^{-\frac13}\\
                                                      &= |D_0|^{-\frac13}\delta_0(a)a|D_0|^{-\frac13}+|D_0|^{-\frac13}a\cdot\delta_0(a)|D_0|^{-\frac13}.
        \end{align*}
        By Lemma \ref{z<p lemma} and the H\"older inequality, we have
        $$[|D_0|,|D_0|^{-\frac13}a^2|D_0|^{-\frac13}]\in\mathcal{L}_{3,\infty}\cdot\mathcal{L}_{3,\infty}\subset\mathcal{L}_{\frac{3}{2},\infty}.$$
    \end{proof}
    
    For the case $p=2$, we instead use:
    \begin{lem}\label{verification p=2} 
        Let $(\mathcal{A},H,D)$ be a $2-$dimensional spectral triple satisfying Hypothesis \ref{main assumption}. If $0\leq a\in\mathcal{A},$ then the operators $A=a^4$ and $B=(1+D^2)^{-\frac16}$ satisfy Condition \ref{conditions for analyticity} with $p=6.$
    \end{lem}
    \begin{proof} 
        This proof is similar to those of Lemmas \ref{verification p>2} and \ref{verification p=1}.
    
        Let $D_0=F(1+D^2)^{\frac12}.$ By Lemma \ref{pass to spectral gap 1}, the $2-$dimensional spectral triple $(\mathcal{A},H,D_0)$ satisfies Hypothesis \ref{main assumption}. By rescaling $D$ if necessary,
        we may assume without loss of generality that the spectrum of $D_0$ does not intersect the interval $(-1,1)$.

        Let us establish Condition \ref{conditions for analyticity}.\eqref{anacond1}. We have
        $$B^pA=|D_0|^{-2}a^4\in\mathcal{L}_{1,\infty}$$
        by Hypothesis \ref{main assumption}.

        Next we establish Condition \ref{conditions for analyticity}.\eqref{anacond2}. Let $q\in(6,7).$ We have
        \begin{align*}
            B^{q-2}[B,A] &= |D_0|^{\frac{2-q}{3}}[|D_0|^{-\frac13},a^4]\\
                         &= -|D_0|^{\frac{1-q}{3}}[|D_0|^{\frac13},a^4]|D_0|^{-\frac13}\\
                         &= -[|D_0|^{\frac13},|D_0|^{\frac{1-q}3}a^4|D_0|^{-\frac13}].
        \end{align*}
        By Lemma \ref{13 lemma}, we have
        $$\|B^{q-2}[B,A]\|_1\leq c_{abs}\|[|D_0|,|D_0|^{\frac{1-q}3}a^4|D_0|^{-\frac13}]\|_1.$$
        However,
        \begin{align*}
            [|D_0|,|D_0|^{\frac{1-q}3}a^4|D_0|^{-\frac13}] &= |D_0|^{\frac{1-q}3}\delta_0(a^4)|D_0|^{-\frac13}\\
                                                           &= |D_0|^{\frac{1-q}3}\delta_0(a^2)a^2|D_0|^{-\frac13}+|D_0|^{\frac{1-q}3}a^2\delta_0(a^2)|D_0|^{-\frac13}.
        \end{align*}
        By Lemma \ref{z<p lemma}, we have by the H\"older inequality:
        $$[|D_0|,|D_0|^{\frac{1-q}3}a^4|D|^{-\frac13}]\in\mathcal{L}_{\frac{6}{q-1},\infty}\cdot\mathcal{L}_{6,\infty}\subset\mathcal{L}_{\frac{6}{q},\infty}$$
        Since $q > 6$, we have that $\mathcal{L}_{6/q,\infty} \subset \mathcal{L}_1$, and so $B^{q-2}[B,A] \in \mathcal{L}_1$.

        Now we prove Condition \ref{conditions for analyticity}.\eqref{anacond3}. We have
        $$A^{\frac12}BA^{\frac12} = a^2|D_0|^{-\frac13}a^2.$$
        Thus,
        $$\|A^{\frac12}BA^{\frac12}\|_{6,\infty}\leq\|a\|_{\infty}^3\|a|D_0|^{-\frac13}\|_{6,\infty}.$$
        The above right hand side is finite by Lemma \ref{z<p lemma}.

        Finally, let us establish Condition \ref{conditions for analyticity}.\eqref{anacond4}. We may compute on $H_\infty$:
        \begin{align*}
            [B,A^{\frac12}] &= [|D_0|^{-\frac13},a^2]\\
                            &= -|D_0|^{-\frac13}[|D_0|^{\frac13},a^2]|D_0|^{-\frac13}\\
                            &= -[|D_0|^{\frac13},|D_0|^{-\frac13}a^2|D_0|^{-\frac13}].
        \end{align*}
        Therefore using Lemma \ref{13 lemma}, we have
        $$\|[B,A^{\frac12}]\|_{3,\infty}\leq c_{abs}\|[|D_0|,|D_0|^{-\frac13}a^2|D_0|^{-\frac13}]\|_{3,\infty}.$$
        Applying the Leibniz rule,
        \begin{align*}
            [|D_0|,|D_0|^{-\frac13}a^2|D_0|^{-\frac13}] &= |D_0|^{-\frac13}\delta_0(a^2)|D_0|^{-\frac13}\\
                                                        &= |D_0|^{-\frac13}\delta_0(a)a|D_0|^{-\frac13}+|D_0|^{-\frac13}a\delta_0(a)|D_0|^{-\frac13}.
        \end{align*}
        By Lemma \ref{z<p lemma}, we then have from the H\"older inequality:
        $$[D_0,|D_0|^{-\frac13}a^2|D_0|^{-\frac13}]\in\mathcal{L}_{6,\infty}\cdot\mathcal{L}_{6,\infty}\subset\mathcal{L}_{3,\infty}.$$
    \end{proof}
    
    We may now at last complete the proof of Theorem \ref{main thm}.
    
    \begin{thm*}
        Assume $p =1$ or $p=2$ and let $(\mathcal{A},H,D)$ be a spectral triple satisfying Hypothesis \ref{main assumption}. If $c\in\mathcal{A}^{\otimes (p+1)}$ is a local Hochschild cycle, then
        for every normalised trace $\varphi$ on $\mathcal{L}_{1,\infty}$ we have:
        \begin{equation*}
            \varphi(\Omega(c)(1+D^2)^{-\frac{p}{2}})=\mathrm{Ch}(c).
        \end{equation*}
    \end{thm*}
    \begin{proof}
        By Theorem \ref{zeta thm}, the function
        \begin{equation*}
            \zeta_{c,D}(z) = \mathrm{Tr}(\Omega(c)(1+D^2)^{-z/2}),\quad \Re(z) > p
        \end{equation*}
        admits an analytic continuation to the set $\{z\;:\;\Re(z)>p-1\}\setminus \{p\}$, and the point $p$ is a simple
        pole with corresponding residue $p\mathrm{Ch}(c)$.

        Let $c = \sum_{j=1}^m a_0^j\otimes \cdots\otimes a_p^j$. Since $c$ is local, we may choose $0 \leq a \in \mathcal{A}$ such that $aa_0^j = a_0^j$ for all $j$.
        
        By exactly the same argument as in the $p>2$ case, we can show that for all $\Re(z) > 0$:
        \begin{equation}\label{omega is still local}
            a^{4z}\Omega(c)=\Omega(c).
        \end{equation}

        We let $A=a^4$ and $B=(1+D^2)^{-\frac16}$ as in Lemmas \ref{verification p=1} and \ref{verification p=2}.
        
        We have that:
        \begin{align*}
            \mathrm{Tr}(\Omega(c)B^zA^z) &= \mathrm{Tr}(a^{4z}\Omega(c)(1+D^2)^{-\frac{z}{6}})\\
                                 &= \mathrm{Tr}(\Omega(c)(1+D^2)^{-\frac{z}{6}}),\quad \Re(z) > 3p.
        \end{align*}
        We recognise the above function as being precisely $z\mapsto \zeta_{c,D}(z/3)$.
        Hence by Theorem \ref{zeta thm}, the function $z\mapsto \mathrm{Tr}(\Omega(c)B^zA^z)$ admits an analytic continuation to the set $\{z \;:\; \Re(z) > 3(p-1)\}\setminus \{3p\}$,
        with a simple pole at $3p$ with corresponding residue $3p\mathrm{Ch}(c)$.
        
        Assume now that $p=1.$ By Lemma \ref{verification p=1}, Condition \ref{conditions for analyticity} holds for $A$ and $B$ (with $p=3$). By Theorem \ref{analyticity theorem I} the function
        $$z\to{\rm Tr}\Big(\Omega(c)\Big(B^zA^z-(A^{\frac12}BA^{\frac12})^z\Big)\Big)$$
        admits an analytic continuation to the set $\{\Re(z)>2\}.$
        
        By Lemma \ref{verification p=1}, $A^{\frac12}BA^{\frac12}\in\mathcal{L}_{3,\infty}.$ Hence, the function (defined {\it a priori} for $\Re(z)>3$)
        $$z\to{\rm Tr}(\Omega(c)(A^{\frac12}BA^{\frac12})^z)$$
        admits an analytic continuation to the set $\{\Re(z)>2\}\setminus \{3\}.$ with a pole at $z=3$ and corresponding residue $3\mathrm{Ch}(c).$ 
        Define $V_1:=(A^{\frac12}BA^{\frac12})^3\in\mathcal{L}_{1,\infty}.$ 
        Then,
        \begin{equation*}
            \mathrm{Tr}(\Omega(c)(A^{1/2}BA^{1/2})^{z}) = \mathrm{Tr}(\Omega(c)V_1^{z/3}).
        \end{equation*}
        We now know that function
        $$z\mapsto{\rm Tr}(\Omega(c)V_1^{z/3})$$
        admits an analytic continuation to the set $\{\Re(z)>2\}\setminus \{3\}$ with a simple pole at $3$ and corresponding residue $3\mathrm{Ch}(c)$. So by rescaling the argument, 
        we can equivalently say that the function $$z\mapsto \mathrm{Tr}(\Omega(c)V_1^z)$$ has analytic continuation to the set $\{z\;:\;\Re(z) > 2/3\}\setminus \{1\}$ with a simple
        pole at $1$ with corresponding residue $\mathrm{Ch}(c)$.
        
        Thus by Theorem \ref{zeta measurability theorem}, for any continuous normalised trace $\varphi$ on $\mathcal{L}_{1,\infty}$ we have
        \begin{equation*}
            \varphi(\Omega(c)V_1) = \mathrm{Ch}(c).
        \end{equation*}
        Due to Lemma \ref{verification p=1}, we have that $V_1-B^3A^3 \in \mathcal{L}_1$, and since $\varphi$ vanishes on $\mathcal{L}_1$ it follows that
        \begin{equation*}
            \varphi(\Omega(c)B^3A^3) =\mathrm{Ch}(c).
        \end{equation*}
        So
        $$\varphi(a^{12}\Omega(c)(1+D^2)^{-1/2}) = \mathrm{Ch}(c).$$
        By taking $z = 3$ in \eqref{omega is still local}, we have that $a^{12}\Omega(c) = \Omega(c)$, this completes the proof in the case $p=1$.
        
        Now assume that $p=2$.      
        By Lemma \ref{verification p=2}, Condition \ref{conditions for analyticity} holds for $A$ and $B$ (with $p=6$). 
        By Theorem \ref{analyticity theorem I} the function
        $$z\to{\rm Tr}\Big(\Omega(c)\Big(B^zA^z-(A^{\frac12}BA^{\frac12})^z\Big)\Big)$$
        admits an analytic extension to the set $\{\Re(z)>5\}.$

        By Lemma \ref{verification p=2}, $A^{\frac12}BA^{\frac12}\in\mathcal{L}_{6,\infty}.$ Hence, the function (defined {\it a priori} for $\Re(z)>6$)
        $$z\to{\rm Tr}(\Omega(c)(A^{\frac12}BA^{\frac12})^z)$$
        admits an analytic extension to the set $\{\Re(z)>5\}\setminus\{6\}.$ The point $z=6$ is a simple pole with corresponding residue $6\mathrm{Ch}(c).$ Consider $V_2=(A^{\frac12}BA^{\frac12})^6\in\mathcal{L}_{1,\infty}.$ 
        We have so far shown that the function
        $$z\to{\rm Tr}(\Omega(c)V_2^z)$$
        admits an analytic extension to the set $\{z\;:\;\Re(z)>\frac56\}\setminus \{1\}.$ The point $z=1$ is a simple pole with corresponding residue $\mathrm{Ch}(c)$.
        
        Hence, by Theorem \ref{zeta measurability theorem}, for any continuous normalised trace $\varphi$ on $\mathcal{L}_{1,\infty}$, we have
        \begin{equation*}
            \varphi(\Omega(c)V_2) = \mathrm{Ch}(c).
        \end{equation*}
        
        By Lemma \ref{verification p=2}, the operator $V_2-B^6A^6$ is trace class. Thus,
        \begin{equation*}
            \varphi(\Omega(c)B^6A^6) = \mathrm{Ch}(c).
        \end{equation*}
        So $\varphi(a^{12}\Omega(c)(1+D^2)^{-1}) = \mathrm{Ch}(c)$. Since $a^{12}\Omega(c) = \Omega(c)$, this completes the proof for the case $p=2$.
    \end{proof}

\appendix
\chapter{Appendix}

\section{Properties of the algebra $\mathcal{B}$}\label{b prop app}
    For this section, $(\mathcal{A},H,D)$ is a smooth spectral triple.
    Recall from Definition \ref{smoothness definition} that $\mathcal{B}$ is the $*-$algebra generated by all elements of the form $\delta^k(a)$ or $\partial(\delta^k(a)),$ $k\geq0,$ $a\in\mathcal{A}.$
    Recall that we define $H_\infty := \bigcap_{k\geq 1} \mathrm{dom}(D^k)$, and that for all $T \in \mathcal{B}$, we have $T:H_\infty\to H_\infty$, and for all $k\geq 0$ we have $D^k,|D|^k:H_\infty\to H_\infty$.
    
    The following should be compared with \cite[Lemma 6.2]{CPRS2}. See also the discussion following \cite[Lemma 10.22]{GVF}.
    \begin{lem}\label{left to right lemma}
        For every $x\in\mathcal{B}$ and for every $m\geq0,$ we have the following equalities
        of linear (potentially unbounded) operators on $H_\infty$:
        \begin{align*}
            |D|^mx &= \sum_{k=0}^m \binom{m}{k}\delta^{m-k}(x)|D|^k\text{ and, }\\
            x|D|^m &= \sum_{k=0}^m (-1)^{m-k}\binom{m}{k}|D|^k\delta^{m-k}(x).
        \end{align*}
    \end{lem}
    \begin{proof}
        We prove only the first equality as the proof of the second one follows by an identical argument.

        This formula can be seen by induction on $m.$ Indeed, for $m=1$, this is simply the claim
        that since $\mathcal{B} \subseteq \mathrm{dom}_\infty(\delta)$ we have an equality of operators on $H_\infty$:
        \begin{equation*}
            |D|x = \delta(x)+x|D|.
        \end{equation*}
        
        Now suppose that the claim is true for $m-1$. Then on $H_\infty$ we have
        \begin{align*}
            |D|^mx &= |D|\cdot|D|^{m-1}x \\
                   &= |D|\cdot\sum_{k=0}^{m-1}\binom{m-1}{k}\delta^{m-1-k}(x)|D|^k\\
                   &= \sum_{k=0}^{m-1}\binom{m-1}{k}|D|\delta^{m-1-k}(x)\cdot|D|^k\\
                   &= \sum_{k=0}^{m-1}\binom{m-1}{k}\delta^{m-1-k}(x)|D|^{k+1}+\sum_{k=0}^{m-1}\binom{m-1}{k}\delta^{m-k}(x)|D|^k\\
                   &= \sum_{k=1}^m\binom{m-1}{k-1}\delta^{m-k}(x)|D|^k+\sum_{k=0}^{m-1}\binom{m-1}{k}\delta^{m-k}(x)|D|^k\\
                   &= \sum_{k=0}^m\binom{m}{k}\delta^{m-k}(x)|D|^k.
        \end{align*} 
        and so the statement follows for $m$.
    \end{proof}

    \begin{lem}\label{left to right corollary} 
        Let $(\mathcal{A},H,D)$ be a smooth spectral triple, and assume that $D$ has a spectral gap at $0.$ Then for all $x \in \mathcal{B}$ and $m\geq 0$ we have
        \begin{enumerate}[{\rm (i)}]
            \item\label{left to right 1} The operators $|D|^{-m}x|D|^m,|D|^mx|D|^{-m}:H_\infty\to H_\infty$ have bounded extension
            \item\label{left to right 2} $|D|^{1-m}[|D|^m,x]:H_\infty\to H_\infty$ has bounded extension.
        \end{enumerate}
    \end{lem}
    \begin{proof} 
        By Lemma \ref{left to right lemma}, on $H_\infty$ we have:
        \begin{align*}
            |D|^mx|D|^{-m} &= \sum_{k=0}^m\binom{m}{k}\delta^{m-k}(x)|D|^{k-m},\\
            |D|^{-m}x|D|^m &= \sum_{k=0}^m(-1)^{m-k}\binom{m}{k}|D|^{k-m}\delta^{m-k}(x).  
        \end{align*}
        Clearly, the expressions on the right hand side have bounded extension. This proves the first assertion.

        By Lemma \ref{left to right lemma}, we have
        \begin{align*}
            [|D|^m,x] &= |D|^mx-\sum_{k=0}^m(-1)^{m-k}\binom{m}{k}|D|^k\delta^{m-k}(x)\\
                      &= \sum_{k=0}^{m-1}(-1)^{m-k-1}\binom{m}{k}|D|^k\delta^{m-k}(x).
        \end{align*}
        Therefore,
        $$|D|^{1-m}[|D|^m,x]=\sum_{k=0}^{m-1}(-1)^{m-k-1}\binom{m}{k}|D|^{k+1-m}\delta^{m-k}(x).$$
        Since $x \in \mathrm{dom}_\infty(\delta)$, for each $0 \leq k \leq m-1$ the operator $|D|^{k+1-m}\delta^{m-k}(x)$ has bounded extension. Hence, $|D|^{1-m}[|D|^m,x]$
        has bounded extension.
    \end{proof}
    
    \begin{lem}\label{schwartz lemma} 
        Assume that $(\mathcal{A},H,D)$ satisfies Hypothesis \ref{main assumption}. Let $h$ be a Borel function on $\mathbb{R}$ such that
        \begin{equation*}
            t \mapsto (1+t^2)^{\frac{p+1}{2}}h(t), \quad t \in \mathbb{R}.
        \end{equation*}
        is bounded.
        Then for all $x \in \mathcal{B}$ and $s > 0$ the operator $xh(sD)$ is in $\mathcal{L}_1$, and:
        \begin{equation*}
            \|xh(sD)\|_1=O(s^{-p}),\quad s\downarrow0.
        \end{equation*}
    \end{lem}
    \begin{proof} 
        Let $s > 0$. Clearly,
        $$(1+s^2D^2)^{-\frac{p+1}{2}}=|(1-isD)^{-p-1}|.$$
        Setting $\lambda=\frac1s,$ we obtain from Hypothesis \ref{main assumption} that:
        \begin{align*}
            \|x(1+s^2D^2)^{-\frac{p+1}{2}}\|_1 &= s^{-p-1}\|x(D+\frac{i}{s})^{-p-1}\|_1\\
                                               &= s^{-p-1}\cdot O(s)\\
                                               &= O(s^{-p}),\quad s\downarrow0.
        \end{align*}

        Since the operator $(1+s^2D^2)^{\frac{p+1}{2}}h(sD)$ is bounded, with
        \begin{equation*}
            \|(1+s^2D^2)^{\frac{p+1}{2}}h(sD)\|_\infty \leq \sup_{t \in \mathbb{R}} (1+t^2)^{\frac{p+1}{2}}|h(t)|
        \end{equation*}
        We can conclude that:
        \begin{align*}
            \|xh(sD)\|_1 &\leq \|x(1+s^2D^2)^{-\frac{p+1}{2}}\|_1\|(1+s^2D^2)^{\frac{p+1}{2}}h(sD)\|_\infty\\
                         &= O(s^{-p}),\quad s\downarrow 0.
        \end{align*}
    \end{proof}
    We note in particular that the assumption on $h$ in Lemma \ref{schwartz lemma} is satisfied if $h$ is a Schwartz function.

    \begin{lem}\label{first decay lemma} 
        Let $(\mathcal{A},H,D)$ satisfy Hypothesis \ref{main assumption}, assume that $D$ has a spectral gap at $0$ and let $x\in\mathcal{B}$. For all non-negative integers $m_1,m_2 \geq 0$ with $m_1+m_2<p,$ we have
        \begin{equation*}
            \||D|^{-m_1}x|D|^{-m_2}e^{-s^2D^2}\|_1 = O(s^{m_1+m_2-p}),\quad s\downarrow0.
        \end{equation*}
    \end{lem}
    \begin{proof} 
        Suppose first that $m_1>0.$ Using the triangle inequality:
        \begin{align*}
            \||D|^{-m_1}x|D|^{-m_2}e^{-s^2D^2}\|_1 &\leq \sum_{k,l=0}^{\infty}\|\chi_{[2^k,2^{k+1}]}(|D|)|D|^{-m_1}x|D|^{-m_2}e^{-s^2D^2}\chi_{[2^l,2^{l+1}]}(|D|)\|_1\\
                                                   &\leq \sum_{k,l=0}^{\infty}2^{-km_1-lm_2}\cdot e^{-2^{2l}s^2}\cdot\|\chi_{[2^k,2^{k+1}]}(|D|)x\chi_{[2^l,2^{l+1}]}(|D|)\|_1\\
                                                   &\leq \sum_{k,l=0}^{\infty}2^{-km_1-lm_2}\cdot e^{-2^{2l}s^2}\cdot\|\chi_{[0,2^{k+1}]}(|D|)x\chi_{[0,2^{l+1}]}(|D|)\|_1.
        \end{align*}

        If $m=\min\{k,l\},$ then
        \begin{align*}
            \|\chi_{[0,2^{k+1}]}(|D|)x\chi_{[0,2^{l+1}]}(|D|)\Big\|_1 &\leq \min\Big\{\Big\|x\chi_{[0,2^{m+1}]}(|D|)\Big\|_1,\Big\|\chi_{[0,2^{m+1}]}(|D|)x\Big\|_1\Big\}\\
                                                                      &\leq e\min\Big\{\Big\|xe^{-2^{-2(m+1)}|D|^2}\Big\|_1,\Big\|x^*e^{-2^{-2(m+1)}|D|^2}\Big\|_1\Big\}\\
                                                                      &= O(2^{mp}).
        \end{align*}

        Thus,
        \begin{equation*}
            \||D|^{-m_1}x|D|^{-m_2}e^{-s^2D^2}\|_1 \leq \Big(\sum_{k,l=0}^{\infty}2^{-km_1-lm_2}\cdot e^{-2^{2l}s^2}\cdot 2^{p\cdot\min\{k,l\}}\Big)\cdot O(1).
        \end{equation*}

        Since $p-m_1>0,$ it follows that
        \begin{align*}
            \sum_{\substack{k,l\geq0\\ k\leq l}}2^{-km_1-lm_2}\cdot e^{-2^{2l}s^2}\cdot 2^{p\cdot\min\{k,l\}} &= \sum_{\substack{k,l\geq0\\ k\leq l}}2^{(p-m_1)k}\cdot 2^{-lm_2}\cdot e^{-2^{2l}s^2}\\
                                                                                                              &\leq 2\cdot \sum_{l=0}^{\infty}2^{(p-m_1)l}\cdot 2^{-lm_2}\cdot e^{-2^{2l}s^2}.
        \end{align*}
        Now due to our assumption that $m_1>0,$ it follows that
        \begin{align*}
            \sum_{\substack{k,l\geq0\\ k\geq l}}2^{-km_1-lm_2}\cdot e^{-2^{2l}s^2}\cdot 2^{p\cdot\min\{k,l\}} &= \sum_{\substack{k,l\geq0\\ k\geq l}}2^{-km_1}\cdot 2^{(p-m_2)l}\cdot e^{-2^{2l}s^2}\\
                                                                                                              &\leq 2\cdot \sum_{l=0}^{\infty}2^{-lm_1}\cdot 2^{(p-m_2)l}\cdot e^{-2^{2l}s^2}.
        \end{align*}
        Note that for $m > 0$, we have
        \begin{equation*}
            \sum_{l=0}^\infty 2^{lm}e^{-2^{2l}s^2} = O(s^{-m}), \quad s\downarrow 0.
        \end{equation*}

        Therefore,
        \begin{align*}
            \||D|^{-m_1}x|D|^{-m_2}e^{-s^2D^2}\|_1 &\leq \Big(\sum_{l=0}^{\infty}2^{l(p-m_1-m_2)}\cdot e^{-2^{2l}s^2}\Big)\cdot O(1)\\
                                                   &= O(s^{m_1+m_2-p}).
        \end{align*}
        This completes the proof for the case $m_1>0.$

        To complete the proof we now deal with the case where $m_1=0.$ We have
        \begin{align*}
               \||D|^{-m_1}x|D|^{-m_2}e^{-s^2D^2}\|_1 &\leq \sum_{l=0}^{\infty}\Big\|x|D|^{-m_2}e^{-s^2D^2}\chi_{[2^l,2^{l+1}]}(|D|)\Big\|_1\\
                                                      &\leq \sum_{l=0}^{\infty}2^{-lm_2}\cdot e^{-2^{2l}s^2}\cdot\Big\|x\chi_{[2^l,2^{l+1}]}(|D|)\Big\|_1\\
                                                      &\leq \sum_{l=0}^{\infty}2^{-lm_2}\cdot e^{-2^{2l}s^2}\cdot\Big\|x\chi_{[0,2^{l+1}]}(|D|)\Big\|_1.
        \end{align*}
        Thus,
        \begin{align*}
            \||D|^{-m_1}x|D|^{m_2}e^{-s^2D^2}\|_1 &\leq \Big(\sum_{l=0}^{\infty}2^{l(p-m_2)}\cdot e^{-2^{2l}s^2}\Big)\cdot O(1)\\
                                                  &= O(s^{m_2-p}).
        \end{align*}
        This completes the proof for the case $m_1=0.$
    \end{proof}

    \begin{lem}\label{left convergence estimate} 
        Let $(\mathcal{A},H,D)$ be a spectral triple satisfying Hypothesis \ref{main assumption} and assume that $D$ has spectral gap at $0.$ For all $x\in\mathcal{B}$ we have
        \begin{equation*}
            \|x|D|^{-p-1}(1-e^{-s^2D^2})\|_1 = O(s),\quad s\downarrow0.
        \end{equation*}
    \end{lem}
    \begin{proof}
        By the triangle inequality, we have
        \begin{align}
            \|x|D|^{-p-1}(1-e^{-s^2D^2})\|_1 &\leq \|x|D|^{-p-1}(1-e^{-s^2D^2})\chi_{(\frac1s,\infty)}(|D|)\|_1\nonumber\\
                                             &\quad +\|x|D|^{-p-1}(1-e^{-s^2D^2})\chi_{[0,\frac1s]}(|D|)\|_1\label{split half line}.
        \end{align}

        Let us estimate the first summand of \eqref{split half line}. Since for $t > 1$ we have:
        $$t^{-p-1}(1-e^{-t^2})\leq t^{-p-1}\leq 2^{\frac{p+1}{2}}\cdot (t^2+1)^{-\frac{p+1}{2}},$$
        it follows that
        \begin{equation*}
            |D|^{-p-1}(1-e^{-s^2D^2})\chi_{(\frac1s,\infty)}(|D|) \leq s^{p+1}\cdot 2^{\frac{p+1}{2}}\cdot (1+s^2D^2)^{-\frac{p+1}{2}}.
        \end{equation*}
        So we may estimate the first summand by:
        \begin{align*}
            \|x|D|^{p-1}(1-e^{-s^2D^2})\chi_{(\frac{1}{s},\infty)}(|D|)\|_1 &\leq 2^{\frac{p+1}{2}}\Big\|x(D^2+s^{-2})^{-\frac{p+1}{2}}\Big\|_1 \\
                                                                            &= 2^{\frac{p+1}{2}}\Big\|x(D+\frac{i}{s})^{-p-1}\Big\|_1\\
                                                                            &= O(s)
        \end{align*}
        Where the final step follows from Hypothesis \ref{main assumption}.\eqref{ass2}.

        Let us estimate the second summand of \eqref{split half line}. We have
        $$t^{-p-1}(1-e^{-t^2})\leq t^{1-p}\leq et^{1-p}\cdot e^{-t^2},\quad t\in[0,1],$$
        so it follows that
        \begin{equation*}
            |D|^{-p-1}(1-e^{-s^2D^2})\chi_{[0,\frac1s]}(|D|) \leq es^{p+1}\cdot (s|D|)^{1-p}\cdot e^{-s^2D^2}.
        \end{equation*}
        So, the second summand in \eqref{split half line} can be estimated by:
        \begin{equation*}
            \|x|D|^{-p-1}(1-e^{-s^2D^2})\chi_{[0,\frac{1}{s})]}(|D|)\|_1 \leq es^2\Big\|x|D|^{1-p}e^{-s^2D^2}\Big\|_1.
        \end{equation*}
        From Lemma \ref{first decay lemma}, this is $O(s)$ as $s\downarrow 0$.
    \end{proof}

    \begin{lem}\label{right convergence estimate}
        Let $(\mathcal{A},H,D)$ be a spectral triple satisfying Hypothesis \ref{main assumption} and assume that $D$ has a spectral gap at $0.$ For every $x\in\mathcal{B},$ we have
        \begin{equation*}
            \||D|^{-p-1}x(1-e^{-s^2D^2})\|_1 = O(s),\quad s\downarrow0.
        \end{equation*}
    \end{lem}
    \begin{proof}
        On the subspace $H_\infty$, we have:
        \begin{align*}
            |D|^{-p-1}x &= x|D|^{-p-1}+[|D|^{-p-1},x]\\
                        &= x|D|^{-p-1}-|D|^{-p-1}[|D|^{p+1},x]|D|^{-p-1}.
        \end{align*}
        By Lemma \ref{left to right lemma}, we have (again on $H_\infty$):
        \begin{align*}
            [|D|^{p+1},x] &= |D|^{p+1}x-x|D|^{p+1}\\
                          &= \left(\sum_{k=0}^{p+1}\binom{p+1}{k}\delta^{p+1-k}(x)|D|^k\right)-x|D|^{p+1}\\
                          &= \sum_{k=0}^p\binom{p+1}{k}\delta^{p+1-k}(x)|D|^k.
        \end{align*}
        Thus on $H_\infty$:
        \begin{equation*}
            |D|^{-p-1}x = x|D|^{-p-1}-\sum_{k=0}^p\binom{p+1}{k}|D|^{-p-1}\delta^{p+1-k}(x)|D|^{k-p-1}.
        \end{equation*}
        Multiplying on the right by $(1-e^{-s^2D^2})$, on $H_\infty$ we have:
        \begin{equation}\label{left to right expansion in trace class}
            |D|^{-p-1}x(1-e^{-s^2D^2}) = x|D|^{-p-1}(1-e^{-s^2D^2})-\sum_{k=0}^p \binom{p+1}{k}|D|^{-p-1}\delta^{p+1-k}(x)|D|^{k-p-1}(1-e^{-s^2D^2})
        \end{equation}
        From Lemma \ref{left convergence estimate}, we have that $x|D|^{-p-1}(1-e^{-s^2D^2})$ has bounded extension to an operator in $\mathcal{L}_1$ with norm bounded by $O(s)$.

        For $0\leq k\leq p,$ we have
        \begin{align*}
            \Big\||D|^{-p-1}\delta^{p+1-k}&(x)|D|^{k-p-1}(1-e^{-s^2D^2})\Big\|_1\\
                                          &\leq\Big\||D|^{-p-1}\delta^{p+1-k}(x)\Big\|_1\cdot\Big\||D|^{k-p}\Big\|_{\infty}\\
                                          &\quad\cdot s\Big\|(s|D|)^{-1}(1-e^{-s^2D^2})\Big\|_{\infty}.
        \end{align*}
        The first factor here is finite by Remark \ref{sp gap fact}. The second factor is finite because $k\leq p.$ The third factor is $O(s)$ by the functional calculus. Thus,
        \begin{equation*}
            \||D|^{-p-1}\delta^{p+1-k}(x)|D|^{k-p-1}(1-e^{-s^2D^2})\|_1=O(s),\quad s\downarrow0.
        \end{equation*}        
        Hence, each summand in \eqref{left to right expansion in trace class} extends to a trace class operator with norm bounded by $O(s)$, $s\downarrow 0$. By the triangle
        inequality, this completes the proof.        
    \end{proof}

\section{Integral formulae for commutators}\label{integral app}
    In this section of the appendix, we collect results concerning formulae for commutators with functions of $D$. Many of the results of this section
    will be known to the expert reader, but since they are scattered around various sources we provide them here with short and self-contained proofs.

    In this section, $(\mathcal{A},H,D)$ is a smooth $p$-dimensional spectral triple. 
    
    The following is essentially a consequence of the classical Duhamel formula.
    \begin{lem}\label{exp commutator lemma} 
        Let $(\mathcal{A},H,D)$ be a smooth spectral triple. If $x\in\mathcal{B}$ then for all $t \in \mathbb{R}$:
        \begin{equation*}
            [e^{it|D|},x] = it\int_0^1 e^{it(1-v)|D|}\delta(x)e^{itv|D|}\,dv.
        \end{equation*}
        Here, the integral is understood in the weak operator topology sense.
    \end{lem}
    \begin{proof}
        For $n \geq 0$, we define the projection
        \begin{equation*}
            p_n := \chi_{[0,n]}(|D|).
        \end{equation*}
        Since $|D|$ is non-negative, as $n\to\infty$ the sequence of projections $\{p_n\}_{n\geq 0}$ converges in the strong operator topology to the identity.
        We define the functions $\xi$ and $\eta$ by:
        \begin{align*}
             \xi(v) &:= p_n\exp(it(1-v)|D|),\\
            \eta(v) &:= x\exp(itv|D|)p_n, \quad v \in \mathbb{R}.
        \end{align*}
        Since $p_n|D| \leq n$, the operator valued functions $\xi$ and $\eta$ are continuous and differentiable in the uniform norm. 
        
        Since $\xi$ and $\eta$ are continuous and differentiable, we have:
        \begin{equation*}
            \xi(1)\eta(1)-\xi(0)\eta(0) = \int_0^1 \xi'(v)\eta(v)+\xi(v)\eta'(v)\,dv.
        \end{equation*}
        where since $\xi,$ $\xi',$ $\eta$ and $\eta'$ are continuous, this may be considered as a Bochner integral. Therefore in particular,
        this is an integral in the weak operator topology.
        
        We can compute the terms in the integrand as:
        \begin{align*}
            \xi'(v)\eta(v) &= -itp_n\exp(it(1-v)|D|)|D|x\exp(itv|D|)p_n,\\
            \xi(v)\eta'(v) &= itp_n\exp(it(1-v)|D|)x|D|\exp(itv|D|)p_n.
        \end{align*}
        Thus, 
        \begin{equation*}
            \xi(1)\eta(1)-\xi(0)\eta(0) = -it\int_0^1 p_n\exp(it(1-v)|D|)\delta(x)\exp(itv|D|)p_n\,dv. 
        \end{equation*}
        Since the operator $\exp(it(1-v)|D|)\delta(x)\exp(itv|D|)$ is a continuous function of $v$ (in weak operator topology), the weak operator
        topology integral
        \begin{equation*}
            \int_{0}^1 \exp(it(1-v)|D|)\delta(x)\exp(itv|D|)\,dv
        \end{equation*}
        exists, and we have:
        \begin{equation*}
            p_n \int_{0}^1 \exp(it(1-v)|D|)\delta(x)\exp(itv|D|)\,dv\cdot p_n = \int_{0}^1 p_n \exp(it(1-v)|D|)\delta(x)\exp(itv|D|)p_n\,dv.
        \end{equation*}
        Therefore:
        \begin{equation*}
            \xi(1)\eta(1)-\xi(0)\eta(0) = p_n\int_{0}^1 \exp(it(1-v)|D|)\delta(x)\exp(itv|D|)\,dv\cdot p_n.
        \end{equation*}
        
        On the other hand, we can compute $\xi(1), \eta(1), \xi(0)$ and $\eta(0)$ directly:
        \begin{align*}
            \xi(1)\eta(1)-\xi(0)\eta(0) &= p_nx\exp(it|D|)p_n-p_n\exp(it|D|)xp_n\\
                                        &= p_n[x,\exp(it|D|)]p_n.
        \end{align*}
        Thus,
        \begin{equation*}
            p_n[\exp(it|D|),x]p_n = itp_n\int_0^1 \exp(it(1-v)|D|)\delta(x)\exp(itv|D|)\,dv\cdot p_n.
        \end{equation*}
        Since $n \geq 0$ is arbitrary, we may take $n\to\infty$ and since $p_n$ converges in the strong operator topology to the identity
        we get the desired equality.
    \end{proof}
    
    Combining Lemma \ref{exp commutator lemma} with the Fourier inversion theorem yields a formula for $[f(s|D|),x]$ for quite general functions $f$. The following formula is well known and appears in many places, for example \cite[Theorem 3.2.32]{Bratteli-Robinson1}.
    \begin{lem}\label{first commutator rep lemma} 
        If $\widehat{f},\widehat{f'}\in L_1(\mathbb{R}),$ then for all $x \in \mathcal{B}$, and $s > 0$,
        \begin{equation*}
            [f(s|D|),x] = s\int_{-\infty}^{\infty}\left(\int_0^1\widehat{f'}(u)e^{ius(1-v)|D|}\delta(x)e^{iusv|D|}dv\right)\,du.
        \end{equation*}
        Here, the integral is understood in a weak operator topology sense.
    \end{lem}
    \begin{proof} 
        Indeed, by the Fourier inversion formula, by functional calculus we have a weak operator topology integral:
        $$f(s|D|)=\int_{\mathbb{R}}\widehat{f}(u)e^{ius|D|}du.$$
        Therefore,
        $$[f(s|D|),x]=\int_{-\infty}^{\infty}\widehat{f}(u)[e^{ius|D|},x]du.$$
        By Lemma \ref{exp commutator lemma}, we have a weak operator topology integral:
        $$[e^{ius|D|},x] = ius\int_0^1e^{ius(1-v)|D|}\delta(x)e^{iusv|D|}dv.$$
        Thus,
        $$[f(s|D|),x]=s\int_{-\infty}^{\infty}iu\widehat{f}(u)\Big(\int_0^1e^{ius(1-v)|D|}\delta(x)e^{iusv|D|}dv\Big)du.$$
        Since $iu\widehat{f}(u)=\widehat{f'}(u),$ the assertion follows.
    \end{proof}

    \begin{lem}\label{second commutator rep lemma} 
        If $\widehat{f},\widehat{f'},\widehat{f''}\in L_1(\mathbb{R}),$ then for all $x \in \mathcal{B}$ and $s > 0$ we have:
        $$[f(s|D|),x]-sf'(s|D|)\delta(x) = -s^2\int_{-\infty}^{\infty}\left(\int_0^1\widehat{f''}(u)(1-v)e^{ius(1-v)|D|}\delta^2(x)e^{iusv|D|}dv\right)du.$$
        $$[f(s|D|),x]-s\delta(x)f'(s|D|) = -s^2\int_{-\infty}^{\infty}\left(\int_0^1\widehat{f''}(u)(1-v)e^{iusv|D|}\delta^2(x)e^{ius(1-v)|D|}dv\right)du.$$
        Here once again the integrals are understood in the weak operator topology.
    \end{lem}
    \begin{proof} We only prove the first equality as the proof of the second one is similar. 
    
        By the Fourier inversion theorem and functional calculus we have a weak operator topology integral representation:
        $$f'(s|D|)=\int_{\mathbb{R}}\widehat{f'}(u)e^{ius|D|}du.$$
        So multiplying on the left by the bounded operator $\delta(x)$,
        \begin{align*}
            sf'(s|D|)\delta(x)  &= s\int_{\mathbb{R}}\widehat{f'}(u)e^{ius|D|}\delta(x)du\\
                                &= s\int_{-\infty}^{\infty}\left(\int_0^1\widehat{f'}(u)e^{ius|D|}\delta(x)dv\right)du.
        \end{align*}

        Now representing $[f(s|D|),x]$ by the integral representation given by Lemma \ref{first commutator rep lemma}, we infer that
        \begin{align}
            [f(s|D|),x]-sf'(s|D|)\delta(x) &= s\int_{-\infty}^{\infty}\Big(\int_0^1\widehat{f'}(u)\Big(e^{ius(1-v)|D|}\delta(x)e^{iusv|D|}-e^{ius|D|}\delta(x)\Big)dv\Big)du\nonumber\\
                                           &= s\int_{-\infty}^{\infty}\Big(\int_0^1\widehat{f'}(u)\Big(e^{ius(1-v)|D|}[\delta(x),e^{iusv|D|}]\Big)dv\Big)du\label{taylor remainder formula}.
        \end{align}

        Applying Lemma \ref{exp commutator lemma} to $\delta(x) \in \mathcal{B}$, we have:
        \begin{equation}\label{delta(x) commutator formula}
            [\delta(x),e^{iusv|D|}] = -iusv\int_0^1e^{iusv(1-w)|D|}\delta^2(x)e^{iusvw|D|}dw.
        \end{equation}
        Combining \eqref{taylor remainder formula} and \eqref{delta(x) commutator formula}, we obtain
        \begin{equation}\label{taylor expansion triple integral}
            [f(s|D|),x]-sf'(s|D|)\delta(x) =-s^2\int_{-\infty}^{\infty}\Big(\int_0^1\Big(\int_0^1\widehat{f''}(u)ve^{ius(1-vw)|D|}\delta^2(x)e^{iusvw|D|}dw\Big)dv\Big)du.
        \end{equation}
        
        We focus on the inner integral. Performing a linear change of variables, $w_0 = vw$, we get:
        \begin{equation*}
            \int_0^1 ve^{ius(1-vw)|D|}\delta^2(x)e^{iusvw|D|}\,dw = \int_0^v e^{ius(1-w_0)|D|}\delta^2(x)e^{iusw_0|D|}\,dw_0,
        \end{equation*}
        and therefore:
        \begin{equation*}
            \int_0^1 \int_0^1 ve^{ius(1-vw)|D|}\delta^2(x)e^{iusvw|D|}\,dwdv = \int_0^1 \int_{0}^v e^{ius(1-w)|D|}\delta^2(x)e^{iusw|D|}\,dwdv.
        \end{equation*}
        Since the integrand is continuous, we may apply Fubini's theorem to interchange the integrals: 
        \begin{align*}
            \int_0^1 \int_0^1 ve^{ius(1-vw)|D|}\delta^2(x)e^{iusvw|D|}\,dwdv &= \int_0^1 \int_{w}^1 e^{ius(1-w)|D|}\delta^2(x)e^{iusw|D|}\,dvdw\\
                                                                             &= \int_0^1 (1-w)e^{ius(1-w)|D|}\delta^2(x)e^{iusw|D|}\,dw.
        \end{align*}         
        So from \eqref{taylor expansion triple integral}:
        \begin{equation*}
            [f(s|D|),x]-sf'(s|D|)\delta(x) = -s^2\int_{-\infty}^\infty \int_0^1 (1-w)e^{ius(1-w)|D|}\delta^2(x)e^{iusw|D|}\,dwdu.
        \end{equation*}
    \end{proof}

\section{Hochschild coboundary computations}\label{coboundary app}

    In this part of the appendix we include some of the lengthy algebraic computations required for Sections \ref{cohomology section} and \ref{preliminary heat section}.    
    Recall that for a multilinear functional $\theta:\mathcal{A}^{\otimes p}\to\mathbb{C}$ the Hochschild coboundary $b\theta:\mathcal{A}^{\otimes (p+1)}\to\mathbb{C}$ is defined in terms of the Hochschild boundary $b$ by $b\theta(c) = \theta(bc)$.
    
    Explicitly, for an elementary tensor $a_0\otimes\cdots \otimes a_p \in \mathcal{A}^{\otimes (p+1)}$ we have:
    \begin{align*}
        (b\theta)(a_0\otimes\cdots a_p) &= \theta(a_0a_1\otimes a_2\otimes\cdots \otimes a_p)+\sum_{k=1}^{p-1} (-1)^k \theta(a_0\otimes\cdots \otimes a_ka_{k+1}\otimes\cdots\otimes a_p)\\
                                        &\quad + (-1)^p\theta(a_pa_0\otimes a_1\otimes\cdots\otimes a_{p-1}).
    \end{align*}


\subsection{Coboundaries in Section \ref{cohomology section}}
    Let $\mathscr{A} \subseteq \{1,\ldots,p\}$. Let $T :=D^{2-|\mathscr{A}|}|D|^{p+1}e^{-s^2D^2}.$
    Following the notation of Definition \ref{W definition}, we define
    for $a \in \mathcal{A}$,
    \begin{equation*}
        b_k(a) := \begin{cases}
                      \delta(a),\quad k \in \mathscr{A},\\
                      [F,a],\quad k\notin \mathscr{A}.
                  \end{cases}
    \end{equation*}
    Fix $1\leq m\leq p-1$. We introduce a pair of multilinear mappings, $\theta_s^1$ and $\theta_s^2$, defined on $a_0\otimes\cdots\otimes a_{p-1} \in \mathcal{A}^{\otimes p}$ by:
    \begin{equation*}
        \theta_s^1(a_0\otimes\cdots\otimes a_{p-1}) := \mathrm{Tr}\left(\Gamma a_0\left(\prod_{k=1}^{m-2}b_k(a_k)\right)\delta^2(a_{m-1})\left(\prod_{k=m}^{p-1}b_{k+1}(a_k)\right)T\right).
    \end{equation*}
    and
    \begin{equation*}
        \theta_s^2(a_0\otimes\cdots\otimes a_{p-1}) := \mathrm{Tr}\left(\Gamma a_0\left(\prod_{k=1}^{m-2}b_k(a_k)\right)[F,\delta(a_{m-1})]\left(\prod_{k=m}^{p-1}b_{k+1}(a_k)\right)T\right).
    \end{equation*}
    
    The mapping that here is denoted $\theta_s^1$ is exactly the multilinear mapping $\theta_s$ introduced in Lemma \ref{kogom2}, and similarly
    the multilinear mapping $\theta_s^2$ is the multilinear mapping $\theta_s$ introduced in Lemma \ref{kogom3}. For $1\leq k < m$ we also introduce $X_k^1$ and $X_k^2$ defined by:
    \begin{align*}
        X_k^1 &:= {\rm Tr}\left(\Gamma a_0\left(\prod_{l=1}^{k-1}b_l(a_l)\right)a_k\left(\prod_{l=k}^{m-2}b_l(a_{l+1})\right)\delta^2(a_m)\left(\prod_{l=m+1}^pb_l(a_l)\right)\cdot T\right),\\
        X_k^2 &:= {\rm Tr}\left(\Gamma a_0\left(\prod_{l=1}^{k-1}b_l(a_l)\right)a_k\left(\prod_{l=k}^{m-2}b_l(a_{l+1})\right) [F,\delta(a_m)]\left(\prod_{l=m+1}^pb_l(a_l)\right)\cdot T\right).
    \end{align*}
    
    Now if $m\leq k \leq p$, we define $Y^1_k$ and $Y^2_k$ by:
    \begin{align*}
          Y_k^1 &:= {\rm Tr}\left(\Gamma a_0\left(\prod_{l=1}^{m-2}b_l(a_l)\right)\delta^2(a_{m-1})\left(\prod_{l=m}^{k-1}b_{l+1}(a_l)\right) a_k\left(\prod_{l=k+1}^pb_l(a_l)\right)\cdot T\right),\\
          Y_k^2 &:= {\rm Tr}\left(\Gamma a_0\left(\prod_{l=1}^{m-2}b_l(a_l)\right)[F,\delta(a_{m-1})]\left(\prod_{l=m}^{k-1}b_{l+1}(a_l)\right)a_k\left(\prod_{l=k+1}^pb_l(a_l)\right)\cdot T\right).
    \end{align*}

    \begin{lem}\label{kogom app 1}
        For $j = 1,2$ and $a_0\otimes\cdots\otimes a_p\in \mathcal{A}^{\otimes (p+1)}$. we have:
        \begin{equation*}
            \theta^j_s(a_0a_1\otimes a_2\otimes\cdots\otimes a_p) = X_1^j
        \end{equation*}
    \end{lem}
    \begin{proof}
        This follows immediately from the definition.
    \end{proof}

    \begin{lem}\label{kogom app 2}
        For $j = 1,2$, $1 \leq k < m-1$ and $a_0\otimes\cdots\otimes a_p \in \mathcal{A}^{\otimes (p+1)}$ we have
        \begin{equation*}
            \theta_s^j(a_0\otimes \cdots\otimes a_{k}a_{k+1}\otimes\cdots\otimes a_p) = X_k^j+X_{k+1}^j.
        \end{equation*}
    \end{lem}
    \begin{proof} 
        We will only describe the $j=1$ case, since the $j=2$ case is identical. By the definition of $\theta_s^j,$ we have
        \begin{align*}
            \theta_s^1(a_0\otimes &a_1\otimes\cdots\otimes a_{k-1}\otimes a_{k}a_{k+1}\otimes a_{k+2}\otimes\cdots\otimes a_p)\\
                                  &= {\rm Tr}\left(\Gamma a_0\left(\prod_{l=1}^{k-1}b_l(a_l)\right)b_k(a_ka_{k+1})\left(\prod_{l=k+1}^{m-2}b_l(a_{l+1})\right)\delta^2(a_m)\left(\prod_{l=m}^{p-1}b_{l+1}(a_{l+1})\right)\cdot T\right).
        \end{align*}
        Now applying the Leibniz rule to $b_k(a_ka_{k+1})$,
        \begin{align*}
            \theta_s^1(&a_0\otimes a_1\otimes\cdots\otimes a_{k-1}\otimes a_{k}a_{k+1}\otimes a_{k+2}\otimes\cdots\otimes a_p)\\
                                  &= {\rm Tr}\left(\Gamma a_0\left(\prod_{l=1}^{k-1}b_l(a_l)\right) a_k b_k(a_{k+1})\left(\prod_{l=k+1}^{m-2}b_l(a_{l+1})\right)\delta^2(a_m)\left(\prod_{l=m}^{p-1}b_{l+1}(a_{l+1})\right)\cdot T\right)\\
                                  &\quad + {\rm Tr}\left(\Gamma a_0\left(\prod_{l=1}^{k-1}b_l(a_l)\right)b_k(a_k)a_{k+1}\left(\prod_{l=k+1}^{m-2}b_l(a_{l+1})\right)\delta^2(a_m)\left(\prod_{l=m}^{p-1}b_{l+1}(a_{l+1})\right)\cdot T\right)\\
                                  &= {\rm Tr}\left(\Gamma a_0\left(\prod_{l=1}^{k-1}b_l(a_l)\right) a_k\left(\prod_{l=k}^{m-2}b_l(a_{l+1})\right)\delta^2(a_m)\left(\prod_{l=m+1}^p b_l(a_l)\right) T\right)\\
                                  &\quad + {\rm Tr}\left(\Gamma a_0\left(\prod_{l=1}^kb_l(a_l)\right) a_{k+1}\left(\prod_{l=k+1}^{m-2}b_l(a_{l+1})\right)\delta^2(a_m)\left(\prod_{l=m+1}^pb_l(a_l)\right)\cdot T\right)\\
                                  &= X_k^1+X_{k+1}^1
        \end{align*}
        as required.
    \end{proof}

    \begin{lem}\label{kogom app 3}
        For $m\leq k<p,$ $j=1,2$ and $a_0\otimes\cdots\otimes a_p \in \mathcal{A}^{\otimes (p+1)}$ we have
        \begin{equation*}
            \theta_s^j(a_0\otimes a_1\otimes\cdots\otimes a_{k-1}\otimes a_{k}a_{k+1}\otimes a_{k+2}\otimes\cdots\otimes a_p) = Y_k^j+Y_{k+1}^j.
        \end{equation*}
    \end{lem}
    \begin{proof} 
        Again we demonstrate only the $j=1$ case since the $j=2$ case is identical. By definition we have:
        \begin{align*}
            \theta_s^1(&a_0\otimes a_1\otimes\cdots\otimes a_{k-1}\otimes a_{k}a_{k+1}\otimes a_{k+2}\otimes\cdots\otimes a_p)\\
                                &= {\rm Tr}\left(\Gamma a_0\left(\prod_{l=1}^{m-2}b_l(a_l)\right)\delta^2(a_{m-1})\left(\prod_{l=m}^{k-1}b_{l+1}(a_l)\right)b_{k+1}(a_ka_{k+1})\left(\prod_{l=m+1}^{p-1}b_{l+1}(a_{l+1})\right)\cdot T\right).
        \end{align*}
        Applying the Leibniz rule to $b_{k+1}(a_ka_{k+1})$ we have:
        \begin{align*}
        \theta_s^1(&a_0\otimes a_1\otimes\cdots\otimes a_{k-1}\otimes a_{k}a_{k+1}\otimes a_{k+2}\otimes\cdots\otimes a_p)\\
                              &={\rm Tr}\left(\Gamma a_0\left(\prod_{l=1}^{m-2}b_l(a_l)\right)\delta^2(a_{m-1})\left(\prod_{l=m}^{k-1}b_{l+1}(a_l)\right)a_k b_{k+1}(a_{k+1})\left(\prod_{l=k+1}^{p-1}b_{l+1}(a_{l+1})\right)\cdot T\right)\\
                              &\quad+{\rm Tr}\left(\Gamma a_0\left(\prod_{l=1}^{m-2}b_l(a_l)\right)\delta^2(a_{m-1})\left(\prod_{l=m}^{k-1}b_{l+1}(a_l)\right)b_{k+1}(a_k)a_{k+1}\left(\prod_{l=k+1}^{p-1}b_{l+1}(a_{l+1})\right)\cdot T\right)\\
                              &={\rm Tr}\left(\Gamma a_0\left(\prod_{l=1}^{m-2}b_l(a_l)\right)\delta^2(a_{m-1})\left(\prod_{l=m}^{k-1}b_{l+1}(a_l\right) a_k\left(\prod_{l=k+1}^pb_l(a_l)\right)\cdot T\right)\\
                              &\quad+{\rm Tr}\left(\Gamma a_0\left(\prod_{l=1}^{m-2}b_l(a_l)\right)\delta^2(a_{m-1})\left(\prod_{l=m}^kb_{l+1}(a_l)\right) a_{k+1}\left(\prod_{l=k+2}^pb_l(a_l)\right)\cdot T\right)\\
                              &=Y_k^1+Y_{k+1}^1.
        \end{align*}
    \end{proof}
    
    We recall also the multilinear maps $\mathcal{W}_{\mathscr{A}}$ from Definition \ref{W definition}. 

    \begin{lem}\label{kogom app 4}
        Let $c=a_0\otimes a_1\otimes\cdots\otimes a_p \in \mathcal{A}^{\otimes (p+1)}$. If we assume that $m-1,m \in \mathscr{A}$, then we have:
        \begin{align*}
            \theta_s^1(&a_0\otimes a_1\otimes\cdots\otimes a_{m-2}\otimes a_{m-1}a_m\otimes a_{m+1}\otimes\cdots\otimes a_p)\\
                       &=X_{m-1}^1+Y_m^1+2{\rm Tr}(\mathcal{W}_{\mathscr{A}}(c)\cdot T).
        \end{align*}
        Now if $\mathscr{B} \subseteq \{1,\ldots, p\}$ is such that $|\mathscr{B}|=|\mathscr{A}|$ and $\mathscr{A}\Delta \mathscr{B} = \{m-1,m\}$ (where $\Delta$ denotes the symmetric difference) then
        \begin{align*}
            \theta_s^2(&a_0\otimes a_1\otimes\cdots\otimes a_{m-2}\otimes a_{m-1}a_m\otimes a_{m+1}\otimes\cdots\otimes a_p)\\
                       &= X_{m-1}^2+Y_m^2+{\rm Tr}(\mathcal{W}_{\mathscr{A}}(c)\cdot T)+{\rm Tr}(\mathcal{W}_{\mathscr{B}}(c)\cdot T).
        \end{align*}
    \end{lem}
    \begin{proof} 
        Using the repeated Leibniz rule, we obtain
        \begin{align*}
            \delta^2(a_{m-1}a_m)   &= \delta^2(a_{m-1})a_m+2\delta(a_{m-1})\delta(a_m)+a_{m-1}\delta^2(a_m),\\
            [F,\delta(a_{m-1}a_m)] &= [F,\delta(a_{m-1})]a_m+[F,a_{m-1}]\delta(a_m)+\delta(a_{m-1})[F,a_m]+a_{m-1}[F,\delta(a_m)].
        \end{align*}
        Let us focus on proving the assertion relating to $\theta^1_s$, since the other assertion is identical.

        By the definition of $\theta_s^1,$ we have
        \begin{align*}
            \theta_s^1(&a_0\otimes a_1\otimes\cdots\otimes a_{m-2}\otimes a_{m-1}a_m\otimes a_{m+1}\otimes\cdots\otimes a_p)\\
                     &= {\rm Tr}\left(\Gamma a_0\left(\prod_{l=1}^{m-2}b_l(a_l)\right)\delta^2(a_{m-1}a_m)\left(\prod_{l=m}^{p-1}b_{l+1}(a_{l+1})\right)\cdot T\right)\\
                     &= {\rm Tr}\left(\Gamma a_0\left(\prod_{l=1}^{m-2}b_l(a_l)\right)\delta^2(a_{m-1}a_m)\left(\prod_{l=m+1}^pb_l(a_l)\right)\cdot T\right).
        \end{align*}
        Applying the Leibniz rule to the $\delta^2(a_{m-1}a_m)$ term:
        \begin{align*}
            \theta_s^1(&a_0\otimes a_1\otimes\cdots\otimes a_{m-2}\otimes a_{m-1}a_m\otimes a_{m+1}\otimes\cdots\otimes a_p)\\
                     &= {\rm Tr}\left(\Gamma a_0\left(\prod_{l=1}^{m-2}b_l(a_l)\right) a_{m-1}\delta^2(a_m)\left(\prod_{l=m+1}^pb_l(a_l)\right)\cdot T\right)\\
                     &\quad + {\rm Tr}\left(\Gamma a_0\left(\prod_{l=1}^{m-2}b_l(a_l)\right) 2\delta(a_{m-1})\delta(a_m)\left(\prod_{l=m+1}^pb_l(a_l)\right)\cdot T\right)\\
                     &\quad + {\rm Tr}\left(\Gamma a_0\left(\prod_{l=1}^{m-2}b_l(a_l)\right)\delta^2(a_{m-1})a_m\left(\prod_{l=m+1}^pb_l(a_l)\right)\cdot T\right)\\
                     &=X_{m-1}+2{\rm Tr}(\mathcal{W}_{\mathscr{A}}(c)\cdot T)+Y_m.
        \end{align*}
    \end{proof}

    \begin{lem}\label{kogom app 5}
        For $a_0\otimes \cdots\otimes a_p \in \mathcal{A}^{\otimes(p+1)}$, we have
        \begin{align*}
            \theta_s^1(&a_pa_0\otimes a_1\otimes a_2\otimes\cdots\otimes a_{p-1})\\
                     &= {\rm Tr}\left(\Gamma a_0\left(\prod_{l=1}^{m-2}b_l(a_l)\right)\delta^2(a_{m-1})\left(\prod_{l=m}^{p-1}b_{l+1}(a_l)\right)[T,a_p]\right)+Y^1_p.
        \end{align*}
        We also have:
        \begin{align*}
            \theta_s^2(&a_pa_0\otimes a_1\otimes a_2\otimes\cdots\otimes a_{p-1})\\
                       &={\rm Tr}\left(\Gamma a_0\left(\prod_{l=1}^{m-2}b_l(a_l)\right)[F,\delta(a_{m-1})]\left(\prod_{l=m}^{p-1}b_{l+1}(a_l)\right)[T,a_p]\right)+Y_p^2.
        \end{align*}
    \end{lem}
    \begin{proof} 
        We prove only the assertion involving $\theta_s^1$, since one can prove the other result by an identical argument.
        Since $\Gamma$ commutes with $a_p$, we have:
        \begin{align*}
            \theta_s^1(&a_pa_0\otimes a_1\otimes a_2\otimes\cdots\otimes a_{p-1})\\
                     &= {\rm Tr}\left(\Gamma a_pa_0\left(\prod_{l=1}^{m-2}b_l(a_l)\right)\delta^2(a_{m-1})\left(\prod_{l=m}^{p-1}b_{l+1}(a_l)\right)\cdot T\right)\\
                     &= {\rm Tr}\left(a_p\Gamma a_0\left(\prod_{l=1}^{m-2}b_l(a_l)\right)\delta^2(a_{m-1})\left(\prod_{l=m}^{p-1}b_{l+1}(a_l)\right)\cdot T\right).
        \end{align*}
        Using the cyclicity of the trace, we have
        \begin{align*}
            \theta_s^1(&a_pa_0\otimes a_1\otimes a_2\otimes\cdots\otimes a_{p-1})\\
                     &={\rm Tr}\left(\Gamma a_0\left(\prod_{l=1}^{m-2}b_l(a_l)\right)\delta^2(a_{m-1})\left(\prod_{l=m}^{p-1}b_{l+1}(a_l)\right)\cdot Ta_p\right)\\
                     &={\rm Tr}\left(\Gamma a_0\left(\prod_{l=1}^{m-2}b_l(a_l)\right)\delta^2(a_{m-1})\left(\prod_{l=m}^{p-1}b_{l+1}(a_l)\right)\cdot [T,a_p]\right)+Y_p^1.
        \end{align*}
    \end{proof}

    Note that for $j = 1,2$, the telescopic sum
    \begin{align*}
        X_1^j+&\sum_{k=1}^{m-1}(-1)^k(X_k^j+X_{k+1}^j)+(-1)^{m-1}(X_{m-1}^j+Y_m^j)\\
            &+\sum_{k=m}^{p-1}(-1)^k(Y_k^j+Y_{k+1}^j)+(-1)^pY_p^j.
    \end{align*}
    vanishes.

    Therefore, combining Lemmas \ref{kogom app 1}, \ref{kogom app 2}, \ref{kogom app 3}, \ref{kogom app 4} and \ref{kogom app 5} we have:
    \begin{align*}
        (b\theta_s^1)(a_0\otimes\cdots\otimes a_p) &= 2\cdot(-1)^{m-1}\cdot{\rm Tr}(\mathcal{W}_{\mathscr{A}}(c)\cdot T)\\
                                                   &\quad + (-1)^p\cdot {\rm Tr}\left(\Gamma a_0\left(\prod_{l=1}^{m-2}b_l(a_l)\right)\delta^2(a_{m-1})\left(\prod_{l=m}^{p-1}b_{l+1}(a_l)\right)\cdot[T,a_p]\right).
    \end{align*}
    Similarly, if $\mathscr{B}$ is such that $|\mathscr{A}| = |\mathscr{B}|$ and $\mathscr{A}\Delta \mathscr{B} = \{m-1,m\}$ then:
    \begin{align*}
        (b\theta_s^2)(a_0\otimes\cdots\otimes a_p) &= (-1)^{m-1}\cdot{\rm Tr}(\mathcal{W}_{\mathscr{A}}(c)\cdot T)+(-1)^{m-1}\cdot{\rm Tr}(\mathcal{W}_{\mathscr{B}}(c)\cdot T)\\
                                                   &\quad +(-1)^p\cdot {\rm Tr}\left(\Gamma a_0\left(\prod_{l=1}^{m-2}b_l(a_l)\right)[F,\delta(a_{m-1})]\left(\prod_{l=m}^{p-1}b_{l+1}(a_l)\right)\cdot [T,a_p]\right).
    \end{align*}
    This completes the computation of the coboundaries for $\theta^1_s$ and $\theta^2_s$.

\subsection{Coboundaries in Section \ref{preliminary heat section}}

    Again in this subsection, $(\mathcal{A},H,D)$ is a smooth spectral triple where $D$ has a spectral gap at $0$. Let $T = Fe^{-s^2D^2}$. Note that this is different
    to $T$ in the preceding section.

    We define the multilinear mapping $\mathcal{L}_s:\mathcal{A}^{\otimes p}\to\mathbb{C}$ on $a_0\otimes \cdots\otimes a_{p-1} \in \mathcal{A}^{\otimes p}$ by:
    \begin{equation*}
        \mathcal{L}_s(a_0\otimes\cdots\otimes a_{p-1}) := {\rm Tr}\left(\Gamma a_0\left(\prod_{k=1}^{p-1}[F,a_k]\right)\cdot T\right).
    \end{equation*}
    We also define the multilinear mapping $\mathcal{K}_s:\mathcal{A}^{\otimes (p+1)}\to\mathbb{C}$ by:
    \begin{equation*}
        \mathcal{K}_s(a_0\otimes\cdots\otimes a_p) := {\rm Tr}\left(\Gamma a_0\left(\prod_{k=1}^{p-1}[F,a_k]\right)\cdot [T,a_p]\right).
    \end{equation*}
    By definition, $\mathcal{K}_s$ is exactly the mapping from Theorem \ref{first cycle thm}.

    For $1\leq m\leq p$, we define $X_m$ by:
    \begin{equation*}
        X_m := {\rm Tr}\left(\Gamma a_0\left(\prod_{k=1}^{m-1}[F,a_k]\right)a_m\left(\prod_{k=m+1}^p[F,a_k]\right) T\right).
    \end{equation*}
    We have
    \begin{align*}
        \mathcal{L}_s(a_0a_1\otimes a_2\otimes\cdots\otimes a_p) &= {\rm Tr}(\Gamma a_0a_1\prod_{k=2}^p[F,a_k]\cdot T)\\
                                                         &=X_1.
    \end{align*}
    
    Applying the Leibniz rule to $[F,a_{m-1}a_m]$:
    \begin{align*}
        \mathcal{L}_s(&a_0\otimes a_1\otimes\cdots\otimes a_{m-2}\otimes a_{m-1}a_m\otimes a_{m+1}\otimes\cdots\otimes a_p)\\
              &= {\rm Tr}\left(\Gamma a_0\left(\prod_{k=1}^{m-2}[F,a_k]\right)[F,a_{m-1}a_m]\left(\prod_{k=m+1}^p[F,a_k]\right)\cdot T\right)\\
              &= {\rm Tr}\left(\Gamma a_0\left(\prod_{k=1}^{m-2}[F,a_k]\right) a_{m-1}\left(\prod_{k=m}^p[F,a_k]\right)\cdot T\right)\\
              &\quad +{\rm Tr}\left(\Gamma a_0\left(\prod_{k=1}^{m-1}[F,a_k]\right)a_m\left(\prod_{k=m+1}^p[F,a_k]\right)\cdot T\right)\\
              &= X_m+X_{m+1}.
    \end{align*}
    Finally,
    \begin{align*}
        \mathcal{L}_s(a_pa_0\otimes a_1\otimes\cdots\otimes a_{p-1}) &= {\rm Tr}(\Gamma a_pa_0\prod_{k=1}^{p-1}[F,a_k]\cdot T)\\
                                                             &= {\rm Tr}(\Gamma a_0\prod_{k=1}^{p-1}[F,a_k]\cdot Ta_p)\\
                                                             &= {\rm Tr}(\Gamma a_0\prod_{k=1}^{p-1}[F,a_k]\cdot [T,a_p])+{\rm Tr}(\Gamma a_0\prod_{k=1}^{p-1}[F,a_k]\cdot a_pT)\\
                                                             &= \mathcal{K}_s(a_0\otimes\cdots\otimes a_p)+X_p.
    \end{align*}

    Thus,
    \begin{align*}
        (b\mathcal{L}_s)(a_0\otimes\cdots\otimes a_p) &= \mathcal{K}_s(a_0\otimes\cdots\otimes a_p)\\
                                              &\quad+ X_1 +\Big(\sum_{m=1}^{p-1}(-1)^m(X_m+X_{m+1})\Big)+(-1)^pX_p.
    \end{align*}
    The latter sum telescopes and indeed vanishes, so it follows that $b\mathcal{L}_s=\mathcal{K}_s.$

\subsection{Coboundaries in Section \ref{heat section}}
    In this section, $(\mathcal{A},H,D)$ is a smooth spectral triple where $D$ has a spectral gap at $0$, and $T := |D|^{2-p}e^{-s^2D^2}$.
    We define the multilinear mapping $\theta_s:\mathcal{A}^{\otimes p}\to\mathbb{C}$ on $a_0\otimes\cdots a_{p-1} \in \mathcal{A}^{\otimes p}$ by:
    \begin{equation*}
        \theta_s(a_0\otimes\cdots\otimes a_{p-1}) = {\rm Tr}\left(\left(\prod_{k=0}^{p-1}\partial(a_k)\right) T\right).
    \end{equation*}

    For $0 \leq k \leq p$ we also define:
    \begin{equation*}
        X_k = {\rm Tr}\left(\left(\prod_{l=0}^{k-1}\partial(a_l)\right)a_k\left(\prod_{l=k+1}^p\partial(a_l)\right)T\right).
    \end{equation*}
    So in particular,
    \begin{equation*}
        X_0 = \mathrm{Tr}\left(a_0\left(\prod_{l=1}^{p}\partial(a_l)\right)T\right).
    \end{equation*}

    Applying the Leibniz rule to $\partial(a_ka_{a+1})$ we get:
    \begin{align*}
        \theta_s(&a_0\otimes\cdots\otimes a_{k-1}\otimes a_ka_{k+1}\otimes a_{k+2}\otimes\cdots\otimes a_p)\\
                 &= {\rm Tr}\left(\left(\prod_{l=0}^{k-1}\partial(a_l)\right)\partial(a_ka_{k+1})\left(\prod_{l=k+2}^p\partial(a_l)\right)\cdot T\right)\\
                 &= {\rm Tr}\left(\left(\prod_{l=0}^{k-1}\partial(a_l)\right) a_k\left(\prod_{l=k+1}^p\partial(a_l)\right)\cdot T\right)\\
                 &\quad +{\rm Tr}\left(\left(\prod_{l=0}^k\partial(a_l)\right) a_{k+1}\left(\prod_{l=k+2}^p\partial(a_l)\right)\cdot T\right)\\
                 &=X_k+X_{k+1}.
    \end{align*}

    We also have
    \begin{align*}
        \theta_s(&a_pa_0\otimes a_1\otimes\cdots\otimes a_{p-1})\\
                 &= {\rm Tr}\left(\left(\prod_{k=0}^{p-1}\partial(a_k)\right)\cdot Ta_p\right)\\
                 &\quad + {\rm Tr}\left(a_0\left(\prod_{l=1}^{p-1}\partial(a_k)\right)\cdot T\partial(a_p)\right)\\
                 &= X_p + {\rm Tr}\left(\left(\prod_{k=0}^{p-1}\partial(a_k)\right)\cdot [T,a_p]\right)\\
                 &\quad + X_0+{\rm Tr}\left(a_0\left(\prod_{k=1}^{p-1}\partial(a_k)\right)\cdot [T,\partial(a_p)]\right).
    \end{align*}

    If we assume that $p$ is even, then:
    \begin{align*}
        (b\theta_s)(&a_0\otimes\cdots\otimes a_p)\\
                    &= \left(\sum_{k=0}^{p-1}(-1)^k(X_k+X_{k+1})\right)+(X_p+X_0)\\
                    &\quad +{\rm Tr}\left(\left(\prod_{k=0}^{p-1}\partial(a_k)\right)\cdot [T,a_p]\right)+{\rm Tr}\left(a_0\left(\prod_{k=1}^{p-1}\partial(a_k)\right)\cdot [T,\partial(a_p)]\right)\\
                    &= 2X_0+{\rm Tr}\left(a_0\left(\prod_{k=1}^{p-1}\partial(a_k)\right) [T,\partial a_p]\right) +{\rm Tr}\left(\left(\prod_{k=0}^{p-1}\partial(a_k)\right)[T,a_p]\right)\\
                    &= 2{\rm Tr}\left(a_0\left(\prod_{k=1}^p\partial(a_k)\right)\cdot T\right)+{\rm Tr}\left(a_0\left(\prod_{k=1}^{p-1}\partial(a_k)\right) [T,\partial a_p]\right)\\
                    &\quad+{\rm Tr}\left(\left(\prod_{k=0}^{p-1}\partial(a_k)\right)[T,a_p]\right).
    \end{align*}
    This completes the computation of the coboundaries.

\section{Technical estimates for Section 4.4.1}\label{app 3 section}

    For this section, $(\mathcal{A},H,D)$ is assumed to be a spectral triple satisfying Hypothesis \ref{main assumption} and we
    assume that $D$ has a spectral gap at $0$.

    The results of this section are very similar to that of Lemma \ref{w bounded lemma}. However additional technicalities make the proofs more involved and therefore are included here in the appendix.
    We make use of the mappings $b_k$ from Definition \ref{W definition}, defined in terms of $m\geq 1$ and a subset $\mathscr{B} \subseteq \{1,\ldots,m\}$ on $a \in \mathcal{A}$
    \begin{equation*}
        b_k(a) := \begin{cases}
                      \delta(a),\quad k \in \mathscr{B},\\
                      [F,a],\quad k\notin \mathscr{B}.
                  \end{cases}
    \end{equation*}

    \begin{lem}\label{c induction lemma}    
        Let $m\geq 1$ and Let $n\in \mathbb{Z}$. Then for all $\mathscr{B} \subseteq \{1,\ldots,m\}$ the operator on $H_\infty$ given by:
        \begin{equation*}
            |D|^{-n}\left(\prod_{k=1}^m b_k(a_k)\right)|D|^{n+m-|\mathscr{B}|}
        \end{equation*}
        has bounded extension, where $b_k$ is defined in terms of $\mathscr{B}$.
    \end{lem}
    \begin{proof} 
        We prove the assertion by induction on $m.$ If $m=1$ then there are two possible cases, $\mathscr{B} = \emptyset$ and $\mathscr{B} = \{1\}$. 
        
        If $\mathscr{B} = \emptyset$, then on $H_\infty$ we have:
        \begin{align*}
            |D|^{-n}\left(\prod_{k=1}^m b_k(a_k)\right)|D|^{n+m-|\mathscr{B}|} &= |D|^{-n}[F,a_1]|D|^{n+1}\\
                                                                    &= |D|^{-n}L(a_1)|D|^n\\
                                                                    &= |D|^{-n}\partial(a_1)|D|^n-F|D|^{-n}\delta(a_1)|D|^n.
        \end{align*}
        By Lemma \ref{left to right corollary}, the operators $|D|^{-n}\partial(a_1)|D|^n$ and $|D|^{-n}\delta(a_1)|D|^n$ have bounded extension, so this proves
        the $\mathscr{B} = \emptyset$ case. On the other hand, if $\mathscr{B} = \{1\},$ then on $H_\infty$:
        \begin{equation*}
            |D|^{-n}\left(\prod_{k=1}^m b_k(a_k) \right)|D|^{n+m-|\mathscr{B}|} = |D|^{-n}\delta(a_1)|D|^n.
        \end{equation*}
        Again by Lemma \ref{left to right corollary}, the operator $|D|^{-n}\delta(a_1)|D|^n$ has bounded extension. This completes the proof of the $\mathcal{B} = \{1\}$ case.

        Now assume that $m > 1$ and the assertion is true for $m-1$. Define $\mathscr{C} := \mathscr{B}\setminus \{m\}.$ and let $n_1 = n+(m-1)-|\mathscr{C}|$. Then on $H_\infty$:
        \begin{align} 
            |D|^{-n}&\left(\prod_{k=1}^mb_k(a_k)\right)|D|^{n+m-|\mathscr{B}|}\nonumber\\
                                                &=\left(|D|^{-n}\left(\prod_{k=1}^{m-1}b_k(a_k)\right)|D|^{n+(m-1)-|\mathscr{C}|}\right)\left(|D|^{-n_1}b_m(a_m)|D|^{n_1+1-|\mathscr{B}\cap\{m\}|}\right).\label{inductive expanded product}
        \end{align} 
        By the inductive assumption, the first factor has bounded extension.

        If $m\in\mathscr{B},$ then the second factor in \eqref{inductive expanded product} is:
        \begin{equation*}
            |D|^{-n_1}\delta(a_m)|D|^{n_1}
        \end{equation*}
        which has bounded extension by Lemma \ref{left to right corollary}.\eqref{left to right 1}. On the other hand, if $m\notin\mathscr{B},$ then the second factor in \eqref{inductive expanded product} is
        \begin{equation*}
            |D|^{-n_1}F(a_m)|D|^{n_1+1} = |D|^{-n_1}(\partial(a_m)-F\delta(a_m))|D|^{n_1}
        \end{equation*}
        which also has bounded extension by Lemma \ref{left to right corollary}.\eqref{left to right 1}. In either case, the second factor of \eqref{inductive expanded product} has bounded extension. So the assertion is proved
        for $m$, completing the proof by induction.
    \end{proof}

    \begin{lem}\label{c first alert lemma} 
        Let $\mathscr{A}\subseteq\{1,\cdots,p\}$ Assume that there is $m > 1$ be such that $m-1,m\in\mathscr{A}.$ 
        The operator on $H_\infty$ given by:
        \begin{equation*}
            \Gamma a_0\left(\prod_{k=1}^{m-2}b_k(a_k)\right)\delta^2(a_{m-1})\left(\prod_{k=m}^{p-1}b_{k+1}(a_k)\right)|D|^{p-|\mathscr{A}|}
        \end{equation*}
        has bounded extension.
    \end{lem}
    \begin{proof} 
        Let $n_1=|\{1,\cdots,m-2\}\setminus\mathscr{A}|$ and $n_2=|\{m+1,\cdots,p\}\setminus\mathscr{A}|$ so that immediately $n_1 \leq m-2-|\mathscr{A}|$ and $n_2 \leq p-m-|\mathscr{A}|$. On $H_\infty$ we have
        $$\Gamma a_0\left(\prod_{k=1}^{m-2}b_k(a_k)\right)\delta^2(a_{m-1})\left(\prod_{k=m}^{p-1}b_{k+1}(a_k)\right)|D|^{p-|\mathscr{A}|}= \mathrm{I}\cdot \mathrm{II}\cdot \mathrm{III}.$$
        Here,
        \begin{align*}
            \mathrm{I}   &= \Gamma a_0\left(\prod_{k=1}^{m-2}b_k(a_k)\right)|D|^{n_1},\\
            \mathrm{II}  &= |D|^{-n_1}\delta^2(a_{m-1})|D|^{n_1},\\
            \mathrm{III} &= |D|^{-n_1}\left(\prod_{k=m}^{p-1}b_{k+1}(a_k)\right)|D|^{n_1+n_2}.
        \end{align*}
        The operators $\mathrm{I}$ and $\mathrm{III}$ have bounded extension by Lemma \ref{c induction lemma}. On the other hand, $\mathrm{II}$ has bounded extension due to Lemma \ref{left to right corollary}.
    \end{proof}

    \begin{lem}\label{c second alert lemma} 
        Let $\mathscr{A}\subset\{1,\cdots,p\}$ and assume that there is $m > 1$ be such that $m-1\in\mathscr{A}$ and $m\notin\mathscr{A}.$ The operator on $H_\infty$ given by
        \begin{align*}
            \Gamma a_0\left(\prod_{k=1}^{m-2}b_k(a_k)\right)[F,\delta(a_{m-1})]\left(\prod_{k=m}^{p-1}b_{k+1}(a_k)\right)|D|^{p-|\mathscr{A}|}
        \end{align*}
        has bounded extension.
    \end{lem}
    \begin{proof} 
        Let $n_1=|\{1,\cdots,m-2\}\setminus\mathscr{A}|$ and $n_2=|\{m+1,\cdots,p\}\setminus\mathscr{A}|,$ so that we immediately have $n_1 \leq m-2-|\mathscr{A}|$ and $n_2\leq p-m-|\mathscr{A}|$ as in Lemma \ref{c first alert lemma}. We have
        $$\Gamma a_0\left(\prod_{k=1}^{m-2}b_k(a_k)\right)[F,\delta(a_{m-1})]\left(\prod_{k=m}^{p-1}b_{k+1}(a_k)\right)|D|^{p-|\mathscr{A}|}=\Gamma a_0\cdot \mathrm{I}\cdot \mathrm{II}\cdot \mathrm{III}.$$
        Here,
        \begin{align*}
              \mathrm{I} &= \left(\prod_{k=1}^{m-2}b_k(a_k)\right)|D|^{n_1},\\
             \mathrm{II} &= |D|^{-n_1}[F,\delta(a_{m-1})]|D|^{1+n_1},\\
            \mathrm{III} &= |D|^{-n_1-1}\left(\prod_{k=m}^{p-1}b_{k+1}(a_k)\right)|D|^{1+n_1+n_2}.
        \end{align*}
        The operators $\mathrm{I}$ and $\mathrm{III}$ have bounded extension by Lemma \ref{c induction lemma}. On the other hand, 
        \begin{equation*}
            \mathrm{II} = |D|^{-n_1}(\partial(\delta(a_{m-1}))-F\delta^2(a_{m-1}))|D|^{n_1}
        \end{equation*}
        which has bounded extension by Lemma \ref{left to right corollary}.\eqref{left to right 1}.
    \end{proof}

    \begin{lem}\label{c third alert lemma} 
        For every $m\geq1,$ the operator on $H_\infty$ given by
        \begin{equation*}
            \Big(\prod_{k=0}^m[F,a_k]\Big)|D|^{m+1}
        \end{equation*}
        has bounded extension.
    \end{lem}
    \begin{proof} 
        We prove the assertion by induction on $m.$ For $m=0,$ on $H_\infty$ we have
        $$[F,a_0]|D|=\partial(a_0)-F\delta(a_0)$$
        which has has bounded extension.

        Now let $m>1$ and assume that the assertion is true for $m-1$. On $H_\infty$ we write
        $$\Big(\prod_{k=0}^m[F,a_k]\Big)|D|^{m+1} = \Big(\Big(\prod_{k=0}^{m-1}[F,a_k]\Big)|D|^m\Big)\cdot\Big(|D|^{-m}[F,a_m]|D|^{m+1}\Big).$$
        The first factor has bounded extension by the induction assumption. The second factor is
        $$|D|^{-m}L(a_m)|D|^m=|D|^{-m}\partial(a_m)|D|^m-F\cdot|D|^{-m}\delta(a_m)|D|^m$$
        which has bounded extension by Lemma \ref{left to right corollary}.\eqref{left to right 1}. Hence, the assertion holds for $m$, completing the inductive proof.
    \end{proof}

\section{Subkhankulov's computation}\label{subhankulov app}

    The following assertion is identical to \cite[Lemma 2.1.1]{subhankulov}. However to the best of our knowledge there is no published proof in English,
    and \cite{subhankulov} is not easily accessible. For the convenience of the reader we include a proof here.
    
    \begin{prop}\label{subhankulov compute} 
        For all $u\in (0,1)$ and $v\in\mathbb{R},$ we have
        $$\frac1{2\pi}\int_{-1}^1\frac{(1-t^2)^2}{u+it}e^{(u+it)v}dt=(1+u^2)^2\chi_{[0,\infty]}(v)+e^{uv}\cdot\min\{1,v^{-2}\}\cdot O(1).$$
    \end{prop}
    \begin{proof}
        We will deal separately with the $v \geq 0$ and $v < 0$ cases. First, assume that $v \geq 0$.
    
        Let $\gamma_0$ be a smooth curve in $\mathbb{C}$ without self-intersections such that
        \begin{enumerate}
            \item $\gamma_0$ starts at $i$ and ends at $-i.$
            \item $\gamma_0$ lies in the half-plane $\{\Re(z)\leq0\},$
            \item The distance between $\gamma_0$ and the interval $[-1,0]$ is greater than or equal to $1$,
            \item the length of $\gamma_0$ is at most $10.$
            \item $\gamma_0$ is contained in the disc $\{z\;:\;|z| \leq 10\}$.
        \end{enumerate}
        Let $\gamma_1$ be the interval starting at $-i$ and ending at $i$ and let the contour $\gamma$ be the concatenation of $\gamma_0$ and $\gamma_1.$
        
        Define 
        \begin{equation*}
            f(z) := \frac{(1+z^2)^2}{u+z},\quad z\in\mathbb{C}\setminus \{-u\}.
        \end{equation*}
        So that by definition:
        \begin{equation}\label{wick rotation}
            \frac{1}{2\pi} \int_{-1}^1 \frac{(1-t^2)^2}{u+it}e^{(u+it)v}\,dt = \frac{1}{2\pi i} e^{uv}\int_{\gamma_1} f(z)e^{zv}\,dz.
        \end{equation}
        Since $z\mapsto e^{zv}$ is entire, the function $z\to f(z)e^{zv}$ is holomorphic in the set $\mathbb{C}\setminus \{-u\}$ and has a simple pole at at $z=-u$
        with corresponding residue $(1+u^2)^2e^{-uv}.$ By construction,
        the point $-u$ is in the interior of the curve $\gamma$ and so         
        by the Cauchy integral formula we have:
        \begin{equation*}
            \frac{1}{2\pi i}\int_{\gamma}f(z)e^{zv}dz = (1+u^2)^2e^{-uv}.
        \end{equation*}
        Since $\gamma = \gamma_0\cup\gamma_1$:
        \begin{equation}\label{v>0 eq1}
            \frac1{2\pi i}\int_{\gamma_0}f(z)e^{zv}dz+\frac1{2\pi i}\int_{\gamma_1}f(z)e^{zv}dz = (1+u^2)^2e^{-uv}.
        \end{equation}
        
        By definition $\gamma_0$ has length at most $10$, so by the triangle inequality we have the bound:
        \begin{equation*}
            \left|\int_{\gamma_0}f(z)e^{zv}dz\right| \leq 10\sup_{z\in\gamma_0} |f(z)||e^{zv}|.
        \end{equation*}
        
        Since $\gamma_0$ is contained in the half-plane $\{z\;:\;\Re(z) \leq 0\}$ we also have that $\sup_{z \in \gamma_0} |e^{zv}| \leq 1$
        and therefore:
        \begin{equation*}
            \left|\int_{\gamma_0}f(z)e^{zv}dz\right| \leq 10\sup_{z\in\gamma_0} |f(z)|.
        \end{equation*}
        
        For $z \in \gamma$ we have that $|z|\leq 10$ and $|u+z|\geq1$ so it follows that
        $$\sup_{z\in\gamma_0}|f(z)|\leq (1+10^2)^2\leq 10^5.$$
        Therefore, we have
        \begin{equation}\label{v>0 eq2}
            \int_{\gamma_0} f(z)e^{zv}dz =O(1).
        \end{equation}

        Using integration by parts twice and taking into account that $f(\pm i)=f'(\pm i)=0,$ we obtain
        $$\int_{\gamma_0}f(z)e^{zv}dz=v^{-2}\int_{\gamma_0}f''(z)e^{vz}dz.$$
        Thus,
        $$v^2\left|\int_{\gamma_0}f(z)e^{zv}dz\right| \leq 10\sup_{z\in\gamma_0}|f''(z)|.$$
        We may compute $f''(z)$ directly as:
        $$f''(z) = (4+12z)\frac1{u+z}-8z(1+z^2)\frac1{(u+z)^2}+(1+z^2)^2\frac2{(u+z)^3}.$$
        Since $|z|\leq 10$ and $|u+z|\geq1$ for every $z\in\gamma_0,$ it follows that
        \begin{align*}
            \sup_{z\in\gamma_0}|f''(z)| &\leq (4+12\cdot 10)+8\cdot 10\cdot (1+10^2)+2\cdot(1+10^2)^2\\
                                        &\leq 10^5.
        \end{align*}
        Therefore, we have:
        \begin{equation}\label{v>0 eq3}
            \int_{\gamma_0}f(z)e^{zv}dz=O(v^{-2}),
        \end{equation}
        Hence,
        \begin{equation}\label{v>0 eq4}
            \int_{\gamma_0}f(z)e^{zv}dz = \min\{1,v^{-2}\}\cdot O(1).
        \end{equation}
        
        Combining \eqref{v>0 eq1} and \eqref{v>0 eq4}, we obtain
        $$\frac1{2\pi i}\int_{\gamma_1}f(z)e^{zv}dz=(1+u^2)^2e^{-uv}+\min\{1,v^{-2}\}\cdot O(1).$$
                
        So using \eqref{wick rotation} and \eqref{v>0 eq1}:
        $$\frac1{2\pi}\int_{-1}^1\frac{(1-t^2)^2}{u+it}e^{(u+it)v}dt=(1+u^2)^2+e^{uv}\cdot\min\{1,v^{-2}\}\cdot O(1).$$
        This completes the proof of the $v \geq 0$.
        
        Now assume that $v < 0$.
        This proof is similar, but instead we consider a contour in the half plane $\{z\;:\;\Re(z)\geq 0\}$. 
        Let $\gamma_2$ be a smooth curve without self-intersections such that
        \begin{enumerate}
            \item $\gamma_2$ starts at $-i$ and ends at $i.$
            \item $\gamma_2$ lies in the half-plane $\{\Re(z)\geq0\}.$
            \item the distance between $\gamma_2$ and $[-1,0]$ is greater than or equal to $1$.
            \item the length of $\gamma_2$ is at most $10.$
            \item $\gamma_2$ is contained in the disc $\{z\;:\;|z| \leq 10\}$.
        \end{enumerate}
        As in the $v \geq 0$ case, $\gamma_1$ denotes the interval joining $-i$ and $i$, and now write $\gamma'$ for the concatenation of $\gamma_1$ and $\gamma_2$.

        Since $f$ is holomorphic in the half-plane $\Re(z) \geq 0$, we have:
        $$\int_{\gamma'}f(z)e^{zv}dz=0.$$
        Since $\gamma' = \gamma_1\cup\gamma_2$ we have:
        \begin{equation}\label{v<0 eq1}
            \frac1{2\pi i}\int_{\gamma_1}f(z)e^{zv}dz+\frac1{2\pi i}\int_{\gamma_2}f(z)e^{zv}dz=0.
        \end{equation}
        
        Since by definition $\gamma_2$ has length at most $10$, and we are assuming $v < 0$ we have:,
        \begin{align*}
            \left|\int_{\gamma_2}f(z)e^{zv}dz\right| &\leq 10\sup_{z\in\gamma_2}|f(z)e^{zv}|\\
                                                     &\leq 10\sup_{z \in \gamma_2}|f(z)|\\
                                                     &\leq 10(1+10^2)^2.\\
                                                     &=O(1).
        \end{align*}

        Using integration by parts and taking into account once again that $f(\pm i)=f'(\pm i)=0,$ we obtain in also in the $v < 0$ case that:
        $$\int_{\gamma_2}f(z)e^{zv}dz=v^{-2}\int_{\gamma_2}f''(z)e^{vz}dz.$$
        Thus,
        \begin{align*}
            v^2\left|\int_{\gamma_2}f(z)e^{zv}dz\right| &\leq 10\sup_{z\in\gamma_2}|f''(z)||e^{zv}|\\
                                                        &\leq 10^6\\
                                                        &= O(1).
        \end{align*}
        Therefore,
        \begin{equation}\label{v<0 eq2}
            \int_{\gamma_2}f(z)e^{zv}dz=\min\{1,v^{-2}\}\cdot O(1).
        \end{equation}

        Combining \eqref{v<0 eq1} and \eqref{v<0 eq2}, we obtain
        \begin{equation*}
            \frac1{2\pi i}\int_{\gamma_1}f(z)e^{zv}dz=\min\{1,v^{-2}\}\cdot O(1).
        \end{equation*}
        Hence, by \eqref{wick rotation} we conclude the proof for the $v < 0$ case.
    \end{proof}

\backmatter

\end{document}